\documentclass[12pt,oneside]{amsart}

\usepackage{geometry}                
\usepackage{graphicx}
\usepackage{amssymb}
\usepackage{color}
\usepackage{bbm, dsfont}
\usepackage[hidelinks]{hyperref}

\hypersetup{
	colorlinks=false,
	pdfborder={0 0 0},
	pdfborderstyle={/S/U/W 0},
}

\numberwithin{equation}{section}

\usepackage{verbatim}

\usepackage[dvipsnames]{xcolor}

\usepackage{ulem}
\newtheorem{theorem}{Theorem}[section]
\newtheorem{lemma}[theorem]{Lemma}
\newtheorem{corollary}[theorem]{Corollary}

\theoremstyle{definition}

\theoremstyle{remark}
\newtheorem{remark}[theorem]{Remark}

\newtheorem*{remark*}{Note}

\numberwithin{equation}{section}

\usepackage{amsmath}
\usepackage{mathrsfs}

\usepackage{verbatim}
\usepackage{amsmath,amssymb,amsthm}           
\usepackage{amssymb}            
\usepackage{amsfonts}            
\usepackage{mathrsfs}          
\usepackage{amsthm}
\usepackage{algorithm}  
\usepackage{algorithmicx}  
\usepackage{algpseudocode}  
\usepackage{mathtools}
\usepackage{commath}
\usepackage{physics}
\usepackage{bm}
\usepackage{graphicx}

\usepackage{float}
\usepackage{listings}
\usepackage{subfigure}
\usepackage{multirow}
\usepackage{color}
\usepackage{bbm}
\usepackage[export]{adjustbox}
\usepackage{enumerate}
\usepackage{bbm}

\usepackage{amsmath}
\usepackage{mathrsfs}

\newcommand{\RNum}[1]{\uppercase\expandafter{\romannumeral #1\relax}}

\newcommand{\specificthanks}[1]{\@fnsymbol{#1}}

\DeclareFontFamily{OML}{rsfs}{\skewchar\font'177}
\DeclareFontShape{OML}{rsfs}{m}{n}{ <5> <6> rsfs5 <7> <8> <9>
	rsfs7 <10> <10.95> <12> <14.4> <17.28> <20.74> <24.88> rsfs10 }{}
\DeclareMathAlphabet{\mathfs}{OML}{rsfs}{m}{n}

\newcounter{cnstcnt}
\newcommand{\cl}{%
	\refstepcounter{cnstcnt}%
	\ensuremath{c_{\thecnstcnt}}}
\newcommand{\cref}[1]{\ensuremath{c_{\ref*{#1}}}}

\newcounter{newcnstcnt}
\newcommand{\Cl}{%
	\refstepcounter{newcnstcnt}%
	\ensuremath{C_{\thenewcnstcnt}}}
\newcommand{\Cref}[1]{\ensuremath{C_{\ref*{#1}}}}

\DeclareFontFamily{U}{mathx}{}
\DeclareFontShape{U}{mathx}{m}{n}{<-> mathx10}{}
\DeclareSymbolFont{mathx}{U}{mathx}{m}{n}
\DeclareMathAccent{\widehat}{0}{mathx}{"70}
\DeclareMathAccent{\widecheck}{0}{mathx}{"71}

\begin{document}

	\title{Heterochromatic two-arm probabilities for metric graph Gaussian free fields
}


	
	
		\author{Zhenhao Cai$^1$}
		\address[Zhenhao Cai]{Faculty of Mathematics and Computer Science, Weizmann Institute of Science}
		\email{zhenhao.cai@weizmann.ac.il}
		\thanks{$^1$Faculty of Mathematics and Computer Science, Weizmann Institute of Science}

		\author{Jian Ding$^2$}
		\address[Jian Ding]{New Cornerstone Science Laboratory, School of Mathematical Sciences, Peking University}
		\email{dingjian@math.pku.edu.cn}
		\thanks{$^2$New Cornerstone Science Laboratory, School of Mathematical Sciences, Peking University}
	
	\maketitle
	%
	%
	
	 	\begin{abstract}

	 	For the Gaussian free field on the metric graph of $\mathbb{Z}^d$ ($d\ge 3$), we consider the heterochromatic two-arm probability, i.e., the probability that two points $v$ and $v'$ are contained in distinct clusters of opposite signs with diameters at least $N$. For all $d\ge 3$ except the critical dimension $d_c=6$, we prove that this probability is asymptotically proportional to $N^{-[(\frac{d}{2}+1)\land 4]}$. Furthermore, we prove that conditioned on this two-arm event, the volume growth of each involved cluster is comparable to that of a typical (unconditioned) cluster; precisely, each cluster has a volume of order $M^{(\frac{d}{2}+1)\land 4}$ within a box of size $M$.



  	 	\end{abstract}

\section{Introduction}\label{section_intro}

In this paper, we study the sign clusters of Gaussian free fields (GFF) on metric graphs, with a particular focus on understanding the interactions between them. For clarity, we start with a review of some basic definitions. We denote the edge set of the $d$-dimensional integer lattice $\mathbb{Z}^d$ by $\mathbb{L}^d:=\{\{x,y\}: x,y\in \mathbb{Z}^d,\|x-y\|=1\}$, where $\|\cdot\|$ represents the Euclidean norm. Throughout this paper, we assume that $d\ge 3$. For each $e = \{x, y\} \in \mathbb{L}^d$, consider a compact interval $I_e$ of length $d$, with endpoints identified with $x$ and $y$ respectively. The metric graph of $\mathbb{Z}^d$, denoted by $\widetilde{\mathbb{Z}}^d$, is defined as the union of all these intervals. The GFF on $\widetilde{\mathbb{Z}}^d$, which we denote by $\{\widetilde{\phi}_v\}_{v\in \widetilde{\mathbb{Z}}^d}$, was introduced in \cite{lupu2016loop} as a natural extension of the discrete GFF on $\mathbb{Z}^d$. To be precise, $\{\widetilde{\phi}_v\}_{v\in \widetilde{\mathbb{Z}}^d}$ can be constructed in the following two steps.  
\begin{enumerate}
	\item[(i)]    Let $\widetilde{\phi}_x$ for $x\in \mathbb{Z}^d$ be a discrete GFF on $\mathbb{Z}^d$, i.e., a family of mean-zero Gaussian random variables whose covariance is given by 
	\begin{equation}
		\mathbb{E}\big[ \widetilde{\phi}_{x}\widetilde{\phi}_{y} \big] = G(x,y), \ \ \forall x,y\in \mathbb{Z}^d. 
	\end{equation}
	Here $G(x,y)$ represents the Green's function on $\mathbb{Z}^d$---the expected number of times that a simple random walk on $\mathbb{Z}^d$ starting from $x$ visits $y$.

	\item[(ii)]  For any $e = \{x,y\} \in \mathbb{L}^d$, $\{\widetilde{\phi}_v\}_{v\in I_e}$ is distributed as a Brownian bridge with boundary conditions $\widetilde{\phi}_x$ at $x$ and $\widetilde{\phi}_y$ at $y$, generated by a one-dimensional Brownian motion with variance $2$ at time $1$.

\end{enumerate}
In fact, $\{\widetilde{\phi}_v\}_{v\in \widetilde{\mathbb{Z}}^d}$ can be equivalently considered as a mean-zero Gaussian field on $\widetilde{\mathbb{Z}}^d$ with covariance $\mathbb{E}\big[ \widetilde{\phi}_{v}\widetilde{\phi}_{w} \big] = \widetilde{G}(v,w)$ for all $v,w\in \widetilde{\mathbb{Z}}^d$. Here $\widetilde{G}(\cdot,\cdot )$ is the Green's function for the canonical Brownian motion on $\widetilde{\mathbb{Z}}^d$ (the precise definition will be provided in Section \ref{section3_notation}). Moreover, $\widetilde{G}(\cdot, \cdot )$ can be expressed as a linear interpolation of the Green's function on the lattice $\mathbb{Z}^d$: suppose that $v \in I_{\{x_1,x_2\}}$ and $w\in I_{\{y_1,y_2\}}$ (where $\{x_1,x_2\},\{y_1,y_2\}\in \mathbb{L}^d$), then one has
\begin{equation}\label{def_Green}
	\begin{split}
		\widetilde{G}(v,w)= d^{-2}\sum\nolimits_{ j,k \in \{1,2\} }  |v-x_{3-j}|\cdot |w-y_{3-k}| \cdot   G(x_j,y_k),
	\end{split}
\end{equation}
where $|u-u'|$ denotes the graph distance between $u$ and $u'$ on $\widetilde{\mathbb{Z}}^d$, with respect to the Lebesgue measure. In particular, $\widetilde{G}(v,w)=G(v,w)$ when $v,w\in \mathbb{Z}^d$.

Many properties of the GFF level-sets $\widetilde{E}^{\ge h}:=\{v\in \widetilde{\mathbb{Z}}^d: \widetilde{\phi}_v\ge h \}$ for $h\in \mathbb{R}$ have been established in previous works. In particular, it was proved in \cite{lupu2016loop} that $\widetilde{E}^{\ge h}$ percolates (i.e., contains an infinite connected component) if and only if $h<0$. As in most percolation models, the connecting probabilities at the critical level $\widetilde{h}_*=0$ have attracted major research interest. To be precise, for any disjoint subsets $A_1,A_2\subset \widetilde{\mathbb{Z}}^d$, we denote by $A_1\xleftrightarrow{\ge 0} A_2$ the event that there exists a path in $\widetilde{E}^{\ge 0}$ connecting $A_1$ and $A_2$ (we will omit the braces when $A_i$ is a singleton $\{v\}$ with $v\in \widetilde{\mathbb{Z}}^d$). \cite[Proposition 5.2]{lupu2016loop} showed that the two-point function satisfies
 \begin{equation}\label{two-point1}
 	\mathbb{P}\big(v\xleftrightarrow{\ge 0} w  \big)= \pi^{-1}\arcsin\Big( \tfrac{\widetilde{G}(v,w)}{\sqrt{\widetilde{G}(v,v)\widetilde{G}(w,w)}} \Big)\asymp (|v-w|+1)^{2-d} 
 \end{equation}
 for all $v,w\in \widetilde{\mathbb{Z}}^d$. Here $f\asymp g$ means that for the functions $f$ and $g$, there exist constants $C>c>0$ depending only on $d$ such that $cg\le f\le Cg$. Based on this result, the one-arm probability $\theta_d(N):=\mathbb{P}(\bm{0}\xleftrightarrow{\ge 0} \partial B(N))$ has been estimated rigorously in a series of works \cite{ding2020percolation, drewitz2023critical, cai2024high, drewitz2023arm, drewitz2024critical, cai2024one}, where $\bm{0}:=(0,0,...,0)$ denotes the origin of $\mathbb{Z}^d$, $B(N):=[-N,N]^d\cap \mathbb{Z}^d$ is the box of side length $\lfloor 2N \rfloor$ in $\mathbb{Z}^d$ centered at $\bm{0}$, and $\partial A:=\{x\in A: \exists y\in \mathbb{Z}^d\setminus A\ \text{such that}\ \{x,y\}\in \mathbb{L}^d\}$ represents the boundary of $A\subset \mathbb{Z}^d$. Specifically, it has been proved that 
 \begin{align}
	&\text{when}\ 3\le d<6,\ \ \ \ \ \ \ \ \  \theta_d(N) \asymp N^{-\frac{d}{2}+1};\label{one_arm_low} \\
	&\text{when}\ d=6,\ \ \ \ \  N^{-2}\lesssim \theta_6(N) \lesssim N^{-2+\varsigma(N)}, \  \text{where}\ \varsigma(N):= \tfrac{\ln\ln(N)}{[\ln(N)]^{1/2}}\ll 1; \label{one_arm_6} \\
	&\text{when}\ d>6,\ \ \ \ \ \ \ \ \ \ \ \ \ \ \ \theta_d(N) \asymp N^{-2}.\label{one_arm_high}
\end{align}
 Here $f\lesssim g$ means that $f\le Cg$ for some constant $C>0$ depending only on $d$. The exact order of $\theta_6(N)$ remains open, and it has been conjectured to be $N^{-2}[\ln(N)]^\delta$ for some $\delta>0$ (see \cite[Remark 1.5]{cai2024one}).

 (\textbf{P.S.} Due to the diverging disparity between the current upper and lower bounds on $\theta_6(N)$, the analysis at the critical dimension $d_c=6$ involves additional complexities compared to other dimensions. Nevertheless, the arguments for $d\neq 6$ can usually be extended to $d=6$, albeit yielding weaker estimates with error terms similar to that in (\ref{one_arm_6}). Accordingly, we assume $d \neq 6$ throughout the remainder of this paper, unless stated otherwise.)

  As a natural extension of the one-arm probability $\theta_d(N)$, the crossing probability $\rho_d(n,N):=\mathbb{P}(B(n)\xleftrightarrow{\ge 0}\partial B(N))$ (where $N> n\ge 1$) was proved in \cite{cai2024one} to satisfy
   \begin{align}
	&\text{when}\ 3\le d<6,\  \rho_d(n,N) \asymp \big(n/N\big)^{\frac{d}{2}-1};\label{crossing_low} \\
	&\text{when}\ d>6, \ \ \ \ \   \ \    \rho_d(n,N) \asymp  (n^{d-4}N^{-2})\land 1. \label{crossing_high}
\end{align}
    For the connecting probabilities between general sets, numerous regularity properties have been established in \cite{cai2024quasi}. Among these, the property known as quasi-multiplicativity played a crucial role in \cite{cai2024incipient}, which constructed four types of incipient infinite clusters (IIC), including the initial version introduced by \cite{kesten1986incipient}---the positive cluster containing $\bm{0}$ under the limiting measure 
    \begin{equation}
    	\mathbb{P}_{\mathrm{IIC}}(\cdot):=\lim_{N\to \infty}\mathbb{P}\big(\cdot \mid \bm{0}\xleftrightarrow{\ge 0} \partial B(N) \big),
    \end{equation}
    and established the equivalence between them. See also \cite{ganguly2024ant, ganguly2024critical} for properties of the random walk and the chemical distance on the IIC in the high-dimensional case (i.e., $d\ge 7$). Recently, a powerful switching identity was introduced in \cite{werner2025switching}, which, among many applications, provides an elegant and precise description of the IIC. Furthermore, building on the analysis of connecting probabilities, the volume growth of clusters in the critical level-set $\widetilde{E}^{\ge 0}$ has also been extensively studied. To be precise, for any $v\in \widetilde{\mathbb{Z}}^d$, we denote by $\mathcal{C}^{+}_v:=\{w\in \widetilde{\mathbb{Z}}^d: w\xleftrightarrow{\ge 0 }v \}$ the cluster of $\widetilde{E}^{\ge 0}$ containing $v$. The decay rate of the cluster volume is defined as $\nu_d(M):=\mathbb{P}\big(\mathrm{vol}(\mathcal{C}^{+}_{\bm{0}})\ge M \big)$ for $M\ge 1$, where $\mathrm{vol}(A)$ (for $A\subset \widetilde{\mathbb{Z}}^d$) denotes the cardinality of $A\cap \mathbb{Z}^d$. Through the collective efforts in \cite{cai2024high, cai2024quasi, drewitz2024cluster}, it has been established that 
       \begin{align}
	&\text{when}\ 3\le d<6,\  \nu_d(M) \asymp M^{-\frac{d-2}{d+2}};\label{volume_low} \\
	&\text{when}\ d>6, \ \ \ \ \   \ \    \nu_d(M) \asymp  M^{-\frac{1}{2}}. \label{volume_high}
\end{align}
In addition, the volume growth rate in large clusters is also well understood. Formally, for any $x\in \widetilde{\mathbb{Z}}^d$ and $M\ge 1$, let $\mathcal{V}^{+}_x(M):=\mathrm{vol}(\mathcal{C}^{+}_{x}\cap B_x(M))$ denote the volume of the cluster $\mathcal{C}^{+}_x$ inside the box $B_x(M):=x+B(M)$. It was proved in \cite[Theorem 1.2]{cai2024incipient} that for some constant $C(d)>0$, the infimum 
\begin{equation}\label{result_typical_volume}
	 \inf_{M\ge 1,N\ge CM} \mathbb{P}\big( \lambda^{-1}M^{(\frac{d}{2}+1)\boxdot 4} \le  \mathcal{V}^{+}_{\bm{0}}(M) \le \lambda M^{(\frac{d}{2}+1)\boxdot 4}  \mid \bm{0} \xleftrightarrow{\ge 0} \partial B(N) \big)
\end{equation} 
converges to $1$ as $\lambda \to \infty$. Here we employ ``$\boxdot$'' to denote the operation $f\boxdot g(d):= f(d)\cdot \mathbbm{1}_{d\le 6}+ g(d)\cdot \mathbbm{1}_{d>6}$ for any functions $f$ and $g$ that depend on $d$. Note that the property in (\ref{result_typical_volume}) automatically applies to the IIC. Similar results in terms of the largest cluster in a given box were obtained in \cite{drewitz2024cluster}.

  As mentioned above, all existing results on the sign clusters of the GFF $\widetilde{\phi}_{\cdot}$ (i.e., maximal subgraphs where the GFF values have the same sign) focus on properties of a single cluster. As a complement, in this paper we aim to study the coexistence of two distinct large sign clusters originating from two given points. To this end, the following two questions arise naturally: \begin{enumerate}

 	\item[\textbf{Q1}:] How likely is this coexistence?

 	\item[\textbf{Q2}:] How does this coexistence alter the geometric properties of the two involved clusters? Do they still behave similarly to a single cluster?

 \end{enumerate} 
Next, we explore these two questions in Sections \ref{subsetion_two_arm} and \ref{subsection_volume} respectively.

\subsection{Heterochromatic connecting probabilities}\label{subsetion_two_arm}

 We formulate \textbf{Q1} by introducing the notion of heterochromatic two-arm probability. Precisely, for any subsets $\{A_i\}_{1\le i\le 4}$ of $\widetilde{\mathbb{Z}}^d$, we define the event 
 \begin{equation}\label{113}
 	\mathsf{H}^{A_1,A_2}_{A_3,A_4}:= \big\{A_1\xleftrightarrow{\ge 0} A_2,A_3\xleftrightarrow{\le 0} A_4  \big\}, 
 \end{equation} 
 where $A_3\xleftrightarrow{\le 0} A_4$ represents the event that a certain path in the negative cluster $\widetilde{E}^{\le 0}:=\{v\in \widetilde{\mathbb{Z}}^d: \widetilde{\phi}_v\le 0 \}$ connects $A_3$ and $A_4$. Here heterochromaticity is imposed to capture the distinction between the two involved sign clusters. In particular, for any $N\ge 1$ and $v,v'\in \widetilde{B}(N)$ (where $\widetilde{B}(N):=\cup_{e\in \mathbb{L}^d:I_e \cap (-N,N)^d\neq \emptyset}I_e$ denotes the box in $\widetilde{\mathbb{Z}}^d$ of radius $N$ centered at $\bm{0}$), we define the heterochromatic two-arm event 
\begin{equation}
	\mathsf{H}_{v'}^{v}(N):= \mathsf{H}^{v,\partial B(N)}_{v',\partial B(N)}.
\end{equation}


  Our first result establishes the decay rate of $\mathbb{P}\big(\mathsf{H}_{v'}^{v}(N) \big)$ as $N\to \infty$, and its exact orders in different regimes, including the case when $|v-v'|$ is large, so that the macroscopic graph structure of $\widetilde{\mathbb{Z}}^d$ plays a role, and the case when $|v-v'|$ is small, with stronger influence from the local, one-dimensional geometry of $\widetilde{\mathbb{Z}}^d$.

  
 \begin{theorem}\label{thm1}
	For any $d\ge 3$ with $d\neq 6$, there exist constants $\Cl\label{const_thm1_1}(d), \cl\label{const_thm1_2}(d)>0$ such that the following hold for all $N\ge \Cref{const_thm1_1}$ and $v,v'\in \widetilde{B}(\cref{const_thm1_2}N)$: 
	 	 \begin{align}
	 	 &\text{when}\ \chi=:|v-v'|\le 1, \  \mathbb{P}\big(\mathsf{H}_{v'}^{v}(N) \big)\asymp \chi^{\frac{3}{2}}N^{-[(\frac{d}{2}+1)\boxdot 4]};   \label{thm1_small_n}\\
	&\text{when}\ \chi\ge 1,\   \ \ \ \  \  \ \ \ \ \ \ \   \   	\mathbb{P}\big(\mathsf{H}_{v'}^{v}(N) \big)\asymp \chi^{(3-\frac{d}{2})\boxdot 0}N^{-[(\frac{d}{2}+1)\boxdot 4]}.   \label{thm1_large_n} 
	\end{align}
\end{theorem}

  	The isomorphism theorem \cite[Theorem 1.1]{lupu2016loop} shows that the sign clusters of the GFF $\widetilde{\phi}_\cdot$ have the same distribution as the loop clusters of the loop soup $\widetilde{\mathcal{L}}_{1/2}$. Here $\widetilde{\mathcal{L}}_{1/2}$ is a Poisson point process consisting of loops---paths that start and end at the same point (the precise definition of $\widetilde{\mathcal{L}}_{1/2}$ will be provided in Section \ref{section3_notation}), and the term ``loop clusters'' refers to the connected components formed by loops in $\widetilde{\mathcal{L}}_{1/2}$. This isomorphism theorem allows us to switch between the two models, which greatly facilitates the proofs. Moreover, all estimates in Theorem \ref{thm1} also apply to their analogues for the loop clusters of $\widetilde{\mathcal{L}}_{1/2}$.

 \begin{remark}[distance between macroscopic clusters]\label{remark1.2}
 	 Theorem \ref{thm1} implies that the expected number of edges within a box of size $N$, whose endpoints are contained in distinct sign clusters with diameter at least $N$, is of order $N^{(\frac{d}{2}-1)\boxdot (d-4)}$. This suggests that two macroscopic sign clusters may nearly touch at many locations, resulting in a vanishing graph distance between them. In the companion paper \cite{inpreparation_second_monent}, we confirm this property through the second moment method. Notably, this contrasts sharply with the two-dimensional case, where the typical distance between macroscopic clusters is comparable to their diameters (see \cite{lupu2019convergence, sheffield2012conformal}). 
 \end{remark}

 \begin{remark}[pivotal edge]

 The phenomenon noted in Remark \ref{remark1.2} suggests that in a sign cluster connecting two distant sets, there are numerous pivotal edges---those whose removal breaks the connectivity between the sets. From the perspective of the loop cluster, this property indicates that removing all loops of length $1$ might significantly reduce the connecting probability. In a companion paper \cite{inpreparation}, we verify this reduction through a multi-scale analysis, and establish that the dimension of clusters in a Brownian loop soup of intensity $\frac{1}{2}$ on $\mathbb{R}^3$ is strictly less than that of loop clusters on $\widetilde{\mathbb{Z}}^3$. This shows that already in dimension $3$, microscopic randomness (on top of the macroscopic loops) is involved in the construction of loop soup cluster, which corrects the scenario proposed in \cite[Section 3]{werner2021clusters}.

 
  \end{remark}

 As a byproduct of proving Theorem \ref{thm1}, we also obtain the exact order of the heterochromatic four-point function, which is interesting in its own right.

  \begin{theorem}\label{theorem_four_point}
  	For any $d\ge 3$ with $d\neq 6$, there exist constants $\Cl\label{const_four_point_1}(d),\cl\label{const_four_point_2}(d)>0$ such that  the following hold for all $N\ge \Cref{const_four_point_1}$, $v,v'\in \widetilde{B}(\cref{const_four_point_2}N)$, and $w,w'\in \widetilde{B}(10N)\setminus \widetilde{B}(N)$ with $|w-w'|\ge N$: 
  	 	 \begin{align}
	 	 &\text{when}\ \chi=:|v-v'|\le 1, \  \mathbb{P}\big(\mathsf{H}^{v,w}_{v',w'} \big)\asymp \chi^{\frac{3}{2}}N^{-[(\frac{3d}{2}-1)\boxdot (2d-4)]};   \label{thm1_small_n_four_point}\\
	&\text{when}\ \chi \ge 1,\   \ \ \ \  \  \ \ \ \ \ \ \   \   	\mathbb{P}\big(\mathsf{H}^{v,w}_{v',w'} \big)\asymp \chi^{(3-\frac{d}{2})\boxdot 0}N^{-[(\frac{3d}{2}-1)\boxdot (2d-4)]}.   \label{thm1_large_n_four_point} 
	\end{align}
  \end{theorem}

 \begin{remark}[Extension to the critical dimension]\label{remark_two_arm_for_6d}
	As noted below (\ref{one_arm_high}), the proofs of Theorems \ref{thm1} and \ref{theorem_four_point} in Sections \ref{section5_four_point} and \ref{section_estimate_coexistence} remain valid at the critical dimension $d_c=6$, and in particular yield the exponents of the heterochromatic two-arm probability and four-point function for $d=6$. Precisely, when $d=6$, there exists a constant $C>0$ such that under the same conditions as in Theorems \ref{thm1} and \ref{theorem_four_point} (recall that $\varsigma(N) = \tfrac{\ln\ln(N)}{[\ln(N)]^{1/2}}$),    
	\begin{equation}
  (\chi \land 1)^{\frac{3}{2}}N^{-4-C \cdot \varsigma(N) }
 	\lesssim   	\mathbb{P}\big(\mathsf{H}_{v'}^{v}(N) \big) \lesssim   (\chi\land 1)^{\frac{3}{2}}N^{-4+C \cdot \varsigma(N) }, 
  	\end{equation}
	\begin{equation}
 	(\chi \land 1)^{\frac{3}{2}}N^{-8-C \cdot \varsigma(N) }
 	\lesssim 	\mathbb{P}\big(\mathsf{H}^{v,w}_{v',w'} \big) \lesssim  (\chi \land 1)^{\frac{3}{2}}N^{-8+C \cdot \varsigma(N) }.
 	\end{equation}
	 
\end{remark}

 	 \subsection{Volume growth under the coexistence}\label{subsection_volume}
On the heterochromatic two-arm event $\mathsf{H}_{v'}^{v}(N)$, the magnitude of the positive cluster $\mathcal{C}^{+}_v$ imposes a constraint on the growth of the negative cluster $\mathcal{C}^{-}_{v'}:=\{w\in \widetilde{\mathbb{Z}}^d: w \xleftrightarrow{\le 0} v'\}$. To study the interaction described in \textbf{Q2}, we aim to analyze how this constraint influences the geometric properties of the two involved clusters, particularly their volume growth. According to (\ref{result_typical_volume}), conditioned on a point being contained in a large cluster, its volume within the $M$-neighborhood of this point is typically of order $M^{(\frac{d}{2}+1)\boxdot 4}$. Naturally, one would ask whether conditioning on $v$ and $v'$ being contained in two distinct large clusters will impose a constraint that alters their inherent volume growth rates. At the first glance, it may seem plausible that the volume growth will be affected when $\mathsf{H}_{v'}^{v}(N)\ll [\theta_d(N)]^2$ (which holds except when $d>6$ or $|v-v'|\asymp N$), as it indicates that the constraint between the two clusters is significant (note that the volume of a macroscopic cluster has fluctuation of the same order as its typical value, as proved by \cite[Theorem 1.3]{cai2024incipient}; this implies that a prior, a significant constraint might change the volume significantly). Interestingly, our next result shows that on the heterochromatic two-arm event, the two involved clusters actually maintain their inherent volume growth rates.

Recall $\mathcal{V}_x^{+}(M)$ above (\ref{result_typical_volume}). For any $v\in \widetilde{\mathbb{Z}}^d$, let $\bar{v}$ denote the lattice point closest to $v$ (with ties broken in a predetermined manner), and define $\mathcal{V}_v^{+}(M):=\mathcal{V}_{\bar{v}}^{+}(M)$. We denote by $\mathcal{V}_v^{-}(M)$ the analogue of $\mathcal{V}_v^{+}(M)$ obtained by replacing $\widetilde{E}^{\ge 0}$ with $\widetilde{E}^{\le 0}$.

 	 \begin{theorem}\label{thm3_volume}
 	 Assume that $N,v,v'$ satisfy the conditions in Theorem \ref{thm1}. There exist constants $\cl\label{const_thm3_volume_1}(d),\cl\label{const_thm3_volume_2}(d)\in (0,1)$ and $\Cl\label{const_thm3_volume_3}(d)>1$ such that for any $1\le M\le \cref{const_thm3_volume_1}N$, 
 	 	\begin{equation}\label{ineq_thm3}
 	 	\mathbb{P}\big(  \mathcal{V}_v^{+}(M), \mathcal{V}_{v'}^{-}(M) \in [\Cref{const_thm3_volume_3}^{-1}M^{(\frac{d}{2}+1)\boxdot 4},\Cref{const_thm3_volume_3}M^{(\frac{d}{2}+1)\boxdot 4}]   \mid \mathsf{H}_{v'}^{v}(N) \big) \ge \cref{const_thm3_volume_2}. 
 	 	\end{equation}  
 	 \end{theorem}

 	  	We expect that the right-hand side of (\ref{ineq_thm3}) can be enhanced to a number arbitrarily close to $1$ (as in (\ref{result_typical_volume})), although the techniques developed in this paper are not yet sufficient to achieve this. More precisely, as shown in \cite{cai2024incipient}, the proof of (\ref{result_typical_volume}) relies on a multi-scale argument. The key idea is that the volume growth of a cluster within an increasing sequence of annuli at scale $M$ is asymptotically independent across annuli, while each annulus potentially adds a term of order $M^{(\frac{d}{2}+1)\boxdot 4}$ to the total volume. In fact, this asymptotic independence is built on the property that  large crossing loops (i.e., loops that traverse an annulus with a big ratio of outer to inner diameters) are rare in a loop cluster. Intuitively, this property should persist under the co-existence of two clusters, although the correlation between clusters will significantly complicate the adaptation of the approach in \cite{cai2024incipient}. Nevertheless, we believe that the decomposition argument of loop clusters tailored for two-arm events in the companion paper \cite{inpreparation_second_monent} should be able to resolve this difficulty, albeit requiring additional technical efforts.  
 	  	 

 	 {\color{blue}

  	 }

\textbf{Statements about constants.} We use the letters $C$ and $c$ to denote constants, which may vary from line to line and depend on the context. Moreover, constants labeled with numerical subscripts (such as $C_1,C_2, c_1, c_2, \cdots$) remain fixed. The letter $C$ (with or without additional indices) is reserved for large constants, while $c$ is used for small constants. Unless stated otherwise, constants are assumed to depend only on the dimension $d$. If a constant depends on other parameters, we will specify these parameters in parentheses.

\textbf{Organization of this paper.} In Section \ref{section_heuristic}, we present a heuristic for Theorem \ref{thm1}. In Section \ref{section3_notation}, we review the relevant notations and preliminary results. In Section \ref{section_regularity_co_exist}, we establish a series of regularity properties of connecting probabilities involving two clusters. Subsequently, in Section \ref{section5_four_point} we derive the upper bound in (\ref{thm1_small_n_four_point}) and the lower bound in (\ref{thm1_large_n_four_point}). In Section \ref{section_estimate_coexistence}, we prove Theorem \ref{thm1} and the remaining bounds in Theorem \ref{theorem_four_point}. Finally, we provide the proof of Theorem \ref{thm3_volume} in Section \ref{section_clusters_volume}.

 \section{Heuristic for the heterochromatic two-arm exponent}\label{section_heuristic}

In this section, we provide some intuitions for the exponents in Theorems \ref{thm1} and \ref{theorem_four_point}. We take a large number $N$, and two points $x_-:= (-N,0,...,0)$ and $x_+:= (N,0,...,0)$ in $\mathbb{Z}^d$. Suppose that $x_-$ is contained in a sign cluster $\mathcal{C}$ with diameter $cN$, which occurs with probability $\theta_d(cN)\asymp N^{-[(\frac{d}{2}-1)\boxdot 2]}$. The strong Markov property of the GFF implies that conditioned on the configuration of $\mathcal{C}$, $\{\widetilde{\phi}_v\}_{v\in \widetilde{\mathbb{Z}}^d\setminus \mathcal{C}}$ is distributed as an independent metric graph GFF with zero boundary condition on $\mathcal{C}$. Therefore, for each edge $\{y,y'\}\in \mathbb{L}^d$ such that $y\in \mathcal{C}$ and $y'\notin \mathcal{C}$ (for convenience, we assume that $y$ is exactly on the boundary on $\mathcal{C}$), according to the two-point function estimate for general graphs (see (\ref{29})), given the cluster $\mathcal{C}$, the conditional probability of $y'$ and $x_+$ being connected by a sign cluster is proportional to the Green's function at $(y',x_+)$ with an absorbing boundary on $\mathcal{C}$, denoted by $\widetilde{G}_{\mathcal{C}}(y',x_+)$ (see (\ref{def_green_function})). Moreover, using the strong Markov property of random walks and the transience of $\mathbb{Z}^d$ for $d\ge 3$, it follows that $\widetilde{G}_{\mathcal{C}}(y',x_+)$ is of the same order as $N^{2-d}$ times the probability that a simple random walk starting from $y$ never returns to the cluster $\mathcal{C}$, denoted by $\mathrm{Es}_{\mathcal{C}}(y')$. In conclusion, the total number $\mathbf{M}$ of edges $\{y,y'\}$ such that $y$ and $y'$ are connected to $x_-$ and $x_+$ respectively by distinct sign clusters typically satisfies 
\begin{equation}
	\begin{split}
		  \mathbf{M} \asymp \theta_d(cN) \cdot N^{2-d}\cdot \sum\nolimits_{\{y,y'\}\in \mathbb{L}^d:y\in \mathcal{C},y'\notin \mathcal{C}} \mathrm{Es}_{\mathcal{C}}(y). 
		  	\end{split} 
\end{equation} 
It is well-known that the sum on the right-hand side equals the capacity of $\mathcal{C}$, denoted by $\mathrm{cap}(\mathcal{C})$ (see e.g., \cite[Section 6.5]{lawler2010random}). Moreover, the decay rate of $\mathrm{cap}(\mathcal{C})$ has been computed in \cite{ding2020percolation, drewitz2023critical}: there exists $C>0$ such that for any $T>0$, 
\begin{equation}\label{asymp_cap_cluster}
	\mathbb{P}\big(\mathrm{cap}(\mathcal{C}) \ge T \big) \sim C T^{-\frac{1}{2}}.
\end{equation}
Here ``$f(T) \sim g(T)$'' means that $\lim\limits_{T\to \infty}f(T)/g(T)=1$. Combined with the estimates (\ref{one_arm_low}) and (\ref{one_arm_high}) on one-arm probabilities, it shows that $\mathrm{cap}(\mathcal{C})$ is typically of order $N^{(d-2)\boxdot 4}$ (note that $\mathbb{P} (\mathrm{cap}(\mathcal{C}) \ge N^{(d-2)\boxdot 4}  )\asymp \theta_d(cN)$). To sum up, 
\begin{equation}\label{final2.2}
	 \mathbf{M} \asymp \theta_d(cN)  \cdot N^{2-d}\cdot N^{(d-2)\boxdot 4}= N^{-[(\frac{d}{2}-1)\boxdot (d-4)]}. 
\end{equation}

{\color{blue}

}


By Fubini's theorem, the expectation of $\mathbf{M}$ equals four times the sum of heterochromatic four-point functions $\mathbb{P}(\mathsf{H}^{y,x_+}_{y',x_-})$ over all edges $\{y,y'\} \in \mathbb{L}^d$ (where the factor of four arises because the sign of each sign cluster is chosen uniformly at random). To relate this sum to the probability $\mathbb{P}(\mathsf{H}^{v_+,x_+}_{v_-,x_-})$ of interest (where $v_+:=(1,0,...,0)$ and $v_-:=\bm{0}$), we establish the following regularity (see Lemma \ref{lemma_roots}): for any edges $\{y_1,y_1'\}$ and $\{y_2,y_2'\}$ (except those are too close to $x_+$ or $x_-$), $\mathbb{P}(\mathsf{H}^{y_1,x_+}_{y'_1,x_-})$ and $\mathbb{P}(\mathsf{H}^{y_2,x_+}_{y_2',x_-})$ are comparable. These two properties together yield  
\begin{equation}\label{final2.3}
	\mathbb{E}[\mathbf{M}]\asymp N^{d}\cdot \mathbb{P}(\mathsf{H}^{v_+,x_+}_{v_-,x_-}). 
\end{equation}
 Combining (\ref{final2.2}) and (\ref{final2.3}), we arrive at 
    \begin{equation}
    	\mathbb{P}(\mathsf{H}^{v_+,x_+}_{v_-,x_-})  \asymp N^{-[(\frac{3d}{2}-1)\boxdot (2d-4)]},
    \end{equation}
 which explains the exponent of $N$ in Theorem \ref{theorem_four_point}. In fact, this estimate indicates that to accommodate two \textit{distinct} sign clusters with diameter $N$ originating from a pair of neighboring points, the energy cost is 
 \begin{equation} 
 \xi_N:=	\frac{\mathbb{P}(\mathsf{H}^{v_+,x_+}_{v_-,x_-})}{\mathbb{P}\big(v_+\xleftrightarrow{\ge 0} x_+ \big)\cdot \mathbb{P}\big(v_-\xleftrightarrow{\le 0} x_-\big)}\asymp  N^{(\frac{d}{2}-3)\boxdot 0}. 
 \end{equation}
 Combined with the estimates for $\theta_d(\cdot)$ in (\ref{one_arm_low}) and (\ref{one_arm_high}), it suggests 
 \begin{equation}\label{final2.6}
 	\mathbb{P}\big(\mathsf{H}^{v_+}_{v_-}(N)\big)\asymp  \xi_N\cdot [\theta_d(N)]^2 \asymp N^{-[(\frac{d}{2}+1)\boxdot 4]},
 \end{equation} 
which matches the exponent of $N$ in Theorem \ref{thm1}.

Next, we discuss the exponent of $\chi$ in the cases $\chi \ll 1$ and $\chi\gg 1$ separately.

\textbf{When $\chi \ll 1$.} Arbitrarily take $v,v'\in I_{\{v_-,v_+\}}$ satisfying $|v-v'|=\chi$. For clarity, we assume that $v$ is closer to $v_+$ than $v'$, and that neither $v$ nor $v'$ is too close to the endpoints $\{v_-,v_+\}$, i.e., $|v-v_{\diamond}|\land |v'-v_{\diamond}| \ge 1$ for $\diamond\in \{+,-\}$. With (\ref{final2.6}) at hand, deriving the exponent of $\chi$ in (\ref{thm1_small_n}) reduces to proving that 
 \begin{equation}
 	\begin{split}
 		\mathbb{P}\big( \mathsf{H}^{v}_{v'}(N)  \mid \mathsf{H}^{v_+}_{v_-}(N)  \big) \asymp  \chi^{\frac{3}{2}}. 
 	\end{split}
 \end{equation}
 This bound follows from the following three observations: 
\begin{enumerate}

	\item[(i)] Conditioned on the event $\mathsf{H}^{v_+}_{v_-}(N)$, the GFF values at $v_+$ and $v_-$	are typically of order $1$.

	\item[(ii)] Given the values of $\widetilde{\phi}_{v_+}$ and $\widetilde{\phi}_{v_-}$, which are of order $1$ and satisfy $\widetilde{\phi}_{v_+}>0$ and $\widetilde{\phi}_{v_-}<0$, the conditional probability of $\{v \xleftrightarrow{\ge 0} v_+\}\cap \{v\xleftrightarrow{\ge 0} v' \}^c$ is proportional to $\chi$.

	\item[(iii)] Under the same conditioning as in Item (ii) and given the positive cluster of $v_+$ (which reaches $v$ and is separated from $v'$ by distance $c\chi $), the conditional probability of $\{v' \xleftrightarrow{\le 0} v_-\}$ is of order $\chi^{\frac{1}{2}}$.

\end{enumerate}
 Next, we explain Observations (i)-(iii) separately. For Observation (i), we recall from \cite[Lemma 3.3]{cai2024one} that conditioned on a point being contained in a large sign cluster, the GFF value of this point is typically of order $1$. We prove a regularity property (see Lemma \ref{lemma_separation}) that conditioned on the existence of two distinct large clusters, one containing $v_+$ and the other containing $v_-$, these clusters are well separated in neighborhoods of $v_+$ and $v_-$. This property allows us to extend \cite[Lemma 3.3]{cai2024one} to the two-arm event, and hence yields Observation (i). For Observation (ii), note that given $\{\widetilde{\phi}_v=a\}$ (where $a>0$ is sufficiently small), the diameter of the positive cluster containing $v$ is typically of order $a^2$ (here the square arises from the Brownian scaling relation, since the GFF restricted to an interval behaves as a Brownian bridge). Hence, the occurrence of $\{v \xleftrightarrow{\ge 0} v' \}^c$ imposes the constraint $0<\widetilde{\phi}_v\lesssim |v-v'|^{\frac{1}{2}}= \chi^{\frac{1}{2}}$, which occurs with probability $\asymp \chi^{\frac{1}{2}}$. Moreover, given that $0<\widetilde{\phi}_v\lesssim  \chi^{\frac{1}{2}}$ and $\widetilde{\phi}_{v_+}\asymp 1$, the probability of $\{v \xleftrightarrow{\ge 0} v_+ \}$ is proportional to $\chi^{\frac{1}{2}}$ (see (\ref{formula_two_point})). Combining these two estimates, we obtain Observation (ii). For Observation (iii), under zero boundary condition imposed by the cluster $\mathcal{C}^+_{v_+}$ (whose distance to $v'$ is $c\chi$), the variance of $\widetilde{\phi}_{v'}$ is of order $\chi$ (see (\ref{23})). In other words, $|\widetilde{\phi}_{v'}|$ typically takes values of order $\chi^{\frac{1}{2}}$. This together with $|\widetilde{\phi}_{v_-}|\asymp 1$ (as ensured by Observation (i)) and the formula (\ref{formula_two_point}) yields that the conditional probability of $\{v' \xleftrightarrow{\le 0} v_-\}$ is of order $\chi^{\frac{1}{2}}$.

     The exponent of $\chi$ in (\ref{thm1_small_n_four_point}) arises from the same mechanism, so we skip it.

\textbf{When $\chi \gg 1$.} As noted below \cite[Theorem 1.2]{cai2024incipient}, the sign cluster is expected to exhibit self-similarity, meaning that its geometric properties remain invariant under scaling transformations. This indicates that the energy cost $\hat{\xi}_{\chi,N}$ of accommodating two clusters that cross the annulus $B(N)\setminus B(\chi)$ (where $N\gg \chi$) is a polynomial of $N/\chi$, and has the same exponent as $\xi_N$. I.e., 
\begin{equation}
	\hat{\xi}_{\chi,N} \asymp \big(N/\chi\big)^{(\frac{d}{2}-3)\boxdot 0}. 
\end{equation}
 For the same reasoning as in (\ref{final2.6}), it yields that for $v,v'\in \widetilde{B}(cN)$ with $|v-v'|=\chi$,
\begin{equation}
	\mathbb{P}\big(\mathsf{H}_{v'}^{v}(N) \big)\asymp \hat{\xi}_{\chi,N}\cdot   [\theta_d(N)]^2 \asymp \chi^{(3-\frac{d}{2})\boxdot 0} N^{-[(\frac{d}{2}+1)\boxdot 4]}, 
\end{equation}
which coincides with (\ref{thm1_large_n}) in Theorem \ref{thm1}. For Theorem \ref{theorem_four_point}, the bound (\ref{thm1_large_n_four_point}) can be derived similarly as follows: for any $w,w'\in \widetilde{B}(10N)\setminus \widetilde{B}(N)$ with $|w-w'|\ge N$,
 \begin{equation}
 	\mathbb{P}\big(\mathsf{H}_{v',w'}^{v,w}  \big) \asymp \hat{\xi}_{\chi,N}\cdot \mathbb{P}\big(v \xleftrightarrow{\ge 0} w  \big)\cdot \mathbb{P}\big(v'\xleftrightarrow{\le 0} w'\big) \asymp  \chi^{(3-\frac{d}{2})\boxdot 0}N^{-[(\frac{3d}{2}-1)\boxdot (2d-4)]}.  
 \end{equation}

 \section{Preliminaries}\label{section3_notation}

In this section, we review some basic definitions and useful results.

 \subsection{Brownian motion on $\widetilde{\mathbb{Z}}^d$}

We recall the definition of the  canonical Brownian motion on $\widetilde{\mathbb{Z}}^d$ (denoted by $\{\widetilde{S}_t\}_{t\ge 0}$) as follows. Precisely, within each edge interval $I_e$ ($e\in \mathbb{L}^d$), the process $\widetilde{S}_\cdot$ behaves as a standard one-dimensional Brownian motion. Upon reaching a lattice point $x\in \mathbb{Z}^d$, it uniformly chooses an edge $e=\{x,y\}$ incident to $x$ and performs a Brownian excursion along the interval $I_e$ away from $x$. When the excursion hits a lattice point (either returning to $x$ or reaching $y$), the process resumes from that lattice point, and continues in the same manner. We denote the transition density of this process by $\widetilde{q}_t(v,w)$ for $t>0$ and $v,w\in \widetilde{\mathbb{Z}}^d$. When $\widetilde{S}_\cdot$ starts from $v\in \widetilde{\mathbb{Z}}^d$, we denote its law by $\widetilde{\mathbb{P}}_v$. Let $\widetilde{\mathbb{E}}_v$ represent the expectation under $\widetilde{\mathbb{P}}_v$.

   For any $D\subset \widetilde{\mathbb{Z}}^d$, we denote the hitting time $\tau_D:=\inf\{t\ge 0:\widetilde{S}_t\in D \}$. For completeness, we set $\tau_\emptyset=\inf \emptyset 
   := +\infty$. When $D=\{v\}$ for some $v\in \widetilde{\mathbb{Z}}^d$, we omit the braces and write $\tau_{v}:=\tau_{\{v\}}$. Moreover, we may also use $\tau_D(\widetilde{\eta})$ to denote the hitting time for a given path $\widetilde{\eta}$ on $\widetilde{\mathbb{Z}}^d$.

 (\textbf{P.S.} Unless otherwise specified, when we use the letter $D$ to denote a subset of $\widetilde{\mathbb{Z}}^d$, we assume that it consists of finitely many compact connected components. Since $\widetilde{\mathbb{Z}}^d$ is locally one-dimensional, the boundary of any $D\subset \widetilde{\mathbb{Z}}^d$, defined by $\widetilde{\partial}D:=\{v\in D: \inf_{w\in \widetilde{\mathbb{Z}}^d\setminus D}|v-w|=0 \}$, is a countable set.)

   Similar to \cite[Lemma 2.1]{cai2024quasi}, when $v,v',w\in \widetilde{\mathbb{Z}}^d$ satisfy $|v-v'|\ge |v-w|$, for the Brownian motion starting from $v$, imposing an absorbing point at $v'$ reduces the hitting probability of $w$ by a factor depending only on $d$. I.e., 
    \begin{equation}\label{new2.1}
 	\widetilde{\mathbb{P}}_v\big(\tau_{w} < \tau_{v'} \big)\asymp \widetilde{\mathbb{P}}_v\big(\tau_{w} < \infty \big) \asymp  (|v-w|+1)^{2-d}. 
 \end{equation}


 \textbf{Green's function with absorbing boundary.} For any $D\subset \widetilde{\mathbb{Z}}^d$, the Green's function for $D$ is defined by 
 \begin{equation}\label{def_green_function}
 	\widetilde{G}_D(v,w):= \int_0^{\infty} \Big\{\widetilde{q}_t(v,w)-\widetilde{\mathbb{E}}_v\big[ \widetilde{q}_{t-\tau_D}(\widetilde{S}_{\tau_D},w)\cdot \mathbbm{1}_{\tau_D<t} \big]  \Big\} dt, \ \ \forall v,w\in  \widetilde{\mathbb{Z}}^d. 
 \end{equation}
Note that $\widetilde{G}_D(v,w)$ is finite, symmetric, continuous, and is decreasing in $D$. Moreover, it equals zero when $v$ or $w$ belongs to $D$. When $D=\emptyset$, $\widetilde{G}_{\emptyset}(\cdot, \cdot )$ is identical to $\widetilde{G}(\cdot, \cdot)$ defined in (\ref{def_Green}). The strong Markov property of $\widetilde{S}_\cdot$ implies that 
\begin{equation}\label{app_green}
	\widetilde{G}_D(v,w)= \widetilde{\mathbb{P}}_v(\tau_w<\tau_D)\cdot \widetilde{G}_D(w,w). 
\end{equation}
 Moreover, for any $D_1\subset D_2\subset \widetilde{\mathbb{Z}}^d$ and $v,w\in \widetilde{\mathbb{Z}}^d\setminus D_2$, one has 
 \begin{equation}\label{formula_two_green}
 	\widetilde{G}_{D_1}(v,w) - \widetilde{G}_{D_2}(v,w)= \sum\nolimits_{z\in (\widetilde{\partial} D_2)\setminus D_1 } \widetilde{\mathbb{P}}_v\big( \tau_{D_2}=\tau_{z}<\infty\big)\cdot  \widetilde{G}_{D_1}(z,w). 
 \end{equation}
 For any $D,D'\subset \widetilde{\mathbb{Z}}^d$, let $\mathrm{dist}(D,D'):=\inf_{v\in D,v'\in D'}|v-v'|$ denote the graph distance between $D$ and $D'$ (where we may omit the braces when $D$ or $D'$ is a singleton $\{v\}$ for some $v\in \widetilde{\mathbb{Z}}^d$). It can be derived using the potential theory that 
  \begin{equation}\label{23}
 	\widetilde{G}_D(v,v)\asymp \mathrm{dist}(v,D)\land 1.
  \end{equation}


 \textbf{Boundary excursion kernel.} For any $D\subset \widetilde{\mathbb{Z}}^d$, the boundary excursion kernel for $D$ is defined by 
 \begin{equation}\label{24}
 	\mathbb{K}_D(v,w):=\lim\limits_{\epsilon \downarrow 0} (2\epsilon)^{-1}\sum\nolimits_{v'\in \widetilde{\mathbb{Z}}^d:|v-v'|=\epsilon}\widetilde{\mathbb{P}}_{v'}\big(\tau_{D}=\tau_{w}<\infty \big),\ \forall v\neq w\in \widetilde{\partial}D.
 \end{equation}
It follows from the time-reversal symmetry of Brownian motion that $\mathbb{K}_D(\cdot,\cdot )$ is symmetric, i.e., $\mathbb{K}_D(v,w)=\mathbb{K}_D(w,v)$. In addition, $\mathbb{K}_D(\cdot,\cdot)$ is decreasing in $D$, i.e., for any $D\subset  D'\subset \widetilde{\mathbb{Z}}^d$ and any $v,w\in (\widetilde{\partial}D)\cap (\widetilde{\partial}D')$, one has $\mathbb{K}_D(v,w)\ge \mathbb{K}_{D'}(w,v)$. As pointed out in \cite[Section 2]{lupu2018random}, $\mathbb{K}_D(\cdot,\cdot)$ can be viewed as the equivalent effective conductance of the graph $\widetilde{\mathbb{Z}}^d\setminus D$. An excellent introduction to effective conductance on graphs can be found in \cite[Section 2]{lyons2017probability}.

 For any $v,w\in I_e$, let $I_{[v,w]}$ denote the subinterval of $I_e$ with endpoints $v$ and $w$. For any $z\in I_{[v,w]}$, the optional stopping theorem (see e.g., \cite[(2.19)]{cai2024one}) implies that 
\begin{equation}\label{ost}
	\widetilde{\mathbb{P}}_z\big( \tau_v<\tau_w \big)=|z-w| / |v-w|. 
\end{equation}
This immediately yields the equivalent effective conductance of a line segment: 
\begin{equation}\label{conduct_line}
	\mathbb{K}_{\widetilde{\mathbb{Z}}^d\setminus I_{[v,w]}}(v,w)= \lim\limits_{\epsilon \downarrow 0} (2\epsilon)^{-1} \cdot \tfrac{\epsilon}{|v-w|}= (2|v-w|)^{-1}. 
\end{equation}
Consequently, for any disjoint subintervals $\{I_{[x,w_i]}\}_{1\le i\le k}$ (all incident to $x$),   
\begin{equation}\label{ostnew}
	\widetilde{\mathbb{P}}_x\big( \tau_{\{w_i\}_{1\le i\le k}} =\tau_{w_j} \big)= \tfrac{|x-w_j|^{-1}}{\sum_{1\le i\le k}|x-w_i|^{-1}},  \ \forall 1\le j\le k. 
\end{equation}

\begin{lemma}\label{lemma21}


  For any $d\ge 3$ and $v\neq w\in \widetilde{\mathbb{Z}}^d$ with $|v-w|\ge 10d$, 
	\begin{equation}\label{newadd28}
		\mathbb{K}_{\{v,w\}}(v,w)\asymp |v-w|^{2-d}.
	\end{equation}
\end{lemma}
 \begin{proof}

 We divide the proof into two cases, according to whether $v$ is a lattice point.

 	\textbf{When $v\in \mathbb{Z}^d$.} Using the identity (\ref{ost}), we have  
 	\begin{equation}\label{new214}
 		\mathbb{K}_{\{v,w\}}(v,w)=  (2d)^{-1} \sum\nolimits_{y\in \mathbb{Z}^d: \{y,v\}\in \mathbb{L}^d} \widetilde{\mathbb{P}}_y\big( \tau_{w}<\tau_{v} \big)\overset{(\ref{new2.1})}{\asymp }|v-w|^{2-d}. 
  	\end{equation}

 		\textbf{When $v\notin \mathbb{Z}^d$.} Let $I_{\{x_1,x_2\}}$ denote the interval that contains $v$. Without loss of generality, we assume that $|v-x_1|\ge |v-x_2|$ (which implies $|v-x_1|\ge \frac{d}{2}$). Applying the identity (\ref{ost}), we have 
 		\begin{equation}\label{newadd217}
 			\begin{split}
 				0\le & \mathbb{K}_{\{v,w\}}(v,w) - (2|v-x_1|)^{-1}\cdot \widetilde{\mathbb{P}}_{x_1}\big( \tau_{w}<\tau_{v} \big) \\
 				= &(2|v-x_2|)^{-1}\cdot \widetilde{\mathbb{P}}_{x_2}\big( \tau_{w}<\tau_{v} \big)\\
 				\overset{(\ref{ostnew})}{=} & (2|v-x_2|)^{-1}\cdot  \frac{\sum_{z\in\mathbb{Z}^d: \{z,x_2\}\in \mathbb{L}^d,z\neq x_{1}}d^{-1}\widetilde{\mathbb{P}}_z\big( \tau_{w}<\tau_{v} \big)}{(2d-1)\cdot d^{-1}+|v-x_2|^{-1}}\overset{(\ref{new2.1})}{\asymp } |v-w|^{2-d}.
 			\end{split}
 		\end{equation}
 		Moreover, by $|v-x_1|\ge \frac{d}{2}$ and (\ref{new2.1}), one has 
 		\begin{equation}\label{newadd218}
 			(2|v-x_1|)^{-1}\cdot \widetilde{\mathbb{P}}_{x_1}\big( \tau_{w}<\tau_{v} \big) \overset{(\ref{new2.1})}{\asymp } |v-w|^{2-d}. 
 		\end{equation}
 		Combining (\ref{newadd217}) and (\ref{newadd218}), we get (\ref{newadd28}). 
 \end{proof}
 		
 		 \begin{lemma}\label{lemma_Kappa}
	For any $d\ge 3$, $D\subset \widetilde{\mathbb{Z}}^d$ and $v,w\in \widetilde{\mathbb{Z}}^d\setminus D$ with $|v-w|\ge d$, 
	\begin{equation}\label{newnew2.5}
		\mathbb{K}_{D\cup \{v,w\}}(v,w)\asymp \big[ \widetilde{G}_D(w,w)\big]^{-1}\cdot \widetilde{\mathbb{P}}_w\big( \tau_v<\tau_D \big). 
	\end{equation}
\end{lemma}
\begin{proof}
	We denote by $F$ the collection of lattice points $y\in \mathbb{Z}^d$ such that one of the intervals incident to $y$ contains $w$. Let $\widehat{F}:= \{y\in F: D\cap I_{[y,w]}=\emptyset\}$. Applying the identity (\ref{ost}), we have 
	\begin{equation}\label{newto323}
			\mathbb{K}_{D\cup \{v,w\}}(v,w)  = \sum\nolimits_{y\in \widehat{F}}\big(2|y-w|\big)^{-1}\cdot   \widetilde{\mathbb{P}}_y\big( \tau_{v}<\tau_{D\cup \{w\}} \big). 
	\end{equation}
	Meanwhile, using the strong Markov property of Brownian motion, one has 
	\begin{equation}\label{newto324}
	\begin{split}
			\widetilde{\mathbb{P}}_y\big( \tau_{v}<\tau_{D} \big)= 	&\widetilde{\mathbb{P}}_y\big( \tau_{v}<\tau_{D\cup \{w\}} \big) + \widetilde{\mathbb{P}}_y\big(\tau_{w}< \tau_{v}<\tau_{D} \big) \\
			=& \widetilde{\mathbb{P}}_y\big( \tau_{v}<\tau_{D\cup \{w\}} \big) + \widetilde{\mathbb{P}}_y\big(\tau_{w}<\tau_{D\cup \{v\}} \big) \cdot \widetilde{\mathbb{P}}_w\big( \tau_{v}<\tau_{D} \big),  
	\end{split}
	\end{equation}
	\begin{equation}\label{newto325}
		\begin{split}
		\widetilde{\mathbb{P}}_w\big( \tau_{v}<\tau_{D} \big) = \mathbb{J}^{-1}\sum\nolimits_{y\in \widehat{F}}    \big(2|y-w|\big)^{-1}\cdot \widetilde{\mathbb{P}}_y\big(\tau_{v}<\tau_{D} \big), 
		\end{split}
	\end{equation}
	where the quantity $\mathbb{J}$ is defined by 
	\begin{equation}
		\mathbb{J}:= \sum\nolimits_{y\in \widehat{F}}  \big(2|y-w|\big)^{-1}+ \sum\nolimits_{y\in F\setminus \widehat{F}} [2\cdot \mathrm{dist}(w,D\cap I_{[y,w]})]^{-1}.
	\end{equation}
	Note that (\ref{23}) implies 
	\begin{equation}\label{newtoJ}
		\mathbb{J}\asymp \big[\mathrm{dist}(w,D)\land 1\big]^{-1}\asymp \big[\widetilde{G}_D(w,w)\big]^{-1}. 
	\end{equation}
	Combining (\ref{newto323}), (\ref{newto324}) and (\ref{newto325}), we have 
	\begin{equation}\label{newto326}
		\begin{split}
			\mathbb{K}_{D\cup \{v,w\}}(v,w) \overset{(\ref{newto323}),(\ref{newto324})}{ =} & \sum\nolimits_{y\in \widehat{F}}\big(2|y-w|\big)^{-1}\cdot   \widetilde{\mathbb{P}}_y\big( \tau_{v}<\tau_{D} \big) \\
			& -\widetilde{\mathbb{P}}_w\big( \tau_{v}<\tau_{D} \big) \sum\nolimits_{y\in \widehat{F}} \tfrac{ \widetilde{\mathbb{P}}_y(\tau_{w}<\tau_{D\cup \{v\}}  ) }{ 2|y-w|  }    \\
 \overset{(\ref{newto325})}{=}&   \Big[ \mathbb{J}- \sum\nolimits_{y\in \widehat{F}}\tfrac{ \widetilde{\mathbb{P}}_y(\tau_{w}<\tau_{D\cup \{v\}}  ) }{ 2|y-w|  }   \Big] \cdot \widetilde{\mathbb{P}}_w\big( \tau_{v}<\tau_{D} \big).
		\end{split}
	\end{equation}

	For the upper bound in (\ref{newnew2.5}), using (\ref{newtoJ}) and (\ref{newto326}), we obtain  
	\begin{equation}\label{newto3.29}
		\begin{split}
			\mathbb{K}_{D\cup \{v,w\}}(v,w) \overset{(\ref{newto326})}{\le}    \mathbb{J}\cdot \widetilde{\mathbb{P}}_w\big( \tau_{v}<\tau_{D} \big) \overset{(\ref{newtoJ})}{\asymp }\big[\widetilde{G}_D(w,w)\big]^{-1} \cdot \widetilde{\mathbb{P}}_w\big( \tau_{v}<\tau_{D} \big).
		\end{split}
	\end{equation}
	For the lower bound, since $1-\widetilde{\mathbb{P}}_y\big(\tau_{w}<\tau_{D\cup \{v\}} \big)\ge \widetilde{\mathbb{P}}_y\big(\tau_{w}=\infty\big)\asymp |y-w|$, it follows from (\ref{newtoJ}) and (\ref{newto326}) that 
	\begin{equation}\label{newto3.30}
		\begin{split}
			\mathbb{K}_{D\cup \{v,w\}}(v,w) \overset{(\ref{newto326})}{=} & \Big( \sum\nolimits_{y\in   \widehat{F}}  \big(2|y-w|\big)^{-1}\cdot \big[1-\widetilde{\mathbb{P}}_y\big(\tau_{w}<\tau_{D\cup \{v\}} \big)\big] \\
			&+  \sum\nolimits_{y\in F\setminus \widehat{F}} [2\cdot \mathrm{dist}(w,D\cap I_{[y,w]})]^{-1} \Big) \cdot  \widetilde{\mathbb{P}}_w\big( \tau_{v}<\tau_{D} \big) \\ 
			\ge &\Big( c|\widehat{F}|+ \sum_{y\in F\setminus \widehat{F}} [2\cdot \mathrm{dist}(w,D\cap I_{[y,w]})]^{-1}  \Big)\cdot \widetilde{\mathbb{P}}_w\big( \tau_{v}<\tau_{D} \big) \\
			\gtrsim & \big[\mathrm{dist}(w,D)\land 1\big]^{-1}\cdot \widetilde{\mathbb{P}}_w\big( \tau_{v}<\tau_{D} \big)\\
			\overset{(\ref{23})}{\asymp } & \big[\widetilde{G}_D(w,w)\big]^{-1}\cdot  \widetilde{\mathbb{P}}_w\big( \tau_{v}<\tau_{D} \big). 
		\end{split}
	\end{equation}
	By (\ref{newto3.29}) and (\ref{newto3.30}), we complete the proof of this lemma.
\end{proof}

  \textbf{Capacity.} The equilibrium measure for $D$ is given by  
 \begin{equation}
 	\mathbb{Q}_D(v):=\lim\limits_{N\to \infty} \sum\nolimits_{w\in \partial B(N)} \mathbb{K}_{D\cup \partial B(N)}(v,w), \ \forall v\in \widetilde{\partial}D. 
 \end{equation} 
 The capacity of $D$ is defined as the total mass of the equilibrium measure, i.e.,
 \begin{equation}
 	\mathrm{cap}(D):= \sum\nolimits_{v\in \widetilde{\partial}D} \mathbb{Q}_D(v). 
 \end{equation}
 It is well-known that $\mathrm{cap}(D)$ is increasing in $D$ (i.e., $\mathrm{cap}(D_1)\le \mathrm{cap}(D_2)$ if $D_1\subset D_2$), sub-additive (i.e., $\mathrm{cap}(D_1\cup D_2)\le\mathrm{cap}(D_1)+\mathrm{cap}(D_2) $ for all $D_1,D_2\subset \widetilde{\mathbb{Z}}^d$), and satisfies $\mathrm{cap}\big(\widetilde{B}(N)\big)\asymp N^{d-2}$ for $N\ge 1$. Another key property is its relation to the hitting probability: for any $D\subset \widetilde{B}(N)$, $x\in \partial B(2N)$ and $D'\subset [\widetilde{B}(3N)]^c$,  
 \begin{equation}\label{310}
 	\widetilde{\mathbb{P}}_x\big(\tau_{D}<\tau_{D'} \big)\asymp N^{2-d}\cdot \mathrm{cap}(D). 
 \end{equation}

 \subsection{Properties of GFFs}

   For any $D\subset \widetilde{\mathbb{Z}}^d$, we denote by $\mathbb{P}^D$ the law of $\{\widetilde{\phi}_v\}_{v\in \widetilde{\mathbb{Z}}^d}$ with zero boundary condition on $D$, i.e., conditioned on $\widetilde{\phi}_v=0$ for all $v\in D$. We denote by $\mathbb{E}^D$ the expectation under $\mathbb{P}^D$. Clearly, the GFF $\widetilde{\phi}_\cdot\sim \mathbb{P}^D$ is mean-zero. In addition, its covariance is given by 
  \begin{equation}
 	\mathbb{E}^{D}\big[\widetilde{\phi}_v\widetilde{\phi}_w \big]=\widetilde{G}_D(v,w),\ \ \forall v,w\in \widetilde{\mathbb{Z}}^d.
 \end{equation}

The formula in (\ref{two-point1}) for the two-point function also works under zero boundary condition. To be precise, for any $D\subset \widetilde{\mathbb{Z}}^d$ and any distinct $v,w\in \widetilde{\mathbb{Z}}^d\setminus D$,  
\begin{equation}\label{29}
\begin{split}
		\mathbb{P}^D\big( v\xleftrightarrow{\ge 0} w\big)=  &\pi^{-1}\arcsin\Big( \tfrac{\widetilde{G}_D(v,w)}{\sqrt{\widetilde{G}_D(v,v)\widetilde{G}_D(w,w)}} \Big) \\
		\asymp &   \tfrac{\widetilde{G}_D(v,w)}{\sqrt{\widetilde{G}_D(v,v)\widetilde{G}_D(w,w)}}= \widetilde{\mathbb{P}}_v\big( \tau_w <\tau_D\big)  \sqrt{\tfrac{\widetilde{G}_D(w,w)}{\widetilde{G}_D(v,v)}}.  
\end{split}
\end{equation}
 The powerful formula of the conditional two-point function  in \cite[(18)]{lupu2018random} states that
\begin{equation}\label{formula_two_point}
	\mathbb{P}^D\big( v\xleftrightarrow{\ge 0} w\mid \widetilde{\phi}_v=a,\widetilde{\phi}_w=b \big)=1-e^{-2ab\cdot \mathbb{K}_{D\cup \{v,w\}}(v,w)}, \ \forall a,b\ge 0. 
\end{equation}

As shown in the following lemma, the difference between the two-point functions under different zero boundary conditions can be approximated by the corresponding difference of Green's functions.

\begin{lemma}\label{lemma_different_two_point}
	For $d\ge 3$, $n\ge 1$, $D\subset \widetilde{\mathbb{Z}}^d$ and $v,w\in \widetilde{\mathbb{Z}}^d\setminus D$,
	\begin{equation}\label{newadd233}
  \big|  \mathbb{P} (v\xleftrightarrow{\ge 0} w  ) -\mathbb{P}^{D}(v\xleftrightarrow{\ge 0} w ) \big|  \lesssim  \widetilde{G}(v,w)-\widetilde{G}_D(v,w). 
\end{equation}
\end{lemma}
\begin{proof}
Without loss of generality, we assume $\widetilde{G}_D(w,w)\le \widetilde{G}_D(v,v)$. By (\ref{29}),  
	\begin{equation}\label{addnew234}
		\begin{split}
	& \big| 	\mathbb{P}(v\xleftrightarrow{\ge 0} w )-\mathbb{P}^{D}(v\xleftrightarrow{\ge 0} w ) \big|  \asymp  \big|\tfrac{\widetilde{G}(v,w)}{\sqrt{\widetilde{G}(v,v)\widetilde{G}(w,w)}}- \tfrac{\widetilde{G}_D(v,w)}{\sqrt{\widetilde{G}_D(v,v)\widetilde{G}_D(w,w)}} \big|\\
	\le &\mathbb{I}_1+\mathbb{I}_2+\mathbb{I}_3:=\tfrac{\widetilde{G}(v,w)-\widetilde{G}_D(v,w)}{\sqrt{\widetilde{G}(v,v)\widetilde{G}(w,w)}} +\tfrac{\widetilde{G}_D(v,w)}{\sqrt{\widetilde{G}(w,w)}} \cdot \big( \tfrac{1}{\sqrt{\widetilde{G}_D(v,v)}}- \tfrac{1}{\sqrt{\widetilde{G}(v,v)}} \big) \\
	&\ \ \ \ \  \ \ \ \ \ \ \ \ \ \ \ \ \ \ + \tfrac{\widetilde{G}_D(v,w)}{\sqrt{\widetilde{G}_D(v,v)}} \cdot \big(  \tfrac{1}{\sqrt{\widetilde{G}_D(w,w)}} - \tfrac{1}{\sqrt{\widetilde{G}(w,w)}} \big). 
		\end{split}
	\end{equation}
	Since $\widetilde{G}(v,v)\asymp  \widetilde{G}(w,w)\asymp 1$, one has 
	\begin{equation}
		\mathbb{I}_1 \asymp \widetilde{G}(v,w)-\widetilde{G}_D(v,w). 
	\end{equation}
	 Meanwhile, by $\widetilde{G}_D(v,v)\le \widetilde{G}(v,v)\asymp 1$ and (\ref{app_green}), we have 
	 \begin{equation}
	 	\begin{split}
	 			\mathbb{I}_2= &  \tfrac{\widetilde{G}_D(v,w)\cdot [ \widetilde{G}(v,v) - \widetilde{G}_D(v,v) ] }{\sqrt{\widetilde{G}(w,w) \widetilde{G}_D(v,v) \widetilde{G}(v,v)}\cdot \big[\sqrt{\widetilde{G}_D(v,v)} + \sqrt{ \widetilde{G}(v,v)} \big]}  \\
	 			\lesssim &   \widetilde{\mathbb{P}}_w(\tau_{v}<\tau_{D})\big[ \widetilde{G}(v,v) - \widetilde{G}_D(v,v) \big]  \le  \widetilde{G}(v,w)-\widetilde{G}_D(v,w). 
	 	\end{split}
	 \end{equation}
	 Similarly, using $\widetilde{G}_D(w,w)\le \widetilde{G}_D(v,v)$, we have 
	 \begin{equation}
	 	\begin{split}
	 		\mathbb{I}_3=& \tfrac{\widetilde{G}_D(v,w)\cdot [ \widetilde{G}(w,w) - \widetilde{G}_D(w,w) ] }{\sqrt{\widetilde{G}_D(v,v) \widetilde{G}_D(w,w) \widetilde{G}(w,w)}\cdot \big[\sqrt{\widetilde{G}_D(w,w)} + \sqrt{ \widetilde{G}(w,w)} \big]} \\
	 		\lesssim  &  \widetilde{\mathbb{P}}_v(\tau_{w}<\tau_{D})\big[ \widetilde{G}(w,w) - \widetilde{G}_D(w,w) \big] \le   \widetilde{G}(v,w)-\widetilde{G}_D(v,w). 
	 	\end{split}
	 \end{equation}
	 Plugging these estimates on $\mathbb{I}_1$, $\mathbb{I}_2$ and $\mathbb{I}_3$ into (\ref{addnew234}), we complete the proof. 
 \end{proof}

\textbf{Sign clusters.} For any $A\subset \widetilde{\mathbb{Z}}^d$, let $\mathcal{C}_A^{+}:=\{v\in \widetilde{\mathbb{Z}}^d: v\xleftrightarrow{\ge 0} A\}$ (resp. $\mathcal{C}_A^{-}:=\{v\in \widetilde{\mathbb{Z}}^d: v\xleftrightarrow{\le 0} A\}$) denote the positive (resp. negative) cluster containing $A$. We then define the sign cluster $\mathcal{C}_A^{\pm}:=\mathcal{C}_A^{+}\cup \mathcal{C}_A^{-}$. Note that $\mathcal{C}_A^{+}\cap \mathcal{C}_A^{-}=A$. For $\diamond\in \{+,-,\pm\}$, we denote by $\mathcal{F}_{\mathcal{C}_A^{\diamond}}$ the $\sigma$-field generated by $\mathcal{C}_A^{\diamond}$ and all GFF values on $\mathcal{C}_A^{\diamond}$.

 We record the following corollary of the strong Markov property of $\widetilde{\phi}_\cdot \sim \mathbb{P}^{D}$ (see e.g., \cite[Theorem 8]{ding2020percolation}), which we will use multiple times. Suppose that $\diamond\in \{+,-,\pm\}$ and $A\subset D'\subset \widetilde{\mathbb{Z}}^d\setminus D$. Conditioned on $\mathcal{F}_{\mathcal{C}_A^{\diamond}}$, on the event $\{\mathcal{C}_A^{\diamond}=D'\}$, one has 
\begin{equation}\label{revise_new_334}
	\big\{ \widetilde{\phi}_v\big\}_{v\in \widetilde{\mathbb{Z}}^d\setminus (D\cup D')} \overset{\mathrm{d}}{=} \big\{ \widetilde{\phi}_v' + \mathcal{H}_v(D\cup D') \big\}_{v\in \widetilde{\mathbb{Z}}^d\setminus (D\cup D')} .
\end{equation}
Here $\widetilde{\phi}_{\cdot}'$ is distributed under $ \mathbb{P}^{D\cup D'}$, and $\mathcal{H}_{\cdot}(D\cup D')$ is the harmonic function on $\widetilde{\mathbb{Z}}^d\setminus (D\cup D')$ with boundary conditions $\{\widetilde{\phi}_v\}_{v\in D\cup D'}$. Moreover, by the continuity of $\widetilde{\phi}_\cdot$, one has $\widetilde{\phi}_v=0$ for all $v\in    ( \widetilde{\partial}D')\setminus A$, which yields that  
\begin{equation} 
	\begin{split}
		\mathcal{H}_v(D\cup D'):=&\sum\nolimits_{w\in (\widetilde{\partial}D')\cap A} \widetilde{\mathbb{P}}_v\big(\tau_{D\cup D'}=\tau_w<\infty \big)\cdot \widetilde{\phi}_w.
			\end{split}
\end{equation}
In particular, when $A$ is a singleton or $\diamond=\pm$, the continuity of $\widetilde{\phi}_\cdot$ also implies that $\widetilde{\phi}_\cdot$ vanishes on $\widetilde{\partial}\mathcal{C}_A^{\diamond}$, and thus $\mathcal{H}_v(D\cup D')=0$ for all $v\in \widetilde{\mathbb{Z}}^d\setminus (D\cup D')$.

  \textbf{Regularity of connecting probabilities.} \cite{cai2024quasi} established a series of regularity properties for the connecting probabilities of $\widetilde{E}^{\ge 0}$. In what follows, we review some of these properties that are useful for later proofs. The first property states that Harnack's inequality holds for the point-to-set connecting probabilities.

 \begin{lemma}[{\cite[Proposition 1.6]{cai2024quasi}}]\label{lemma_harnack}
 	For any $d\ge 3$ with $d\neq 6$, there exists $C>0$ such that for any $N\ge 1$, $v,w\in \widetilde{B}(N)$ and $A,D\subset [\widetilde{B}(CN)]^c$,  
 	\begin{equation}
 		\mathbb{P}^{D}\big( v\xleftrightarrow{\ge 0} A \big) \asymp \mathbb{P}^{D}\big( w\xleftrightarrow{\ge 0} A\big). 
 		 	\end{equation}
 \end{lemma}

 \begin{remark}
 In the original statement of \cite[Proposition 1.6]{cai2024quasi}, the points $v,w$ were required to lie in the lattice. In Lemma \ref{lemma_harnack}, however, we allow $v,w$ to be located in the interior of some interval of $\widetilde{\mathbb{Z}}^d$. This extension is valid since for any $v\in I_{\{x,y\}}$, it follows from a union bound that $\mathbb{P}^{D}( v\xleftrightarrow{\ge 0} A)$ is bounded from above by a linear combination of $\mathbb{P}^{D}(x\xleftrightarrow{\ge 0} A)$ and $\mathbb{P}^{D}( y\xleftrightarrow{\ge 0} A)$; meanwhile, a direct application of the FKG inequality yields that $\mathbb{P}^{D}( v\xleftrightarrow{\ge 0} A)$ is at least proportional to $\mathbb{P}^{D}(x\xleftrightarrow{\ge 0} A)$ (see e.g.,  \cite[(2.26)]{cai2024quasi}). Similarly, such an extension also applies to \cite[Lemma 5.3]{cai2024quasi}. 
 \end{remark}


 For any $v\in \widetilde{\mathbb{Z}}^d$ and $R\ge 1$, we denote the Euclidean ball $\mathcal{B}_v(R):=\big\{y\in \mathbb{Z}^d: \|\bar{v}-y\|\le R\big\}$ (recalling that $\bar{v}$ represents the closest lattice point to $v$). Similarly, we denote $B_v(R):=\bar{v}+B(R)$ and $\widetilde{B}_v(R):=\bar{v}+\widetilde{B}(R)$. When $x=\bm{0}$, we may omit the subscript $\bm{0}$. Referring to \cite[(5.2)]{cai2024quasi}, we have the following quantitative relation between the point-to-set and boundary-to-set connection probabilities.


 \begin{lemma}\label{lemma_onecluster_box_point}
We denote $F_1:=[\widetilde{B}(10d^2N)]^c$ and $F_{-1}:=\widetilde{B}(\frac{N}{10d^2} )$. For any $d\ge 3$ with $d\neq 6$, $D\subset \widetilde{\mathbb{Z}}^d$, $N\ge 1$, $i \in \{1,-1\}$, and $v\in (\widetilde{\mathbb{Z}}^d\setminus D)\cap  F_i$,
 	\begin{equation}\label{new.234}
 	 \begin{split}
 	 	\mathbb{P}^D\big( v\xleftrightarrow{\ge 0} \partial B(N) \big) \lesssim & N^{-(\frac{d}{2}\boxdot 3)}\sum\nolimits_{x\in \partial \mathcal{B}(d^{i}N)} \mathbb{P}^D\big( v\xleftrightarrow{\ge 0}x  \big). 	 \end{split}
 	\end{equation} 
 \end{lemma}

 The loop decomposition argument developed in \cite[Section 5.2]{cai2024quasi} provides a series of upper bounds on the connecting probability between multiple sets. We recall a basic application of these bounds as follows.

  
 \begin{lemma}[{\cite[Lemma 5.3]{cai2024quasi}}]\label{lemma_cite_three_point}
 	For any $3\le d\le 5$, there exists $C>0$ such that for any $M\ge 1$, $v,w\in \widetilde{B}(M)$ and $A,D\subset [\widetilde{B}(CM)]^c$, 
 	\begin{equation}\label{ineq_new_lemma_cite_three_point}
 		\mathbb{P}^{D}\big(A\xleftrightarrow{\ge 0} v,A\xleftrightarrow{\ge 0} w \big)\lesssim  \big(|v-w|+1\big)^{-\frac{d}{2}+1}\mathbb{P}^{D}\big(A\xleftrightarrow{\ge 0} v \big). 
 	\end{equation}
 \end{lemma}

Next, we present two properties of the positive cluster in the high-dimensional case (i.e., $d\ge 7$). The first property will be used several times in our application of the second moment method. Since its proof involves numerous computations and are relatively standard, we defer them to Appendix \ref{appendix2}.

\begin{lemma}\label{lemme_technical_high_dimension}
	For $d \ge 7$, there exists $C>0$ such that for any $M\ge 1$, $N\ge CM$ and $y\in \partial B(N)$, 
	\begin{equation}\label{3.12_3.39}
	\max_{x_1\in B_y(M)}	  \sum\nolimits_{x_2\in B_y(M)}   \mathbb{P}\big(\bm{0}\xleftrightarrow{\ge 0} x_1,\bm{0}\xleftrightarrow{\ge 0} x_2\big)  
  		\lesssim  M^{4} N^{2-d}. 
	\end{equation} 
\end{lemma}

When $d\ge 7$, it follows from (\ref{result_typical_volume}) that with high probability, the volume of a large sign cluster within a box of size $M$ is at least $cM^4$. The following property is an analogue of (\ref{result_typical_volume}) under a different conditioning, where $\{\bm{0} \xleftrightarrow{\ge 0} \partial B(N)\}$ is replaced by the event that two distant points are connected. As the proof closely follows the argument in \cite[Section 5.1]{cai2024incipient}, we defer it to Appendix \ref{appendix3}.

 \begin{lemma}\label{lemma_hit_capacity}
 Assume that $d,M,N,y$ satisfy all conditions in Lemma \ref{lemme_technical_high_dimension}. For any $\epsilon>0$, there exists $c(d,\epsilon)>0$ such that  
 	  	\begin{equation}
 	  	  \mathbb{P}\big( \mathcal{V}^+_y(M)  \le cM^4 \mid  \bm{0}\xleftrightarrow{\ge 0} y \big) \le \epsilon.  \end{equation}  
  \end{lemma}

 {\color{blue}

 	}

 	 For any $v\in \widetilde{\mathbb{Z}}^d$ and $r>0$, let $\mathbf{B}_v(r):=\{w\in \widetilde{\mathbb{Z}}^d: |v-w|\le r\}$ denote the ball of radius $r$ centered at $v$ under the graph distance on $\widetilde{\mathbb{Z}}^d$.

The subsequent result shows that the two-point function satisfies the maximum principle, up to certain correction factors.

  \begin{lemma}\label{new_lemma_210}
 	For any $d\ge 3$, $D\subset \widetilde{\mathbb{Z}}^d$, $n\ge 10$ and $v,w\in \widetilde{\mathbb{Z}}^d\setminus D$ with $|v-w|\ge 2n$, 
 	\begin{equation}\label{newadd265}
 	\begin{split}
  		\mathbb{P}^D\big(v \xleftrightarrow{\ge 0} w \big) \lesssim  \big( \mathrm{dist}(v,D)\land 1 \big)^{\frac{1}{2}}  
  		\cdot |\partial\mathcal{B}_{{v}}(n)|^{-1}  \sum\nolimits_{z\in \partial\mathcal{B}_{{v}}(n)}  \mathbb{P}^D\big( z  \xleftrightarrow{\ge 0} w  \big).
  		 	\end{split}
 	\end{equation}		 
 \end{lemma} 
 \begin{proof}
 	 By the strong Markov property of $\widetilde{S}_\cdot$, $\widetilde{\mathbb{P}}_v\big( \tau_w<\tau_D \big)$ is upper-bounded by 
 	  \begin{equation}\label{2.14}
 \begin{split}
 	 & \widetilde{\mathbb{P}}_v\big( \tau_{\widetilde{\partial}\mathbf{B}_v(\frac{1}{2})}<\tau_D \big)\max_{y_1\in \widetilde{\partial}\mathbf{B}_v(\frac{1}{2}) } \sum_{y_2  \in \partial\mathcal{B}_{{v}}(n)}\widetilde{\mathbb{P}}_{y_1}\big(\tau_{\partial\mathcal{B}_{{v}}(n)}=\tau_{y_2}\big) \cdot \widetilde{\mathbb{P}}_{y_2}\big( \tau_{w} <\tau_D  \big). 
 \end{split}
 \end{equation}
 	Since $D$ is compact, there exists $v_\dagger\in D$ such that $|v-v_\dagger|=\mathrm{dist}(v,D)$. Thus, by $\tau_{D}\le \tau_{v_\dagger}$ and the potential theory of Brownian motion, we have 
\begin{equation}\label{2.15}
	\widetilde{\mathbb{P}}_v\big( \tau_{\widetilde{\partial}\mathbf{B}_v(\frac{1}{2})}<\tau_D \big)\le \widetilde{\mathbb{P}}_v\big( \tau_{\widetilde{\partial}\mathbf{B}_v(\frac{1}{2})}<\tau_{v_\dagger} \big)\asymp  \mathrm{dist}(v,D)\land 1.
\end{equation}
 Meanwhile, since the hitting probability on $\partial\mathcal{B}_{{v}}(n)$ for a Brownian motion starting from $ \widetilde{\partial}\mathbf{B}_v(\frac{1}{2})$ is proportional to the uniform distribution on $\partial\mathcal{B}_{{v}}(n)$ (see e.g. \cite[Lemma 6.3.7]{lawler2010random}), we have: for any $y_1\in \widetilde{\partial}\mathbf{B}_v(\frac{1}{2})$ and $y_2 \in \partial\mathcal{B}_v(n)$, 
 	\begin{equation}\label{new217}
 		\widetilde{\mathbb{P}}_{y_1}\big(\tau_{\partial\mathcal{B}_{{v}}(n)}=\tau_{y_2}\big)  \asymp |\partial\mathcal{B}_{{v}}(n)|^{-1}. 
 	\end{equation}
 	Plugging (\ref{2.15}) and (\ref{new217}) into (\ref{2.14}), we obtain 
 \begin{equation}\label{2.17}
 	\begin{split}
 		\widetilde{\mathbb{P}}_v\big( \tau_w<\tau_D \big)\lesssim    \big(\mathrm{dist}(v,D)\land 1 \big)\cdot   |\partial\mathcal{B}_{{v}}(n)|^{-1}  \sum\nolimits_{z\in \partial\mathcal{B}_{{v}}(n)}   \widetilde{\mathbb{P}}_{z}\big( \tau_{w} <\tau_D  \big). 
 	\end{split}
 \end{equation}
  	 Combined with (\ref{23}) and (\ref{29}), it implies the desired bound (\ref{newadd265}):   
 \begin{equation}
 	\begin{split}
 	\mathbb{P}^D\big(v \xleftrightarrow{\ge 0} w \big) \overset{(\ref{23}),(\ref{29})}{\asymp}	&\widetilde{\mathbb{P}}_v\big( \tau_w<\tau_D \big)\sqrt{ \widetilde{G}_D(w,w)}\cdot \big(\mathrm{dist}(v,D)\land 1 \big)^{-\frac{1}{2}}\\
 	\overset{(\ref{2.17}),\widetilde{G}_D(z,z)\lesssim 1}{\lesssim } & \big(\mathrm{dist}(v,D)\land 1 \big)^{\frac{1}{2}}\cdot |\partial\mathcal{B}_{{v}}(n)|^{-1}  \\
 	&\cdot   \sum\nolimits_{z\in \partial\mathcal{B}_{{v}}(n)}  \widetilde{\mathbb{P}}_{z}\big( \tau_{w} <\tau_D  \big) \sqrt{ \tfrac{\widetilde{G}_D(w,w)}{\widetilde{G}_D(z,z)}}\\
 	\overset{(\ref{29})}{\asymp} & \big( \mathrm{dist}(v,D)\land 1 \big)^{\frac{1}{2}} \cdot   |\partial\mathcal{B}_{{v}}(n)|^{-1}  \sum_{z\in \partial\mathcal{B}_{{v}}(n)}  \mathbb{P}^D\big( z \xleftrightarrow{\ge 0} w  \big). \qedhere
 	\end{split}
 \end{equation}	
 
  \end{proof}

Subsequently, we provide an upper bound on the conditional connecting probability between two points, given the GFF value at one of them. The proof needs the following basic facts: 
\begin{itemize}
	\item  For any random variable $\mathbf{X}$, by Jensen's inequality one has 
 \begin{equation}\label{new2.24}
		\mathbb{E}[\mathbf{X}\land 1]\le  \mathbb{E}[\mathbf{X}]\land 1. 
			\end{equation}

	\item For any $u_1>0$ and $0<u_2\le C$, we have 
	\begin{equation}\label{new226}
		\big(\tfrac{u_1}{u_2}\cdot(u_1+u_2) \big)\land 1 \lesssim \tfrac{u_1}{\sqrt{u_2}}.  
	\end{equation}
	For (\ref{new226}), we consider the following two cases separately. When $u_1\le u_2$, 
	\begin{equation}
		\big(\tfrac{u_1}{u_2}\cdot(u_1+u_2) \big)\land 1 \overset{(u_1+u_2\le 2u_2)}{\le }  2u_1  \overset{(u_2\le C)}{\lesssim} \tfrac{u_1}{\sqrt{u_2}}. 
	\end{equation} 
	When $u_1\ge u_2$, since $a\land 1\le \sqrt{a}$ for $a>0$, one has 
	\begin{equation}
		\big(\tfrac{u_1}{u_2}\cdot(u_1+u_2) \big)\land 1 \overset{(u_1+u_2\le 2u_1)}{\le } \tfrac{2u_1^2}{u_2}\land 1\lesssim \tfrac{u_1}{\sqrt{u_2}}. 
	\end{equation}
	To sum up, we obtain (\ref{new226}).

	\item  For any $\mu,\sigma\ge 0$ and $\mathbf{Y}\sim N(\mu,\sigma^2)$, one has (see e.g., \cite[(3)]{leone1961folded})  
	 	\begin{equation}\label{new2.25}
	\begin{split}
		\mathbb{E}\big[\mathbf{Y}\cdot \mathbbm{1}_{\mathbf{Y}\ge 0} \big]=\tfrac{\sigma}{\sqrt{2\pi}} e^{-\frac{\mu^2}{2\sigma^2}}+ \tfrac{\mu}{2}\cdot \big(1-2\Phi(-\tfrac{\mu}{\sigma}) \big)\le \sigma+ \mu,
		\end{split}	
	\end{equation} 
where $\Phi(t)=\mathbb{P}(\mathbf{Z}\le t)$ for $\mathbf{Z}\sim N(0,1)$.

\end{itemize}

 \begin{lemma}\label{lemma_conditional_two_points}
 	 For any $a>0$, $d\ge 3$, $D\subset \widetilde{\mathbb{Z}}^d$ and $v,w\in \widetilde{\mathbb{Z}}^d\setminus D$ with $|v-w|\ge d$, 
 	 \begin{equation}\label{232}
	 \mathbb{P}^D\big(  v\xleftrightarrow{\ge 0}w \mid  \widetilde{\phi}_v =a \big) \lesssim  \tfrac{a\cdot \mathbb{P}^{D}(v\xleftrightarrow{\ge 0} w )}{\sqrt{\widetilde{G}_{D}(v,v)}}. 
\end{equation} 
 \end{lemma}
 \begin{proof}
 	Using (\ref{formula_two_point}), (\ref{new2.24}) and Lemma \ref{lemma_Kappa}, we have
\begin{equation}\label{225}
\begin{split}
	 & \mathbb{P}^D\big(  v\xleftrightarrow{\ge 0} w \mid  \widetilde{\phi}_v =a \big)\\\overset{(\ref{formula_two_point})
}{\asymp} & \mathbb{E}^{D}\big[ \big( a\cdot \widetilde{\phi}_w\cdot \mathbb{K}_{D\cup \{v,w\}}(v,w)\cdot \mathbbm{1}_{\widetilde{\phi}_w \ge 0} \big) \land 1  \mid  \widetilde{\phi}_v =a \big]\\
	 \overset{(\ref{new2.24}),\mathrm{Lemma}\ \ref{lemma_Kappa}}{\lesssim } &  \Big( \tfrac{a\cdot \widetilde{\mathbb{P}}_w( \tau_v<\tau_D  )}{\widetilde{G}_D(w,w)} \cdot \mathbb{E}^{D}\big[  \widetilde{\phi}_w\cdot   \mathbbm{1}_{\widetilde{\phi}_w\ge 0} \mid   \widetilde{\phi}_v =a  \big]\Big)\land 1.  
\end{split}
\end{equation}
By the strong Markov property of $\widetilde{\phi}_\cdot$, the distribution of $\widetilde{\phi}_w$ under $ \mathbb{P}^{D}(\cdot \mid \widetilde{\phi}_v =a)$ is given by $N\big(a\cdot \widetilde{\mathbb{P}}_w(\tau_{v}<\tau_D ),\widetilde{G}_{D\cup \{v\}}(w,w)\big)$. Thus, by (\ref{new2.25}) and $\widetilde{G}_{D\cup \{v\}}(w,w)\le \widetilde{G}_{D}(w,w)$, we have 
\begin{equation}\label{new231}
	\begin{split}
		\mathbb{E}^{D}\big[  \widetilde{\phi}_w\cdot   \mathbbm{1}_{\widetilde{\phi}_w\ge 0} \mid   \widetilde{\phi}_v =a  \big] \le  a\cdot \widetilde{\mathbb{P}}_w(\tau_{v}<\tau_D ) + \widetilde{G}_{D}(w,w). 
	\end{split}
\end{equation}
By plugging (\ref{new231}) into (\ref{225}) and then applying (\ref{new226}), we get 
\begin{equation}
	\begin{split}
		\mathbb{P}^D\big(  v\xleftrightarrow{\ge 0} w \mid  \widetilde{\phi}_v =a \big) \lesssim \tfrac{a\cdot \widetilde{\mathbb{P}}_w(\tau_{v}<\tau_D )}{\sqrt{\widetilde{G}_{D}(w,w)}} \overset{(\ref{29})}{ \asymp }  \tfrac{a\cdot \mathbb{P}^{D}(v\xleftrightarrow{\ge 0} w )}{\sqrt{\widetilde{G}_{D}(v,v)}}.  \qedhere
	\end{split}
\end{equation}
 \end{proof}

 \subsection{Loop soup}

 A rooted loop on $\widetilde{\mathbb{Z}}^d$ is a path $\widetilde{\varrho}:[0,T)\to \widetilde{\mathbb{Z}}^d$ (where $T\in \mathbb{R}^+$) satisfying $\widetilde{\varrho}(0)=\widetilde{\varrho}(T)$. For any $\alpha>0$, the loop soup of intensity $\alpha$ is the Poisson point process with intensity measure $\alpha\cdot \widetilde{\mu}$. Here the measure $\widetilde{\mu}(\cdot)$ is supported on the space of rooted loops and is defined by 
   \begin{equation}\label{3.10_3.54}
   	\widetilde{\mu}(\cdot) = \int_{v\in \widetilde{\mathbb{Z}}^d} \mathrm{d}\mathrm{m}(v)\int_{0<t<\infty} t^{-1}\widetilde{q}_t(v,v)\widetilde{\mathbb{P}}^t_{v,v}(\cdot) \mathrm{d}t,  
   \end{equation} 
   where $\mathrm{m}(\cdot)$ denotes the Lebesgue measure on $\widetilde{\mathbb{Z}}^d$, and $\widetilde{\mathbb{P}}^t_{x,y}(\cdot)$ denotes the law of the Brownian bridge on $\widetilde{\mathbb{Z}}^d$ with duration $t$, starting from $x$ and ending at $y$, i.e., the Markov process with transition density $\frac{\widetilde{q}_s(x,\cdot)\widetilde{q}_{t-s}(\cdot,y)}{\widetilde{q}_t(x,y)}$ for $0\le s\le t$.


    We say that two rooted loops are equivalent if one can be transformed into the other by a certain time shift. We refer each equivalence class of rooted loops under this relation as a loop (usually denoted by $\widetilde{\ell}$). Since the densities of two equivalent rooted loops in $	\widetilde{\mu}$ are identical, $\widetilde{\mu}$ can be viewed as a measure on the space of loops by forgetting the root of each rooted loop. We use $\mathrm{ran}(\cdot)$ to denote the range of a path, i.e., the collection of points visited by the path. With a slight abuse of notation, we define $\mathrm{ran}(\widetilde{\ell})$ as $\mathrm{ran}(\widetilde{\varrho})$ for any $\widetilde{\varrho}\in \widetilde{\ell}$ (the choice of $\widetilde{\varrho}$ is irrelevant). For a point measure $\mathcal{L}$ consisting of paths, we denote by $\cup \mathcal{L}$ the union of the ranges of all paths in $\mathcal{L}$.

     For any $D\subset \widetilde{\mathbb{Z}}^d$, let $\widetilde{\mu}^D$ denote the restriction of $\widetilde{\mu}$ to the set of loops that are disjoint from $D$. It follows that the truncated point measure $\widetilde{\mathcal{L}}_{\alpha}^{D}:=\widetilde{\mathcal{L}}_{\alpha}\cdot \mathbbm{1}_{\mathrm{ran}(\widetilde{\ell})\cap D=\emptyset}$ is a Poisson point process with intensity measure $\alpha\widetilde{\mu}^{D}$. For any $A,A'\subset \widetilde{\mathbb{Z}}^d$, we denote by $A\xleftrightarrow{(D)} A'$ (resp. $A\xleftrightarrow{} A'$) the event that there exists a path within $\cup \widetilde{\mathcal{L}}_{1/2}^{D}$ (resp. $\cup \widetilde{\mathcal{L}}_{1/2}$) connecting $A$ and $A'$. Moreover, we denote by $\mathcal{C}^{D}_{A}:=\{v\in \widetilde{\mathbb{Z}}^d:v\xleftrightarrow{(D)}A\}$ the union of clusters of $\cup \widetilde{\mathcal{L}}_{1/2}^D$ intersecting $A$. By the thinning property of Poisson point processes, conditioned on the event $\{\mathcal{C}^{D}_{A}=F\}$ (for any proper configuration $F$ of $\mathcal{C}^{D}_{A}$), the point measure $\widetilde{\mathcal{L}}_{1/2}^{D}\cdot \mathbbm{1}_{\mathrm{ran}(\widetilde{\ell})\cap \mathcal{C}^{D}_{A}=\emptyset}$, consisting of loops that are not used in the construction of $\mathcal{C}^{D}_{A}$, has the same distribution as $\widetilde{\mathcal{L}}_{1/2}^{D\cup F}$. This is commonly known as the ``restriction property'' of the loop soup. (P.S. When the superscript $D$ is $\emptyset$, it may be omitted. The same convention applies to notations introduced later.)

        \textbf{Isomorphism theorem.}  \cite[Proposition 2.1]{lupu2016loop} shows that for any $D\subset \widetilde{\mathbb{Z}}^d$, there exist a coupling between $\widetilde{\mathcal{L}}_{1/2}^{D}$ and $\widetilde{\phi}_\cdot\sim \mathbb{P}^D$ satisfying the following properties:
   \begin{enumerate}

   	\item[(i)]  For any $v\in \widetilde{\mathbb{Z}}^d\setminus D$, the total local time at $v$ of loops in $\widetilde{\mathcal{L}}_{1/2}^{D}$, denoted by $\widehat{\mathcal{L}}_{1/2}^{D,v}$, equals to $\frac{1}{2}\widetilde{\phi}_v^2$;

   	\item[(ii)]   Each sign cluster of $\widetilde{\phi}_\cdot$ is exactly a loop cluster of $\widetilde{\mathcal{L}}_{1/2}^{D}$.

   \item[(iii)] On each loop cluster of $\widetilde{\mathcal{L}}_{1/2}^{D}$, the sign of $\widetilde{\phi}_\cdot$ is independently and uniformly chosen from $\{+,-\}$.

   \end{enumerate}
   Here we list some useful corollaries of this isomorphism theorem: 
    \begin{enumerate}

    	\item For any $D\subset \widetilde{\mathbb{Z}}^d$ and any $A_i^{(1)},A_i^{(2)}\subset \widetilde{\mathbb{Z}}^d$ for $1\le i\le k$ ($k\in \mathbb{N}^+$), we denote the events $\mathsf{A}:=\cap_{1\le i\le k} \big\{A_i^{(1)}\xleftrightarrow{\ge 0}  A_i^{(2)}\big\}$ and $\overline{\mathsf{A}}^D:=\cap_{1\le i\le k} \big\{A_i^{(1)}\xleftrightarrow{(D)}  A_i^{(2)}\big\}$. The isomorphism theorem then implies that 
    	  \begin{equation}\label{coro_ineq_iso}
    	2^{-k}\cdot  \mathbb{P}\big(\overline{\mathsf{A}}^D \big) 	\le  \mathbb{P}^D\big(\mathsf{A} \big) \le \mathbb{P}\big(\overline{\mathsf{A}}^D\big). 
    	  \end{equation}
    	Here the second bound follows immediately from Property (ii). To see the first bound, observe that on $\overline{\mathsf{A}}^D$, there exists a collection of at most $k$ loop clusters of $\widetilde{\mathcal{L}}_{1/2}^{D}$ that can certify $\overline{\mathsf{A}}^D$. Thus, for any realization of $\widetilde{\mathcal{L}}_{1/2}^{D}$ under which $\overline{\mathsf{A}}^D$ occurs, one can force the signs of $\widetilde{\phi}_\cdot$ on at most $k$ loop cluster to be $+$ (whose probability is at least $2^{-k}$ by Property (iii)) so that $\mathsf{A}$ holds.

       	\item  For any $D_1\subset D_2\subset \widetilde{\mathbb{Z}}^d$ and $A,A'\subset  \widetilde{\mathbb{Z}}^d$, since $\widetilde{\mathcal{L}}_{1/2}^{D}$ is decreasing in $D$,   
       	\begin{equation}
       		\mathbb{P}\big(A\xleftrightarrow{(D_2)} A' \big)\le \mathbb{P}\big(A\xleftrightarrow{(D_1)} A' \big). 
       	 \end{equation}
       	Combined with (\ref{coro_ineq_iso}), it implies that 
       	 \begin{equation}\label{315}
    	 	\mathbb{P}^{D_2}\big(A\xleftrightarrow{\ge 0} A' \big)\le 2\mathbb{P}^{D_1}\big(A\xleftrightarrow{\ge 0} A' \big). 
    	 \end{equation}

    \end{enumerate}

Next, we estimate the decay rate of the local time on a large loop cluster.

\begin{lemma}\label{lemma25}
	For any $d\ge 3$, there exist $C,c>0$ such that for any $T>0$, $D\subset \widetilde{\mathbb{Z}}^d$ and $v,w\in \widetilde{\mathbb{Z}}^d\setminus D$ with $|v-w|\ge d$, 
	\begin{equation}\label{ineq_lemma_25}
		\mathbb{P}\big( \widehat{\mathcal{L}}_{1/2}^{D,v}\ge T \mid  v\xleftrightarrow{(D)} w \big) \le Ce^{-cT}. 
	\end{equation}
\end{lemma}
\begin{proof}
 By the isomorphism theorem, it suffices to prove that for any $K>0$,  
	\begin{equation}\label{vg37}
		\mathbb{P}^D\big( \widetilde{\phi}_v\ge K \mid  v\xleftrightarrow{\ge 0} w \big) \le Ce^{-cK^2}. 
	\end{equation}
In fact, since $\widetilde{\phi}_v\sim N\big(0,\widetilde{G}_D(v,v)\big)$, $\mathbb{P}^D\big( \widetilde{\phi}_v\ge K , v\xleftrightarrow{\ge 0} w \big)$ can be written as  
\begin{equation}
	\begin{split} 
	 &\int_{a\ge K}  \mathbb{P}^D\big(  v\xleftrightarrow{\ge 0} w \mid  \widetilde{\phi}_v =a \big) \cdot \tfrac{1}{\sqrt{2\pi \widetilde{G}_{D}(v,v)}}e^{-\frac{a^2}{2\widetilde{G}_D(v,v)}} \mathrm{d}a \\
	\overset{(\text{Lemma}\ \ref{lemma_conditional_two_points})}{\lesssim} & \mathbb{P}^{D}(v\xleftrightarrow{\ge 0} w )  \int_{a\ge K}  \tfrac{a }{ \widetilde{G}_{D}(v,v)}\cdot  e^{-\frac{a^2}{2\widetilde{G}_D(v,v)}} \mathrm{d}a \\
	\lesssim  &\mathbb{P}^{D}(v\xleftrightarrow{\ge 0} w )  \cdot e^{-\frac{K^2}{2\widetilde{G}_D(v,v)}} 
	\overset{(\widetilde{G}_D(v,v)\lesssim 1)}{\lesssim }   \mathbb{P}^D(v\xleftrightarrow{\ge 0} w)\cdot e^{-cK^2}. 
		\end{split}
\end{equation}
This yields (\ref{vg37}) and thus completes the proof. 
\end{proof}

    \textbf{Excursion decomposition of loops.} For any path $\widetilde{\eta}:[0,T]\to \widetilde{\mathbb{Z}}^d$ ($T\in \mathbb{R}^+$) and any $0\le t_1<t_2\le T$, we denote by $\widetilde{\eta}[t_1,t_2]$ the sub-path satisfying
\begin{equation}
	\widetilde{\eta}[t_1,t_2](s):= \widetilde{\eta}(t_1+s), \ \forall 0\le s\le t_2-t_1.
\end{equation}

 For any $v\in \widetilde{\mathbb{Z}}^d$ and any loop $\widetilde{\ell}$ intersecting $v$, arbitrarily take a rooted loop $\widetilde{\varrho}\in \widetilde{\ell}$ with $\widetilde{\varrho}(0)=\widetilde{\varrho}(T)=v$ (where $T$ is the duration of $\widetilde{\ell}$). For any $t\in (0,T)$, let $\delta_t^{-}=\delta_t^{-}(\widetilde{\varrho},v):=\sup\{s\le t:\widetilde{\varrho}(s)=v\} $ and $\delta_t^{+}=\delta_t^{+}(\widetilde{\varrho},v):=\inf\{s\ge t:\widetilde{\varrho}(s)=v\}$. We define $\mathscr{E}_v(\widetilde{\varrho})$ as the point measure supported on $\{\widetilde{\varrho}[\delta_t^{-},\delta_t^{+}]:0<t<T,\delta_t^{-}<\delta_t^{+}\}$. Since $\mathscr{E}_v(\widetilde{\varrho})$ is invariant under the selection of $\widetilde{\varrho}\in \widetilde{\ell}$, we also denote $\mathscr{E}_v(\widetilde{\varrho})$ by $\mathscr{E}_v(\widetilde{\ell})$. For completeness, we set $\mathscr{E}_v(\widetilde{\ell}):=0$ when $\widetilde{\ell}$ does not intersect $v$.

Referring to \cite[Section 1]{werner2025switching}, for any $D\subset \widetilde{\mathbb{Z}}^d$, $v\in \widetilde{\mathbb{Z}}^d\setminus D$ and $a>0$, conditioned on the event $\big\{\widehat{\mathcal{L}}_{1/2}^{D,v}=a\big\}$, the point measure $\widetilde{\mathcal{E}}_{v}^D:= \sum\nolimits_{\widetilde{\ell}\in \widetilde{\mathcal{L}}^{D}_{1/2}} \mathscr{E}_v(\widetilde{\ell})$ is a Poisson point process consisting of excursions from $v$. (This property of the loop soup is derived from the rewiring property \cite{werner2016spatial}, which is specific to the intensity $\frac{1}{2}$.) Moreover, its intensity measure can be written as $a\cdot \mathbf{e}_{v,v}^{D}$, where the ``Brownian excursion measure'' $\mathbf{e}_{v,v}^{D}$ is an infinite but $\sigma$-finite measure. Notably, similar to the well-known ``It\^o's excursion measure'', it is not convenient to construct $\mathbf{e}_{v,v}^{D}$ directly. For any distinct $v, w\in \widetilde{\mathbb{Z}}^d$, the measure $\mathbf{e}_{v,w}^{D}$ on the space of excursions from $v$ to $w$ can be obtained by first restricting $\mathbf{e}_{v,v}^{D}$ to the paths that hits $w$, and then keeping only the part of each path until the first time it hits $w$. Note that $\mathbf{e}_{v,w}^{D}$ is a finite measure when $v\neq w$. In addition, $\mathbf{e}_{v,w}^{D}$ (for both the cases $v=w$ and $v\neq w$) satisfies the strong Markov property. I.e., for an excursion $\widetilde{\eta}$ sampled from $\mathbf{e}_{v,w}^{D}$ and for any $A\subset \widetilde{\mathbb{Z}}^d\setminus (D\cup \{v,w\})$, given $\{\widetilde{\eta}(t)\}_{0\le t\le \tau_{A}(\widetilde{\eta})}$ (suppose that $\widetilde{\eta}(\tau_{A}(\widetilde{\eta}))=z$), the law of the remaining part of $\widetilde{\eta}$ is given by
\begin{equation}
 	\widetilde{\mathbb{P}}_z\big( \big\{\widetilde{S}_t \big\}_{0 \le t\le \tau_{w}} \in \cdot  \mid \tau_{D\cup \{v,w\}}=\tau_{w}<\infty \big). 
 \end{equation}
 Moreover, according to the restriction property, $\mathbf{e}_{v,w}^{D}$ can be obtained by restricting $\mathbf{e}_{v,w}:=\mathbf{e}_{v,w}^{\emptyset}$ to excursions that stay disjoint from $D$.

For more details on excursion measures, readers may refer to \cite[Chapter XII]{revuz2013continuous} or \cite[Chapters III and IV]{blumenthal2012excursions}. Subsequently, we present a relation between the total mass of the Brownian excursion measure and the boundary excursion kernel. Although its analogue for one-dimensional Brownian motions is a classical result, we present a proof in the context of metric graphs for the sake of completeness.

\begin{lemma}\label{lemma_total_measure_K}
	For any $d\ge 3$, $D\subset \widetilde{\mathbb{Z}}^d$ and $v\neq w\in  \widetilde{\mathbb{Z}}^d\setminus D$, the total mass of the measure $\mathbf{e}_{v,w}^{D}$ equals to $\mathbb{K}_{D\cup \{v,w\}}(v,w)$. Consequently,
		\begin{equation}
		\bar{\mathbf{e}}_{v,w}^{D}(\cdot):= [\mathbb{K}_{D\cup \{v,w\}}(v,w)]^{-1}\cdot \mathbf{e}_{v,w}^{D}(\cdot)
	\end{equation}
	defines a probability measure.  
\end{lemma}
\begin{proof}
By the strong Markov property of $\mathbf{e}_{v,v}^{D}$, the measure $\mathbf{e}^{D}_{v,v}\big(\big\{\widetilde{\eta}:w\in \mathrm{ran} (\widetilde{\eta})\big\}\big)$ equals the product of the total mass of $\mathbf{e}^{D}_{v,w}$ and the probability $\widetilde{\mathbb{P}}_w(\tau_v<\tau_D)$. Thus, recalling (\ref{24}), it suffices to show that for all sufficiently small $\epsilon>0$, 
	\begin{equation}\label{221}
	\begin{split}
			&\mathbf{e}^{D}_{v,v}\big(\big\{\widetilde{\eta}:w\in \mathrm{ran} (\widetilde{\eta})\big\}\big)\\
			=& \widetilde{\mathbb{P}}_w(\tau_v<\tau_D)\cdot  \big(\tfrac{1}{2\epsilon}+o(\epsilon^{-1}) \big) \sum\nolimits_{z\in \widetilde{\partial}\mathbf{B}_v(\epsilon)}\widetilde{\mathbb{P}}_z\big( \tau_{D\cup \{v,w\}} = \tau_w<\infty\big).
	\end{split}
		\end{equation}
	To see this, for any sufficiently small $\epsilon>0$ satisfying $\mathbf{B}_v(\epsilon)\cap D=\emptyset$, and any $z\in \widetilde{\partial}\mathbf{B}_v(\epsilon)$, it follows from the strong Markov property of $\mathbf{e}_{v,v}^D$ that  
	\begin{equation}\label{222}
	\begin{split}
		 	\frac{	\mathbf{e}^{D}_{v,v}\big(\big\{ \widetilde{\eta}:w\in \mathrm{ran} (\widetilde{\eta}),  \tau_{\widetilde{\partial}\mathbf{B}_v(\epsilon)}(\widetilde{\eta}) =\tau_{z} (\widetilde{\eta}) \big\}\big) }{ 	\mathbf{e}^{D}_{v,v}\big(\big\{ \widetilde{\eta}:  \tau_{\widetilde{\partial}\mathbf{B}_v(\epsilon)}(\widetilde{\eta}) =\tau_{z} (\widetilde{\eta}) \big\}\big)  }  
			= \frac{\widetilde{\mathbb{P}}_z\big( \tau_w<\tau_v<\tau_D \big)}{\widetilde{\mathbb{P}}_z\big(\tau_v<\tau_D \big)}.  
	\end{split}
	\end{equation}
	Note that $\widetilde{\mathbb{P}}_z\big(\tau_v<\tau_D \big)$ converges to $1$ as $\epsilon\to 0$. Meanwhile, by the strong Markov property of Brownian motion, one has $\widetilde{\mathbb{P}}_z\big( \tau_w<\tau_v<\tau_D \big)=\widetilde{\mathbb{P}}_z\big(\tau_{D\cup \{v,w\}}=\tau_w<\infty \big)\cdot \widetilde{\mathbb{P}}_w\big(\tau_{v} < \tau_{D} \big)$. Plugging these two facts into (\ref{222}), we obtain that
	\begin{equation}\label{223}
		\begin{split}
			&\mathbf{e}^{D}_{v,v}\big(\big\{ \widetilde{\eta}:w\in \mathrm{ran} (\widetilde{\eta}) \big)= \widetilde{\mathbb{P}}_w(\tau_v<\tau_D)\cdot   (1+o(1)) \\
			 &\ \ \ \ \ \ \ \ \cdot \sum\nolimits_{z\in \widetilde{\partial}\mathbf{B}_v(\epsilon)} \widetilde{\mathbb{P}}_z\big(\tau_{D\cup \{v,w\}}=\tau_w<\infty \big)\cdot \mathbf{e}^{D}_{v,v}\big(\big\{ \widetilde{\eta}:  \tau_{\widetilde{\partial}\mathbf{B}_v(\epsilon)}(\widetilde{\eta}) =\tau_{z} (\widetilde{\eta}) \big\}\big), 
		\end{split}
	\end{equation}
where the $o(1)$ term depends only on $\epsilon$. Moreover, \cite[Theorem (1.10) in Chapter VI]{revuz2013continuous} shows that $\epsilon$ times the number of Brownian excursions from $v$ to $z$ converges to half of the total local time at $v$ in the $L_p$ norm for all $p\ge 1$, as $\epsilon \to 0$. Meanwhile, recall that conditioned on $\big\{ \widehat{\mathcal{L}}_{1/2}^{D,v}=1\big\}$, the number of excursions in $\widetilde{\mathcal{E}}_{v}^{D}$ hitting $z$ is a Poisson random variable with parameter $\mathbf{e}^{D}_{v,v}\big(\big\{ \widetilde{\eta}:  \tau_{\widetilde{\partial}\mathbf{B}_v(\epsilon)}(\widetilde{\eta}) =\tau_{z} (\widetilde{\eta}) \big\}\big)$. Putting these two observations together, we conclude that 
\begin{equation}\label{224}
	\mathbf{e}^{D}_{v,v}\big(\big\{ \widetilde{\eta}:  \tau_{\widetilde{\partial}\mathbf{B}_v(\epsilon)}(\widetilde{\eta}) =\tau_{z} (\widetilde{\eta}) \big\}\big) = \tfrac{1}{2\epsilon}+ o(\epsilon^{-1}). 
\end{equation} 
Plugging (\ref{224}) into (\ref{223}), we obtain (\ref{221}) and thus complete the proof.
\end{proof}

  Based on Lemmas \ref{lemma21} and \ref{lemma_total_measure_K}, we derive the following bound on the conditional probability that two nearby points are not connected by the loop cluster.

 \begin{lemma}\label{lemma_connect_close_points}
	For any $d\ge 3$, there exists $c>0$ such that for any $a>0$, $v\neq w\in \widetilde{\mathbb{Z}}^d$ with $|v-w|\le 1$, and $D\subset \widetilde{\mathbb{Z}}^d$ with $\mathrm{dist}(D,w)\ge |v-w|$,
	\begin{equation}\label{3.12_3.69}
	   	\mathbb{P}^{D}\big( \{ v\xleftrightarrow{\ge 0} w \}^c \mid \widetilde{\phi}_v=a \big) \le  e^{-ca^2|v-w|^{-1}}. 
	\end{equation}
\end{lemma}
\begin{proof}
	By the isomorphism theorem, it suffices to show that for any $T>0$
	\begin{equation}
		   	\mathbb{P}\big( \{ v\xleftrightarrow{(D)} w \}^c \mid \widetilde{\mathcal{L}}_{1/2}^{D,v}= T  \big) \le Ce^{-cT|v-w|^{-1}}.
	\end{equation}
	To see this, note that on the event $\{v\xleftrightarrow{(D)} w \}^c$, no excursion in $\widetilde{\mathcal{E}}^{D}_v$ reaches $w$. In addition, conditioned on $\{ \widetilde{\mathcal{L}}_{1/2}^{D,v}= T\}$, the number of excursions in $\widetilde{\mathcal{E}}^{D}_v$ intersecting $w$ is a Poisson random variable with parameter 
	\begin{equation}\label{3.11_3.71}
		\begin{split}
			&	T\cdot \mathbf{e}^{D}_{v,v}(\{\widetilde{\eta}:w\in \mathrm{ran}(\widetilde{\eta})\}) \\
				 \overset{(\text{Lemma}\ \ref{lemma_total_measure_K})}{=}&  T\cdot \mathbb{K}_{D\cup \{v,w\}}(v,w) \cdot \widetilde{\mathbb{P}}_w\big( \tau_{v}<\tau_{D} \big)  \overset{(\ref{new2.1}),(\ref{conduct_line})}{\gtrsim } T\cdot |v-w|^{-1}. 
		\end{split}
	\end{equation}
	Consequently, we obtain 
	\begin{equation}
		\begin{split}
				&\mathbb{P}\big( \{ v\xleftrightarrow{(D)} w \}^c \mid \widetilde{\mathcal{L}}_{1/2}^{D,v}= T  \big)\\
				 \le & \mathbb{P}\big( \{ \widetilde{\eta}\in \widetilde{\mathcal{E}}^{D}_v: w\in \mathrm{ran}(\widetilde{\eta})   \} =\emptyset \mid \widetilde{\mathcal{L}}_{1/2}^{D,v}= T  \big) 
				\le  e^{-cT\cdot |v-w|^{-1}}.  \qedhere
		\end{split}
	\end{equation}

\end{proof}

Next, we present two estimates on the total mass of certain excursions.

 \begin{lemma}\label{lemma_excursion_hitting_point}
	For any $d\ge 3$ and $v_1,v_2,v_3\in \widetilde{\mathbb{Z}}^d$ with $|v_i-v_j|\ge d$ for all $i\neq j$, 
	\begin{equation}
		\bar{\mathbf{e}}_{v_1,v_2}(\{\widetilde{\eta}: v_3\in \mathrm{ran}(\{\widetilde{\eta}) \}) \lesssim \tfrac{|v_1-v_3|^{2-d}|v_3-v_2|^{2-d}}{|v_1-v_2|^{2-d}}. 
	\end{equation} 
\end{lemma} 
\begin{proof}
Similar to (\ref{222}), we have: for any $z\in \partial \mathbf{B}_{v_1}(1)$, 
\begin{equation}
		 	\frac{	\mathbf{e}_{v_1,v_2}\big(\big\{ \widetilde{\eta}:v_3\in \mathrm{ran} (\widetilde{\eta}),  \tau_{\widetilde{\partial}\mathbf{B}_{v_1}(1)}(\widetilde{\eta}) =\tau_{z} (\widetilde{\eta}) \big\}\big) }{ 	\mathbf{e}_{v_1,v_2}\big(\big\{ \widetilde{\eta}:  \tau_{\widetilde{\partial}\mathbf{B}_{v_1}(1)}(\widetilde{\eta}) =\tau_{z} (\widetilde{\eta}) \big\}\big)  }  
			= \frac{\widetilde{\mathbb{P}}_{z}\big( \tau_{v_3}<\tau_{v_2}<\tau_{v_1} \big)}{\widetilde{\mathbb{P}}_{z}\big(\tau_{v_2}<\tau_{v_1} \big)}. 
\end{equation}
Note that $	\mathbf{e}_{v_1,v_2}\big(\big\{ \widetilde{\eta}:  \tau_{\widetilde{\partial}\mathbf{B}_{v_1}(1)}(\widetilde{\eta}) =\tau_{z} (\widetilde{\eta}) \big\}\big)\asymp 1$ and $\widetilde{\mathbb{P}}_{z}\big(\tau_{v_2}<\tau_{v_1} \big)\asymp |v_1-v_2|^{2-d}$. In addition, by the strong Markov property of $\widetilde{S}_\cdot\sim \widetilde{\mathbb{P}}_{z}$, one has 
\begin{equation*}
	\widetilde{\mathbb{P}}_{z}\big( \tau_{v_3}<\tau_{v_2}<\tau_{v_1} \big)\le  \widetilde{\mathbb{P}}_{z}\big( \tau_{v_3}<\infty  \big)\cdot \widetilde{\mathbb{P}}_{v_3}\big(  \tau_{v_2}<\infty \big) \overset{}{\asymp} |v_1-v_3|^{2-d}|v_3-v_2|^{2-d}. 
\end{equation*}
 Combining these estimates, we obtain this lemma.   
\end{proof}

\begin{lemma}\label{lemma_new216}
	For any $d\ge 3$, $v\in \widetilde{\mathbb{Z}}^d$ and $N\ge 1$, 
	\begin{equation}\label{newadd..279}
		\mathbf{e}_{v,v}\big(\big\{\widetilde{\eta}: \mathrm{ran}(\widetilde{\eta})\cap \partial B(N)\neq \emptyset \big\} \big)\lesssim N^{2-d}. 
	\end{equation}
\end{lemma}
 \begin{proof}
As in (\ref{222}), for any $z\in \partial \mathbf{B}_v(1)$, we have   
 	\begin{equation*}
 			\frac{	\mathbf{e}_{v,v}\big(\big\{ \widetilde{\eta}:\mathrm{ran}(\widetilde{\eta})\cap \partial B(N)\neq \emptyset,  \tau_{\widetilde{\partial}\mathbf{B}_{v}(1)}(\widetilde{\eta}) =\tau_{z} (\widetilde{\eta}) \big\}\big) }{ 	\mathbf{e}_{v ,v}\big(\big\{ \widetilde{\eta}:  \tau_{\widetilde{\partial}\mathbf{B}_{v}(1)}(\widetilde{\eta}) =\tau_{z} (\widetilde{\eta}) \big\}\big)  }  
			= \frac{\widetilde{\mathbb{P}}_{z}\big( \tau_{\partial B(N)}<\tau_{v}<\infty \big)}{\widetilde{\mathbb{P}}_{z}\big(\tau_{v}<\infty \big)}. 
 	\end{equation*}
 	Here $	\mathbf{e}_{v ,v}\big(\big\{ \widetilde{\eta}:  \tau_{\widetilde{\partial}\mathbf{B}_{v}(1)}(\widetilde{\eta}) =\tau_{z} (\widetilde{\eta}) \big\}\big) $ and $ \widetilde{\mathbb{P}}_{z}\big(\tau_{v}<\infty \big)$ are both of order $1$. Moreover, by the strong Markov property of $\widetilde{S}_\cdot\sim \widetilde{\mathbb{P}}_{z}$, one has 
 	\begin{equation}
 		\widetilde{\mathbb{P}}_{z}\big( \tau_{\partial B(N)}<\tau_{v}<\infty \big)\le \max_{x\in \partial B(N)} \widetilde{\mathbb{P}}_x\big(\tau_v<\infty \big) \lesssim N^{2-d}. 
 	\end{equation}
 Putting these estimates together, we obtain the desired bound (\ref{newadd..279}).  \end{proof}

\subsection{Werner's switching identity}

In this subsection, we review a recent major progress in the study of the loop soup clusters. Precisely, the main result in \cite{werner2025switching} (referred to as the ``switching identity'') provides a workable description of the loop cluster that connects two given points. When restricted to $\widetilde{\mathbb{Z}}^d$, it yields: 

  \begin{lemma}[switching identity]\label{lemma_switching}
 	 For any $d\ge 3$, $D\subset \widetilde{\mathbb{Z}}^d$, $v\neq w\in \widetilde{\mathbb{Z}}^d\setminus D$, and $a,b>0$, conditioned on $\big\{ v\xleftrightarrow{} w,  \widehat{\mathcal{L}}_{1/2}^{D,v}=a, \widehat{\mathcal{L}}_{1/2}^{D,w}=b  \big\}$, the occupation field $\big\{ \widehat{\mathcal{L}}_{1/2}^{D,z} \big\}_{z\in \widetilde{\mathbb{Z}}^d}$ has the same distribution as the total local time of the following four independent components:
 	 \begin{itemize}
 	 	\item[(1)] loops in the loop soup $\widetilde{\mathcal{L}}_{1/2}^{D\cup \{v,w\}}$;

 	 	\item[(2)] a Poisson point process with intensity measure $a\cdot \mathbf{e}_{v,v}^{ D\cup  \{w\}}$;

 	 	\item[(3)] a Poisson point process with intensity measure $b \cdot \mathbf{e}_{w,w}^{ D\cup  \{v\}}$;

 	 	\item[(4)] a Poisson point process with intensity measure $\sqrt{ab} \cdot \mathbf{e}_{v,w}^D$, where the number of excursions is conditioned to be odd.

 	 \end{itemize}

 \end{lemma}

%
%
%
%
%
%
%
%
%
%

 To facilitate the use of this switching identity, we record the following notations.  \begin{itemize}

 	\item  We denote the conditional probability measure 
 	\begin{equation}\label{density_p}
 	\widehat{\mathbb{P}}_{v\leftrightarrow w,a,b}^D(\cdot)	:=\mathbb{P}\big(\cdot \mid v\xleftrightarrow{(D)} w,  \widehat{\mathcal{L}}_{1/2}^{D,v}=a, \widehat{\mathcal{L}}_{1/2}^{D,w}=b  \big). 
 	\end{equation}
 	Let $\mathfrak{p}_{v\leftrightarrow w,a,b}^D$ denote the density (at the point $(a, b)$) of the event $\big\{v\xleftrightarrow{(D)} w,  \widehat{\mathcal{L}}_{1/2}^{D,v}=a, \widehat{\mathcal{L}}_{1/2}^{D,w}=b\big\}$. By the isomorphism theorem and (\ref{formula_two_point}), one has 
 	 	\begin{equation}\label{250}
 		\mathfrak{p}_{v\leftrightarrow w,a,b}=(ab)^{-\frac{1}{2}}\cdot \mathfrak{g}_{v,w}(\sqrt{2a},\sqrt{2b} )\cdot  \big(1-e^{-4\sqrt{ab}\cdot \mathbb{K}_{ \{v,w\}}(v,w)}\big),
 	\end{equation}
 	where $\mathfrak{g}_{v,w}(\cdot,\cdot)$ is the density of $(\widetilde{\phi}_{v},\widetilde{\phi}_{w})$ (see e.g., \cite[Page 329]{whittaker1924calculus}), i.e., 
 	\begin{equation}\label{251}
 		\mathfrak{g}_{v,w}(s,t)= \tfrac{\mathrm{exp}\big( -\frac{1}{2(1-\delta^2)} \cdot \big[\frac{s^2}{\widetilde{G}(v,v)}-\frac{2\delta st}{\sqrt{\widetilde{G}(v,v)\widetilde{G}(w,w)}}+\frac{t^2}{\widetilde{G}(w,w)} \big] \big)}{2\pi \sqrt{\widetilde{G}(v,v)\widetilde{G}(w,w)(1- \delta^2)}},\ \forall s,t\ge 0. 
 	\end{equation}
 		Here we denote $\delta:=\frac{\widetilde{G}(v,w)}{\sqrt{\widetilde{G}(v,v)\widetilde{G}(w,w)}}$.


 	\item We define $\widecheck{\mathbb{P}}_{v\leftrightarrow w,a,b}^D(\cdot)$ as the probability measure generated by Components (1)-(4) in Lemma \ref{lemma_switching} (we also denote its expectation by $\widecheck{\mathbb{E}}_{v\leftrightarrow w,a,b}^D[\cdot]$). Under the measure $\widecheck{\mathbb{P}}^D_{v\leftrightarrow w,a,b}(\cdot)$, for each $i\in \{1,2,3,4\}$, we denote the Poisson point process in Component $(i)$ by $\mathcal{P}^{(i)}$. We define 
    \begin{equation}
    	\widecheck{\mathcal{C}}:=\big\{ z\in \widetilde{\mathbb{Z}}^d: z\xleftrightarrow{\cup  (\sum_{1\le i\le 4 }\mathcal{P}^{(i)})} v \big\}     \end{equation}
 	 as the cluster containing $v$, formed by all loops and excursions in $\sum_{1\le i\le 4} \mathcal{P}^{(i)}$.

 \end{itemize}

When $D=\emptyset$, we drop the superscript $\emptyset$ in these notations and write $\widehat{\mathcal{L}}_{1/2}^{v}$, $\widehat{\mathbb{P}}_{v\leftrightarrow w,a,b}$, $\mathfrak{p}_{v\leftrightarrow w,a,b}$, $\widecheck{\mathbb{P}}_{v\leftrightarrow w,a,b}$ and $\widecheck{\mathbb{E}}_{v\leftrightarrow w,a,b}$.

 The following fact will be useful in the analysis of $\mathcal{P}^{(4)}$. Let $\mathbf{X}$ be a Poisson random variable with parameter $0<\lambda \lesssim 1$, and let $\{\mathbf{q}_j\}_{j\in \mathbb{N}}$ be a family of i.i.d. Bernoulli random variables with expectation $q$. Define $\mathbf{Y}:=\sum_{1\le j\le \mathbf{X}}\mathbf{q}_j$. Then  
 \begin{equation}\label{property_poisson_odd_number}
 \begin{split}
 	 \mathbb{P}\big( \mathbf{Y}>0 \mid \mathbf{X}\ \text{is odd} \big)=&  1-  \frac{\sum_{k\ge 1: k\ \text{odd}}  e^{-\lambda}\cdot \frac{\lambda^k}{k!} \cdot (1-q)^k  }{\sum_{k\ge 1: k\ \text{odd}}  e^{-\lambda}\cdot \frac{\lambda^k}{k!} }\\
 	 =&\frac{\sinh(\lambda)- \sinh(\lambda(1-q))}{\sinh(\lambda)}\asymp q.
 \end{split}
 \end{equation}

%
%

 \section{Regularity of coexistence}\label{section_regularity_co_exist}

 This section is devoted to establishing a series of regularity properties of  connecting probabilities involving two clusters in the low-dimensional case (i.e., $3\le d\le 5$), which are not only useful for deriving the main results but also interesting in their own right. For convenience, we formulate them in terms of the loop soup to avoid dealing with signs. We first introduce some notations as follows:  
 \begin{itemize}

 	\item For any events $\mathsf{A}_1$ and $\mathsf{A}_2$ measurable with respect to $\widetilde{\mathcal{L}}_{1/2}$, we define $\mathsf{A}_1 \Vert  \mathsf{A}_2$ as the event that $\mathsf{A}_1$ and $\mathsf{A}_2$ are certified by different clusters of $\cup \widetilde{\mathcal{L}}_{1/2}$.

 	\item For any $n\ge 1$ and $N\ge 10n$, let $\Psi(n,N)$ denote the collection of quadruples $(v_1,v_2,w_1,w_2)$ such that $v_1,v_2\in \widetilde{B}(10 n)\setminus \widetilde{B}(n)$ with $|v_1-v_2|\ge  n$, and $w_1,w_2\in \widetilde{B}(10 N)\setminus \widetilde{B}(N)$ with $|w_1-w_2|\ge  N$.

 	\item For any $\psi=(v_1,v_2,w_1,w_2)$, we define the event 
  \begin{equation}\label{newadd328}
  	\mathsf{C}[\psi]:= \big\{v_1 \xleftrightarrow{} w_1 \Vert v_2 \xleftrightarrow{} w_2 \big\}. 
  \end{equation}

%
%

 \end{itemize}

 The first property shows that the probability of $\mathsf{C}[(v_1,v_2,w_1,w_2)]$ is stable under perturbations of  $v_1,v_2,w_1$ and $w_2$.

 \begin{lemma}\label{lemma_roots}
 	 For any $3\le d\le 5$, there exists a constant $\Cl\label{const_lemma_roots1}>0$ such that for any $n\ge \Cref{const_lemma_roots1}$, $N\ge \Cref{const_lemma_roots1}n$ and $\psi,\psi'\in \Psi(n,N)$,  	   	
 	  \begin{equation}\label{3.18_ineq_lemma4.1}
 	 	\mathbb{P}\big( \mathsf{C}[\psi]   \big)  \asymp \mathbb{P}\big( \mathsf{C}[\psi']  \big). 
 	 \end{equation} 
 \end{lemma}

For clarity of exposition, we state the remaining properties first and provide their proofs afterwards. The second property shows that the two coexisting clusters are unlikely to enter a specific small region or to reach a faraway area. We first fix some notations as follows. 
\begin{itemize}

	\item   For simplicity, when referring to $\psi=(v_1,v_2,w_1,w_2)\in \Psi(n,N)$, we denote 
	\begin{equation}
		z^{\mathrm{in}}_i:=v_i,\  z^{\mathrm{out}}_i:=w_i, \ m_{\mathrm{in}}:=n \  \text{and} \ m_{\mathrm{out}}:=N. 
	\end{equation}
	For example, we may denote the event $\mathsf{C}[\psi]$ by $\big\{    z_1^{\mathrm{in}} \xleftrightarrow{} z_1^{\mathrm{out}}   \Vert    z_2^{\mathrm{in}} \xleftrightarrow{} z_2^{\mathrm{out}}   \big\}$. Moreover, we also use ``$-\mathrm{in}$'' to represent ``$\mathrm{out}$'', and vice versa.

	\item For any $R\ge 1$ and $\delta>0$, let $\mathfrak{B}_{R,\delta}^{\mathrm{in}}:=\widetilde{B}(\delta R)$ and $\mathfrak{B}_{R,\delta}^{\mathrm{out}}:=[\widetilde{B}(\delta^{-1}R)]^{c}$.

	\item For any $\psi\in \Psi(n,N)$, $\delta \in (0,1)$, $\diamond\in \{\mathrm{in},\mathrm{out}\}$, $i\in \{ 1,2\}$ and $A\subset \widetilde{\mathbb{Z}}^d$, let 	\begin{equation}
		\begin{split}
			\mathsf{R}_i^{\diamond}[\psi, \delta, A]:= \big\{ z_i^{\diamond}\xleftrightarrow{} \cup_{x\in A} \widetilde{B}_x(\delta m_{\diamond}) \cup \mathfrak{B}_{m_{\diamond},\delta}^{\diamond} \big\}. 
		\end{split}
	\end{equation}
	 We then define $\mathsf{S}_i[\psi,\delta]:=\cup_{\diamond\in \{\mathrm{in},\mathrm{out}\}} \mathsf{R}_i^{\diamond}[\psi, \delta, \{z_{3-i}^{\diamond}\}]$, $\mathsf{S}[\psi,\delta]:=\cup_{i\in \{1,2\}}\mathsf{S}_i[\psi,\delta]$ and $\widehat{\mathsf{C}}[\psi,\delta]:=\mathsf{C}[\psi]\cap (\mathsf{S}[\psi,\delta])^c$.


\end{itemize}



 \begin{lemma}\label{lemma_separation}
	For any $3\le d\le 5$ and $\epsilon>0$, there exist $\Cl\label{const_lemma_separation1}(d,\epsilon),\cl\label{const_lemma_separation2}(d,\epsilon)>0$ such that for any $n\ge \Cref{const_lemma_separation1}$, $N\ge \Cref{const_lemma_separation1}n$, $\psi\in \Psi(n,N)$, $\diamond\in \{\mathrm{in},\mathrm{out}\}$, $i\in \{ 1,2\}$, and any line segment $\widetilde{\eta}\subset \mathfrak{B}_{m_{\diamond},100}^{\diamond}$ satisfying $\mathrm{vol}(\widetilde{\eta})\le 100m_\diamond$ and $\mathrm{dist}(z_i^{\diamond},\widetilde{\eta})\ge \frac{m_\diamond}{100}$, 
	 \begin{equation}\label{3.5_ineq_lemma_separation_1}
	 	\mathbb{P}\big( \mathsf{R}_i^{\diamond}[\psi, \cref{const_lemma_separation2}, \widetilde{\eta}] \mid    \mathsf{C}[\psi] \big) \le \tfrac{1}{4} \epsilon. 
	 \end{equation}
	Particularly, by taking $\widetilde{\eta}=\{z_{3-i}^\diamond\}$, this implies 
\begin{equation}
		\mathbb{P}\big( \mathsf{S}[\psi,\cref{const_lemma_separation2}] \mid    \mathsf{C}[\psi] \big)\le   \epsilon.
	\end{equation}
	Consequently, there exists $\cl\label{const_lemma_separation3}>0$ such that 
	\begin{equation}\label{3.5_ineq_lemma_separation_3}
		\mathbb{P}\big(     \widehat{\mathsf{C}}[\psi, \cref{const_lemma_separation3}] \big)\asymp  \mathbb{P}\big(  \mathsf{C}[\psi] \big), \ \forall \psi\in \Psi(n,N). 
	\end{equation}
\end{lemma}

 It has been proved in \cite[Lemma 5.3]{cai2024quasi} that for any $3\le d\le 5$ and $v,w,x\in \widetilde{\mathbb{Z}}^d$,
 \begin{equation}
	\mathbb{P}\big( v\xleftrightarrow{} x \mid v \xleftrightarrow{} w \big) \lesssim    [\mathrm{dist}(x,\{v ,w \})+1]^{-\frac{d}{2}+1}. 
\end{equation}
The following lemma provides an analogue of this estimate, with the conditioning on $\{v \xleftrightarrow{} w\}$ replaced by conditioning on an event of the form $\mathsf{C}[\psi]$.


\begin{lemma}\label{lemma_five_point}
	For any $3\le d\le 5$, there exists $\Cl\label{const_lemma_five_point1}>0$ such that for any $n\ge \Cref{const_lemma_five_point1}$, $N\ge \Cref{const_lemma_five_point1}n$, $\psi=(v_1,v_2,w_1,w_2)\in \Psi(n,N)$, $x\in \widetilde{\mathbb{Z}}^d$ and $i\in \{1,2\}$, 
	\begin{equation}\label{ineq_lemma_five_point_new}
			\mathbb{P}\big( v_i \xleftrightarrow{} x \mid \mathsf{C}[\psi] \big) \lesssim   [\mathrm{dist}(x,\{v_i,w_i\} )+1]^{-\frac{d}{2}+1}.  
	\end{equation}
\end{lemma}

The proofs of Lemmas \ref{lemma_roots}-\ref{lemma_five_point} mainly rely on the so-called ``separation lemma'', whose statement is quite plausible (and may even appear obvious at first glance)---conditioned on two random sets being disjoint, their extremities are significantly separated. In our application, the relevant sets are the ranges of two Brownian motions stopped upon reaching a fixed distance, under the conditioning that they are not connected by the loop soup (the switching identity allows us to relate two disjoint loop clusters to two such Brownian motions). The separation lemma has been extensively applied, primarily in the study of Brownian motion \cite{Lawler1998, lawler2002sharp, lawler2012fast}, random walk \cite{lawler1996hausdorff, masson2009growth, lawler2021convergence, 4966635c-2c6a-31bf-9599-a4e32c6a97f5} and Bernoulli percolation \cite{kesten1987scaling, smirnov2001critical, nolin2008near, schramm2010quantitative, garban2013pivotal, du2024sharp}. See also a recent application to percolation in two-dimensional loop soups \cite{gao2024percolation,gao2024percolation2}. In the existing percolation literature, proofs of the separation lemma typically require multi-scale analysis and, to some extent, rely on the planarity of the underlying graphs (which allows a single curve to separate one set from another). However, since $\widetilde{\mathbb{Z}}^d$ with $d\ge 3$ is no longer planar, such a property is absent in the context of this paper. Instead, we show that if two Brownian motions come sufficiently close, they are very likely to be connected by the loop soup. The key ingredients are to locate regions where the loop soup is dense, and to establish suitable independence between the loop soup within these regions and the conditioning event.

   We fix some notations before showing the statement of the separation lemma. 
   \begin{itemize}
   

   	\item  For any $D\subset \widetilde{\mathbb{Z}}^d$, $x_1\in \mathfrak{B}_{N,1/d}^{\mathrm{in}}\setminus D$, $x_{-1}\in \mathfrak{B}_{N,1/d}^{\mathrm{out}}\setminus D$ and $i\in \{1,-1\}$, let
   	\begin{equation}
   	\hat{\mathbf{e}}_{x_i,x_{-i}}^{D,N}(\cdot) := 	\bar{\mathbf{e}}_{x_i,x_{-i}}^D( \{\widetilde{\eta}(t)\}_{0\le t\le \tau_{\partial \mathcal{B}(N)}} \in \cdot ). 
   	\end{equation}
   	In particular, when $D=\emptyset$, we suppress the $\emptyset$ in the superscript and write $\hat{\mathbf{e}}_{x_i,x_{-i}}^{N}$ and $\hat{\mathbf{e}}_{x_i,x_{-i}'}^{N}$.


   	\item   For any possible configurations $\eta,\eta'$ under the law $\hat{\mathbf{e}}_{x_i,x_{-i}}^{D,N}$ (with endpoints $z$ and $z'$ respectively), we define  
   	\begin{equation}\label{3.9_4.17}
   		\mathcal{Q}(\eta,\eta'):= \sup \big\{ \delta>0:  \big[ \mathrm{ran}(\eta)\cup \widetilde{B}_z(\delta N)\big]\cap \big[\mathrm{ran}(\eta')\cup \widetilde{B}_{z'}(\delta N)\big] =\emptyset \big\}. 
   	\end{equation}

   	\item  For any $v,v',w,w'\in \widetilde{\mathbb{Z}}^d$ and $a,a',b,b'>0$, let $\mathcal{P}_{v,v',w,w'}^{a,a',b,b'}$ be the point process consisting of $\widetilde{\mathcal{L}}_{1/2}^{\{v,v',w,w'\}}$ and four independent Poisson point processes with intensity measures $a\cdot \mathbf{e}_{v,v}^{ \{w\}}$, $b \cdot \mathbf{e}_{w,w}^{  \{v\}}$, $a'\cdot  \mathbf{e}_{v',v'}^{ \{v,w,w'\}}$ and $b'\cdot  \mathbf{e}_{w',w'}^{ \{v,w,v'\}}$ respectively. When $a=a'=b=b'=0$, we simply write $\mathcal{P}_{v,v',w,w'}$ for $\mathcal{P}_{v,v',w,w'}^{a,a',b,b'}$.

%
%
%
%
%
%
%
%
%
%
%

   \end{itemize}

\begin{lemma}[separation lemma]\label{new_lemma_separation}
For any $3\le d \le 5$, there exist $\Cl\label{const_ls1},\Cl\label{const_ls4},\cl\label{const_ls2},\cl\label{const_ls3}>0$ such that for any $i\in \{\mathrm{in},\mathrm{out}\}$, $N\ge \Cref{const_ls1}$, $x_i,x_i'\in \mathfrak{B}_{N,\cref{const_ls2}}^{i}$, $x_{-i},x_{-i}',x_{-i}'',x_{-i}'''\in \mathfrak{B}_{N,\cref{const_ls2}}^{-i}$ and $a,a',b,b'>0$,
 \begin{equation}\label{2_27_4.15}
 \begin{split}
 	 	&\mathbb{P} \big(  \mathcal{Q}(\hat{\eta}, \hat{\eta}') \ge \cref{const_ls3} , \{ \mathrm{ran}(\hat{\eta}) \xleftrightarrow{\cup \mathcal{P}_{x_i,x_i',x_{-i}'',x_{-i}'''}^{a,a',b,b'}} \mathrm{ran}(\hat{\eta}')\}^c \big)  \\
 	 	\gtrsim & e^{-\Cref{const_ls4}(a+a'+b+b')(|x_i-x_i'|+1)^{2-d}}\cdot \mathbb{P}\big(   \{ \mathrm{ran}(\hat{\eta}) \xleftrightarrow{\cup \mathcal{P}_{x_i,x_i',x_{-i}'',x_{-i}'''}^{a,a',b,b'}} \mathrm{ran}(\hat{\eta}')\}^c \big), 
 \end{split}
 	\end{equation} 
 where on each side, $\hat{\eta}\sim 	\hat{\mathbf{e}}_{x_i,x_{-i}}^{N}$, $\hat{\eta}'\sim \hat{\mathbf{e}}_{x_i',x_{-i}'}^{\{x_i,x_{-i}\},N}$ and $ \mathcal{P}_{x_i,x_i',x_{-i}'',x_{-i}'''}^{a,a',b,b'}$ are independent. 
 \end{lemma}



  The remainder of this section is organized as follows. In Sections \ref{subsection_new_4.1}-\ref{lemma_five_point}, we establish Lemmas \ref{lemma_roots}-\ref{lemma_five_point} assuming Lemma \ref{new_lemma_separation}. The proof of Lemma \ref{new_lemma_separation} is then presented in Section \ref{subsection4.5}.


{\color{violet}

}

 {\color{red}
 
%
%
%
%
%
%

 }

\subsection{Proof of Lemma \ref{lemma_roots}}\label{subsection_new_4.1}

%
%
%
%
%
%
%
%
%
%

We start with the following preparatory lemma. 
 
 \begin{lemma}\label{lemma_four_point_local_time_bound}
 	For any $d\ge 3$, there exist $C,c>0$ such for any $T>0$, any distinct $v_1,v_2,w_1,w_2\in \widetilde{\mathbb{Z}}^d$ with $|v_1-w_1|\land |v_2-w_2|\ge d$, and any $z\in \{v_1,v_2,w_1,w_2\}$, 
 	 	\begin{equation}
 		\begin{split}
 			\mathbb{P}\big(  \widehat{\mathcal{L}}_{1/2}^{z} \ge T  \mid \{v_1\xleftrightarrow{} w_1\Vert v_2\xleftrightarrow{} w_2 \} \big)\le Ce^{-cT}. 
 		\end{split}
 	\end{equation}
 \end{lemma}
 \begin{proof}
 	 This is a direct corollary of the restriction property and Lemma \ref{lemma25}.
 \end{proof}

We now turn to the proof of Lemma \ref{lemma_roots}. Arbitrarily fix $\psi=(v_1,v_2,w_1,w_2)\in \Psi(n,N)$. For any $T>0$, we define 
\begin{equation}\label{3.15_4.14}
	\overline{\mathsf{C}}[\psi,T]:= \cap_{z\in \{v_1,v_2,w_1,w_2\}} \{0< \widehat{\mathcal{L}}_{1/2}^{z} \le T \} \cap  \mathsf{C}[\psi]. 
\end{equation}
By Lemma \ref{lemma_four_point_local_time_bound}, there exists $C_\dagger>0$ such that  
\begin{equation}\label{3.15_4.15}
	  \mathbb{P}\big(\overline{\mathsf{C}}[\psi,C_\dagger] \big) \asymp \mathbb{P}\big(  \mathsf{C}[\psi] \big). 
\end{equation}
According to Lemma \ref{lemma_switching}, given that $\{v_1 \xleftrightarrow{} w_1,  \widehat{\mathcal{L}}_{1/2}^{v_1}=a_1, \widehat{\mathcal{L}}_{1/2}^{w_1}=b_1 \}$ occurs, the conditional probability of $\overline{\mathsf{C}}[\psi,C_\dagger]$ is 
    \begin{equation}\label{2_26_4.16}
    	\begin{split}
    		\widecheck{\mathbb{P}}_{v_1\leftrightarrow w_1,a_1,b_1}\big( v_2\xleftrightarrow{\cup\mathcal{P}^{(1)}} w_2,   \big\{ v_2\xleftrightarrow{\cup \mathcal{P}^{(1)}} \cup \big(  \sum_{2\le i\le 4}\mathcal{P}^{(i)} \big)   \big\}^c , \widehat{\mathcal{L}}_{1/2}^{v_2}\land \widehat{\mathcal{L}}_{1/2}^{w_2}\le  C_\dagger \big).
    	\end{split}
    \end{equation}
   Applying Lemma \ref{lemma_switching} once again, we can rewrite the probability in (\ref{2_26_4.16}) as  
    \begin{equation}\label{3.1_4.19}
    	\begin{split}
    				\int_{0<a_2,b_2\le C_\dagger}   \mathbb{P}\big(   \{ \cup \mathfrak{L}_1   \xleftrightarrow{\cup \mathcal{P}} \cup \mathfrak{L}_2   \}^c  \big)\mathfrak{p}_{v_2\leftrightarrow w_2,a_2,b_2}^{\{v_1,w_1\}} \mathrm{d}a_2\mathrm{d}b_2.
    	\end{split}
    \end{equation}
  Here $\mathfrak{L}_1$ (resp. $\mathfrak{L}_2$) has the same law as $\mathcal{P}^{(4)}$ under $\widecheck{\mathbb{P}}_{v_1\leftrightarrow w_1,a_1,b_1}$ (resp. $\widecheck{\mathbb{P}}^{\{v_1,w_1\}}_{v_2\leftrightarrow w_2,a_2,b_2}$); moreover, $\mathfrak{L}_1$, $\mathfrak{L}_2$ and $\mathcal{P}:=\mathcal{P}_{v_1,v_2,w_1,w_2}^{a_1,a_2,b_1,b_2}$ are independent. Denote $\hat{N}:=\cref{const_ls2}N$. Since $\cup \mathfrak{L}_1$ (resp. $\cup \mathfrak{L}_2$) stochastically dominates $\hat{\eta}_1 \sim \hat{\mathbf{e}}_{v_1,w_1}^{\hat{N}}$ (resp. $\hat{\eta}_2 \sim \hat{\mathbf{e}}_{v_2,w_2}^{\{v_1,w_1\},\hat{N}}$),
  \begin{equation}\label{3.1_4.20}
  	\begin{split}
  		   \mathbb{P}\big(   \{ \cup \mathfrak{L}_1   \xleftrightarrow{\cup \mathcal{P}} \cup \mathfrak{L}_2   \}^c  \big)  
  		 \le  \mathbb{I}:=  \mathbb{P}\big(   \{ \mathrm{ran}(\hat{\eta}_1)     \xleftrightarrow{\cup \mathcal{P}} \mathrm{ran}(\hat{\eta}_2)  \}^c  \big).
  	\end{split}
  \end{equation}
 Let $\bar{\mathfrak{p}}:= \mathfrak{p}_{v_1\leftrightarrow w_1,a_1,b_1}\cdot \mathfrak{p}_{v_2\leftrightarrow w_2,a_2,b_2}^{\{v_1,w_1\}}$. Combining (\ref{3.15_4.15})-(\ref{3.1_4.20}), we have 
  \begin{equation}\label{new_3.2_4.21}
  	 \begin{split}
\mathbb{P}\big(  \mathsf{C}[\psi] \big)   \lesssim  \int_{0<a_1,b_1,a_2,b_2\le C_\dagger}   \mathbb{I}\cdot \bar{\mathfrak{p}}\ \mathrm{d}a_1\mathrm{d}b_1\mathrm{d}a_2\mathrm{d}b_2. 
  	 \end{split}
  \end{equation}

   Next, we show that for any $\psi'=(v_1,v_2,w_1',w_2')\in \Psi(n,N)$, 
   \begin{equation}\label{new3.2_4.22}
   	\begin{split}
   		\mathbb{P}\big(  \mathsf{C}[\psi'] \big) \gtrsim  \int_{0<a_1,b_1,a_2,b_2\le  C_\dagger}   \mathbb{I}\cdot \bar{\mathfrak{p}}'\ \mathrm{d}a_1\mathrm{d}b_1\mathrm{d}a_2\mathrm{d}b_2, 
   	\end{split}
   \end{equation}
   where $\bar{\mathfrak{p}}':=\mathfrak{p}_{v_1\leftrightarrow w_1',a_1,b_1}\cdot \mathfrak{p}_{v_2\leftrightarrow w_2',a_2,b_2}^{\{v_1,w_1'\}}$. 
 To see this, for the same reason as in (\ref{2_26_4.16}) and (\ref{3.1_4.19}), $\mathbb{P}( \overline{\mathsf{C}}[\psi',C_\dagger])$ can be written as 
    \begin{equation}\label{3.2_4.23}
    	\begin{split}
    		\int_{0<a_1,b_1,a_2,b_2\le  C_\dagger}  \mathbb{P} \big(   \{ \cup \mathfrak{L}_1'   \xleftrightarrow{\cup \mathcal{P}'} \cup \mathfrak{L}_2'   \}^c  \big) \cdot \bar{\mathfrak{p}}' \  \mathrm{d}a_1\mathrm{d}b_1\mathrm{d}a_2\mathrm{d}b_2,
    	\end{split}
    \end{equation}
   where $\mathcal{P}':=\mathcal{P}_{v_1,v_2,w_1',w_2'}^{a_1,a_2,b_1,b_2}$, and $\mathfrak{L}_j'$ is the analogue of $\mathfrak{L}_j$ with $w_1,w_2$ replaced by $w_1',w_2'$. Lemma \ref{lemma_total_measure_K} implies that the total masses of $\mathbf{e}_{v_1,w_1'}$ and $\mathbf{e}_{v_2,w_2'}^{\{v_1,w_1'\}}$ equal to $\mathbb{K}_{v_1,w_1'}(v_1,w_1')$ and $\mathbb{K}_{v_2,w_2'}^{\{v_1,w_1'\}}(v_2,w_2')$ respectively, and hence are both $O(N^{2-d})$ (by Lemma \ref{lemma21}). Also recall that $|\mathfrak{L}_1'|$ and $|\mathfrak{L}_2'|$ are both conditioned to be odd. Thus,  
    \begin{equation}\label{3.2_4.24}
   	\begin{split}
   		\mathbb{P}\big(   |\mathfrak{L}_1'| =  |\mathfrak{L}_2'| =1   \big)  \asymp  1, \ \forall 0<a_1,b_1,a_2,b_2\le  N^{d-2}. 
   	\end{split}
   \end{equation}
    On $\{|\mathfrak{L}_j'|=1\}$, let $\eta_j'$ denote the unique excursion in $\mathfrak{L}_j'$ (from $v_j$ to $w_j'$). We define $\hat{\eta}_j'$ (resp. $\check{\eta}_j'$) as the sub-path of $\eta_j'$ up to (resp. after) time $\tau_{\partial \mathcal{B}(\hat{N})}(\eta_j')$. Let $\mathbf{z}_j$ denote the endpoint of $\hat{\eta}_j'$. We take two paths $\ell_1$ and $\ell_2$ satisfying the following:
   \begin{itemize}

   	\item[-]   Each $\ell_j$ starts from $\mathbf{z}_j$ and ends at $w_j'$;

   	\item[-]  Each $\ell_j$ consists of at most $2d$ line segments and has length at most $dN$;

  	 \item[-]   $\mathrm{ran}(\ell_j) \subset \widetilde{B}_{\mathbf{z}_j}( \cref{const_ls3}^3\hat{N})\cup \big[\mathcal{B}\big((1+\cref{const_ls3}^4)\hat{N}\big)\big]^c$;

   \item[-]  $\mathrm{dist}\big( \mathrm{ran}(\ell_1), \mathrm{ran}(\ell_2)  \big) \ge \cref{const_ls3}^2 \hat{N}$.

   \end{itemize}
For $j\in \{1,2\}$, we define $D_j:=\cup_{z\in \mathrm{ran}(\ell_j)}\widetilde{B}_z(\cref{const_ls3}^5\hat{N})$ and $D_j^+:=\cup_{z\in \mathrm{ran}(\ell_j)}\widetilde{B}_z(\cref{const_ls3}^4\hat{N})$. Note that $D_1^+\cup D_2^+ \subset \widetilde{B}_{\mathbf{z}_j}(2\cref{const_ls3}^3\hat{N})\cup [\mathcal{B}(\hat{N})]^c$. The event $\{ \cup \mathfrak{L}_1'   \xleftrightarrow{\cup \mathcal{P}'} \cup \mathfrak{L}_2'   \}^c$ occurs if all of the following events occur:
   \begin{itemize}

   	\item  $\mathsf{A}_1:=\{  |\mathfrak{L}_1'| =  |\mathfrak{L}_2'| =1\}$;

   	\item  $\mathsf{A}_2:= \{ \mathcal{Q}(\hat{\eta}_1',\hat{\eta}_2')\ge \cref{const_ls3} \}\cap   \{ \mathrm{ran}(\hat{\eta}_1')   \xleftrightarrow{\cup \mathcal{P}'} \mathrm{ran}(\hat{\eta}_2')  \}^c  $;

   	\item  $\mathsf{A}_3:= \cap_{j\in \{1,2\}} \{ \mathrm{ran}(\check{\eta}_j') \subset D_j \}$;

   	\item  $\mathsf{A}_4:= \cap_{j\in \{1,2\}} \{D_j\xleftrightarrow{\cup \mathcal{P}'} \widetilde{\partial }D_j^+\}^c$.

   \end{itemize}
    Given that $\mathsf{A}_1\cap \mathsf{A}_2$ occurs, by the strong Markov property of Brownian excursions and the invariance principle, the conditional probability of $\mathsf{A}_3$ is uniformly bounded away from $0$. Meanwhile, $\mathsf{A}_2$ and $\mathsf{A}_4$ are both decreasing with respect to the Poisson point process $\mathcal{P}'$. To sum up, for any $0<a_1,b_1,a_2,b_2\le  N^{d-2}$, 
    \begin{equation}
    	 \mathbb{P}\big(\{ \cup \mathfrak{L}_1'   \xleftrightarrow{\cup \mathcal{P}'} \cup \mathfrak{L}_2'   \}^c  \big) \gtrsim  \mathbb{P}(\mathsf{A}_2)\cdot \mathbb{P}(\mathsf{A}_4). 
    \end{equation}
    In fact, $\mathbb{P}(\mathsf{A}_4)$ is of order $1$. First, it follows from (\ref{crossing_low}) and the FKG inequality that $\mathbb{P}(\cap_{j\in \{1,2\}} \{D_j\xleftrightarrow{\{v_1,v_2,w_1',w_2'\}} \widetilde{\partial }D_j^+\}^c)\asymp 1$. Second, if an excursion in $\mathcal{P}'$ starting from $\{v_1,v_2,w_1',w_2'\}$ contributes to the event $\{D_j\xleftrightarrow{\cup \mathcal{P}'} \widetilde{\partial }D_j^+\}$, then it must intersect $\widetilde{\partial }D_j\cup \widetilde{\partial }D_j^+$, whose probability is at most $1-\mathrm{exp}(-C(a_1+a_2+b_1+b_2)N^{2-d})$ (by Lemma \ref{lemma_new216}). These two observations imply that $\mathbb{P}(\mathsf{A}_4)\asymp 1$. As a result, for any $0<a_1,b_1,a_2,b_2\le  N^{d-2}$,
   \begin{equation}\label{3_2_4.25}
   	\begin{split}
   	  \mathbb{P}\big(\{ \cup \mathfrak{L}_1'   \xleftrightarrow{\cup \mathcal{P}'} \cup \mathfrak{L}_2'   \}^c  \big)  
   		\gtrsim    \mathbb{P}(\mathsf{A}_2),
   	\end{split}
   \end{equation}
   where $\hat{\eta}_1'\sim \hat{\mathbf{e}}_{v_1,w_1'}^{\hat{N}}$ and $\hat{\eta}_2'\sim \hat{\mathbf{e}}_{v_2,w_2'}^{\{v_1,w_1'\},\hat{N}}$. Note that $\hat{\mathbf{e}}_{v_1,w_1'}^{\hat{N}}$ (resp. $\hat{\mathbf{e}}_{v_2,w_2'}^{\{v_1,w_1'\},\hat{N}}$) is equivalent to $\hat{\mathbf{e}}_{v_1,w_1}^{\hat{N}}$ (resp. $\hat{\mathbf{e}}_{v_2,w_2}^{\{v_1,w_1\},\hat{N}}$), i.e., their Radon-Nikodym derivative is uniformly bounded away from $0$ and $\infty$.  This together with Lemma \ref{new_lemma_separation} implies 
\begin{equation}\label{3.2_4.26}
	 \begin{split}
	 \mathbb{P}(\mathsf{A}_2) \asymp & \mathbb{P} \big(  \mathcal{Q}(\hat{\eta}_1,\hat{\eta}_2)\ge \cref{const_ls3},  \{ \mathrm{ran}(\hat{\eta}_1)   \xleftrightarrow{\cup \mathcal{P}'} \mathrm{ran}(\hat{\eta}_2)  \}^c  \big)  \asymp \mathbb{I}  
	 \end{split}
\end{equation}
for all $0<a_1,b_1,a_2,b_2\le  C_\dagger$. By (\ref{3.2_4.23}), (\ref{3_2_4.25}) and (\ref{3.2_4.26}), we obtain (\ref{new3.2_4.22}).

 Note that whenever $a_1,b_1,a_2,b_2\in (0, C_\dagger]$, the densities $\bar{\mathfrak{p}}$ and $\bar{\mathfrak{p}}'$ are both of order $N^{4-2d}$ (see (\ref{250})). Combined with (\ref{new_3.2_4.21}) and (\ref{new3.2_4.22}), it implies that for any $\psi=(v_1,v_2,w_1,w_2),\psi'=(v_1,v_2,w_1',w_2')\in \Psi(n,N)$, 
 \begin{equation}\label{3.2_4.30}
 	\mathbb{P}( \mathsf{C}[\psi'] ) \gtrsim \mathbb{P}( \mathsf{C}[\psi] ). 
 \end{equation}
  By the same argument, we also have that for any $\psi''=(v_1',v_2',w_1',w_2')\in \Psi(n,N)$, 
 \begin{equation}\label{3.2_4.31}
 	\mathbb{P}( \mathsf{C}[\psi''] ) \gtrsim \mathbb{P}( \mathsf{C}[\psi'] ). 
 \end{equation} 
 Combining (\ref{3.2_4.30}) and (\ref{3.2_4.31}), and then swapping the roles of $\psi$ and $\psi''$, we complete the proof of this lemma.  \qed

 \begin{remark}\label{remark_order_escape_prob}
 We record the following consequence for later use: the disconnecting probability $\mathbb{I}(a_1,b_1,a_2,b_2)$, defined in (\ref{3.1_4.20}), is of the same order for all $a_1,b_1,a_2,b_2\in (0,C_\dagger]$. Indeed, since $\mathbb{I}$ is decreasing in $a_1,b_1,a_2$ and $b_2$, one has 
 \begin{equation}\label{3.18_good_4.28}
 	\mathbb{I}(C_\dagger,C_\dagger,C_\dagger,C_\dagger)\le \mathbb{I}(0,0,0,0). 
 \end{equation}
 On the other hand, by the separation lemma, conditioned on $\{ \mathrm{ran}(\hat{\eta}_1)     \xleftrightarrow{\cup \mathcal{P}_{v_1,v_2,w_1,w_2} } \mathrm{ran}(\hat{\eta}_2)  \}^c $, with a uniformly positive probability each $ \hat{\eta}_i$ is disjoint from $\widetilde{B}_{v_{3-i}}(C)$. (Precisely, we first reverse $\hat{\eta}_1$ and $\hat{\eta}_2$, thereby obtaining two paths from $\partial \mathcal{B}(\cref{const_ls2}N)$ to $v_1$ and $v_2$ respectively. We then apply to these reversed paths the analogue of Lemma \ref{new_lemma_separation} with $i=\mathrm{out}$, involving two Brownian motions conditioned to hit $v_1$ and $v_2$ respectively before exiting $\mathcal{B}(\cref{const_ls2}N)$. The same argument as in Section \ref{subsection4.5} is valid for this analogue.) Moreover, for each $z\in \{v_1,v_2,w_1,w_2 \}$, the probability that all excursions in $\mathcal{P}_{v_1,v_2,w_1,w_2}^{C_\dagger,C_\dagger,C_\dagger,C_\dagger}$ starting from $z$ are contained in $\widetilde{B}_z(C)$ is of order $1$. Combining these observations with the FKG inequality, we have  
 \begin{equation}\label{3.18_good_4.29}
 	\mathbb{I}(C_\dagger,C_\dagger,C_\dagger,C_\dagger) \gtrsim \mathbb{I}(0,0,0,0), 
 \end{equation}
 which together with (\ref{3.18_good_4.28}) and the monotonicity of $\mathbb{I}$ yields 
  \begin{equation}\label{3.19_4.30_new}
  	\mathbb{I}(a_1,b_1,a_2,b_2) \asymp \mathbb{I}(0,0,0,0), \ \forall  a_1,b_1,a_2,b_2\in (0,C_\dagger].
  \end{equation} 
   Combining (\ref{new_3.2_4.21}), (\ref{new3.2_4.22}), (\ref{3.19_4.30_new}) and the fact that $\bar{\mathfrak{p}}\asymp N^{4-2d}$ for $a_1,b_1,a_2,b_2\in (0,C_\dagger]$, we obtain that 
   \begin{equation}
   	\mathbb{I}(a_1,b_1,a_2,b_2) \asymp N^{2d-4}\cdot \mathbb{P}(\mathsf{C}[\psi])  
   \end{equation}
  uniformly over all $a_1,b_1,a_2,b_2\in (0,C_\dagger]$. This estimate will be used in Section \ref{section5.2_lower_volume}.

 \end{remark}

\subsection{Proof of Lemma \ref{lemma_separation}}

We restrict to the case $\diamond=\mathrm{out}$ and $i=1$, since the other cases are analogous. Recall the event $\overline{\mathsf{C}}[\psi,T]$ in (\ref{3.15_4.14}). By Lemma \ref{lemma_four_point_local_time_bound}, there exists $C_{\star}=C_{\star}(d,\epsilon)>0$ such that 
\begin{equation}\label{3.15_4.27}
	\mathbb{P}\big(\mathsf{C}[\psi] \setminus \overline{\mathsf{C}}[\psi,C_{\star}] \big) \le \tfrac{1}{8}\epsilon\cdot \mathbb{P}\big(\mathsf{C}[\psi]  \big).  
\end{equation}
Denote $\widetilde{D}=\widetilde{D}(\delta):=\cup_{x\in \widetilde{\eta}} \widetilde{B}_x(\delta N) $ and $\mathcal{P}:=\mathcal{P}_{v_1,v_2,w_1,w_2}^{a_1,a_2,b_1,b_2}$. As in (\ref{2_26_4.16}) and (\ref{3.1_4.19}), $\mathbb{P}(\overline{\mathsf{C}}[\psi,C_{\star}],w_1\xleftrightarrow{} \widetilde{D}\cup \partial B(\delta^{-1}N))$ can be written as 
\begin{equation}\label{3.4_4.33}
	\begin{split}
			\int_{0<a_1,b_1,a_2,b_2\le C_{\star}}    \mathbb{P} \big(   \{ \cup \mathfrak{L}_1   \xleftrightarrow{\cup \mathcal{P}} \cup \mathfrak{L}_2   \}^c, \cup \mathfrak{L}_1 \xleftrightarrow{\cup \mathcal{P}}  \widetilde{D} \cup \partial B(\tfrac{N}{\delta}) \big) 
		 \bar{\mathfrak{p}}   \mathrm{d}a_1\mathrm{d}b_1\mathrm{d}a_2\mathrm{d}b_2.
	\end{split}
\end{equation}
 Here $\mathfrak{L}_1$ and $\mathfrak{L}_2$ are defined as in (\ref{3.1_4.19}), and  $\bar{\mathfrak{p}}:=\mathfrak{p}_{v_1\leftrightarrow w_1,a_1,b_1}\cdot \mathfrak{p}_{v_2\leftrightarrow w_2,a_2,b_2}^{\{v_1,w_1\}}$. We employ the notation $\mathbb{I}$ from (\ref{3.1_4.20}). By (\ref{3.1_4.20}) and (\ref{3.2_4.24}), we have 
  \begin{equation}\label{3.4_4.34}
 	\begin{split}
 		\mathbb{P} \big(   \{ \cup \mathfrak{L}_1   \xleftrightarrow{\cup \mathcal{P}} \cup \mathfrak{L}_2   \}^c, \{ |\mathfrak{L}_1|=|\mathfrak{L}_2|=1 \}^c  \big) \lesssim (a_1b_1+a_2b_2)N^{4-2d}\cdot \mathbb{I}.  
 	\end{split}
 \end{equation} 
 Denote $\hat{N}:=\cref{const_ls2}N$. On the event $\{|\mathfrak{L}_j|=1\}$, let $\eta_j$ be the unique excursion in $\mathfrak{L}_j$, and write $\hat{\eta}_j$ for its sub-path up to time $\tau_{\partial \mathcal{B}(\hat{N})}$. We denote by $\mathbf{z}_j$ the endpoint of $\hat{\eta}_j$. By the Markov property of Brownian excursions, conditional on $\hat{\eta}_j$, the remaining sub-path of $\eta_j$ is distributed as $\widetilde{\mathbb{P}}_{\mathbf{z}_j}( \{\widetilde{S}_t\}_{0\le t\le \tau_{w_j}} \in \cdot  \mid \tau_{w_j}<\infty)$. Therefore, applying the FKG inequality, we have  
 \begin{equation}\label{3.4_4.35}
 	 \begin{split}
 	 	& \mathbb{P} \big(   \{ \cup \mathfrak{L}_1   \xleftrightarrow{\cup \mathcal{P}} \cup \mathfrak{L}_2   \}^c,  |\mathfrak{L}_1|=|\mathfrak{L}_2|=1 ,  
 	 	   \cup \mathfrak{L}_1 \xleftrightarrow{\cup \mathcal{P}}  \widetilde{D} \cup \partial B(\tfrac{N}{\delta})  \big)  \\
 	 	   \le &  \mathbb{I} \cdot  \max_{z\in \partial \mathcal{B}(\hat{N})} \mathbb{J}_z:=\mathbb{I} \cdot  \max_{z\in \partial \mathcal{B}(\hat{N})} \mathbb{P}\big( \mathrm{ran}( \check{\eta}^z)\xleftrightarrow{\cup \mathcal{P}}  \widetilde{D} \cup \partial B(\tfrac{N}{\delta})   \big), 
 	 \end{split}
 \end{equation}
where $\check{\eta}^z\sim \widetilde{\mathbb{P}}_{z}( \{\widetilde{S}_t\}_{0\le t\le \tau_{w_1}} \in \cdot  \mid \tau_{w_1}<\infty)$ is independent of $\mathcal{P}$. We denote by $\mathcal{C}_\delta$ the union of all clusters in $\cup \mathcal{L}_{1/2}^{\{v_1,v_2,w_1,w_2\}}$ intersecting $\widetilde{D} \cup \partial B(\delta^{-1}N)$. For each $x\in \{v_1,v_2,w_1,w_2 \}$, let $\mathcal{P}_x$ denote the point process consisting of excursions in $\mathcal{P}$ that start and end at $x$. We also denote $\mathcal{P}^0:=\mathcal{P}_{v_1,v_2,w_1,w_2}$. Note that 
\begin{equation}\label{inclusion_3.9_4.36}
\begin{split}
		& \big\{\mathrm{ran}( \check{\eta}^z)\xleftrightarrow{\cup \mathcal{P}}  \widetilde{D} \cup \partial B(\tfrac{N}{\delta})  \big\}  \\ 
		=  & \cup_{x\in \{v_1,v_2,w_1,w_2 \}} \big\{ (\cup \mathcal{P}_x) \cap  \mathcal{C}_\delta \neq \emptyset\big\} \cup \big\{\mathrm{ran}( \check{\eta}^z)  \xleftrightarrow{\cup \mathcal{P}^0}  \widetilde{D} \cup \partial B(\tfrac{N}{\delta}) \big\}. 
\end{split}
\end{equation}
On $\{(\cup \mathcal{P}_x) \cap  \mathcal{C}_\delta \neq \emptyset\}$, either $\cup \mathcal{P}_x$ hits $\partial B_x(N^{1/3})$, or $B_x(N^{1/3})$ is connected to $\partial B_x(cN)$ by $\cup \mathcal{L}_{1/2}$. Moreover, the probability of the first event is at most proportional to the local time of $\mathcal{P}_x$ at $x$ times $N^{-1/3}$ (by Lemma \ref{lemma_new216}), and the probability of the second event is $O(N^{-1/3})$ (by (\ref{crossing_low})). Thus, 
\begin{equation}\label{3.5_4.37}
	\begin{split}
		\mathbb{J}_z\lesssim (a_1+a_2+b_1+b_2+1)N^{-1/3}+ \hat{\mathbb{J}}_z + \tilde{\mathbb{J}}_z,
	\end{split}
\end{equation}
where $\hat{\mathbb{J}}_z:=\mathbb{P}\big( \mathrm{ran}( \check{\eta}^z)\xleftrightarrow{\cup \mathcal{P}^0}  \partial B(\delta^{-1}N)   \big)$ and $\tilde{\mathbb{J}}_z:=\mathbb{P}\big( \mathrm{ran}( \check{\eta}^z)\xleftrightarrow{\cup \mathcal{P}^0}  \widetilde{D} \big)$.

We claim the following estimates: for all sufficiently large $N$, and $z\in \partial \mathcal{B}(\hat{N})$, 
\begin{equation}\label{3.18_claim4.34}
	(1)\ \hat{\mathbb{J}}_z \lesssim \delta^{\frac{1}{3}}; \ \ \  (2)\ \tilde{\mathbb{J}}_z\lesssim [\ln(1/\delta)]^{-\frac{1}{150}}. 
\end{equation}
In fact, combining these two claims with (\ref{3.4_4.33}), (\ref{3.4_4.34}), (\ref{3.4_4.35}) and (\ref{3.5_4.37}), we have
\begin{equation}\label{3.15_4.34}
	\begin{split}
			& \mathbb{P}\big(\overline{\mathsf{C}}[\psi,C_{\star}],w_1\xleftrightarrow{} \widetilde{D}\cup \partial B(\tfrac{N}{\delta})  \big) \\
			\lesssim &   \int_{0<a_1,b_1,a_2,b_2\le C_\star } \big[ (a_1b_1+a_2b_2)N^{4-2d} + (a_1+a_2+b_1+b_2+1)N^{-\frac{1}{3}}  \\
		&\ \ \ \ \ \ \ \ \ \ \ \ \ \ \ \ \ \ \ \ \ +\delta^{\frac{1}{3}} + [\ln(1/\delta)]^{-\frac{1}{150}}   \big] \cdot  \mathbb{I}\cdot \bar{\mathfrak{p}}\    \mathrm{d}a_1\mathrm{d}b_1\mathrm{d}a_2\mathrm{d}b_2   \\ 
		\lesssim & \big(  C_\star^2N^{4-2d}+ C_\star N^{-\frac{1}{3}} +  [\ln(1/\delta)]^{-\frac{1}{150}} \big) \int_{0<a_1,b_1,a_2,b_2\le C_\star }  \mathbb{I}\cdot \bar{\mathfrak{p}}\    \mathrm{d}a_1\mathrm{d}b_1\mathrm{d}a_2\mathrm{d}b_2.
	\end{split}
\end{equation}
Similar to (\ref{new3.2_4.22}), the integral on the right-hand side is at most a constant multiple of $\mathbb{P}(\mathsf{C}[\psi])$, where the constant may depend on $C_\star$. Combined with (\ref{3.15_4.27}) and (\ref{3.15_4.34}), it yields the desired bound (\ref{3.5_ineq_lemma_separation_1}). 

In what follows, we provide the proofs of Claims (1) and (2) in (\ref{3.18_claim4.34}).

\textbf{Proof of Claim (1).} On $\{\mathrm{ran}( \check{\eta}^z)\xleftrightarrow{\cup \mathcal{P}^0}  \partial B(\delta^{-1}N)\}$, either $\check{\eta}^z$ hits $\partial B(\delta^{-1/3}N)$, or $B(\delta^{-1/3}N)$ is connected to $\partial B(\delta^{-1}N)$ by $\cup \widetilde{\mathcal{L}}_{1/2}$. It follows from (\ref{310}) and (\ref{crossing_low}) that the probabilities of these two events are both $O(\delta^{\frac{1}{3}(d-2)})$. Thus, 
\begin{equation}
	\hat{\mathbb{J}}_z \lesssim \delta^{\frac{1}{3}(d-2)} \lesssim \delta^{\frac{1}{3}}.
\end{equation}


\textbf{Proof of Claim (2).} Let $\lambda:=[\ln(1/\delta)]^{\frac{1}{2d(3d-2)}}$, and define $\widetilde{\mathcal{C}}$ as the union of all clusters in $\cup \widetilde{\mathcal{L}}_{1/2}$ intersecting $\widetilde{D}$. We need the following estimate:
   \begin{equation}\label{3.5_4.41}
     \overline{\mathbb{P}}_z\big( \tau_{\widetilde{\mathcal{C}}}<\infty \big)\lesssim  \lambda^{-3d+2}, 
   \end{equation}
 where $\overline{\mathbb{P}}_z$ is the product measure of $\widetilde{\mathbb{P}}_z$ and the probability measure of $\widetilde{\mathcal{L}}_{1/2}$. We first prove Claim (2) using (\ref{3.5_4.41}). By $\mathrm{dist}(\widetilde{D},w_1)\gtrsim N$ and (\ref{crossing_low}), we have 
 	 \begin{equation}\label{fixnew465}
 	\mathbb{P}(\mathsf{F}_{1}):=	\mathbb{P}\big( \widetilde{\mathcal{C}}   \cap   B_{w_1}(\lambda^{-3}N)  \neq \emptyset \big)  \lesssim   \lambda^{-\frac{3d}{2}+3}.
 	\end{equation}
 In addition, combining (\ref{3.5_4.41}) and Markov's inequality, one has 
  	\begin{equation}\label{fixnew467}
 				 \mathbb{P}(\mathsf{F}_{2}):=  \mathbb{P}\big(  \widetilde{\mathbb{P}}_{z}(\tau_{\widetilde{\mathcal{C}}}< \infty) \ge   \lambda^{-3d+4}  \big)\lesssim  \lambda^{-2}. 
 	\end{equation}
 	On the event $\mathsf{F}_{1}^c\cap \mathsf{F}_{2}^c$, by the strong Markov property of Brownian motion, 
 \begin{equation}
 	\begin{split}
 		  \widetilde{\mathbb{P}}_z\big( \tau_{\widetilde{\mathcal{C}}}< \tau_{w_1}<\infty  \big)  
 		\le  & \sum\nolimits_{x\in \widetilde{\partial}\widetilde{\mathcal{C}}  } \widetilde{\mathbb{P}}_z\big( \tau_{\widetilde{\mathcal{C}}}= \tau_{x}<\infty  \big) \cdot \widetilde{\mathbb{P}}_x\big(   \tau_{w_1}<\infty  \big)\\
 		\overset{(\widetilde{\mathcal{C}}   \cap  B_{w_1}(\lambda^{-3}N)  \neq \emptyset)}{\lesssim }&  (\lambda^{-3}N)^{2-d}\cdot \widetilde{\mathbb{P}}_z\big( \tau_{\widetilde{\mathcal{C}}} <\infty  \big) 
 		\lesssim   \lambda^{-4}N^{2-d}. 
 	\end{split}
 \end{equation}
 Combined with $\widetilde{\mathbb{P}}_z (\tau_{w_1}<\infty)\asymp N^{2-d}$, it yields 
 \begin{equation}\label{3.6_4.46}
 	\mathbb{P}\big( \mathsf{F}_{1}^c\cap \mathsf{F}_{2}^c, \check{\eta}^z \ \text{hits}\ \widetilde{\mathcal{C}} \big) \lesssim \lambda^{-4}. 
 \end{equation}
 Putting (\ref{fixnew465}), (\ref{fixnew467}) and (\ref{3.6_4.46}) together, we obtain the desired bound: 
 \begin{equation}
\tilde{\mathbb{J}}_z \le   \mathbb{P}\big(  \check{\eta}^z \ \text{hits}\ \widetilde{\mathcal{C}} \big)  \lesssim \lambda^{-\frac{3d}{2}+3} + \lambda^{-2} + \lambda^{-4} \lesssim [\ln(1/\delta)]^{-\frac{1}{150}}. 
 \end{equation}

Next, we prove the bound (\ref{3.5_4.41}). Let $\bar{\lambda}:=\lambda^{3d-2}=[\ln(1/\delta)]^{\frac{1}{2d}} $ and $\widetilde{\mathcal{C}}^+:=\mathcal{C}^+_{\widetilde{D}}$. Define $\mathring{\mathbb{P}}_z$ as the product measure of $\widetilde{\mathbb{P}}_z$ and the probability measure of the GFF on $\widetilde{\mathbb{Z}}^d$. Using the isomorphism theorem, we have 
 \begin{equation}
 	\begin{split}
 			& \overline{\mathbb{P}}_z\big( \tau_{\widetilde{\mathcal{C}}}< \infty  \big)\asymp  \mathring{\mathbb{P}}_v\big( \tau_{\widetilde{\mathcal{C}}^{+}}< \infty  \big)\\ 
 			\le & \mathring{\mathbb{P}}_z\big( \tau_{\widetilde{\mathcal{C}}^{+}}< \infty, \widetilde{\mathcal{C}}^{+} \subset  \widetilde{B}(\bar{\lambda}^3N) ,v\notin \widetilde{\mathcal{C}}^{+}\big)\\
 			& +  \mathbb{P}\big(\widetilde{D} \xleftrightarrow{\ge 0}  \partial B( \bar{\lambda}^{3}N) \big)+  \mathbb{P}\big(v\xleftrightarrow{\ge 0} \widetilde{D}  \big)\\
 			\overset{(\ref{one_arm_low}),(\ref{crossing_low})}{\lesssim  } &  \mathring{\mathbb{P}}_z\big( \tau_{\widetilde{\mathcal{C}}^{+}}< \infty, \widetilde{\mathcal{C}}^{+} \subset  \widetilde{B}(\bar{\lambda}^{3}N)  ,v\notin \widetilde{\mathcal{C}}^{+}\big)   +   \bar{\lambda}^{-\frac{3d}{2}+3}   +N^{-\frac{d}{2}+1}. 
 	\end{split}
 \end{equation}
Thus, it suffices to show that 
  	\begin{equation}\label{fix430}
 		\mathring{\mathbb{P}}_z\big( \tau_{\widetilde{\mathcal{C}}^{+}}< \infty, \widetilde{\mathcal{C}}^{+} \subset  \widetilde{B}(\bar{\lambda}^{3}N)  ,v\notin \widetilde{\mathcal{C}}^{+}\big)   \lesssim  \bar{\lambda}^{-1}. 
 	\end{equation}
 	In what follows, we establish (\ref{fix430}) using the exploration martingale argument, which was originally introduced in \cite{lupu2018random} and later developed in \cite{ding2020percolation, cai2024one}. Precisely, for any $t\ge 0$, we define $\mathcal{C}_t$ as the collection of points $z\in \widetilde{\mathbb{Z}}^d$ such that there exists a path of length at most $t$ in $\widetilde{E}^{\ge 0}$ connecting $z$ and $\widetilde{D}$. Note that $\mathcal{C}_0=\widetilde{D}$ and $\mathcal{C}_\infty= \widetilde{\mathcal{C}}^{+}$. In addition, since $\mathcal{C}_t$ is increasing in $t$, we have that 
 	\begin{equation}
 		\mathcal{M}_t:= \mathbb{E}\big[ \widetilde{\phi}_v \mid \mathcal{F}_{\mathcal{C}_t} \big]
 	\end{equation}
 	for $t\ge 0$ is a martingale. In particular,  
 	\begin{equation}
 		\mathcal{M}_0=\sum\nolimits_{x\in \widetilde{\partial}\widetilde{D}} \widetilde{\mathbb{P}}_{z}\big(\tau_{\widetilde{D}}=\tau_{x}<\infty \big)\cdot \widetilde{\phi}_x.
 	\end{equation}
 	On the event $\{z \notin \widetilde{\mathcal{C}}^{+}\}$, one has $\mathcal{M}_\infty=0$. Moreover, referring to \cite[Corollary 10]{ding2020percolation}, the quadratic variation of $\mathcal{M}_t$ can be written as  
\begin{equation}
	\langle \mathcal{M} \rangle_{t}= \sum\nolimits_{x \in \widetilde{\partial}\mathcal{C}_t} \widetilde{\mathbb{P}}_{z}\big(\tau_{\mathcal{C}_t}=\tau_{x}<\infty \big)\cdot \widetilde{G}(x,z). 
\end{equation}
In addition, when $\{\widetilde{\mathcal{C}}^{+} \subset  \widetilde{B}(\bar{\lambda}^{3}N) \}$ occurs, one has $\widetilde{G}(z,v)\gtrsim (\bar{\lambda}^{3}N)^{2-d}$ and thus, 
\begin{equation}
		\langle \mathcal{M} \rangle_{t} \gtrsim (\bar{\lambda}^{3}N)^{2-d}\cdot \widetilde{\mathbb{P}}_{z}\big(\tau_{\mathcal{C}_t}<\infty \big), \ \ \forall t\ge 0.
\end{equation}
As a result, there exists $c_{\dagger}>0$ such that for any $a\ge \bar{\lambda}^{-1}$, 
 \begin{equation}\label{fix4.35}
 \begin{split}
 	 	&\mathring{\mathbb{P}}_z \big( \widetilde{\mathcal{C}}^{+} \subset  \widetilde{B}(\bar{\lambda}^{3}N),z\notin \widetilde{\mathcal{C}}^{+} ,  \widetilde{\mathbb{P}}_{v}\big(\tau_{\mathcal{C}^{+}_{\widetilde{D}}}<\infty \big) \ge 
a \big) \\
 \le  & \mathbb{P}\big(  \widetilde{\mathcal{C}}^{+} \subset  \widetilde{B}(\bar{\lambda}^{3}N),z\notin \widetilde{\mathcal{C}}^{+} ,  \langle \mathcal{M} \rangle_{\infty}  \ge 
 c_{\dagger}a  (\bar{\lambda}^{3}N)^{2-d} \big)\\
 \le &  \mathbb{P}\big( |  \mathcal{M}_0 | \ge \bar{\lambda}^{-2d+1} N^{-\frac{d}{2}+1} \big) \\
 & +   \mathbb{P}\big( | \mathcal{M}_0 |\le \bar{\lambda}^{-2d+1} N^{-\frac{d}{2}+1}  ,  \langle \mathcal{M} \rangle_{\infty}  \ge 
 c_{\dagger}a\bar{\lambda}^{-3d+6}N^{2-d} \big). 
 \end{split}
 \end{equation}
 	 	Note that $\mathcal{M}_0$ is a mean-zero Gaussian random variable with variance 
 	\begin{equation}\label{fix433}
\mathrm{Var}(\mathcal{M}_0)  =	\sum\nolimits_{x_1,x_2\in \widetilde{\partial}\widetilde{D}} \widetilde{\mathbb{P}}_{z}\big(\tau_{\widetilde{D}}=\tau_{x_1}<\infty \big)\cdot \widetilde{\mathbb{P}}_{z}\big(\tau_{\widetilde{D}}=\tau_{x_2}<\infty \big)\cdot \widetilde{G}(x_1,x_2).
 	\end{equation}

  	We claim that 
 	\begin{equation}\label{fix4.37}
 		\mathrm{Var}(\mathcal{M}_0)\lesssim \bar{\lambda}^{-4d}  N^{2-d}. 
 	\end{equation}
  	Before proving this claim, we first derive the bound (\ref{fix430}) using it. For the first term on the right-hand side of (\ref{fix4.35}), by (\ref{fix4.37}) and the tail estimate for the normal distribution (see e.g., \cite[Lemma 2.11]{cai2024one}), we have 
 	\begin{equation}\label{fixnew438}
 		 \mathbb{P}\big( |  \mathcal{M}_0 | \ge \bar{\lambda}^{-2d+1}  N^{-\frac{d}{2}+1} \big)  \lesssim e^{-c\bar{\lambda}^2}. 
 	\end{equation}
 	For the second term, by the martingale representation theorem, the following process is a standard Brownian motion stopped at time $\langle \mathcal{M} \rangle_{\infty} $ (see \cite[Theorem 11]{ding2020percolation}): 
 \begin{equation}
 	W_t :=\left\{
\begin{aligned}
&\mathcal{M}_{T_t}-    \mathcal{M}_0  & \forall 0\le t\le \langle \mathcal{M} \rangle_{\infty}; \\
&\mathcal{M}_\infty -    \mathcal{M}_0  & \forall t\ge  \langle \mathcal{M} \rangle_{\infty}, 
\end{aligned}\right.
 \end{equation} 
 	where $T_t:= \inf\{s>0: \langle \mathcal{M} \rangle_{s}\ge t\}$. Therefore,   	\begin{equation}\label{fix440}
 		\begin{split}
 			&  \mathbb{P}\big(  \langle \mathcal{M} \rangle_{\infty}  \ge 
 c_{\dagger}a\bar{\lambda}^{-3d+6}N^{2-d} \mid \mathcal{F}_{ \mathcal{C}_0 } \big)\cdot \mathbbm{1}_{| \mathcal{M}_0 |\le \bar{\lambda}^{-2d+1} N^{-\frac{d}{2}+1}}\\
\overset{(\text{symmetry})}{ \le} &  \mathbb{P}\big( \sup\nolimits_{ 0\le t\le c_{\dagger}a\bar{\lambda}^{-3d+6}N^{2-d}} W_t\le  \bar{\lambda}^{-2d+1} N^{-\frac{d}{2}+1} \big) \\ 
 \overset{(\text{reflection principle})}{= }&1-2\mathbb{P}\big(W_{c_{\dagger}a\bar{\lambda}^{-3d+6}N^{2-d}} \ge \bar{\lambda}^{-2d+1} N^{-\frac{d}{2}+1} \big)  
 \lesssim  a^{-1} \bar{\lambda}^{-d-4}. 
 		\end{split}
 	\end{equation}
 		Putting (\ref{fix4.35}), (\ref{fixnew438}) and (\ref{fix440}) together, we get: for any $a \ge \bar{\lambda}^{-1}$,  	
 		\begin{equation}
 	\begin{split}
 		 	\mathbb{P}\big( \widetilde{\mathcal{C}}^{+} \subset  \widetilde{B}(\bar{\lambda}^{3}N),z\notin \widetilde{\mathcal{C}}^{+} ,  \widetilde{\mathbb{P}}_{z}\big(\tau_{\widetilde{\mathcal{C}}^{+}}<\infty \big) 
 \ge a \big) 
 \lesssim   e^{-c\bar{\lambda}^2} +  \bar{\lambda}^{-d-3}. 
 	\end{split}	
 	\end{equation}
 	 	This yields the desired bound (\ref{fix430}):
 	\begin{equation}
 		\begin{split}
 		&	\mathring{\mathbb{P}}_z\big( \tau_{\widetilde{\mathcal{C}}^{+}}< \infty, \widetilde{\mathcal{C}}^{+} \subset  \widetilde{B}(\bar{\lambda}^{3}N)   ,z \notin \widetilde{\mathcal{C}}^{+}\big)  \\
 			\le &  \bar{\lambda}^{-1} + \int_{\bar{\lambda}^{-1} \le a \le 1} 	\mathbb{P}\big( \widetilde{\mathcal{C}}^{+} \subset  \widetilde{B}(\bar{\lambda}^{3}N),z\notin \widetilde{\mathcal{C}}^{+} ,  \widetilde{\mathbb{P}}_{z}\big(\tau_{\widetilde{\mathcal{C}}^{+}}<\infty \big) 
 \ge a \big)\mathrm{d}a 	\\
 \lesssim  &  \bar{\lambda}^{-1}   + e^{-c\bar{\lambda}^2} +  \bar{\lambda}^{-d-3}  \lesssim   \bar{\lambda}^{-1}. 
 \end{split}
 	\end{equation}

 		 It remains to verify the claim (\ref{fix4.37}). To achieve this, by reversing the Brownian motion and using the strong Markov property, we have: for any $x \in \widetilde{\partial}\widetilde{D}$,  
 	\begin{equation*}
 		\begin{split}
 			 \widetilde{\mathbb{P}}_{z}\big(\tau_{\widetilde{D}}=\tau_{x}<\infty \big) \lesssim & \sum_{y \in \mathbb{Z}^d: d\le |x-y|\le 2d} \widetilde{\mathbb{P}}_{y}\big(\tau_{z}<\tau_{\widetilde{D}} \big) \\
 			\le &\sum_{y \in \mathbb{Z}^d: d\le |x-y|\le 2d} \widetilde{\mathbb{P}}_{y}\big(\tau_{\partial B_x(\delta N)}<\tau_{\widetilde{D}} \big) \cdot \max_{w\in \partial B_x(\delta N)}   \widetilde{\mathbb{P}}_{w}\big(\tau_{\partial B_{x}(\frac{N}{200})} <\tau_{\widetilde{D}} \big)\\
 			&\ \ \ \ \ \ \ \ \ \ \ \ \ \ \ \ \ \ \  \cdot \max_{w'\in \partial B_{x}(\frac{N}{200})} \widetilde{\mathbb{P}}_{w'}\big(\tau_{z}<\infty \big). 
 		\end{split}
 	\end{equation*}
 	Next, we bound the probabilities on the right-hand side individually. For the first probability, by the potential theory we have 
 	\begin{equation}
 		\widetilde{\mathbb{P}}_{y}\big(\tau_{\partial B_x(\delta N)}<\tau_{\widetilde{D}} \big)\lesssim (\delta N)^{-1}. 
 	\end{equation} 
 	For the second one, recall that $\widetilde{D}$ is an orthotope consisting of $O(\delta^{-1})$ boxes of size $\delta N$. By the invariance principle, $\widetilde{\mathbb{P}}_{w}\big(\tau_{\partial B_{x}(\frac{N}{200})} <\tau_{\widetilde{D}} \big)$ (where $|w-z| \asymp \delta N$) is bounded from above by a constant multiple of the probability that a Brownian motion on $\mathbb{R}^d$, starting at $O(1)$ distance from a cylinder of diameter $1$, reaches distance $\delta^{-1}$ without intersecting this cylinder. By projecting the Brownian motion onto the hyperplane orthogonal to the cylinder's axis, the latter probability corresponds to that of a Brownian motion on $\mathbb{R}^{d-1}$, starting at $O(1)$ distance from a unit ball, reaches distance $\delta^{-1}$ without hitting this ball, which is of order $[ \ln(\delta^{-1})]^{-1}$ for $d=3$, and of order $1$ for $d\ge 4$ (see e.g., \cite[Theorem 3.18]{morters2010brownian}). To sum up, 
 	\begin{equation}
 		\widetilde{\mathbb{P}}_{w}\big(\tau_{\partial B_{x}(\frac{N}{200})} <\tau_{\widetilde{D}} \big) \lesssim \mathbbm{1}_{d=3}\cdot [ \ln(\delta^{-1})]^{-1} + \mathbbm{1}_{d\ge 4}. 	\end{equation} 
	For the third probability, one has 
 	\begin{equation}
 	\max_{w'\in \partial B_{x}(\frac{N}{200})} \widetilde{\mathbb{P}}_{w'}\big(\tau_{z}<\infty \big) \lesssim [\mathrm{dist}(B_{z}(\tfrac{N}{200}), z)]^{2-d}\asymp N^{2-d}. 
 	\end{equation}
 	 Putting these estimates together, we have: for any $x \in \widetilde{\partial}\widetilde{D}$, 
  	\begin{equation}\label{fix434}
 				\widetilde{\mathbb{P}}_{z}\big(\tau_{\widetilde{D}}=\tau_{x}<\infty \big) \lesssim  \delta^{-1}\big( \mathbbm{1}_{d=3}\cdot [ \ln( \delta^{-1})]^{-1} + \mathbbm{1}_{d\ge 4} \big)\cdot N^{1-d}.
 	\end{equation}
 	Plugging (\ref{fix434}) into (\ref{fix433}), we obtain that 
\begin{equation}\label{fix438}
	\mathrm{Var}(\mathcal{M}_0)  \lesssim  \delta^{-2}\big( \mathbbm{1}_{d=3}\cdot [ \ln(\delta^{-1})]^{-1} + \mathbbm{1}_{d\ge 4} \big)^2N^{2-2d} \cdot 	\sum_{x_1,x_2\in \widetilde{\partial}\widetilde{D}}  (|x_1-x_2|+1)^{2-d}. 
\end{equation}
 	Since $\widetilde{\eta}$ is a line segment, there exist $O(\delta^{-1})$ points $y_1,...,y_K$ in $\widetilde{\eta}$ such that $\widetilde{D} =\cup_{1\le j\le K}  \widetilde{B}_{y_j}(\delta N)$ and $|y_j-y_{j'}|\ge \frac{|j-j'|}{10}\cdot \delta N $ for all $j\neq j'$. Therefore,  
\begin{equation}
	\begin{split}
		& \sum\nolimits_{x_1,x_2\in \widetilde{\partial}\widetilde{D}}  (|x_1-x_2|+1)^{2-d} \\
		\le &\sum\nolimits_{1\le j \le K,x\in \partial B_{y_j}(\delta N)} \sum\nolimits_{1\le j' \le K:|j-j'|\le 100}\sum\nolimits_{  x'\in \partial B_{y_{j'}}(\delta N) }  (|x -x'|+1)^{2-d}\\
		&+  \sum\nolimits_{1\le j,j'\le K:|j-j'|> 100} \sum\nolimits_{x\in \partial B_{y_j}(\delta N),x'\in \partial B_{y_{j'}}(\delta N)}(|x -x'|+1)^{2-d}.
	\end{split}
\end{equation}
 	In addition, the first sum on the right-hand side is at most 
\begin{equation}
\begin{split}
	& 	C K\cdot |\partial B_{y_1}(\delta N)|\cdot \max_{x \in \partial B_{y_j}(\delta N),|j-j'|\le 100}\sum_{x'\in \partial B_{y_{j'}}(\delta N)} (|x -x'|+1)^{2-d} \\
		\lesssim &\delta^{-1} \cdot (\delta N)^{d-1}\cdot \delta N= \delta^{d-1}N^{d},
\end{split}
\end{equation}
where the inequality in the second line follows from a direct calculation (see e.g., \cite[(4.4)]{cai2024high}). Meanwhile, since $|x -x'|\gtrsim  |j -j'|\cdot \delta N$ for all $|j-j'|\ge 100$, $x\in \partial B_{y_j}(\delta N)$ and $x'\in \partial B_{y_{j'}}(\delta N)$, the second sum is bounded from above by 
\begin{equation}\label{fix441}
	\begin{split}
		C|\partial B_{y_1}(\delta N)|^2\cdot (\delta N)^{2-d}\cdot  \sum_{1\le j <j'\le K} |j-j'|^{2-d}\lesssim  \delta^d\ln(\delta^{-1})N^d\lesssim \delta^{d-1}N^{d}. 
	\end{split}
\end{equation}
 	Recall that $\bar{\lambda}=[\ln(1/\delta)]^{\frac{1}{2d}} $. Combining (\ref{fix438})-(\ref{fix441}), we derive (\ref{fix4.37}):
  \begin{equation}
 	\begin{split}
 		\mathrm{Var}(\mathcal{M}_0) \lesssim  \delta^{d-3}\big( \mathbbm{1}_{d=3}\cdot [ \ln(\delta^{-1})]^{-1} + \mathbbm{1}_{d\ge 4} \big)^2N^{2-d}\lesssim \bar{\lambda}^{-4d} N^{2-d}. 
 	\end{split}
 \end{equation}
In conclusion, we have established Claim (2), and thus the proof of Lemma \ref{lemma_separation} is complete.   \qed

{\color{blue}

 }


The proof above yields the following byproduct: Lemma \ref{lemma_separation} remains valid when the conditioning on $\mathsf{C}[\psi]$ is replaced by the two-point event.

 \begin{lemma}\label{lemma_avoid_path_two_point}
  Under the same assumptions as in Lemma \ref{lemma_separation}, one has 
 	 \begin{equation}\label{ineq_lemma_avoid_path_two_point}
 	 	\mathbb{P}\big( z_i^{\diamond} \xleftrightarrow{} \cup_{x\in \widetilde{\eta}} \widetilde{B}_x(\cref{const_lemma_separation2} m_\diamond ) \cup \mathfrak{B}^{\diamond}_{m_\diamond,\cref{const_lemma_separation2}}  \mid v_i \xleftrightarrow{} w_i \big) \le  \epsilon. 
 	 \end{equation}
 \end{lemma}
\begin{proof}
	Likewise, we only consider the case $\diamond=\mathrm{out}$ and $i=1$. Let $\widetilde{F}=\widetilde{F}(\delta):=\cup_{x\in \widetilde{\eta}} \widetilde{B}_x(\delta N ) \cup \partial B(\delta^{-1}N)$. By Lemma \ref{lemma_switching}, the conditional probability of $\{v_1\xleftrightarrow{} \widetilde{F},v_1\xleftrightarrow{} w_1\}$ given $\{v_1 \xleftrightarrow{} w_1,  \widehat{\mathcal{L}}_{1/2}^{v_1}=a_1, \widehat{\mathcal{L}}_{1/2}^{w_1}=b_1 \}$ equals 
		\begin{equation}\label{3.8_4.76}
		\begin{split}
		 \widecheck{\mathbb{P}}_{v_1\leftrightarrow w_1,a_1,b_1}\big( \cup_{2\le i\le 4}  \big\{ (\cup \mathcal{P}^{(i)})\cap \mathcal{C}_{\widetilde{F}}^{\{v_1,w_1\}} \neq \emptyset  \big\}     \big).
		\end{split}
	\end{equation}
	By Claims (1) and (2) in (\ref{3.18_claim4.34}), the probability that $\eta\sim \bar{\mathbf{e}}_{v_1,w_1}$ intersects $\mathcal{C}_{\widetilde{F}}^{\{v_1,w_1\}}$ is $O([\ln(1/\delta)]^{-\frac{1}{150}}N^{2-d})$. In addition, by (\ref{3.5_4.41}), the total mass of excursions intersecting $\mathcal{C}_{\widetilde{F}}^{\{v_1,w_1\}}$ under $\mathbf{e}_{v_1,v_1}^{\{w_1\}}$ or $\mathbf{e}_{w_1,w_1}^{\{v_1\}}$ has expectation $O([\ln(1/\delta)]^{-\frac{1}{10}})$. Consequently, we have 
	\begin{equation}
		\begin{split}
		\sum\nolimits_{i\in \{2,3\}}	 \widecheck{\mathbb{P}}_{v_1\leftrightarrow w_1,a_1,b_1}\big(   (\cup \mathcal{P}^{(i)})\cap \mathcal{C}_{\widetilde{F}}^{\{v_1,w_1\}} \neq \emptyset     \big) \lesssim (a_1+b_1) [\ln(1/\delta)]^{-\frac{1}{10}}, 
		\end{split}
	\end{equation}
	  \begin{equation}
	  	\begin{split}
	  		\widecheck{\mathbb{P}}_{v_1\leftrightarrow w_1,a_1,b_1}\big(   (\cup \mathcal{P}^{(4)})\cap \mathcal{C}_{\widetilde{F}}^{\{v_1,w_1\}} \neq \emptyset , |\mathcal{P}^{(4)}| =1  \big) \lesssim \sqrt{a_1b_1}   [\ln(1/\delta)]^{-\frac{1}{150}}. 
	  	\end{split}
	  \end{equation}
	Meanwhile, since $|\mathcal{P}^{(4)}|$ is conditioned to be odd, one has 
	\begin{equation}\label{3.8_4.79}
		\widecheck{\mathbb{P}}_{v_1\leftrightarrow w_1,a_1,b_1}\big(  |\mathcal{P}^{(4)}| \neq 1  \big) \lesssim a_1b_1N^{4-2d}. 
	\end{equation}
 Combining (\ref{3.8_4.76})-(\ref{3.8_4.79}), we obtain  
	\begin{equation*}
		\begin{split}
			&\mathbb{P}\big(v_1\xleftrightarrow{} \widetilde{F},v_1\xleftrightarrow{} w_1  \big) \\
			 \lesssim & \int_{a_1,b_1>0} \big[(a_1+b_1) [\ln(1/\delta)]^{-\frac{1}{10}} + \sqrt{a_1b_1} [\ln(1/\delta)]^{-\frac{1}{150}} + a_1b_1N^{4-2d}  \big] \mathrm{d} \mathfrak{p}_{v_1\leftrightarrow w_1,a_1,b_1} \\
			\overset{(\mathrm{AM-GM})}{ \lesssim} &  \int_{a_1,b_1>0} \big[(a_1+b_1) [\ln(1/\delta)]^{-\frac{1}{150}}  + (a_1^2+b_1^2)N^{4-2d}  \big] \mathrm{d} \mathfrak{p}_{v_1\leftrightarrow w_1,a_1,b_1} \\
			\overset{(\ref{ineq_lemma_25})}{\lesssim } & \big( [\ln(1/\delta)]^{-\frac{1}{150}} + N^{4-2d} \big)\cdot \mathbb{P}\big( v_1\xleftrightarrow{} w_1  \big),
		\end{split}
	\end{equation*}
 which yields the desired bound (\ref{ineq_lemma_avoid_path_two_point}).   
\end{proof}

 We also present the analogue of Lemma \ref{lemma_avoid_path_two_point} for $d\ge 7$, which will be useful in subsequent proofs.
 

 \begin{lemma}\label{lemma_highd_avoid_box_twopoint}
 	 For any $d\ge 7$ and $\epsilon>0$, there exist $\Cl\label{const_highd_avoid_box_twopoint1}(d,\epsilon),\cl\label{const_highd_avoid_box_twopoint2}(d,\epsilon)>0$ such that for any $n\ge  \Cref{const_highd_avoid_box_twopoint1}$, $N\ge \Cref{const_highd_avoid_box_twopoint1}n$, $\psi\in \Psi(n,N)$, $i\in \{ 1,2\}$, and any $x\in \mathfrak{B}_{m_{\diamond},100}^{\diamond}$ satisfying $|z_i^{\diamond}- x | \ge \frac{m_\diamond}{100}$,  
\begin{equation}
		\mathbb{P}\big( z_i^{\diamond} \xleftrightarrow{}  \widetilde{B}_x(\cref{const_highd_avoid_box_twopoint2} m_\diamond ) \cup \mathfrak{B}^{\diamond}_{m_\diamond,\cref{const_highd_avoid_box_twopoint2}}  \mid v_i \xleftrightarrow{} w_i \big) \le  \epsilon. 
\end{equation}
  \end{lemma}

 \begin{proof}
  Using the framework in the proof of Lemma \ref{lemma_avoid_path_two_point}, it suffices to show that  
\begin{enumerate}

	\item The probability that $\eta \sim \bar{\mathbf{e}}_{v_1,w_1}$ intersects $\widetilde{\mathcal{C}}:=\mathcal{C}_{B_x(\delta N)\cup \partial B(\delta^{-1}N)}^{\{v_1,w_1\}}$ is $o_{\delta}(1)$;

	\item The total mass of excursions intersecting $\widetilde{\mathcal{C}}$ under $\mathbf{e}_{v_1,v_1}^{\{w_1\}}$ or $\mathbf{e}_{w_1,w_1}^{\{v_1\}}$ has expectation $o_{\delta}(1)$.

\end{enumerate}

 For Item (1), note that $\widetilde{\mathcal{C}}=\widetilde{\mathcal{C}}_1\cup \widetilde{\mathcal{C}}_2$, where $\mathcal{C}_1:=\mathcal{C}^{\{v_1,w_1\}}_{B_x(\delta N)}$ and $\mathcal{C}_2:=\mathcal{C}^{\{v_1,w_1\}}_{\partial B(\delta^{-1}N)}$. When $\eta$ hits $\mathcal{C}_1$, either $\eta$ hits $B_x(\delta^{\frac{1}{d-2}}N)$, or there exists $y\in [B_x(\delta^{\frac{1}{d-2}}N)]^c$ such that $\eta$ hits $\widetilde{B}_y(1)$ and $y\leftrightarrow B_x(\delta N)$ occurs. By (\ref{310}), the probability of the first event is $O(\delta)$. Applying the union bound and Lemma \ref{lemma_excursion_hitting_point}, the probability of the second event is at most of order 
 	 	\begin{equation} 
 		\begin{split}
 			&   N^{d-2} \sum_{y\in [B_x(\delta^{\frac{1}{d-2}}N)]^c} (|v_1-y|+1)^{2-d}(|y-w_1|+1)^{2-d}\cdot \mathbb{P}\big(y \xleftrightarrow{} B_x(\delta N)\big)\\
 			\lesssim &  N^{d-2}   \sum_{y\in [B_x(\delta^{\frac{1}{d-2}}N)]^c}  (|v_1-y|+1)^{2-d}(|y-w_1|+1)^{2-d}\cdot \big(\tfrac{\delta^{\frac{1}{d-2}}N}{\delta N}\big)^{2-d}(\delta N)^{-2}\\ 
 			\lesssim &  \delta^{d-5}N^{d-4}\cdot (|v_1-w_1|+1)^{4-d}\asymp \delta^{d-5}, 
 		\end{split}
 	\end{equation} 
  	where in the second line we used \cite[Proposition 1.5]{cai2024quasi}, and the last inequality follows from the following fact (see, e.g., \cite[Lemma 4.3]{cai2024high}), which we will use multiple times: for any $v,w\in \widetilde{\mathbb{Z}}^d$,  
  	\begin{equation}\label{computation_2-d_2-d}
	\sum\nolimits_{x\in \mathbb{Z}^d} (|v-x|+1)^{2-d}(|x-w|+1)^{2-d} \lesssim (|v-w|+1)^{4-d}. 
\end{equation} 
  	When $\eta$ hits $\mathcal{C}_2$, either $\eta$ hits $\partial B(\delta^{-1/2}N)$, or there exists $y\in B(\delta^{-1/2}N)$ such that $\eta$ hits $\widetilde{B}_y(1)$ and $\{y\xleftrightarrow{} \partial B(\delta^{-1}N)\}$ occurs. Using (\ref{310}), the probability of the first event is $O(\delta^{\frac{1}{2}(d-2)})$. Applying the union bound, Lemma \ref{lemma_excursion_hitting_point} and (\ref{one_arm_high}), the probability of the second event is at most 
  \begin{equation}
  	\begin{split}
  		  C \sum_{y\in B(\delta^{-1/2}N)} N^{d-2}(|v_1-y|+1)^{2-d} (|y-w_1|+1)^{2-d}\cdot (\delta^{-1}N)^{-2} \overset{(\ref{computation_2-d_2-d})}{\lesssim} \delta^{2}. 
  		  	\end{split}
  \end{equation}
Combining these estimates, we obtain Item (1).

For Item (2), when an excursion starting from $z\in \{v_1,w_1\}$ hits $\widetilde{\mathcal{C}}$, either it reaches $\partial B_z(\delta^{-1})$, or $B_z(\delta^{-1})\xleftrightarrow{} \partial B_z(cN)$ occurs. By Lemma \ref{lemma_new216} we have 
  \begin{equation}
  	\mathbf{e}_{z,z}\big( \{\eta: \eta\ \text{hits}\ \partial B_z(\delta^{-1})\} \big) \lesssim \delta^{d-2}. 
  \end{equation}
Meanwhile, it follows from (\ref{crossing_high}) that for all $N\ge \delta^{3-d}$, 
\begin{equation}
	\mathbb{P}\big(B_z(\delta^{-1})\xleftrightarrow{} \partial B_z(cN)  \big) \lesssim \delta^{4-d}N^{-2} \lesssim  \delta^{d-2}. 
\end{equation}
These two bounds together yield Item (2). 
\end{proof}

 	 {\color{blue}

 }

  {\color{purple}

  }

\subsection{Proof of Lemma \ref{lemma_five_point}}

As an auxiliary result, we present the three-point version of Lemma \ref{lemma25} as follows. 

\begin{lemma}\label{lemme_local_time_three_points}
	For any $d\ge 3$, there exist $C,c>0$ such that for any $T>0$, $D\subset \widetilde{\mathbb{Z}}^d$ and $z_1,z_2,z_3\in \widetilde{\mathbb{Z}}^d\setminus D$ with $\min_{i\neq j}|z_i-z_j|\ge d$, 
	\begin{equation} 
		\mathbb{P}\big( \widehat{\mathcal{L}}_{1/2}^{D,z_1}\ge T \mid  z_1 \xleftrightarrow{(D)} z_2, z_2 \xleftrightarrow{(D)} z_3 \big) \le Ce^{-cT}. 
	\end{equation} 
\end{lemma}
\begin{proof}
	For brevity, we only provide a sketch of the proof. We condition on the event $\{z_2 \xleftrightarrow{(D)} z_3\}$, and split the argument according to whether $z_1$ is in $ \cup(\sum_{2\le i\le 4}\mathcal{P}^{(i)})$.

	\textbf{Case 1: $z_1\notin \cup(\sum_{2\le i\le 4}\mathcal{P}^{(i)})$.} In this case, $\widehat{\mathcal{L}}_{1/2}^{D,z_1}$ relies only on $\mathcal{P}^{(1)}= \widetilde{\mathcal{L}}_{1/2}^{\{z_2,z_3\}}$, and $z_1 \xleftrightarrow{(D)} z_2$ is equivalent to $z_1\xleftrightarrow{\cup \mathcal{P}^{(1)} } \cup(\sum_{2\le i\le 4}\mathcal{P}^{(i)})$. As shown in the proof of \cite[Lemma 3.3]{cai2024one}, conditioned on this connecting event, the local time at $z_1$ of the loop soup $\widetilde{\mathcal{L}}_{1/2}^{\{z_2,z_3\}}$ still has exponential tails.

	\textbf{Case 2: $z_1\in \cup(\sum_{2\le i\le 4}\mathcal{P}^{(i)})$}. Note that $z_1 \xleftrightarrow{(D)} z_2$ must occur, and that $\widehat{\mathcal{L}}_{1/2}^{D,z_1}$ equals the total local time at $z_1$ of $\cup(\sum_{1\le i\le 4}\mathcal{P}^{(i)})$. By the isomorphism theorem, the contribution of $\mathcal{P}^{(1)}$ has exponential tails. Meanwhile, by the transience of $\widetilde{\mathbb{Z}}^d$, the local time contributed by each excursion in $\sum_{2\le i\le 4}\mathcal{P}^{(i)}$ decays exponentially; in addition, the total number of these excursions decays exponentially as well, since both $\widehat{\mathcal{L}}_{1/2}^{D,z_2}$ and $\widehat{\mathcal{L}}_{1/2}^{D,z_3}$ conditioned on $\{z_2 \xleftrightarrow{(D)} z_3\}$ do so (by Lemma \ref{lemma25}). 
	
	To sum up, we derive that the occupation time $\widehat{\mathcal{L}}_{1/2}^{D,z_1}$ given $\{z_1 \xleftrightarrow{(D)} z_2, z_2 \xleftrightarrow{(D)} z_3\}$ has exponential tails.
\end{proof}

We now proceed to the proof of Lemma \ref{lemma_five_point}. Without loss of generality, we only consider $i=1$. By the restriction property and Lemmas \ref{lemma25} and \ref{lemme_local_time_three_points}, one has 
\begin{equation}
	\begin{split}
		\mathbb{P}\big( \widehat{\mathcal{L}}_{1/2}^{D,z}\ge T   \mid v_1 \xleftrightarrow{} x,  \mathsf{C}[\psi] \big) \le Ce^{-cT}  
	\end{split}
\end{equation} 
for all $z\in \{v_1,v_2,w_1,w_2\}$. As a result, there exists $C_\star>0$ such that 
\begin{equation}\label{3.17_newadd_4.79}
	\mathbb{P}\big( v_1 \xleftrightarrow{} x,  \mathsf{C}[\psi]  \big)  \asymp  \mathbb{P}\big(v_1 \xleftrightarrow{} x,  \overline{\mathsf{C}}[\psi, C_\star]   \big). 
\end{equation}
 We divide the remainder of the proof into three cases according to the location of $x$: $x\in [\widetilde{B}(\Cref{const_lemma_five_point1}^{-1/4}N)]^c$, $x\in \widetilde{B}(\Cref{const_lemma_five_point1}^{1/4}n)$ and $x\in \widetilde{B}(\Cref{const_lemma_five_point1}^{-1/4}N)\setminus \widetilde{B}(\Cref{const_lemma_five_point1}^{1/4}n)$.

\textbf{When $x\in [\widetilde{B}(\Cref{const_lemma_five_point1}^{-1/4}N)]^c$.} In this case, one has $\mathrm{dist}(x,\{v_1,w_1\})\asymp R:=|x-w_1|$. As in (\ref{3.4_4.33}), $\mathbb{P}\big( v_1 \xleftrightarrow{} x , \overline{\mathsf{C}}[\psi,C_\star] \big)$ can be written as 
\begin{equation}\label{3.9_4.80}
	\begin{split}
			\int_{0<a_1,b_1,a_2,b_2\le C_\star}    \mathbb{P} \big(   \{ \cup \mathfrak{L}_1   \xleftrightarrow{\cup \mathcal{P}} \cup \mathfrak{L}_2   \}^c, \cup \mathfrak{L}_1 \xleftrightarrow{\cup \mathcal{P}}  x \big) 
		 \bar{\mathfrak{p}} \   \mathrm{d}a_1\mathrm{d}b_1\mathrm{d}a_2\mathrm{d}b_2,
	\end{split}
\end{equation}
where $\mathcal{P}:=\mathcal{P}_{v_1,v_2,w_1,w_2}^{a_1,a_2,b_1,b_2}$ (defined below (\ref{3.9_4.17})), $\mathfrak{L}_1$ and $\mathfrak{L}_2$ are defined as in (\ref{3.1_4.19}), and $\bar{\mathfrak{p}}:=\mathfrak{p}_{v_1\leftrightarrow w_1,a_1,b_1}\cdot \mathfrak{p}_{v_2\leftrightarrow w_2,a_2,b_2}^{\{v_1,w_1\}}$. Let $\hat{N}:=\Cref{const_lemma_five_point1}^{-1/2}N$. Recall the notation $\mathbb{I}$ from (\ref{3.1_4.20}). For the same reason as in (\ref{3.4_4.35}), we have 
 \begin{equation}\label{3.9_4.81}
	\begin{split}
		 & \mathbb{P} \big(   \{ \cup \mathfrak{L}_1   \xleftrightarrow{\cup \mathcal{P}} \cup \mathfrak{L}_2   \}^c,  |\mathfrak{L}_1|= |\mathfrak{L}_2|=1 , \cup \mathfrak{L}_1 \xleftrightarrow{\cup \mathcal{P}}  x \big) \\
		  \le & \mathbb{I} \cdot \max_{z\in \partial \mathcal{B}(\hat{N})} \mathbb{P}\big( \mathrm{ran}( \check{\eta}^z)\xleftrightarrow{\cup \mathcal{P}}  x   \big),
	\end{split}
\end{equation}
where $\check{\eta}^z\sim \widetilde{\mathbb{P}}_{z}( \{\widetilde{S}_t\}_{0\le t\le \tau_{w_1}} \in \cdot  \mid \tau_{w_1}<\infty)$. Denote $\widehat{\mathcal{C}}:=\mathcal{C}_x^{\{v_1,v_2,w_1,w_2\}}$. For each $y\in \{v_1,v_2,w_1,w_2\}$, we denote by $\mathcal{P}_y$ the point process consisting of excursions in $\mathcal{P}$ that start and end at $y$. Let $\mathcal{P}^0:=\mathcal{P}_{v_1,v_2,w_1,w_2}$. As in (\ref{inclusion_3.9_4.36}), one has 
\begin{equation}
	\begin{split}
		  \big\{  \mathrm{ran}( \check{\eta}^z)\xleftrightarrow{\cup \mathcal{P}}  x  \big\}   \subset     \cup_{y \in \{v_1,v_2,w_1,w_2 \}} \big\{ (\cup \mathcal{P}_y) \cap  \widehat{\mathcal{C}} \neq \emptyset\big\} \cup \big\{\mathrm{ran}( \check{\eta}^z)  \xleftrightarrow{\cup \mathcal{P}^0}  x \big\}.
	\end{split}
\end{equation}
On $\{(\cup \mathcal{P}_y) \cap  \widehat{\mathcal{C}}   \neq \emptyset\}$, either $\cup \mathcal{P}_y$ hits $\partial B_y(\frac{R}{\Cref{const_lemma_five_point1}})$, or $\{ x\xleftrightarrow{} \partial B_x(\frac{R}{\Cref{const_lemma_five_point1}})\}$ occurs. By Lemma \ref{lemma_new216}, the probability of the first event is at most proportional to the local time of $\mathcal{P}_y$ at $y$ times $R^{2-d}$. In addition, by (\ref{one_arm_low}), the probability of the second event is $O(R^{-\frac{d}{2}+1})$. Consequently, for any $a_1,b_1,a_2,b_2 \in (0, R^{\frac{d}{2}-1}]$, 
\begin{equation}\label{3.9_4.83}
	\begin{split}
		\mathbb{P}\big( \mathrm{ran}( \check{\eta}^z)\xleftrightarrow{\cup \mathcal{P}}  x   \big) \lesssim  R^{-\frac{d}{2}+1}   +  \mathbb{P}\big( \mathrm{ran}( \check{\eta}^z)\xleftrightarrow{\cup \mathcal{P}^0}  x   \big). 
	\end{split}
\end{equation}
 Combining (\ref{3.4_4.34}), (\ref{3.9_4.80}), (\ref{3.9_4.81}) and (\ref{3.9_4.83}), and then using (\ref{new3.2_4.22}), we have 
\begin{equation}\label{3.9_4.84_new}
	\begin{split}
		\mathbb{P}\big( v_1 \xleftrightarrow{} x , \overline{\mathsf{C}}[\psi,C_\star]\big) \lesssim   \mathbb{P}\big(   \mathsf{C}[\psi] \big)   \big[  R^{-\frac{d}{2}+1}+ \max_{z\in \partial \mathcal{B}(\hat{N})} \mathbb{P}\big( \mathrm{ran}( \check{\eta}^z)\xleftrightarrow{\cup \mathcal{P}^0}  x   \big) \big].
	\end{split}
\end{equation} 
On $\{\mathrm{ran}( \check{\eta}^z)\xleftrightarrow{\cup \mathcal{P}^0}  x\}$, either $\widehat{\mathcal{C}}$ reaches $\partial B_x( \tfrac{R}{\Cref{const_lemma_five_point1}})$, or $\check{\eta}^z$ hits $\widehat{\mathcal{C}}$ inside $\widetilde{B}_x( \tfrac{R}{\Cref{const_lemma_five_point1}})$. Hence, 
\begin{equation}\label{3.9_4.84}
	 \big\{  \mathrm{ran}( \check{\eta}^z)\xleftrightarrow{\cup \mathcal{P}}  x  \big\} \subset  \big\{ x \xleftrightarrow{} \partial B_x( \tfrac{R}{\Cref{const_lemma_five_point1}} )  \big\}\cup \big\{  \mathrm{ran}( \check{\eta}^z) \cap \widetilde{B}_x( \tfrac{R}{\Cref{const_lemma_five_point1}} ) \xleftrightarrow{(\partial B_x( \Cref{const_lemma_five_point1}^{-1}R))}  x  \big\}. 
\end{equation} 
Moreover, the same arguments as in \cite[Lemma 4.2]{cai2024quasi} show that $\mathrm{ran}( \check{\eta}^z) \cap \widetilde{B}_x( \tfrac{R}{\Cref{const_lemma_five_point1}} )$ is stochastically dominated by $(\cup \widetilde{\mathcal{L}}_{1/2}\cdot \mathbbm{1}_{\mathrm{ran}(\widetilde{\ell})\cap \partial B_x(R)\neq \emptyset})\cap \widetilde{B}_x( \tfrac{R}{\Cref{const_lemma_five_point1}} )$. As a result, 
\begin{equation}
	\begin{split}
		\mathbb{P}\big(  \mathrm{ran}( \check{\eta}^z) \cap \widetilde{B}_x( \tfrac{R}{\Cref{const_lemma_five_point1}} ) \xleftrightarrow{(\partial B_x( \Cref{const_lemma_five_point1}^{-1}R ))}  x  \big) \le \mathbb{P}\big(  x \xleftrightarrow{} \partial B_x( \tfrac{R}{\Cref{const_lemma_five_point1}} ) \big). 
	\end{split}
\end{equation}
This together with (\ref{3.9_4.84}) yields that 
\begin{equation}\label{3.9_4.87}
	\begin{split}
		\mathbb{P}\big( \mathrm{ran}( \check{\eta}^z)\xleftrightarrow{\cup \mathcal{P}^0}  x   \big) \lesssim \theta_d(\tfrac{R}{\Cref{const_lemma_five_point1}}  ) \lesssim R^{-\frac{d}{2}+1}. 
	\end{split}
\end{equation}
Plugging (\ref{3.9_4.87}) into (\ref{3.9_4.84_new}) and then applying (\ref{3.17_newadd_4.79}), we obtain (\ref{ineq_lemma_five_point_new}).

\textbf{When $x\in \widetilde{B}(\Cref{const_lemma_five_point1}^{1/4}n)$.} The proof in this case is similar to the one above, so we omit the details.

 \textbf{When $x\in \widetilde{B}(\Cref{const_lemma_five_point1}^{-1/4}N)\setminus \widetilde{B}(\Cref{const_lemma_five_point1}^{1/4}n)$.} In this case, $\mathrm{dist}(x,\{v_1,w_1\})\asymp |x|$. We define $\mathsf{G}$ as the event that either $\{x \xleftrightarrow{\cup \mathcal{P}^0}  \partial B_x(\Cref{const_lemma_five_point1}^{-1/8}|x|)\}$ occurs, or there exists $y\in \{v_1,v_2,w_1,w_2\}$ such that $\mathcal{P}_y$ hits $B_y(\Cref{const_lemma_five_point1}^{-1/8}|x|)$. By (\ref{one_arm_low}) and Lemma \ref{lemma_new216}, 
  \begin{equation}\label{3.10_4.88_new}
  	\mathbb{P}( \mathsf{G} )  \lesssim |x|^{-\frac{d}{2}+1} + (a_1+a_2+b_1+b_2)|x|^{2-d}. 
  \end{equation}
 Note that $\mathsf{G}$ is increasing with respect to $\mathcal{P}$. Hence, by the FKG inequality, 
 \begin{equation}\label{3.10_4.88}
 \begin{split}
 		& \int_{0<a_1,b_1,a_2,b_2\le C_\star}    \mathbb{P} \big(   \{ \cup \mathfrak{L}_1   \xleftrightarrow{\cup \mathcal{P}} \cup \mathfrak{L}_2   \}^c, \mathsf{G} \big) 
		 \cdot \bar{\mathfrak{p}} \   \mathrm{d}a_1\mathrm{d}b_1\mathrm{d}a_2\mathrm{d}b_2\\
		 \lesssim & |x|^{-\frac{d}{2}+1}\cdot  \int_{0<a_1,b_1,a_2,b_2\le C_\star}    \mathbb{P} \big(   \{ \cup \mathfrak{L}_1   \xleftrightarrow{\cup \mathcal{P}} \cup \mathfrak{L}_2   \}^c  \big) 
		 \cdot \bar{\mathfrak{p}} \   \mathrm{d}a_1\mathrm{d}b_1\mathrm{d}a_2\mathrm{d}b_2 \\
		\le & |x|^{-\frac{d}{2}+1}\cdot \mathbb{P}\big(\mathsf{C}[\psi]\big).
 \end{split}
 \end{equation}
 Meanwhile, (\ref{new3.2_4.22}) and (\ref{3.4_4.34}) together yield  
 \begin{equation}\label{3.15_4.87}
 	\begin{split}
 		&\int_{0<a_1,b_1,a_2,b_2\le C_\star}    \mathbb{P} \big(   \{ \cup \mathfrak{L}_1   \xleftrightarrow{\cup \mathcal{P}} \cup \mathfrak{L}_2   \}^c, \{ |\mathfrak{L}_1|= |\mathfrak{L}_2|=1\}^c \big)\cdot \bar{\mathfrak{p}} \   \mathrm{d}a_1\mathrm{d}b_1\mathrm{d}a_2\mathrm{d}b_2 \\
 		 \lesssim & N^{4-2d}\cdot \mathbb{P}\big(\mathsf{C}[\psi]\big)\lesssim  |x|^{-\frac{d}{2}+1}\cdot \mathbb{P}\big(\mathsf{C}[\psi]\big).
 	\end{split}
 \end{equation}

  Next, we bound the probability of the event 
  \begin{equation}
  	\mathsf{H}:= \{ \cup \mathfrak{L}_1   \xleftrightarrow{\cup \mathcal{P}} \cup \mathfrak{L}_2   \}^c\cap \{ |\mathfrak{L}_1|= |\mathfrak{L}_2|=1\}\cap \mathsf{G}^c \cap  \{\mathfrak{L}_1 \xleftrightarrow{\cup \mathcal{P}}  x\}. 
  \end{equation}
   Denote $L^-:=\Cref{const_lemma_five_point1}^{-1/8}|x|$ and $L^+:=\Cref{const_lemma_five_point1}^{1/8}|x|$. For each $j\in \{1,2\}$, when $|\mathfrak{L}_j|=1$, let $\eta_j$ denote the unique excursion in $\mathfrak{L}_j$ (from $v_j$ to $w_j$). We define $\varsigma_j^-$ (resp. $\varsigma_j^+$) as the first (resp. last) time $\eta_j$ visits $\partial \mathcal{B}(L^-)$ (resp. $\partial \mathcal{B}(L^+)$). Let $T_j$ denote the duration of $\eta_j$. Consider the sub-paths $\hat{\eta}_j:= \eta_j[0,\varsigma_j^-]$, $\bar{\eta}_j:=\eta_j[\varsigma_j^-,\varsigma_j^+]$ and $\check{\eta}_j:=[\varsigma_j^+,T_j]$. On the event $\mathsf{H}$, $\cup \mathfrak{L}_1$ must intersect $\widehat{\mathcal{C}}$ inside $\widetilde{B}_x(\Cref{const_lemma_five_point1}^{-1/8}|x|)$; hence, since $(\cup \mathfrak{L}_1)\cap \widetilde{B}_x(\Cref{const_lemma_five_point1}^{-1/8}|x|)\subset \mathrm{ran}(\bar{\eta}_1)$, $\mathbb{P} (\mathsf{H})$ is upper-bounded by 
     \begin{equation}\label{3.10_4.90}
 	\begin{split}
 		 \mathbb{P} \big( \{ \mathrm{ran}(\hat{\eta}_1)\xleftrightarrow{\cup \mathcal{P}}  \mathrm{ran}(\hat{\eta}_2) \}^c,\{ \mathrm{ran}(\check{\eta}_1)\xleftrightarrow{\cup \mathcal{P}}  \mathrm{ran}(\check{\eta}_2) \}^c, \mathsf{G}^c ,\mathrm{ran}(\bar{\eta}_1)  \xleftrightarrow{\cup \mathcal{P}^0}  x  \big). 
 	\end{split}
 \end{equation} 
 Note that $\hat{\eta}_1\sim \hat{\mathbf{e}}_{v_1,w_1}^{L^-}$ and $\hat{\eta}_2\sim \hat{\mathbf{e}}_{v_2,w_2}^{\{v_1,w_1\},L^-}$. Moreover, given $\hat{\eta}_1$ and $\hat{\eta}_2$, the law of the reversed path of $\check{\eta}_1$ (resp. $\check{\eta}_2$) is equivalent to $\hat{\mathbf{e}}_{w_1,v_1}^{L^+}$ (resp. $\hat{\mathbf{e}}_{w_2,v_2}^{\{v_1,w_1\},L^+}$). For each $j\in \{1,2\}$, let $\mathbf{y}_j$ (resp. $\mathbf{z}_j$) denote the endpoint of $\hat{\eta}_j$ (resp. $ \check{\eta}_j$). Recall $\widetilde{\mathbb{P}}^t_{\cdot,\cdot }$ below (\ref{3.10_3.54}). Conditioned on $\hat{\eta}_1, \hat{\eta}_2,\check{\eta}_1$ and $\check{\eta}_2$, the law of $\bar{\eta}_1$ is given by  
 \begin{equation}
 	\begin{split}
\bar{\eta}_{\mathbf{y}_1,\mathbf{z}_1} \sim  \big[ \widetilde{G}(\mathbf{y}_1,\mathbf{z}_1) \big]^{-1} 	\int \widetilde{q}_t(\mathbf{y}_1,\mathbf{z}_1) 	\widetilde{\mathbb{P}}^t_{\mathbf{y}_1,\mathbf{z}_1}(\cdot) \mathrm{d}t. 
 	\end{split}
 \end{equation}
 Thus, it follows from (\ref{3.10_4.90}) and the FKG inequality that 
 \begin{equation}\label{3.10_4.92}
 	\begin{split}
 		\mathbb{P} (\mathsf{H}) \le &  \mathbb{P} (\mathsf{A} )  
 		  \cdot \max_{y\in \partial \mathcal{B}(L^-),z\in \partial \mathcal{B}(L^+)} \mathbb{P}\big( \mathrm{ran}(\bar{\eta}_{y,z} ) \xleftrightarrow{\cup \mathcal{P}^0} x \big), 
 	\end{split}
 \end{equation}
 where $\mathsf{A}:=\{ \mathrm{ran}(\hat{\eta}_1)\xleftrightarrow{\cup \mathcal{P}}  \mathrm{ran}(\hat{\eta}_2) \}^c\cap \{ \mathrm{ran}(\check{\eta}_1)\xleftrightarrow{\cup \mathcal{P}}  \mathrm{ran}(\check{\eta}_2) \}^c \cap  \mathsf{G}^c$.

 Define $\mathsf{I}:=\{  \partial  \mathcal{B}(L^-) \cup   \partial  \mathcal{B}(L^+)\xleftrightarrow{\cup \mathcal{P}^0 } \partial B(|x|)  \}^c$. It follows from (\ref{crossing_low}) that $\mathbb{P}(\mathsf{I})\asymp 1$. Thus, since $\mathsf{A}$ and $\mathsf{I}$ are both decreasing with respect to $\mathcal{P}$, we have 
 \begin{equation}
  \mathbb{P}(\mathsf{A} )\overset{(\text{FKG})}{\asymp}  	\mathbb{P}(\mathsf{A}\cap \mathsf{I}).
 \end{equation}
 Moreover, on $\mathsf{G}^c\cap \mathsf{I}$, the events $\{ \mathrm{ran}(\hat{\eta}_1)\xleftrightarrow{\cup \mathcal{P}}  \mathrm{ran}(\hat{\eta}_2) \}^c$ and $\{ \mathrm{ran}(\check{\eta}_1)\xleftrightarrow{\cup \mathcal{P}}  \mathrm{ran}(\check{\eta}_2) \}^c$ rely on the excursions and loops contained in $\widetilde{B}(|x|)$ and $[\widetilde{B}(|x|)]^c$ respectively. Therefore, by the BKR inequality (see \cite[Lemma 3.3]{cai2024high}), we have  
 \begin{equation}
 	\begin{split}
 		\mathbb{P}(\mathsf{A}\cap \mathsf{I}) \le & \mathbb{P}\big(\{ \mathrm{ran}(\hat{\eta}_1)\xleftrightarrow{\cup \mathcal{P}}  \mathrm{ran}(\hat{\eta}_2) \}^c \big) \cdot \mathbb{P}\big(\{ \mathrm{ran}(\check{\eta}_1)\xleftrightarrow{\cup \mathcal{P}}  \mathrm{ran}(\check{\eta}_2) \}^c  \big)  \\ 
 		\overset{(\text{Lemma}\ \ref{new_lemma_separation})}{\asymp }    &\mathbb{P}(\hat{\mathsf{A}}) \cdot  \mathbb{P}(\check{\mathsf{A}})   
 		  \overset{(\text{FKG})}{\lesssim }  \mathbb{P}(\hat{\mathsf{A}}\cap \check{\mathsf{A}})  
 	\end{split}
 \end{equation}
 for all $a_1,a_2,b_1,b_2\in (0,C_\star]$, where $\hat{\mathsf{A}}:=\{ \mathrm{ran}(\hat{\eta}_1)\xleftrightarrow{\cup \mathcal{P}}  \mathrm{ran}(\hat{\eta}_2) \}^c \cap \{ \mathcal{Q}(\hat{\eta}_1, \hat{\eta}_2) \ge \cref{const_ls3}\}$ and $\check{\mathsf{A}}:=\{ \mathrm{ran}(\check{\eta}_1)\xleftrightarrow{\cup \mathcal{P}}  \mathrm{ran}(\check{\eta}_2) \}^c \cap \{  \mathcal{Q}(\check{\eta}_1, \check{\eta}_2) \ge \cref{const_ls3}\}$. Meanwhile, by the same argument as in (\ref{3.9_4.87}), one has 
  \begin{equation}\label{3.10_4.95}
\max_{y\in \partial \mathcal{B}(L^-),z\in \partial \mathcal{B}(L^+)} \mathbb{P}\big( \mathrm{ran}(\bar{\eta}_{y,z} ) \xleftrightarrow{\cup \mathcal{P}^0} x \big) \lesssim |x|^{-\frac{d}{2}+1}. 
  \end{equation}
  Combining (\ref{3.10_4.92})-(\ref{3.10_4.95}), we derive 
  \begin{equation}\label{3.10_4.96}
  	\begin{split}
  		&\int_{0<a_1,b_1,a_2,b_2\le C_\star}    \mathbb{P}  (\mathsf{H}  ) 
		 \cdot \bar{\mathfrak{p}} \   \mathrm{d}a_1\mathrm{d}b_1\mathrm{d}a_2\mathrm{d}b_2 \\
		 \lesssim &|x|^{-\frac{d}{2}+1} \cdot \int_{0<a_1,b_1,a_2,b_2\le C_\star}    \mathbb{P}(\hat{\mathsf{A}}\cap \check{\mathsf{A}})
		 \cdot \bar{\mathfrak{p}} \   \mathrm{d}a_1\mathrm{d}b_1\mathrm{d}a_2\mathrm{d}b_2. 
  	\end{split}
  \end{equation}


  Similar to (\ref{new3.2_4.22}), we claim that 
   \begin{equation}\label{3.10_4.97}
   	\mathbb{P}(\mathsf{C}[\psi])\gtrsim  \int_{0<a_1,b_1,a_2,b_2\le C_\star}    \mathbb{P}(\hat{\mathsf{A}}\cap \check{\mathsf{A}})
		 \cdot \bar{\mathfrak{p}} \   \mathrm{d}a_1\mathrm{d}b_1\mathrm{d}a_2\mathrm{d}b_2. 
   \end{equation} 
   Before proving this claim, we first complete the proof using it. Combining (\ref{3.10_4.88}), (\ref{3.15_4.87}), (\ref{3.10_4.96}) and (\ref{3.10_4.97}), we obtain 
   \begin{equation}
   	\begin{split}
   		\mathbb{P}\big( \overline{\mathsf{C}}[\psi,C_\star], v_1\xleftrightarrow{}x  \big) \lesssim |x|^{-\frac{d}{2}+1}\cdot \mathbb{P}\big(\mathsf{C}[\psi]\big),
   	\end{split}
   \end{equation}
   which together with (\ref{3.17_newadd_4.79}) implies the desired bound (\ref{ineq_lemma_five_point_new}).


It remains to show (\ref{3.10_4.97}). We take two paths $\ell_1$ and $\ell_2$ satisfying the following:
   \begin{itemize}

   	\item[-]   Each $\ell_j$ starts from $\mathbf{y}_j$ and ends at $\mathbf{z}_j$;

   	\item[-]  Each $\ell_j$ consists of at most $2d$ line segments and has length at most $dN$;

  	 \item[-]   $\mathrm{ran}(\ell_j) \subset \widetilde{B}_{\mathbf{y}_j}( \cref{const_ls3}^3L^-)\cup \widetilde{B}_{\mathbf{z}_j}( \cref{const_ls3}^3L^-)\cup  \big(\mathcal{B} ((1+\cref{const_ls3}^4)L^- )\cup [\mathcal{B} (L^+-\cref{const_ls3}^4L^- )]^c  \big)^c$;

   \item[-]  $\mathrm{dist}\big( \mathrm{ran}(\ell_1), \mathrm{ran}(\ell_2)  \big) \ge \cref{const_ls3}^2 L^-$.

   \end{itemize}
   For $j\in \{1,2\}$, we define $D_j:=\cup_{z\in \mathrm{ran}(\ell_j)}\widetilde{B}_z(\cref{const_ls3}^5L^-)$ and $D_j^+:=\cup_{z\in \mathrm{ran}(\ell_j)}\widetilde{B}_z(\cref{const_ls3}^4L^-)$. On $\{|\mathfrak{L}_1|=|\mathfrak{L}_2|=1\}$, the event $\{ \cup \mathfrak{L}_1   \xleftrightarrow{\cup \mathcal{P}} \cup \mathfrak{L}_2   \}^c $ occurs if $\hat{\mathsf{A}}\cap \check{\mathsf{A}}$ and the following events occur: 
   \begin{equation}
   	\mathsf{D}_1:= \cap_{j\in \{1,2\}} \{ \mathrm{ran}(\bar{\eta}_j) \subset D_j \}, \ \ \mathsf{D}_2:= \cap_{j\in \{1,2\}} \{D_j\xleftrightarrow{\cup \mathcal{P}} \widetilde{\partial}D_j^+\}^c. 
   \end{equation}
    As noted before (\ref{3_2_4.25}), the probabilities of $\mathsf{D}_1$ and $	\mathsf{D}_2$ are both uniformly bounded away from zero. Combined with (\ref{3.2_4.24}), this yields 
     \begin{equation}\label{3.10_4.100}
	\begin{split}
		  \mathbb{P}\big(\{ \cup \mathfrak{L}_1   \xleftrightarrow{\cup \mathcal{P}} \cup \mathfrak{L}_2   \}^c \big)  
   		\gtrsim    \mathbb{P}(\hat{\mathsf{A}}\cap \check{\mathsf{A}}) 
	\end{split}
\end{equation}
for all $a_1,b_1,a_2,b_2\in (0, C_\star]$. Substituting (\ref{3.10_4.100}) into (\ref{3.2_4.23}) (with $\psi'=\psi$), we obtain the claim (\ref{3.10_4.97}), thereby completing the proof of Lemma \ref{lemma_five_point}. \qed

\begin{remark}\label{3.12_remark_4.7}
In fact, the assumption $|v_1-v_2|\ge  \Cref{const_lemma_five_point1}$ in Lemma \ref{lemma_five_point} can be dropped. The reason is that the separation lemma (i.e., Lemma \ref{new_lemma_separation}) does not require the starting points of the excursions to be well separated. This extension will be used in Section \ref{section5.2_lower_volume}.
\end{remark}

\subsection{Proof of Lemma \ref{new_lemma_separation}}\label{subsection4.5}

We only provide the proof for $i=\mathrm{in}$, since the case $i=\mathrm{out}$ is analogous. For brevity, we write $v:=x_{\mathrm{in}}$, $v':=x_{\mathrm{in}}'$, $w:=x_{\mathrm{out}}$, $w':=x_{\mathrm{out}}'$, $w'':=x_{\mathrm{out}}''$ and $w''':=x_{\mathrm{out}}'''$. Let $r_0:= |v-v'|\vee C_\dagger$, where $C_\dagger>0$ is a sufficiently large constant. Without loss of generality, we assume that $v,v'\in \widetilde{B}(r_0)$. We begin by introducing some notation.

\begin{itemize}

	 \item For each $j\in \mathbb{N}$, let $r_j:=C_\dagger^j  r_0$ and $r_j^+:=C_\dagger^{1/2}r_j$, where $C_\dagger>0$ is a large constant. We assume that $N/r_0=C_\dagger^{K}$ for some odd number $K\in \mathbb{N}^+$.

	\item For any $R\ge r_1$, let $\hat{\eta}_{R}$ (resp. $\hat{\eta}_{R}'$) be the sub-path of $\hat{\eta}$ (resp. $\hat{\eta}'$) up to the first time it hits $\partial \mathcal{B}(R)$.

	\item  For each $1\le j\le K$, we define the event $\mathsf{Q}_j(\delta):=\{ \mathcal{Q}(\hat{\eta}_{r_j}, \hat{\eta}_{r_j}') \ge \delta\}$. 
	  

	\item For $y\in \{v,v',w'',w'''\}$, let $\mathcal{P}_y$ denote the point measure consisting of excursions in $\mathcal{P}:=\mathcal{P}_{v,v',w'',w'''}^{a,a',b,b'}$ starting from $y$. Let $\mathcal{P}^0:=\mathcal{P}_{v,v',w'',w'''}$.

	\item For $1\le j\le K$, we define $\mathfrak{U}_j$ as the point measure consisting of the loops in $\mathcal{P}^0$ that are contained in $\mathcal{B}(r_j^+)$, and the excursions in $\mathcal{P}_{v}+\mathcal{P}_{v'}$ that are contained in $\mathcal{B}(r_{0}^+)$. Define $\mathsf{A}_j:= \big\{  \mathrm{ran}(\hat{\eta}_{r_j}) \xleftrightarrow{\cup \mathfrak{U}_j} \mathrm{ran}(\hat{\eta}_{r_j}') \big\}^c$. Note that $\mathbb{P}(\mathsf{Q}_1(\cref{const_ls3})\mid \mathsf{A}_1)\asymp 1$, and that $\mathsf{A}_{j_1}\subset \mathsf{A}_{j_2}$ for all $j_1\ge j_2$.

\end{itemize}

 The key is to establish the following bound: 
 \begin{equation}\label{3.13_key_3.134}
	\begin{split}
	\mathbb{P}\big( \mathsf{Q}_K(\cref{const_ls3}), \mathsf{A}_K  \big)  \gtrsim \mathbb{P}\big( \mathsf{A}_K  \big). 
	\end{split}
\end{equation}
Before presenting its proof, we first prove Lemma \ref{new_lemma_separation} using it. Consider the event 
\begin{equation}
\begin{split}
	\mathsf{F}:= &  \big\{ \mathcal{B}(r_{K})  \xleftrightarrow{ \cup \mathcal{P}^0} \partial   \mathcal{B}(r_{K}^+)   \big\}^c \cap \big\{ \big[ \cup (\mathcal{P}_{v}+\mathcal{P}_{v'} ) \big] \cap \partial \mathcal{B}(r_{0}^+)=\emptyset \big\}\\
&	\cap    \big\{ \big[ \cup (\mathcal{P}_{w''}+\mathcal{P}_{w'''}  )\big]  \cap \partial \mathcal{B}(r_{K}^+) =\emptyset  \big\}. 
\end{split} 
\end{equation}  
The three events defining $\mathsf{F}$ are independent and decreasing. In addition, on the event $\mathsf{F}$, the union of clusters in $\cup \mathcal{P}$ intersecting $\hat{\eta}$ (or $\hat{\eta}'$) coincides with the counterpart of $\mathfrak{U}_K$. Consequently, 
 \begin{equation}\label{3.15_4.102}
 	\begin{split}
 		& \mathbb{P} \big(  \mathcal{Q}(\hat{\eta}, \hat{\eta}') \ge \cref{const_ls3} , \{ \mathrm{ran}(\hat{\eta}) \xleftrightarrow{\cup \mathcal{P}} \mathrm{ran}(\hat{\eta}')\}^c \big)  \\ 
 		\ge &  \mathbb{P} \big( \mathsf{Q}_K(\cref{const_ls3}), \mathsf{A}_K,  \mathsf{F}\big)  \overset{(\text{FKG}),(\ref{3.13_key_3.134})}{\gtrsim }  \mathbb{P}(\mathsf{F})\cdot \mathbb{P}(\mathsf{A}_K) \\
 		\overset{( \mathfrak{U}_K\le  \mathcal{P})}{\ge} &  \mathbb{P}(\mathsf{F})\cdot \mathbb{P}(\{ \mathrm{ran}(\hat{\eta}) \xleftrightarrow{\cup \mathcal{P}} \mathrm{ran}(\hat{\eta}')\}^c). 
 	\end{split}
 \end{equation}
 By (\ref{crossing_low}), the probability of $\{\mathcal{B}(r_{K})  \xleftrightarrow{ \cup \mathcal{P}^0} \partial   \mathcal{B}(r_{K}^+)\}$ is uniformly bounded away from zero. By Lemma \ref{lemma_new216}, under $\mathbf{e}_{v,v}^{\{v'\}}$ or $\mathbf{e}_{w',w'}^{\{w\}}$ (resp. $\mathbf{e}_{w''}^{\{v,v',w'''\}}$ or $\mathbf{e}_{w'''}^{\{v,v',w''\}}$), the total mass of excursions intersecting $\partial \mathcal{B}(r_{0}^+)$ (resp. $\partial \mathcal{B}(r_{K}^+)$) is $O(r_0^{2-d})$. Thus, 
 \begin{equation}\label{3.15_4.103}
 	\begin{split}
 		\mathbb{P}(\mathsf{F}) \gtrsim e^{-C(a+a'+b+b')r_0^{2-d}}. 
 	\end{split}
 \end{equation}
Plugging (\ref{3.15_4.103}) into (\ref{3.15_4.102}), we derive this lemma.

In what follows, we present the proof of the key estimate (\ref{3.13_key_3.134}). We first prove a technical lemma.  

 \begin{lemma}\label{lemma_lowd_cluster_capacity}
 	For any $3\le d\le 5$, there exist $\cl\label{const_lemma_lowd_cluster_capacity1},\cl\label{const_lemma_lowd_cluster_capacity2}>0$ such that for any $r\ge 1$, 
 	 \begin{equation}\label{newadd_3.17_4.107}
 	 	\mathbb{P}\big( \exists x \in B(r)\ \text{such that}\ \mathrm{cap}\big( \mathcal{C}_x^{\partial B(r)} \big) \ge \cref{const_lemma_lowd_cluster_capacity1}r^{d-2} \big) \ge \cref{const_lemma_lowd_cluster_capacity2}.
 	 \end{equation}
 \end{lemma}
 \begin{proof}
   For any $x\in B(r/3)$, by (\ref{one_arm_low}), there exists $C_\dagger>0$ such that 
   \begin{equation}\label{newadd_3.17_4.108}
   	\mathbb{P}\big( x\xleftrightarrow{}  \partial B_x(\tfrac{r}{3}) \big) \le C_\dagger r^{-\frac{d}{2}+1}. 
   \end{equation}
By (\ref{asymp_cap_cluster}), there exists $c_{\bigtriangleup}>0$ such that 
\begin{equation}\label{newadd_3.17_4.109}
		\mathbb{P}\big(  \mathrm{cap}( \mathcal{C}_x) \ge  c_{\bigtriangleup} r^{d-2} \big) \ge  2C_\dagger r^{-\frac{d}{2}+1}.
\end{equation}
  Combining (\ref{newadd_3.17_4.108}) and (\ref{newadd_3.17_4.109}), we have 
  \begin{equation}
  \begin{split}
  	  	& \mathbb{P}\big( \mathrm{cap}\big( \mathcal{C}_x^{\partial B(r)} \big) \ge c_{\bigtriangleup}r^{d-2} \big)  \\
  	  	\ge &  \mathbb{P}\big( \mathrm{cap} ( \mathcal{C}_x  ) \ge c_{\bigtriangleup}r^{d-2}, \{x\xleftrightarrow{}  \partial B_x(\tfrac{r}{3}) \}^c \big)  \ge C_\dagger r^{-\frac{d}{2}+1}.
  \end{split}
  \end{equation}
 As a result, the first moment of $\mathbf{X}:=\sum_{x\in B(cr)}\mathbbm{1}_{\mathrm{cap} ( \mathcal{C}_x^{\partial B(r)}  ) \ge c_{\bigtriangleup}r^{d-2}}$	satisfies 
 \begin{equation}\label{newadd_3.17_4.111}
 	\mathbf{E}[\mathbf{X}] \ge |B(cr)|\cdot C_\dagger r^{-\frac{d}{2}+1} \asymp r^{\frac{d}{2}+1}. 
 \end{equation}
 	For the second moment, it follows from Lemma \ref{lemma_cite_three_point} and (\ref{one_arm_low}) that 
 	\begin{equation}\label{newadd_3.17_4.112}
 		\begin{split}
 			\mathbf{E}\big[\mathbf{X}^2\big] \le &  \sum\nolimits_{x_1,x_2\in B(cr) } \mathbb{P}\big( x_1\xleftrightarrow{} \partial B(\tfrac{r}{10}),x_2\xleftrightarrow{} \partial B(\tfrac{r}{10})  \big) \\
 			\lesssim & r^{-\frac{d}{2}+1} \sum\nolimits_{x_1,x_2\in B(cr) }   (|x_1-x_2|+1)^{-\frac{d}{2}+1} \asymp r^{d+2}. 
 		\end{split}
 	\end{equation}
 Using the Paley-Zygmund inequality, (\ref{newadd_3.17_4.111}) and (\ref{newadd_3.17_4.112}), we obtain (\ref{newadd_3.17_4.107}).  
 \end{proof}

 For any $x\in \mathbb{Z}^d$ and $R\ge 1$, we define the event 
 \begin{equation}
 	\mathsf{D}_{x,R}:= \big\{ \exists y\in B_x(2R)\setminus B_x(R)\ \text{such that}\ \mathrm{cap}\big( \widehat{\mathcal{C}}_y \big)\ge \cref{const_lemma_lowd_cluster_capacity1}R^{d-2}  \big\},
 \end{equation}
 where $\widehat{\mathcal{C}}_y:=\mathcal{C}_y^{\partial [B_x(2R)\setminus B_x(R)]}$. Note that Lemma \ref{lemma_lowd_cluster_capacity} implies $\mathbb{P}(\mathsf{D}_{x,R})\ge \cref{const_lemma_lowd_cluster_capacity2}$, and that for any $R_1\ge 2R_2$, $\mathsf{D}_{x,R_1}$ and $\mathsf{D}_{x,R_2}$ are independent. For any $1\le j\le K-1$ and $\rho\in (0,2^{-100})$, we say a point $x\in \partial \mathcal{B}(r_{j})$ is $(\rho,j)$-bad if 
 \begin{equation}
 	\big|\big\{1\le k\le \log_2(1/\rho)-10 : \mathsf{D}_{x,2^k\rho r_{j} }\ \text{occurs}  \big\} \big| \le \tfrac{1}{2}\cref{const_lemma_lowd_cluster_capacity2}\log_2(1/\rho). 
 \end{equation}
 Applying the Hoeffding's inequality, we have 
 \begin{equation}\label{3.16_4.107}
 	\begin{split}
 		\mathbb{P}\big( x\ \text{is}\  (\rho,j)\text{-bad}  \big) \le  \rho^{2c_{\ddagger}}.  	\end{split}
 \end{equation} 
for some constant $c_{\ddagger}>0$. We define the event 
 \begin{equation}
 	\mathsf{G}_j(\rho) := \big\{  |x\in \partial \mathcal{B}(r_{j}): x \ \text{is}\ (\rho,j)\text{-bad}  | \le \rho^{c_{\ddagger}}  |\partial \mathcal{B}(r_{j})| \big\}. 
 \end{equation}
 By Markov's inequality and (\ref{3.16_4.107}), we obtain  
 \begin{equation}
 	\mathbb{P}\big([\mathsf{G}_j(\rho)]^c\big)\le \big(\rho^{c_{\ddagger}}  |\partial \mathcal{B}(r_{j})|\big)^{-1}\sum\nolimits_{x\in \partial \mathcal{B}(r_{j})} \mathbb{P}\big( x\ \text{is}\  (\rho,j)\text{-bad}  \big) \le \rho^{c_{\ddagger}}. 
 \end{equation}
 For each $1\le j\le K-1$, since the events $\mathsf{A}_{j}$ and $\mathsf{G}_{j+1}$ depend on the loops within $\mathcal{B}(r_j^+)$ and $\mathcal{B}(2r_{j+1})\setminus \mathcal{B}(\frac{1}{2}r_{j+1})$ respectively, we have 
  \begin{equation}\label{3.16_4.110}
	\mathbb{P}\big(  \mathsf{A}_{j},[\mathsf{G}_{j+1}(\rho)]^c   \big) =	\mathbb{P}\big(  \mathsf{A}_{j}    \big) \cdot 	\mathbb{P}\big(  [\mathsf{G}_{j+1}(\rho)]^c   \big)   \le \rho^{c_{\ddagger}}\cdot \mathbb{P}\big( \mathsf{A}_{j}    \big), 
\end{equation}

 We claim the following two inequalities: 
 \begin{itemize}

 	\item   There exists $c_{\bigtriangleup}>0$ such that for any $1\le j \le K-2$, 
 	  \begin{equation}\label{3.16_4.111}
	\mathbb{P}\big(  \mathsf{A}_{j+2}, [\mathsf{Q}_{j+1}(\rho)]^c, \mathsf{G}_{j+1}(\rho) \big)\le  \rho^{c_{\bigtriangleup}} \cdot \mathbb{P}\big(  \mathsf{A}_{j}  \big).  
\end{equation}

    \item   There exists $C_{\star}(\rho)>0$ such that for any $1\le j \le K-1$, 
    \begin{equation}\label{3.16_4.112}
    	\mathbb{P}\big(\mathsf{Q}_{j}(\rho ), \mathsf{A}_{j} \big)  \le  C_{\star}(\rho) \cdot \mathbb{P}\big(\mathsf{Q}_{j+1}(\cref{const_ls3}), \mathsf{A}_{j+1} \big).  
    \end{equation}

 	
 \end{itemize}
 Before proving these claims, we first establish (\ref{3.13_key_3.134}) using them. For $1 \le j \le K-2$, by the union bound, (\ref{3.16_4.110}) and (\ref{3.16_4.111}), $\mathbb{P}\big(\mathsf{A}_{j+2}\big)$ is at most
 \begin{equation}
	\begin{split}
		 & \mathbb{P}\big(\mathsf{A}_{j+2}, \mathsf{Q}_{j+1}(\rho)\big) +\mathbb{P}\big(\mathsf{A}_{j+2}, [\mathsf{G}_{j+1}(\rho)]^c\big) +  \mathbb{P}\big(\mathsf{A}_{j+2},[\mathsf{Q}_{j+1}(\rho)]^c, \mathsf{G}_{j+1}(\rho) \big)  \\
		\le & \mathbb{P}\big(\mathsf{A}_{j+1}, \mathsf{Q}_{j+1}(\rho)\big) + (\rho^{c_{\ddagger}} +\rho^{c_{\bigtriangleup}}  )\cdot \mathbb{P}\big( \mathsf{A}_{j}  \big).  
	\end{split}
\end{equation}
 Iterating this inequality, we have  
 \begin{equation}
 	\begin{split}
 	\mathbb{P}\big(\mathsf{A}_{K}\big) \le & \sum\nolimits_{1\le j\le \frac{K-1}{2}} (\rho^{c_{\ddagger}} +\rho^{c_{\bigtriangleup}}  )^{j-1}\mathbb{P}\big(\mathsf{A}_{K-2j+1}, \mathsf{Q}_{K-2j+1}(\rho)\big)    \\
 	&+ (\rho^{c_{\ddagger}} +\rho^{c_{\bigtriangleup}}  )^{\frac{K-1}{2}}  \mathbb{P}\big(\mathsf{A}_{1} \big)   \\
 	\overset{(\ref{3.16_4.112})}{\lesssim  } &\mathbb{P}\big(\mathsf{Q}_K(\cref{const_ls3}), \mathsf{A}_K \big) \cdot   C_{\star}(\rho) \sum\nolimits_{1\le j\le \frac{K+1}{2}} (\rho^{c_{\ddagger}} +\rho^{c_{\bigtriangleup}}  )^{j-1} [C_{\star}(\cref{const_ls3})]^{2j-1}, 
 	\end{split}
 \end{equation}
 where in the last inequality we used $\mathbb{P}(\mathsf{Q}_1(\cref{const_ls3})\mid \mathsf{A}_1)\asymp 1$. Choosing $\rho >0$ sufficiently small so that $\rho ^{c_{\ddagger}} +\rho ^{c_{\bigtriangleup}}<[2C_{\star}(\cref{const_ls3})]^{-2}$, we obtain (\ref{3.13_key_3.134}).

 It remains to prove the claims (\ref{3.16_4.111}) and (\ref{3.16_4.112}).

\textbf{Proof of (\ref{3.16_4.111}).} We denote by $\mathbf{z}_m$ (resp. $\mathbf{z}_m'$) the endpoint of $\hat{\eta}_{r_m}$ (resp. $\hat{\eta}_{r_m}'$). According to \cite[Lemma 6.3.7]{lawler2010random}, the distributions of $\mathbf{z}_{j+1}$ and $\mathbf{z}_{j+1}'$ given $\hat{\eta}_{r_j}$ and $\hat{\eta}_{r_j}'$ are proportional to the uniform distribution on $\partial\mathcal{B}(r_{j+1})$. Thus, on $\mathsf{G}_{j+1}(\rho)$, 
\begin{equation}
	\begin{split}
	 \mathbb{P}(\mathsf{B}):= \mathbb{P}\big(\mathbf{z}_{j+1}\ \text{or}\ \mathbf{z}_{j+1}'\ \text{is}\  (\rho,j+1)\text{-bad}\mid \hat{\eta}_{r_j}, \hat{\eta}_{r_j}'\big) \lesssim \rho^{c_\ddagger}, 
	\end{split}
\end{equation}
which further implies that 
\begin{equation}\label{3.16_4.116new}
	\mathbb{P}\big( \mathsf{A}_{j+2}, [\mathsf{Q}_{j+1}(\rho)]^c,  \mathsf{B}  \big)  \lesssim  \rho^{c_\ddagger}\cdot \mathbb{P}\big( \mathsf{A}_{j}  \big). 
\end{equation}
On the event $\mathsf{A}_{j+1}\cap [\mathsf{Q}_{j+1}(\rho)]^c\cap \mathsf{B}^c$, either $\hat{\eta}_{r_{j+1}}'$ intersects $\mathbf{B}_{\mathbf{z}_{j+1}}(\rho r_{j+1})$, or $\hat{\eta}_{r_{j+1}}$ intersects $\mathbf{B}_{\mathbf{z}_{j+1}'}(\rho r_{j+1})$ (we denote these two events by $\mathsf{R}$ and $\mathsf{R}'$ respectively). Suppose that $\mathsf{R}$ occurs. Since $\mathbf{z}_{j+1}$ is $(\rho,j+1)$-good, with probability at least $1-\rho^c$, two independent Brownian motions starting from $\mathbf{B}_{\mathbf{z}_{j+1}}(\rho r_{j+1})$ are connected through the loop soup before exiting $B_{\mathbf{z}_{j+1}}(\frac{r_{j+1}}{100})$. Indeed, whenever two such Brownian motions cross an annulus $B_{\mathbf{z}_{j+1}}(2^{k+1})\setminus B_{\mathbf{z}_{j+1}}(2^{k})$ for which $\mathsf{D}_{\mathbf{z}_{j+1},2^k}$ occurs, with a uniformly positive probability they intersect the same loop cluster with capacity at least $\cref{const_lemma_lowd_cluster_capacity1}2^{k(d-2)}$. Since the number of such annuli is of order $\log_2(1/\rho)$, it follows that the probability that no such connection occurs is at most $\rho^c$. The same argument is valid for $\mathsf{R}'$. Consequently,  
\begin{equation}\label{3.16_4.117new}
 	\begin{split}
 		\mathbb{P}\big( \mathsf{A}_{j+2}, [\mathsf{Q}_{j+1}(\rho)]^c,  \mathsf{B} \big)  \le \rho^c\cdot \mathbb{P}\big( \mathsf{A}_{j}  \big). 
 	\end{split}
 \end{equation}
Combining (\ref{3.16_4.116new}) and (\ref{3.16_4.117new}), we obtain the claim (\ref{3.16_4.111}).

 \textbf{Proof of (\ref{3.16_4.112}).} We take two paths $\ell$ and $\ell'$ satisfying the following:
 \begin{itemize}
 	 
 	 \item[-] $\ell$ (resp. $\ell'$) starts from $\mathbf{z}_j$ (resp. $\mathbf{z}_j'$) and is stopped upon hitting $\partial \mathcal{B}(r_{j+1})$;

 	 \item[-] $\ell$ and $\ell'$ both have total length at most $d\cdot r_{j+1}$;

 	\item[-]   $\mathrm{ran}(\ell)\subset \widetilde{B}_{\mathbf{z}_j}(\rho^3r_j)\cup  [\mathcal{B} ((1+\rho^4)r_j ) ]^c$, $\mathrm{ran}(\ell')\subset \widetilde{B}_{\mathbf{z}_j'}(\rho^3r_j)\cup [\mathcal{B} ((1+\rho^4)r_j ) ]^c$;

 	\item[-]   $\mathrm{dist}\big( \mathrm{ran}(\ell), \mathrm{ran}(\ell')  \big) \ge \rho^2r_{j}$.

 	 \item[-]  The distance between the endpoints of $\ell$ and $\ell'$ is at least $10\cref{const_ls3}r_{j+1}$.

 \end{itemize}
We define $D:=\cup_{z\in \mathrm{ran}(\ell)}\widetilde{B}_z(\rho^5r_{j})$ and $D_+:=\cup_{z\in \mathrm{ran}(\ell)}\widetilde{B}_z(\rho^4r_{j})$. Let $D'$ (resp. $D_+'$) denote the analogue of $D$ (resp. $D_+$) for $\ell'$. The event $\mathsf{Q}_{j+1}(\cref{const_ls3})\cap \mathsf{A}_{j+1}$ occurs if $\mathsf{Q}_{j}(\rho)$, $\mathsf{A}_{j}$, $\mathsf{D}_1$ and $\mathsf{D}_2$ all occur, where   
\begin{equation}
 \mathsf{D}_1:=	\big\{  \mathrm{ran}(\hat{\eta}_{r_{j+1}})\setminus  \mathrm{ran}(\hat{\eta}_{r_{j }})  \subset D \big\} \cap \big\{  \mathrm{ran}(\hat{\eta}_{r_{j+1}}')\setminus  \mathrm{ran}(\hat{\eta}_{r_{j }}')  \subset D' \big\}, 
\end{equation}
\begin{equation}
	\begin{split}
		\mathsf{D}_2 :=  \big\{ D \xleftrightarrow{} \widetilde{\partial } D_+ \big\}^c \cap \big\{ D' \xleftrightarrow{} \widetilde{\partial } D_+' \big\}^c. 
	\end{split}
\end{equation}
By the invariance principle, the conditional probability of $\mathsf{D}_1$ given $\hat{\eta}_{r_{j }}$ and $\hat{\eta}_{r_{j }}'$ is bounded from below by some constant depending on $\rho$. As a result, 
 \begin{equation}\label{3.16_new_4.117}
	\begin{split}
		\mathbb{P}\big( \mathsf{Q}_{j+1}(\cref{const_ls3}), \mathsf{A}_{j+1} \big) \ge  &c(\rho) \cdot   \mathbb{P}\big( \mathsf{Q}_{j}(\rho),  \mathsf{A}_{j} , \mathsf{D}_2 \big) \\
		 \overset{(\text{FKG})}{ \ge}&  c(\rho) \cdot   \mathbb{P}\big(   \mathsf{D}_2 \big)  \cdot \mathbb{P}\big( \mathsf{Q}_{j}(\rho),  \mathsf{A}_{j}   \big). 
	\end{split}
\end{equation}
Using (\ref{crossing_low}) and the FKG inequality, we further obtain $\mathbb{P}\big(   \mathsf{D}_2 \big)\ge c'(\rho)>0$. Plugging this into (\ref{3.16_new_4.117}), we derive the claim (\ref{3.16_4.112}).

  In conclusion, we obtain the key estimate (\ref{3.13_key_3.134}), and thus complete the proof of Lemma \ref{new_lemma_separation}.  \qed

 \section{Estimates on heterochromatic four-point functions }\label{section5_four_point}

 The goal of this section is to establish the upper bound in (\ref{thm1_small_n_four_point}) and the lower bound in (\ref{thm1_large_n_four_point}), which is crucial for proving Theorem \ref{thm1}. For the first part, we prove the following estimate. Recall the notation $\mathsf{C}[\cdot]$ in (\ref{newadd328}).

   \begin{lemma}\label{lemma_prob_Cpsi_small_n}
For any $d\ge 3$ with $d\neq 6$, there exist $\Cl\label{const_prob_Cpsi_small_n1},\Cl\label{const_prob_Cpsi_small_n2}>0$ such that for any $v_1 \neq v_2\in \widetilde{B}(\Cref{const_prob_Cpsi_small_n1})$, $N\ge \Cref{const_prob_Cpsi_small_n2}$ and $w_1,w_2\in \widetilde{B}(10N)\setminus \widetilde{B}(N)$ with $|w_1-w_2|\ge N$,  
	\begin{equation}\label{newadd420}
		\mathbb{P}\big( \mathsf{C}[(v_1,v_2,w_1,w_2)] \big)  \lesssim  |v_1-v_2|^{\frac{3}{2}}\cdot  N^{-[(\frac{3d}{2}-1)\boxdot (2d-4)]}.
	\end{equation}
\end{lemma}

 \begin{proof}
For simplicity, we denote $\epsilon_N:=N^{-[(\frac{3d}{2}-1)\boxdot (2d-4)]}$ and $\psi=(v_1,v_2,w_1,w_2)$. We begin by showing that for some constant $C_\dagger>0$, 
	\begin{equation}\label{new.44}
		\mathbb{P}\big( \mathsf{C}[\psi] \big)\lesssim \epsilon_N, \ \ \forall \psi\in \Psi(C_\dagger,N).
	\end{equation}
	When $d\ge 7$, this follows directly from the BKR inequality:
	  \begin{equation}\label{new.45}
  	 \mathbb{P}\big(\mathsf{C}[\psi] \big) \le  \mathbb{P}\big(v_1\xleftrightarrow{} w_1 \big)\cdot  \mathbb{P}\big(v_2\xleftrightarrow{} w_2 \big)\lesssim N^{4-2d}. 
  \end{equation}
 In the low-dimensional case $3\le d\le 5$, we arbitrarily take $w\in \widetilde{B}(10N)\setminus \widetilde{B}(5N)$ and let $\widetilde{S}^w_{\cdot}\sim \widetilde{\mathbb{P}}_w$ be a Brownian motion starting from $w$. By (\ref{one_arm_low}), one has 
	  \begin{equation}\label{46}
  	\begin{split}
  \mathbb{P}(\mathsf{A}):=&\mathbb{P}\big( \widetilde{S}^w_{\cdot}\ \text{hits}\ \mathcal{C}_{\bm{0}}\cap [\widetilde{B}(2N)\setminus \widetilde{B}(N)] \big)\\
  \le &\mathbb{P}\big(  \mathcal{C}_{\bm{0}}\cap [\widetilde{B}(2N)\setminus \widetilde{B}(N)]\neq \emptyset \big)   \le \theta_d(N)\lesssim N^{- \frac{d}{2}+1}. 
  	\end{split}
  \end{equation}
	 Meanwhile, for any $x\in \mathbb{Z}^d$, we define the event 
 \begin{equation}\label{47}
 	\mathsf{A}_x:=\big\{\widetilde{S}^w_{\cdot}\ \text{first hits}\ \mathcal{C}_{\bm{0}}\cap [\widetilde{B}(2N)\setminus \widetilde{B}(N)]\ \text{within}\ \widetilde{B}_x(C_\dagger)  \big\}.   
 \end{equation} 
	Note that $\cup_{x\in B(2N-2C_\dagger)\setminus B(N+C_\dagger)}\mathsf{A}_x \subset \mathsf{A}$, and that $\mathsf{A}_{x_1}$ and $\mathsf{A}_{x_2}$ are disjoint whenever $B_{x_1}(C_\dagger)\cap B_{x_2}(C_\dagger)=\emptyset$. Therefore, 
 \begin{equation}\label{48}
 	\mathbb{P}(\mathsf{A}) \gtrsim \sum\nolimits_{x\in B(2N-2C_\dagger )\setminus B(N+C_\dagger)} \mathbb{P}(\mathsf{A}_x). 
 \end{equation}
 For each $x\in B(2N-2C_\dagger)\setminus B(N+C_\dagger)$, we arbitrarily take $x'\in \partial B_x(C_\dagger-1)$. In order for $\mathsf{A}_x$ to occur, it suffices to ensure that $x\in \mathcal{C}_{\bm{0}}$, and that $\widetilde{S}^w_{\cdot}$ reaches $x'$ before $\mathcal{C}_{\bm{0}}$ and then hit $\mathcal{C}_{\bm{0}}$ before leaving $\widetilde{B}_x(C_\dagger)$. As a result, 
 \begin{equation}\label{49}
 	\begin{split}
 	\mathbb{P}(\mathsf{A}_x) \ge 	\mathbb{E}\big[ \mathbbm{1}_{x\in \mathcal{C}_{\bm{0}}}\cdot \widetilde{\mathbb{P}}_w\big( \tau_{x'}<\tau_{\mathcal{C}_{\bm{0}}} \big)\cdot \widetilde{\mathbb{P}}_{x'}\big(\tau_{\mathcal{C}_{\bm{0}}}<\tau_{\partial B_x(C_\dagger)} \big) \big]. 
 	\end{split}
 \end{equation}
 When $\mathcal{C}_{\bm{0}}$ is disjoint from $\widetilde{B}_w(1)$, one has $\widetilde{G}_{\mathcal{C}_{\bm{0}}}(w,w)\asymp 1$ and hence, 
\begin{equation}\label{410}
	\begin{split}
		\widetilde{\mathbb{P}}_w\big( \tau_{x'}<\tau_{\mathcal{C}_{\bm{0}}} \big)  \gtrsim   	\widetilde{\mathbb{P}}_w\big( \tau_{x'}<\tau_{\mathcal{C}_{\bm{0}}} \big)  \cdot \sqrt{\tfrac{\widetilde{G}_{\mathcal{C}_{\bm{0}}}(x',x') }{\widetilde{G}_{\mathcal{C}_{\bm{0}}}(w,w) }} \overset{(\ref{29})}{\asymp} \mathbb{P}\big( x'\xleftrightarrow{(\mathcal{C}_{\bm{0}})} w \big). 
	\end{split}
\end{equation}
 In addition, when $x\in \mathcal{C}_{\bm{0}}$, one has 
 \begin{equation}\label{411}
 	\widetilde{\mathbb{P}}_{x'}\big(\tau_{\mathcal{C}_{\bm{0}}}<\tau_{\partial B_x(C_\dagger)} \big) \ge 	\widetilde{\mathbb{P}}_{x'}\big(\tau_{x}<\tau_{\partial B_x(C_\dagger)} \big)\gtrsim 1. 
 \end{equation}
Combining (\ref{49}), (\ref{410}) and (\ref{411}), we get 
\begin{equation}
\begin{split}
		\mathbb{P}(\mathsf{A}_x) \ge&  \mathbb{E}\big[ \mathbbm{1}_{x\in \mathcal{C}_{\bm{0}}, \widetilde{B}_w(1)\cap \mathcal{C}_{\bm{0}}=\emptyset}\cdot  \mathbb{P}\big( x'\xleftrightarrow{(\mathcal{C}_{\bm{0}})} w \big) \big]\\
	=&\mathbb{P}\big( \big\{x \xleftrightarrow{}  \bm{0} \Vert x'\xleftrightarrow{} w \big\}, \widetilde{B}_w(1)\cap \mathcal{C}_{\bm{0}}=\emptyset \big)\\
	\overset{(\text{Lemma}\ \ref{lemma_separation})}{\asymp } & \mathbb{P}\big( \big\{x \xleftrightarrow{}  \bm{0} \Vert x'\xleftrightarrow{} w \big\}  \big) \overset{(\text{Lemma}\ \ref{lemma_roots})}{\asymp }  \mathbb{P}\big( \mathsf{C}[\psi] \big), 
\end{split}
\end{equation}
 for all $\psi\in \Psi(C_\dagger,N)$. Putting this together with (\ref{46}) and (\ref{48}), one obtain 
 \begin{equation}\label{413}
 	\mathbb{P}\big( \mathsf{C}[\psi] \big) \lesssim \epsilon_N, \ \ \forall \psi \in \Psi(C_\dagger,N)  
 \end{equation}
 By (\ref{new.45}) and (\ref{413}), we confirm (\ref{new.44}).

  For any $v_1,v_2\in \widetilde{B}(\frac{C_\dagger}{10})$, and any $w_1,w_2\in \widetilde{B}(10N)\setminus \widetilde{B}(N)$ with $|w_1-w_2|\ge N$, on the event $\mathsf{C}[\psi]$, the cluster $\mathcal{C}_{w_1}$ (resp. $\mathcal{C}_{w_2}$) must intersect $\partial B(2C_\dagger)$ (resp. $\partial B(4C_\dagger)$). Combined with (\ref{new.44}), it implies that 
 \begin{equation}
 	\begin{split}
 		\mathbb{P}\big( \mathsf{C}[\psi] \big) \le \sum_{v_1'\in \partial B(2C_\dagger),v_2'\in \partial B(4C_\dagger)} \mathbb{P}\big( \mathsf{C}[(v_1',v_2',w_1,w_2)] \big) \overset{}{\lesssim } \epsilon_N.
 	\end{split}
 \end{equation}
 Thus, it remain to prove (\ref{newadd420}) in the case when $\chi:=|v_1-v_2|$ is sufficiently small.

 Note that one of $\mathbf{B}_{v_1}(\frac{\chi}{3})$ and $\mathbf{B}_{v_2}(\frac{\chi}{3})$ must be contained in some interval $I_e$. Otherwise, there exist $x_1,x_2\in \mathbb{Z}^d$ such that $|x_i-v_i|\le \frac{\chi}{3}$ for $i\in \{1,2\}$. 
 	\begin{itemize}
 		\item  If $x_1=x_2$, then one has $\chi \le |x_1-v_1|+|x_1-v_2|\le \frac{2\chi}{3}$, a contradiction.

 		\item If $x_1\neq x_2$, then one has $\chi \ge |x_1-x_2|- |x_1-v_1|-|x_2-v_2|\ge d-\frac{2\chi}{3}>\chi$, again a contradiction.

 	\end{itemize}
 	Without loss of generality, we assume $\mathbf{B}_{v_1}(\frac{\chi}{3})\subset I_{\{x,y\}}$. By the restriction property,   
 	 	\begin{equation}\label{newadd421}
 		\begin{split}
 				\mathbb{P}\big( \mathsf{C}[\psi] \big)= & \mathbb{E}\big[ \mathbbm{1}_{\{v_1\xleftrightarrow{} w_1\}\cap \{v_1\xleftrightarrow{}v_2\}^c}\cdot \mathbb{P}\big(  v_2\xleftrightarrow{(\mathcal{C}_{v_1})} w_2\mid \mathcal{F}_{\mathcal{C}_{v_1}}\big)   \big]\\
 				\overset{(\text{Lemma}\ \ref{new_lemma_210})}{\lesssim} & \chi^{\frac{1}{2}}\sum_{z\in \partial \mathcal{B}_{{v}_1}(C_\dagger ) } \mathbb{E}\big[ \mathbbm{1}_{\{v_1\xleftrightarrow{} w_1\}\cap \{v_1\xleftrightarrow{}v_2\}^c}\cdot \mathbb{P}\big(  z\xleftrightarrow{(\mathcal{C}_{v_1})} w_2 \mid \mathcal{F}_{\mathcal{C}_{v_1}} \big)   \big]\\
 				\le  & \chi^{\frac{1}{2}} \sum_{z\in \partial \mathcal{B}_{{v}_1}(C_\dagger ) } \mathbb{P}\big(v_1\xleftrightarrow{} w_1\Vert z\xleftrightarrow{} w_2 \big) 
 		\overset{(\ref{new.44})}{\lesssim}   \chi^{\frac{1}{2}} \epsilon_N.
 		\end{split}
 	\end{equation}
 	This estimate does not yet yield (\ref{newadd420}). We proceed by iterating it.

 	For $z\in \partial \mathcal{B}_{{v}_1}(C_\dagger)$, we denote $\mathsf{C}_z:=\{v_1\xleftrightarrow{} w_1\Vert z \xleftrightarrow{} w_2\}$. We claim that  
 	 	\begin{equation}\label{newadd4.22}
 		\begin{split}
 		\mathbb{P}\big(\mathsf{C}[\psi] \big) \lesssim \chi^{\frac{3}{2}}  \max_{z\in \partial \mathcal{B}_{{v}_1}(C_\dagger )}  \mathbb{E}\big[ \mathbbm{1}_{\mathsf{C}_z\cap \{\mathrm{dist}(\mathcal{C}_{w_2},v_1)\ge \frac{\chi}{3}\}} \cdot \big(\mathrm{dist}(\mathcal{C}_{w_2},v_1) \big)^{-1} \big] +\chi^2 \epsilon_N.
 		\end{split}
 	\end{equation}
 	Before proving (\ref{newadd4.22}), we use it to derive (\ref{newadd420}). Note that for any $z\in \partial \mathcal{B}_{{v}_1}(C_\dagger )$, the event $\{z\xleftrightarrow{} w_2\}$ implies that $\mathrm{dist}(\mathcal{C}_{w_2},v_1)\le |z-v_1|\le  10dC_\dagger$. Thus, assuming that $\mathbb{P}\big(\mathsf{C}[\psi] \big)\lesssim f(\chi)\cdot \epsilon_N$ holds for the function $f(\cdot)$ (which has been confirmed by (\ref{newadd421}) for $f(a)=a^{\frac{1}{2}}$), we have 
    \begin{equation}\label{newadd4.23}
    	\begin{split}
    		 & \mathbb{E}\big[  \mathbbm{1}_{\mathsf{C}_z\cap \{\mathrm{dist}(\mathcal{C}_{w_2},v_1)\ge \frac{\chi}{3}\}} \cdot \big(\mathrm{dist}(\mathcal{C}_{w_2},v_1) \big)^{-1}  \big]\\
    		\le & \mathbb{P}(\mathsf{C}_z )+ \int_{(10dC_\dagger)^{-1}\le t\le 3\chi^{-1}} \mathbb{P}\big( \mathsf{C}_z,0<\mathrm{dist}(\mathcal{C}_{w_2},v_1)\le t^{-1} \big) \mathrm{d}t\\
    		\le & \mathbb{P}(\mathsf{C}_z ) +   \int_{(10dC_\dagger)^{-1} \le t\le 3\chi^{-1}} \sum\nolimits_{v_1'\in \widetilde{\mathbb{Z}}^d:|v_1-v_1'|=t^{-1}} \mathbb{P}\big(\mathsf{C}_{v_1'}  \big) \mathrm{d}t\\
    	\overset{}{\lesssim} & \big(1+  \int_{(10dC_\dagger)^{-1} \le t\le 3\chi^{-1}} f(t^{-1})\mathrm{d}t\big)\cdot \epsilon_N.
    	\end{split}
    \end{equation}
     Plugging (\ref{newadd4.23}) into (\ref{newadd4.22}), we obtain that 
     \begin{equation}\label{newadd424}
     	\begin{split}
     	\mathbb{P}\big(\mathsf{C}[\psi] \big) \lesssim  \chi^{\frac{3}{2}}\cdot \big(1+  \int_{(10dC_\dagger)^{-1}\le t\le 3\chi^{-1}} f(t^{-1})\mathrm{d}t\big) \cdot \epsilon_N.
     	\end{split}
     \end{equation}
Combined with (\ref{newadd421}), it implies $\mathbb{P}\big(\mathsf{C}[\psi] \big)\lesssim \chi \cdot \epsilon_N$. Inserting this stronger estimate into (\ref{newadd424}), we further get $\mathbb{P}\big(\mathsf{C}[\psi] \big)\lesssim  \chi^{\frac{3}{2}} \ln(\chi^{-1}) \cdot \epsilon_N$. Iterating this argument once more, we finally arrive at (\ref{newadd420}), since $\int_{1\le t<\infty} t^{-\frac{3}{2}} \ln(t)\mathrm{d}t<\infty$.

We divide the proof of (\ref{newadd4.22}) into two cases, depending on whether $v_2 \in I_{\{x,y\}}$.

 	\textbf{Case (1): $v_2\in I_{\{x,y\}}$.} Recall that the first two lines of (\ref{newadd421}) yield
 	\begin{equation}\label{newadd_425}
 	\begin{split}
 				\mathbb{P}\big(\mathsf{C}[\psi] & \big)\lesssim  \chi^{\frac{1}{2}}\sum_{z\in \partial \mathcal{B}_{{v}_1}(C_\dagger ) } \mathbb{E}\big[ \mathbbm{1}_{\{v_1\xleftrightarrow{} w_1\}\cap \{v_1\xleftrightarrow{}v_2\}^c}\cdot \mathbb{P}\big(  z\xleftrightarrow{(\mathcal{C}_{v_1})} w_2 \mid \mathcal{F}_{\mathcal{C}_{v_1}} \big)   \big] \\ 
 		= & \chi^{\frac{1}{2}}\sum_{z\in \partial \mathcal{B}_{{v}_1}(C_\dagger ) }  \mathbb{E}\big[ \mathbbm{1}_{z\xleftrightarrow{} w_2} \cdot \mathbb{P}\big( \{v_1\xleftrightarrow{(\mathcal{C}_{w_2})} w_1\}\cap \{v_1\xleftrightarrow{(\mathcal{C}_{w_2})}v_2\}^c  \mid \mathcal{F}_{\mathcal{C}_{w_2}} \big)   \big].
 	\end{split}
 	\end{equation}
Without loss of generality, assume that $x$ is the endpoint of $I_{\{x,y\}}$ such that $v_2\in I_{[v_1,x]}$. On $\{v_1\xleftrightarrow{(\mathcal{C}_{w_2})} w_1\}\cap \{v_1\xleftrightarrow{(\mathcal{C}_{w_2})}v_2\}^c$, one of the following occurs:
\begin{enumerate}
	\item[(a)] $\mathcal{C}_{w_2}$ intersects $v_2$, and hence $\mathsf{C}[\psi]$ occurs;

	\item[(b)] $\mathcal{C}_{w_2}$ intersects $y$ but not $v_2$. As a result, $v_1\xleftrightarrow{(\mathcal{C}_{w_2})} w_1$ and $\{v_1\xleftrightarrow{(\mathcal{C}_{w_2})}v_2\}^c $ are incompatible.

	\item[(c)] $\mathcal{C}_{w_2}$ is disjoint from both $y$ and $v_2$, and thus $\mathrm{dist}(\mathcal{C}_{w_2},v_1)\ge \frac{\chi}{3}$.

\end{enumerate}
 	In conclusion, the expectation on the right-hand side of (\ref{newadd_425}) is at most
 	 \begin{equation}\label{newadd_426}
 	\begin{split}
 			\mathbb{P}\big(\mathsf{C}[\psi]  \big)+  \mathbb{E}\big[ \mathbbm{1}_{\{z\xleftrightarrow{} w_2,\mathrm{dist}(\mathcal{C}_{w_2},v_1)\ge \frac{\chi}{3}\}} \mathbb{P}\big( v_1\xleftrightarrow{(\mathcal{C}_{w_2})} w_1, \{v_1\xleftrightarrow{(\mathcal{C}_{w_2})} v_2 \}
  				^c   \mid \mathcal{F}_{\mathcal{C}_{w_2}} \big)  \big].
 	\end{split}
 \end{equation} 	
 	Moreover, by the isomorphism theorem we have 
 	 	 	\begin{equation}\label{ineq423}
 		\begin{split}
 		&	\mathbb{P}\big( v_1\xleftrightarrow{(\mathcal{C}_{w_2})} w_1, \{v_1\xleftrightarrow{(\mathcal{C}_{w_2})} v_2 \}
  				^c   \mid \mathcal{F}_{\mathcal{C}_{w_2}} \big)\\
  				\asymp & \int_{a>0} \tfrac{\mathrm{exp}\big( 	-\frac{a^2}{2\widetilde{G}_{\mathcal{C}_{w_2}}(v_1,v_1)}\big)  }{\sqrt{2\pi \widetilde{G}_{\mathcal{C}_{w_2}}(v_1,v_1)}} \cdot  \mathbb{P}^{\mathcal{C}_{w_2}}\big(v_1\xleftrightarrow{\ge 0} w_1, \{v_1\xleftrightarrow{\ge 0} v_2 \}
  				^c   \mid \widetilde{\phi}_{v_1}=a\big) \mathrm{d}a,
 		\end{split}
 	\end{equation}
 	where the conditional probability, by the FKG inequality, is upper-bounded by
 	\begin{equation}\label{ineq424}
 		\begin{split}
 			& \mathbb{P}^{\mathcal{C}_{w_2}}\big(v_1\xleftrightarrow{\ge 0} w_1  \mid \widetilde{\phi}_{v_1}=a\big) \cdot  \mathbb{P}^{\mathcal{C}_{w_2}}\big(  \{v_1\xleftrightarrow{\ge 0} v_2 \}
  				^c   \mid \widetilde{\phi}_{v_1}=a\big)\\
  				\overset{(\text{Lemmas}\ \ref{lemma_conditional_two_points}\ \text{and}\ \ref{lemma_connect_close_points})}{\lesssim } & \mathbb{P}^{\mathcal{C}_{w_2}}\big(v_1\xleftrightarrow{\ge 0} w_1 \big)\cdot \tfrac{a\cdot e^{-ca^2\chi^{-1}}}{\sqrt{\widetilde{G}_{\mathcal{C}_{w_2}}(v_1,v_1)}}. 
 		\end{split}
 	\end{equation} 	
 	By (\ref{ineq423}), (\ref{ineq424}) and $\widetilde{G}_{\mathcal{C}_{w_2}}(v_1,v_1)\asymp \mathrm{dist}(\mathcal{C}_{w_2},v_1)\land 1$ (by (\ref{23})), we obtain that on the event $\{ z\xleftrightarrow{} w_2 ,\mathrm{dist}(\mathcal{C}_{w_2},v_1) \gtrsim \chi\}$, 
 	\begin{equation}\label{newadd_429}
 		\begin{split}
 		&	\mathbb{P}\big( v_1\xleftrightarrow{(\mathcal{C}_{w_2})} w_1, \{v_1\xleftrightarrow{(\mathcal{C}_{w_2})} v_2 \}
  				^c   \mid \mathcal{F}_{\mathcal{C}_{w_2}} \big)\\
  	 \lesssim &\big( \int_{a>0} a \cdot  e^{-ca^2\chi^{-1}} \mathrm{d}a \big)\cdot \big(\mathrm{dist}(\mathcal{C}_{w_2},v_1)\land 1 \big)^{-1} \cdot  \mathbb{P}^{\mathcal{C}_{w_2}}\big(v_1\xleftrightarrow{\ge 0} w_1 \big) \\
  	 \lesssim & \chi \cdot \big(\mathrm{dist}(\mathcal{C}_{w_2},v_1) \big)^{-1}  \cdot  \mathbb{P}^{\mathcal{C}_{w_2}}\big(v_1\xleftrightarrow{\ge 0} w_1 \big),
 		\end{split}
 	\end{equation}
 	where in the last line we used the inclusion $\{z\xleftrightarrow{} w_2\}\subset \{\mathrm{dist}(\mathcal{C}_{w_2},v_1)\le  10dC_\dagger\}$. Plugging (\ref{newadd_429}) into (\ref{newadd_426}), we get
\begin{equation}\label{newadd431}
	\begin{split}
	&	\mathbb{E}\big[ \mathbbm{1}_{z\xleftrightarrow{} w_2} \cdot \mathbb{P}\big( \{v_1\xleftrightarrow{(\mathcal{C}_{w_2})} w_1\}\cap \{v_1\xleftrightarrow{(\mathcal{C}_{w_2})}v_2\}^c  \mid \mathcal{F}_{\mathcal{C}_{w_2}} \big)   \big]\\
		\lesssim  &\mathbb{P}\big(\mathsf{C}[\psi]  \big)+ \mathbb{E}\big[ \mathbbm{1}_{\mathsf{C}_z\cap \{\mathrm{dist}(\mathcal{C}_{w_2},v_1)\ge \frac{\chi}{3}\}} \cdot \big(\mathrm{dist}(\mathcal{C}_{w_2},v_1) \big)^{-1} \big]. 
	\end{split}
\end{equation}
Inserting (\ref{newadd431}) into (\ref{newadd_425}), and recalling that $\chi$ is small, we obtain (\ref{newadd4.22}).

 	\textbf{Case (2): $v_2\notin I_{\{x,y\}}$.} Without loss of generality, assume that $x$ is the endpoint of $I_{\{x,y\}}$ that is closer to $v_2$, and that $m:=|v_1-x|\ge |v_2-x|$. Since $v_2\in \mathbf{B}_x(m)$, 
 	\begin{equation}\label{newadd432}
 		\begin{split}
 			\mathbb{P}\big(\mathsf{C}[\psi]  \big) \le \mathbb{I}:=  \sum\nolimits_{v_2'\in \widetilde{\partial }\mathbf{B}_x(m): v_2'\notin I_{\{x,y\}}} \mathbb{P}\big(\mathsf{C}[(v_1,v_2',w_1,w_2)]  \big). 		\end{split}
 	\end{equation}
In addition, it follows from (\ref{newadd_425}) that 
\begin{equation}\label{newadd433}
	\mathbb{I} \lesssim \chi^{\frac{1}{2}}\sum\nolimits_{z\in \partial \mathcal{B}_{{v}_1}(C_\dagger )} \sum\nolimits_{v_2'\in \widetilde{\partial }\mathbf{B}_x(m): v_2'\notin I_{\{x,y\}}}  \mathbb{J}_{z,v_2'},   
\end{equation}
 	where $\mathbb{J}_{z,v_2'}:=\mathbb{E}\big[ \mathbbm{1}_{z\xleftrightarrow{} w_2} \cdot \mathbb{P}\big( \{v_1\xleftrightarrow{(\mathcal{C}_{w_2})} w_1\}\cap \{v_1\xleftrightarrow{(\mathcal{C}_{w_2})}v_2'\}^c  \mid \mathcal{F}_{\mathcal{C}_{w_2}} \big)   \big]$. Moreover, on the event $\{v_1\xleftrightarrow{(\mathcal{C}_{w_2})} w_1\}\cap \{v_1\xleftrightarrow{(\mathcal{C}_{w_2})}v_2'\}^c$, one of the following occurs:
 \begin{enumerate}
 
   \item[(i)]  $\mathcal{C}_{w_2}$ intersects $y$. Therefore, since $v_1\xleftrightarrow{(\mathcal{C}_{w_2})} w_1$, the cluster $\mathcal{C}_{w_1}$ must intersect $x$, and thus $\mathsf{C}[(x,v_2',w_1,w_2)]$ occurs.

 	\item[(ii)] $\mathcal{C}_{w_2}$ intersects $\mathbf{B}_x(m)$ but not $y$. Consequently, $\mathsf{C}[(v_1,v_2',w_1,w_2)]$ occurs for some $v_2'\in \widetilde{\partial }\mathbf{B}_x(m)$ with $ v_2'\notin I_{\{x,y\}}$.

 	 \item[(iii)] $\mathcal{C}_{w_2}$ is disjoint from both $y$ and $\mathbf{B}_x(m)$. Hence, $\mathrm{dist}(\mathcal{C}_{w_2},v_1)\ge \frac{\chi}{3}$.

 \end{enumerate}
 	 	As a result, for any $v_2'\in \widetilde{\partial }\mathbf{B}_x(m)$ with $v_2'\notin I_{\{x,y\}}$, we have 
 	\begin{equation}\label{newadd_434}
 		\begin{split}
 		 \mathbb{J}_{z,v_2'}  \lesssim & \mathbb{P}\big(\mathsf{C}[(x,v_2',z,w_2)] \big) + \mathbb{I} \\
 		 &+\mathbb{E}\big[ \mathbbm{1}_{\{z\xleftrightarrow{} w_2,\mathrm{dist}(\mathcal{C}_{w_2},v_1)\ge \frac{\chi}{3}\}} \mathbb{P}\big( v_1\xleftrightarrow{(\mathcal{C}_{w_2})} w_1, \{v_1\xleftrightarrow{(\mathcal{C}_{w_2})} v_2' \}
  				^c   \mid \mathcal{F}_{\mathcal{C}_{w_2}} \big)  \big].
 		\end{split}
 	\end{equation}
 	Recall that it has been proved in Case (1) that 
 	\begin{equation}\label{newadd_435}
 		\mathbb{P}\big(\mathsf{C}[(x,v_2',z,w_2)] \big) \lesssim \chi^{\frac{3}{2}} \epsilon_N. 
 	\end{equation}
 	Meanwhile, by the same argument as in proving (\ref{newadd_429}), we have 
 	\begin{equation}\label{newadd_436}
 		\begin{split}
 			& \mathbb{E}\big[ \mathbbm{1}_{\{z\xleftrightarrow{} w_2,\mathrm{dist}(\mathcal{C}_{w_2},v_1)\ge \frac{\chi}{3}\}} \mathbb{P}\big( v_1\xleftrightarrow{(\mathcal{C}_{w_2})} w_1, \{v_1\xleftrightarrow{(\mathcal{C}_{w_2})} v_2' \}
  				^c   \mid \mathcal{F}_{\mathcal{C}_{w_2}} \big)  \big]\\
  				\lesssim & \mathbb{E}\big[ \mathbbm{1}_{\mathsf{C}_z\cap \{\mathrm{dist}(\mathcal{C}_{w_2},v_1)\ge \frac{\chi}{3}\}} \cdot \big(\mathrm{dist}(\mathcal{C}_{w_2},v_1) \big)^{-1} \big].
 		\end{split}
 	\end{equation}
 	Plugging (\ref{newadd_435}) and (\ref{newadd_436}) into (\ref{newadd_434}), we get
 	\begin{equation}\label{newadd437}
 			 \mathbb{J}_{z,v_2'}  \lesssim \mathbb{E}\big[ \mathbbm{1}_{\mathsf{C}_z\cap \{\mathrm{dist}(\mathcal{C}_{w_2},v_1)\ge \frac{\chi}{3}\}} \cdot \big(\mathrm{dist}(\mathcal{C}_{w_2},v_1) \big)^{-1} \big] + \chi^{\frac{3}{2}}\cdot \epsilon_N+ \mathbb{I} .
 	\end{equation}
 	Inserting (\ref{newadd437}) into (\ref{newadd433}), and using the fact that $\chi$ is small, we obtain 
 	\begin{equation}
 		\mathbb{I}\lesssim \chi^{\frac{3}{2}}  \max_{z\in \partial \mathcal{B}_{{v}_1}(C_\dagger )}  \mathbb{E}\big[ \mathbbm{1}_{\mathsf{C}_z\cap \{\mathrm{dist}(\mathcal{C}_{w_2},v_1)\ge \frac{\chi}{3}\}} \cdot \big(\mathrm{dist}(\mathcal{C}_{w_2},v_1) \big)^{-1} \big] +\chi^2 \epsilon_N.
 	\end{equation}
 	This together with (\ref{newadd432}) implies the claim (\ref{newadd4.22}) and thus completes the proof.
\end{proof}

  For the lower bound in (\ref{thm1_large_n_four_point}), we first address the low-dimensional case (i.e., $3\le d\le 5$) by proving the following lemma.

     \begin{lemma} \label{lemma_prob_Cpsi}
 	For any $3\le d\le 5$, there exist $\Cl\label{const_lemma_prob_Cpsi1},\Cl\label{const_lemma_prob_Cpsi2}>0$ such that for any $n\ge \Cref{const_lemma_prob_Cpsi1}$, $N\ge \Cref{const_lemma_prob_Cpsi2}n$ and $\psi\in \Psi(n,N)$, 
 	\begin{equation}\label{new44}
 		 \mathbb{P}\big(\mathsf{C}[\psi] \big) \gtrsim n^{ 3-\frac{d}{2} }  N^{- \frac{3d}{2}+1}.
 	\end{equation} 
 \end{lemma}

 To derive this, we need an auxiliary lemma as follows.

 \begin{lemma}\label{newlemma43}
 	 For any $3\le d\le 5$, $N\ge 1$ and $z\in \partial B(4N)$,
 	 \begin{equation}\label{new.433}
 	\widetilde{\mathbb{P}}_z\times 	 \mathbb{P}\big(  \tau_{\mathcal{C}_{\bm{0}}  } = \tau_{\mathcal{C}_{\bm{0}}\cap [\widetilde{B}(2N)\setminus \widetilde{B}(N)] }  <\infty \big) \gtrsim N^{-\frac{d}{2}+1}. 	 \end{equation}
 \end{lemma} 
  \begin{proof}
  	We claim that for any $x\in B(\frac{5N}{3})\setminus B(\frac{4N}{3})$, 
  	\begin{equation}\label{new.434}
 	\widetilde{\mathbb{P}}_z\times  \mathbb{P}\big( \bm{0} \xleftrightarrow{} 
 	x ,\tau_{x}<\tau_{\mathcal{C}_{\bm{0}}\cap [\widetilde{B}(2N)\setminus \widetilde{B}(N)]^c }  \big) \gtrsim N^{4-2d}. 
 	\end{equation}
 	In fact, assuming this claim, the desired bound (\ref{new.433}) can be derived via the second moment method. Precisely, we consider the quantity 
 	\begin{equation}
 		\mathbf{X}:= \sum\nolimits_{x\in B(\frac{5N}{3})\setminus B(\frac{4N}{3})} \mathbbm{1}_{\bm{0} \xleftrightarrow{} 
 	x  ,\tau_{x}<\tau_{\mathcal{C}_{\bm{0}}\cap [\widetilde{B}(2N)\setminus \widetilde{B}(N)]^c }}. 
 	\end{equation}
  	It directly follows from (\ref{new.434}) that $\mathbb{E}[\mathbf{X} ]\gtrsim N^{4-d}$. Moreover, using the isomorphism theorem, we have 
   \begin{equation}
   	\begin{split}
   		\mathbb{E}[\mathbf{X}^2]\lesssim & \sum\nolimits_{x_1,x_2\in B(\frac{5N}{3})\setminus B(\frac{4N}{3})}  	  \widetilde{\mathbb{P}}_z(\tau_{x_1}<\infty,\tau_{x_2}<\infty) \mathbb{P}(\bm{0}\xleftrightarrow{\ge 0} x_1,\bm{0}\xleftrightarrow{\ge 0} x_2)  \\
   		\overset{(\ref{ineq_new_lemma_cite_three_point})(\ref{app1})}{\lesssim } & N^{4-2d} \sum\nolimits_{x_1,x_2\in B(\frac{5N}{3})\setminus B(\frac{4N}{3})}  	 (|x_1-x_2|+1)^{-\frac{3d}{2}+3} \overset{(\ref{computation_d-a})}{\lesssim }  N^{-\frac{3d}{2}+7}.
   	\end{split}
   \end{equation}
   Thus, by the Paley-Zygmund inequality, we obtain (\ref{new.433}) as follows:
  	\begin{equation}
  		\begin{split}
  				 \widetilde{\mathbb{P}}_z\times 	 \mathbb{P}\big(  \tau_{\mathcal{C}_{\bm{0}}  } = \tau_{\mathcal{C}_{\bm{0}}\cap [\widetilde{B}(2N)\setminus \widetilde{B}(N)] }  <\infty \big)
  				\ge  \mathbb{P}(\mathbf{X}>0)\gtrsim \frac{(\mathbb{E}[\mathbf{X} ])^2}{\mathbb{E}[\mathbf{X}^2]} \gtrsim N^{-\frac{d}{2}+1}.
  		\end{split}
  	\end{equation}

  	It remains to prove the claim (\ref{new.434}). We take a path $\widetilde{\eta}$ consisting of at most $2d$ line segments, connecting $z$ and $\partial B(\frac{17N}{9})$, and satisfying $\mathrm{ran}(\widetilde{\eta})\cap B(\frac{16N}{9})=\emptyset$ and $\mathrm{vol}(\mathrm{ran}(\widetilde{\eta}))\le CN$. For $\delta>0$, we denote $D_\dagger^{\delta}:=\cup_{v\in \mathrm{ran}(\widetilde{\eta}_\dagger)}\mathbf{B}_{v}(\delta N)$. It follows from Lemma \ref{lemma_avoid_path_two_point} that for a sufficiently small $\delta>0$,
  	 	\begin{equation}\label{new.437}
  		\mathbb{P}\big( \bm{0} \xleftrightarrow{} 
 	x , \{\bm{0} \xleftrightarrow{} 
 	D_\dagger^{\delta} \}^c  \big) \asymp \mathbb{P}\big( x\in \mathcal{C}_{\bm{0}}\big) \asymp  N^{2-d}. 
  	\end{equation} 
  	Meanwhile, note that $D_\dagger^{\delta}\cup [\widetilde{B}(2N)\setminus \widetilde{B}(N)]$ contains both $x$ and $z$, and the distance between its boundary and $\{x,z\}$ is of order $N$. Therefore, by Harnack's inequality we have (letting $w$ be an arbitrary point in $\partial B_x(\frac{\delta N}{10})$)
  	  	\begin{equation}\label{new.438}
  		\widetilde{\mathbb{P}}_z\big( \tau_{x}< \tau_{ (D_\dagger^{\delta}\cup [\widetilde{B}(2N)\setminus \widetilde{B}(N)] )^c} \big)\asymp \widetilde{\mathbb{P}}_w\big( \tau_{x}< \tau_{ (D_\dagger^{\delta}\cup [\widetilde{B}(2N)\setminus \widetilde{B}(N)] )^c} \big) \asymp N^{2-d}. 
  	\end{equation}
  	In particular, when the two events on the left-hand sides of (\ref{new.437}) and (\ref{new.438}) both occur, the Brownian motion starting from $z$ indeed reaches $x$ before hitting the set $\mathcal{C}_{\bm{0}}\cap [\widetilde{B}(2N)\setminus \widetilde{B}(N)]^c$ (since it is disjoint from $D_\dagger^{\delta}\cup [\widetilde{B}(2N)\setminus \widetilde{B}(N)]$). Thus, by (\ref{new.437}) and (\ref{new.438}), we obtain the claim (\ref{new.434}) as follows:  
  	\begin{equation}
	\begin{split}
		& \widetilde{\mathbb{P}}_z\times  \mathbb{P}\big( \bm{0} \xleftrightarrow{} 
 	x ,\tau_{x}<\tau_{\mathcal{C}_{\bm{0}}\cap [\widetilde{B}(2N)\setminus \widetilde{B}(N)]^c }  \big) \\
 	\ge & 	\mathbb{P}\big( \bm{0} \xleftrightarrow{} 
 	x , \{\bm{0} \xleftrightarrow{} 
 	D_\dagger^{\delta} \}^c  \big) \cdot \widetilde{\mathbb{P}}_z\big( \tau_{x}< \tau_{ (D_\dagger^{\delta}\cup [\widetilde{B}(2N)\setminus \widetilde{B}(N)] )^c} \big)  \asymp N^{4-2d}.  \qedhere
	\end{split}
\end{equation}
  \end{proof}

  With Lemma \ref{newlemma43} in hand, we are ready to prove Lemma \ref{lemma_prob_Cpsi}.

    \begin{proof}[Proof of Lemma \ref{lemma_prob_Cpsi}]

  We arbitrarily choose $z\in \partial B(4N)$, and denote by $\widetilde{S}^z_\cdot \sim \widetilde{\mathbb{P}}_z$ a Brownian motion starting from $z$, independent of $\widetilde{\mathcal{L}}_{1/2}$.

  		  Let $\Omega_{n,N}:=\{x\in n\cdot \mathbb{Z}^d: B_x(n)\cap [ B(2N)\setminus B(N)]\neq \emptyset \}$. For $x\in \Omega_{n,N}$, we define 
  		  	  \begin{equation}
  		\overline{\mathsf{A}}_x:=\big\{\widetilde{S}^z_{\cdot}\ \text{first hits}\ \mathcal{C}_{\bm{0}} \ \text{within}\ \widetilde{B}_x(n)  \big\}.
  \end{equation}
  	 By $\widetilde{B}(2N)\setminus \widetilde{B}(N)\subset \cup_{x\in \Omega_{n,N}}\widetilde{B}_x(n)$ and Lemma \ref{newlemma43}, we have 
 \begin{equation}
 	\begin{split}
 		\sum\nolimits_{x\in \Omega_{n,N}} \widetilde{\mathbb{P}}_z\times \mathbb{P}\big(\overline{\mathsf{A}}_x \big) \ge   \widetilde{\mathbb{P}}_z\times 	 \mathbb{P}\big(  \tau_{\mathcal{C}_{\bm{0}}  } = \tau_{\mathcal{C}_{\bm{0}}\cap [\widetilde{B}(2N)\setminus \widetilde{B}(N)] }  <\infty \big)  
 		\gtrsim   N^{-\frac{d}{2}+1}.
 	\end{split}
 \end{equation} 
  		As a result, there exists $x_\dagger\in  \Omega_{n,N}$ such that 
  		\begin{equation}\label{new.442}
  			\widetilde{\mathbb{P}}_z\times \mathbb{P}\big(\overline{\mathsf{A}}_{x_\dagger} \big) \gtrsim |\Omega_{n,N} |^{-1}N^{-\frac{d}{2}+1} \asymp n^dN^{-\frac{3d}{2}+1}. 
  		\end{equation}
  		Since $\overline{\mathsf{A}}_{x_\dagger}\subset \{\bm{0}\xleftrightarrow{} \partial B_{x_\dagger}(n)\}$, we have  
  		\begin{equation}\label{newto444}
  			\widetilde{\mathbb{P}}_z\times \mathbb{P}\big(\overline{\mathsf{A}}_{x_\dagger} \big) \le \mathbb{E}\big[ \mathbbm{1}_{\bm{0}\xleftrightarrow{} \partial B_{x_\dagger}(n)} \cdot \widetilde{ \mathbb{P}}_z\big(\tau_{ B_{x_\dagger}(n)}<\tau_{\mathcal{C}_{\bm{0}}} \big) \big]. 
  		\end{equation}
 By the last-exit decomposition (see e.g. \cite[Section 8.2]{morters2010brownian}), one has 
 \begin{equation}
 \begin{split}
 	& \widetilde{ \mathbb{P}}_z\big(\tau_{ B_{x_\dagger}(n)}<\tau_{\mathcal{C}_{\bm{0}}} \big) \\ 
 	\lesssim & \sum_{y\in \partial B_{x_\dagger}(4dn),y'\in B_{x_\dagger}(4dn-1):\{y,y'\}\in \mathbb{L}^d}  \widetilde{G}_{\mathcal{C}_{\bm{0}}}(z,y) \cdot \widetilde{\mathbb{P}}_{y'}\big(\tau_{ B_{x_\dagger}(n)}<\tau_{\mathcal{C}_{\bm{0}}\cup \partial B_{x_\dagger}(4dn)} \big). 
 \end{split}
 \end{equation} 
Moreover, the formula (\ref{29}) implies that 
\begin{equation}
	 \widetilde{G}_{\mathcal{C}_{\bm{0}} }(z,y)\lesssim \mathbb{P}\big(z\xleftrightarrow{(\mathcal{C}_{\bm{0}})} y \big). 
\end{equation}
Meanwhile, applying the potential theory, one has 
\begin{equation}\label{newto447}
	\widetilde{\mathbb{P}}_{y'}\big(\tau_{ B_{x_\dagger}(n)}<\tau_{\mathcal{C}_{\bm{0}}\cup \partial B_{x_\dagger}(4dn)} \big)\le \widetilde{\mathbb{P}}_{y'}\big(\tau_{ B_{x_\dagger}(n)}<\tau_{  \partial B_{x_\dagger}(4dn)} \big) \lesssim n^{-1}. 
\end{equation}
Combining (\ref{newto444})-(\ref{newto447}), we obtain 
\begin{equation}
	\begin{split}
		\widetilde{\mathbb{P}}_z\times \mathbb{P}\big(\overline{\mathsf{A}}_{x_\dagger} \big)\lesssim & n^{-1}\sum\nolimits_{y\in \partial B_{x_\dagger}(4dn) } \mathbb{E}\big[ \mathbbm{1}_{\bm{0}\xleftrightarrow{} \partial B_{x_\dagger}(n)} \cdot  \mathbb{P}\big(z\xleftrightarrow{(\mathcal{C}_{\bm{0}})} y \big) \big]\\
		\overset{(\text{restriction property})}{=}&n^{-1}\sum\nolimits_{y\in \partial B_{x_\dagger}(4dn) } \mathbb{E}\big[ \mathbbm{1}_{z\xleftrightarrow{} y } \cdot  \mathbb{P}\big(\bm{0}\xleftrightarrow{(\mathcal{C}_{z})} \partial B_{x_\dagger}(n) \big) \big]\\
		\overset{(\text{Lemma}\ \ref{lemma_onecluster_box_point})}{\lesssim } &  n^{-\frac{d}{2}-1}  \sum\nolimits_{y\in \partial B_{x_\dagger}(4dn), w\in \partial \mathcal{B}_{x_\dagger}(d n) }  \mathbb{P}\big( \mathsf{C}[(y,w,z,\bm{0} )] \big), 
	\end{split}
\end{equation}
 where, according to Lemma \ref{lemma_roots}, each $\mathbb{P}\big( \mathsf{C}[(y,w,z,\bm{0} )] \big)$ is proportional to $\mathbb{P}\big( \mathsf{C}[\psi] \big)$ for all $\psi\in \Psi(n,N)$. Consequently,   
\begin{equation}
	\widetilde{\mathbb{P}}_z\times \mathbb{P}\big(\overline{\mathsf{A}}_{x_\dagger} \big)\lesssim  n^{\frac{3d}{2}-1}\cdot \mathbb{P}\big( \mathsf{C}[\psi] \big), \ \ \forall \psi\in \Psi(n,N). 
\end{equation}
  		Combined with (\ref{new.442}), it yields the desired bound (\ref{new44}).
  		\end{proof}

For the high-dimensional case $d\ge 7$, we present the following stronger result, which includes additional restrictions on the cluster volume. Recall $\mathcal{V}_\cdot(\cdot)$ in (\ref{def353}).

  \begin{lemma}\label{lemma_highd_volume_two_arm}
  	For any $d\ge 7$, there exist $\Cl\label{const_highd_volume_two_arm1},\Cl\label{const_highd_volume_two_arm2},\cl\label{const_highd_volume_two_arm3}>0$ such that for any $M\ge 1$,  $n\ge \Cref{const_highd_volume_two_arm1}$, $N\ge \Cref{const_highd_volume_two_arm2}n$ and $\psi\in \Psi(n,N)$, 
  	\begin{equation}\label{revise_lemma547}
  		\mathbb{P}\big(  \mathcal{V}_{v_1} (M) \land \mathcal{V}_{v_2} (M)  \ge \cref{const_highd_volume_two_arm3}  M^{4} ,   
		\mathsf{C}[\psi] \big) \gtrsim N^{4-2d}. 
  	\end{equation}
  \end{lemma}

By Lemmas \ref{lemma_prob_Cpsi_small_n} and \ref{lemma_highd_volume_two_arm}, we already obtain the order of $\mathbb{P} ( \mathsf{C}[\psi] )$ for $d\ge 7$: 
\begin{equation}\label{order_highd_four_points}
	\mathbb{P} ( \mathsf{C}[\psi] ) \asymp N^{4-2d},\ \ \forall \psi\in \Psi(n,N). 
\end{equation}  
  To establish Lemma \ref{lemma_highd_volume_two_arm}, we need the following technical lemma, which shows that in high dimensions, imposing an absorbing boundary on a large cluster decreases the Green's function by only a small proportion.

\begin{lemma}\label{lemma_51_high_dimension}
	For any $d\ge 7$, there exists $\Cl\label{const_disjoint_highd1}>0$ such that for any $n\ge \Cref{const_disjoint_highd1}$, $N\ge \Cref{const_disjoint_highd1}n$, $\psi=(v_1,v_2,w_1,w_2)\in \Psi(n,N)$, 
	\begin{equation}\label{newadd.5.3}
	\mathbb{E}\Big[\mathbbm{1}_{v_1\xleftrightarrow{} w_1}\cdot \big(  \widetilde{G}(v_2,w_2) -  \widetilde{G}_{\mathcal{C}_{v_1}}(v_2,w_2) \big)  \Big]\lesssim n^{6-d}N^{4-2d}. 
	\end{equation} 
\end{lemma}

 \begin{proof} 
	By (\ref{formula_two_green}), the left-hand side of (\ref{newadd.5.3}) can be written as 
	\begin{equation}\label{addnew5.18}
		\begin{split}
			&\mathbb{E}\Big[\mathbbm{1}_{v_1\xleftrightarrow{} w_1}\cdot  \sum\nolimits_{z\in \mathcal{C}_{v_1}}\widetilde{\mathbb{P}}_{v_2}\big( \tau_{\mathcal{C}_{v_1}}=\tau_z <\infty \big)\cdot \widetilde{G}(z,w_2) \Big]\\
			\lesssim  &\mathbb{I}(B_{{v}_1}(\tfrac{n}{10})) + \mathbb{I}(B_{{w}_1}(\tfrac{N}{10})) +\mathbb{J},
		\end{split}
	\end{equation}
	where $\mathbb{I}(\cdot )$ and $\mathbb{J}$ are defined by 
	\begin{equation}
		\mathbb{I}(A):= \sum\nolimits_{x\in A} \mathbb{P}\big(v_1\xleftrightarrow{} w_1, v_1\xleftrightarrow{}x \big)\cdot \widetilde{\mathbb{P}}_{v_2}\big(\tau_{x}<\infty\big)\cdot \widetilde{G}(x,w_2),  
	\end{equation}
	\begin{equation}
	\mathbb{J}:= \mathbb{E}\Big[\mathbbm{1}_{v_1\xleftrightarrow{} w_1}\cdot  \sum\nolimits_{z\in \widetilde{\partial} \widehat{\mathcal{C}}_{v_1}}\widetilde{\mathbb{P}}_{v_2}\big( \tau_{\widehat{\mathcal{C}}_{v_1}}=\tau_z <\infty \big)\cdot \widetilde{G}(z,w_2) \Big] 
	\end{equation}
	with $\widehat{\mathcal{C}}_{v_1}:=\mathcal{C}_{v_1}\setminus [B_{{v}_1}(\tfrac{n}{10})\cup B_{{w}_1}(\tfrac{N}{10})]$.

	For $\mathbb{I}(B_{{v}_1}(\tfrac{n}{10}))$, since $\mathrm{dist}(v_2,B_{{v}_1}(\tfrac{n}{10}))\gtrsim n$ and $\mathrm{dist}(v_2,B_{{w}_1}(\tfrac{N}{10}))\gtrsim N$, we have 
	\begin{equation}\label{addnew.5.21}
		\begin{split}
			\mathbb{I}(B_{{v}_1}(\tfrac{n}{10}))\lesssim & n^{2-d}N^{2-d} \sum\nolimits_{x\in B_{{v}_1}(\frac{n}{10})}  \mathbb{P}\big(v_1\xleftrightarrow{} w_1, v_1\xleftrightarrow{}x \big)\\
			\overset{(\text{Lemma}\ \ref{lemme_technical_high_dimension})}{\lesssim} & n^{6-d}N^{4-2d}.
		\end{split}
	\end{equation}
	Similarly, $\mathbb{I}(B_{{w}_1}(\tfrac{N}{10}))$ can be estimated as follows: 
	\begin{equation}\label{addnew.5.22}
		 \mathbb{I}(B_{{w}_1}(\tfrac{N}{10})) \lesssim  N^{4-2d} \sum\nolimits_{x\in B_{{v}_1}(\frac{N}{10})}  \mathbb{P}\big(v_1\xleftrightarrow{} w_1, v_1\xleftrightarrow{}x \big) 	\overset{(\text{Lemma}\ \ref{lemme_technical_high_dimension})}{\lesssim}   N^{10-3d}.
	\end{equation}

	 It remains to bound the quantity $\mathbb{J}$. By (\ref{formula_two_green}) and Lemma \ref{lemma25}, one has 
	 \begin{equation}\label{addnew5.23}
	 	\begin{split}
	 	\mathbb{J}_1:=	&\mathbb{E}\Big[\mathbbm{1}_{v_1\xleftrightarrow{} w_1, \widehat{\mathcal{L}}_{1/2}^{w_1}\ge [\ln(n)]^2}\cdot  \sum\nolimits_{z\in \widetilde{\partial}\widehat{\mathcal{C}}_{v_1}}\widetilde{\mathbb{P}}_{v_2}\big( \tau_{\widehat{\mathcal{C}}_{v_1}}=\tau_z <\infty \big)\cdot \widetilde{G}(z,w_2) \Big] \\
	 		\overset{}{=}&\mathbb{E}\Big[\mathbbm{1}_{v_1\xleftrightarrow{} w_1, \widehat{\mathcal{L}}_{1/2}^{w_1}\ge [\ln(n)]^2}\cdot \max\nolimits_{z\in \widetilde{\partial}\widehat{\mathcal{C}}_{v_1} }  \widetilde{G}(z,w_2)    \Big]\\
	 		\lesssim &N^{2-d}\cdot  \mathbb{P}\big(v_1\xleftrightarrow{} w_1, \widehat{\mathcal{L}}_{1/2}^{w_1}\ge [\ln(n)]^2  \big)  
	 		\overset{(\text{Lemma}\ \ref{lemma25})}{\lesssim }  e^{-c [\ln(N)]^2 }N^{4-2d}. 
	 	\end{split}
	 \end{equation}
	Let $\mathbb{J}_2:=\mathbb{J}-\mathbb{J}_1$. By the switching identity (see Lemma \ref{lemma_switching}) one has 
	 \begin{equation}\label{524}
	 	\begin{split}
	 		\mathbb{J}_2 \lesssim  & N^{2-d} \int_{0\le a,b\le [\ln(n)]^2} \widecheck{\mathbb{E}}_{v_1\leftrightarrow w_1,a,b}\Big[  \sum\nolimits_{z\in \widetilde{\partial}\widehat{\mathcal{C}}_{v_1}}\widetilde{\mathbb{P}}_{v_2}\big( \tau_{\widehat{\mathcal{C}}_{v_1}}=\tau_z <\infty \big)  \widetilde{G}(z,w_2)  \Big] \\
	 		&\ \ \ \ \ \ \ \ \ \ \ \ \ \ \ \ \ \ \ \ \ \ \ \  \cdot  \mathfrak{p}_{v_1\leftrightarrow w_1,a,b} \mathrm{d}a  \mathrm{d}b.
	 	\end{split}
	 \end{equation}
     Moreover, under the measure $\widecheck{\mathbb{E}}_{v_1\leftrightarrow w_1,a,b}$, the number of excursions in $\mathcal{P}^{(2)}$ that intersect $\partial B_{{v}_1}(\frac{n}{100})$ is a Poisson random variable with parameter
     \begin{equation}\label{new525}
     	a \cdot \mathbf{e}_{v_1,v_1}^{\{w_1\}}\big(\big\{ \widetilde{\eta}: \mathrm{ran}(\widetilde{\eta})\cap \partial B_{{v}_1}(\tfrac{n}{100})\neq \emptyset \big\} \big) \overset{(\text{Lemma}\ \ref{lemma_new216})}{\lesssim} a\cdot n^{2-d}. 
     \end{equation}
	 Therefore, for any $0\le a\le [\ln(n)]^2$, we have 
	 \begin{equation}\label{526}
	 	\begin{split}
	 		&\widecheck{\mathbb{E}}_{v_1\leftrightarrow w_1,a,b}\Big[\mathbbm{1}_{ (\cup \mathcal{P}^{(2)})\cap \partial B_{{v}_1}(\frac{n}{100})\neq  \emptyset} \cdot  \sum\nolimits_{z\in \widetilde{\partial}\widehat{\mathcal{C}}_{v_1}}\widetilde{\mathbb{P}}_{v_2}\big( \tau_{\widehat{\mathcal{C}}_{v_1}}=\tau_z <\infty \big)\cdot \widetilde{G}(z,w_2)  \Big]\\
	 		\overset{}{\le }& \widecheck{\mathbb{E}}_{v_1\leftrightarrow w_1,a,b}\Big[\mathbbm{1}_{ (\cup \mathcal{P}^{(2)})\cap \partial B_{{v}_1}(\frac{n}{100})\neq  \emptyset} \cdot \max\nolimits_{z\in \widetilde{\partial}\widehat{\mathcal{C}}_{v_1}} \widetilde{G}(z,w_2) \Big]\\
	 		\lesssim & N^{2-d}\cdot \widecheck{\mathbb{P}}_{v_1\leftrightarrow w_1,a,b}\big( (\cup \mathcal{P}^{(2)})\cap \partial B_{{v}_1}(\tfrac{n}{100})\neq  \emptyset \big) \overset{(\ref{new525})}{\lesssim } [\ln(n)]^2 n^{2-d}N^{2-d}. 
	 	\end{split}
	 \end{equation}
	 Similarly, we also have: for any $0\le b\le [\ln(n)]^2$, 
	 \begin{equation}\label{527}
	 	\begin{split}
	 		&\widecheck{\mathbb{E}}_{v_1\leftrightarrow w_1,a,b}\Big[\mathbbm{1}_{ (\cup \mathcal{P}^{(3)})\cap \partial B_{{w}_1}(\frac{N}{100})\neq  \emptyset} \cdot  \sum\nolimits_{z\in \widetilde{\partial}\widehat{\mathcal{C}}_{v_1}}\widetilde{\mathbb{P}}_{v_2}\big( \tau_{\widehat{\mathcal{C}}_{v_1}}=\tau_z <\infty \big)\cdot \widetilde{G}(z,w_2)  \Big]\\
	 		\lesssim &N^{2-d} \cdot \widecheck{\mathbb{P}}_{v_1\leftrightarrow w_1,a,b}\big( (\cup \mathcal{P}^{(3)})\cap \partial B_{{w}_1}(\tfrac{N}{100})\neq  \emptyset \big) 	\lesssim   [\ln(n)]^2N^{4-2d}.  
	 	\end{split}
	 \end{equation}
	 Let $\mathsf{F}:=\{\cup \mathcal{P}^{(2)}\subset \widetilde{B}_{{v}_1}(\frac{n}{100}) ,\cup \mathcal{P}^{(3)}\subset \widetilde{B}_{{w}_1}(\frac{N}{100})\}$. By (\ref{524}), (\ref{526}) and (\ref{527}), 
	 \begin{equation}\label{new.add.5.28}
	 	\begin{split}
	 			\mathbb{J}_2\lesssim & N^{2-d} \int_{0\le a,b\le [\ln(n)]^2} \widecheck{\mathbb{E}}_{v_1\leftrightarrow w_1,a,b}\Big[ \mathbbm{1}_{\mathsf{F}}   \sum_{z\in \widetilde{\partial}\widehat{\mathcal{C}}_{v_1}}\widetilde{\mathbb{P}}_{v_2}\big( \tau_{\widehat{\mathcal{C}}_{v_1}}=\tau_z <\infty \big)  \widetilde{G}(z,w_2)  \Big]\\
	 		&\ \ \ \ \ \ \ \ \ \ \ \ \ \ \ \ \ \ \ \ \ \ \ 	\cdot  \mathfrak{p}_{v_1\leftrightarrow w_1,a,b} \mathrm{d}a  \mathrm{d}b + [\ln(n)]^2 n^{2-d}N^{4-2d}.	 	\end{split}
	 \end{equation}
 On the event $\mathsf{F}$, when $\widetilde{S}_\cdot \sim \widetilde{\mathbb{P}}_{v_2}$ first hits $\widehat{\mathcal{C}}_{v_1}$ at $z$, one of the following occurs: 
 \begin{enumerate}

 	\item[(i)]  $z$ is connected to $[\cup (\mathcal{P}^{(2)}+\mathcal{P}^{(4)}) ]\cap B_{{v}_1}(\tfrac{n}{100})$ by $\cup \mathcal{P}^{(1)}$;

 	\item[(ii)] $z$ is connected to $[\cup (\mathcal{P}^{(3)}+\mathcal{P}^{(4)})]\cap B_{{w}_1}(\tfrac{N}{100})$ by $\cup \mathcal{P}^{(1)}$;

 	\item[(iii)] $z$ is connected to $(\cup \mathcal{P}^{(4)}) \setminus [ B_{{v}_1}(\tfrac{n}{100})\cup   B_{{w}_1}(\tfrac{N}{100})]$ by $\cup \mathcal{P}^{(1)}$.

 \end{enumerate}
For $\diamond\in \{\mathrm{i},\mathrm{ii},\mathrm{iii}\}$, we define $\mathsf{F}_\diamond$ as the sub-event of $\mathsf{F}$ restricted to Case $(\diamond)$, and let  
\begin{equation}
	\mathbb{U}_\diamond:=  \widecheck{\mathbb{E}}_{v_1\leftrightarrow w_1,a,b}\Big[ \mathbbm{1}_{\mathsf{F}_\diamond} \cdot    \sum\nolimits_{z\in \widetilde{\partial} \widehat{\mathcal{C}}_{v_1}}\widetilde{\mathbb{P}}_{v_2}\big( \tau_{\widehat{\mathcal{C}}_{v_1}}=\tau_z <\infty \big)  \widetilde{G}(z,w_2)  \Big]. 
\end{equation}

 For $\mathbb{U}_{\mathrm{i}}$, by the union bound, $\mathbb{U}_{\mathrm{i}}$ is at most proportional to 
 \begin{equation}\label{newadd.5.30}
 	\begin{split}
 	 &  \Big(  \widecheck{\mathbb{E}}_{v_1\leftrightarrow w_1,a,b}\big[ \mathrm{vol}\big(\cup \mathcal{P}^{(2)}\cap B_{{v}_1}(\tfrac{n}{100}) \big) \big] + \widecheck{\mathbb{E}}_{v_1\leftrightarrow w_1,a,b}\big[ \mathrm{vol}\big( \cup \mathcal{P}^{(4)}) \cap B_{{v}_1}(\tfrac{n}{100})  \big]  \Big)\\
 		&\cdot \max_{y\in B_{{v}_1}(\frac{n}{100})} \sum_{z \in  [B_{{v}_1}(\frac{n}{10})\cup B_{{w}_1}(\frac{N}{10})]^c}  \mathbb{P}\big(y\xleftrightarrow{} \widetilde{B}_z(1) \big)  \widetilde{\mathbb{P}}_{v_2}\big(\tau_{z}<\infty\big)  \widetilde{G}(z,w_2). 
 	\end{split}
 \end{equation}
 Since $\mathbf{e}_{v_1,v_1}\big( \big\{ \widetilde{\eta}: y\in \mathrm{ran}(\widetilde{\eta}) \big\}  \big)\lesssim |v_1-y|^{4-2d}$ for all $y\in [\mathbf{B}_{v_1}(1)]^c$, one has 
 \begin{equation}\label{addneww531}
 	\begin{split}
 		& \widecheck{\mathbb{E}}_{v_1\leftrightarrow w_1,a,b}\big[ \mathrm{vol}\big(\cup \mathcal{P}^{(2)}\cap B_{{v}_1}(\tfrac{n}{100}) \big) \big] 
 		 \lesssim   \sum_{y\in B_{{v}_1}(\frac{n}{100})} (|v_1-y|+1)^{4-2d} \overset{(\ref{computation_d-a})}{\lesssim} 1. 
 	\end{split}
 \end{equation} 
Meanwhile, under the measure $\mathbf{e}_{v_1,w_1}$, the probability of an excursion intersecting $y\in B_{{v}_1}(\frac{n}{100})$ is at most $(|v_1-y|+1)^{2-d}$. Therefore, by (\ref{property_poisson_odd_number}) we have 
\begin{equation}\label{newaddd_5.33}
	 \widecheck{\mathbb{E}}_{v_1\leftrightarrow w_1,a,b}\big[ \mathrm{vol}\big(\cup \mathcal{P}^{(4)}\cap B_{{v}_1}(\tfrac{n}{100}) \big) \big]  \lesssim  \sum\nolimits_{y \in B_{{v}_1}(\frac{n}{100})} (|v_1-y|+1)^{2-d} \overset{(\ref{computation_d-a})}{\lesssim} n^2.
\end{equation} 
For any $y\in B_{{v}_1}(\frac{n}{100})$ and $z\in [B_{{v}_1}(\frac{n}{10})\cup B_{{w}_1}(\frac{N}{10})]^c$, since $|y-z|\asymp |z-v_1|$, one has 
 \begin{equation}\label{newaddd_5.34}
 	\mathbb{P}\big(y\xleftrightarrow{} \widetilde{B}_z(1) \big)\asymp (|z-v_1|+1)^{2-d}.  \end{equation} 
 	Thus, plugging (\ref{addneww531}), (\ref{newaddd_5.33}) and (\ref{newaddd_5.34}) into (\ref{newadd.5.30}), we derive that  
 \begin{equation}\label{addnew.5.32}
 	\begin{split}
 		\mathbb{U}_{\mathrm{i}} \lesssim n^2 \sum\nolimits_{z \in \mathbb{Z}^d}  (|z-v_1|+1)^{2-d} (|v_2-z|+1)^{2-d}(|z-w_2|+1)^{2-d}. 
 	\end{split}
 \end{equation} 
Since $|v_2-w_2|\gtrsim N$, one has $ |v_2-z|\vee |z-w_2|\gtrsim N$. Restricted to $ |v_2-z|\gtrsim N$, the sum in (\ref{addnew.5.32}) is bounded from above by 
 \begin{equation}\label{addnew.5.33}
 	CN^{2-d}\sum\nolimits_{z\in \mathbb{Z}^d}(|z-v_1|+1)^{2-d} (|z-w_2|+1)^{2-d}\overset{(\ref{computation_2-d_2-d})}{\lesssim} N^{6-2d}.
 \end{equation}
Similarly, when restricted to $|z-w_2|\gtrsim N$, this sum is upper-bounded by 
\begin{equation}\label{addnew.5.34}
	CN^{2-d}\sum\nolimits_{z\in \mathbb{Z}^d}(|z-v_1|+1)^{2-d} (|v_2-z|+1)^{2-d}\lesssim  n^{4-d}N^{2-d}. 
\end{equation}
 Combining (\ref{addnew.5.32}), (\ref{addnew.5.33}) and (\ref{addnew.5.34}), we get 
 \begin{equation}\label{revise_newadd_568}
 	\mathbb{U}_{\mathrm{i}}\lesssim n^{2}\big(N^{6-2d} +n^{4-d}N^{2-d}  \big) \lesssim n^{6-d}N^{2-d}. 
 \end{equation}

 For $\mathbb{U}_{\mathrm{ii}}$, by the same argument as in the proof of (\ref{addnew.5.32}), one has 
 \begin{equation}\label{newadd.536}
 \begin{split}
  	\mathbb{U}_{\mathrm{ii}} \lesssim & N^2 \sum\nolimits_{z \in  \mathbb{Z}^d}  (|z-w_1|+1)^{2-d} (|v_2-z|+1)^{2-d}(|z-w_2|+1)^{2-d}\\
\overset{(|z-w_1|\gtrsim N)}{\lesssim} &	N^{4-d}\sum\nolimits_{z \in  \mathbb{Z}^d}  (|v_2-z|+1)^{2-d}(|z -w_2|+1)^{2-d} \overset{(\ref{computation_2-d_2-d})}{ \lesssim} N^{8-2d}. 
 \end{split}
 \end{equation}

 We now estimate $\mathbb{U}_{\mathrm{iii}}$. Recalling the construction of Case (iii), we have 
 \begin{equation}\label{newadd.5.37}
 	\begin{split}
 	 \mathbb{U}_{\mathrm{iii}}  
 		\lesssim   &\sum_{y\in [B_{{v}_1}(\frac{n}{100})\cup B_{{w}_1}(\frac{N}{100})]^c,z\in  [B_{{v}_1}(\frac{n}{10})\cup B_{{w}_1}(\frac{N}{10})]^c} \widecheck{\mathbb{P}}_{v_1\leftrightarrow w_1,a,b}\big( y\in \cup \mathcal{P}^{(4)} \big)\\
 		& \ \ \ \ \ \ \ \ \ \ \ \ \ \ \ \ \ \ \ \ \ \ \ \ \ \ \ \ \ \cdot \mathbb{P}\big(y \xleftrightarrow{ } \widetilde{B}_z(1) \big) \cdot  \widetilde{\mathbb{P}}_{v_2}\big(\tau_{z}<\infty\big)\cdot \widetilde{G}(z , w_2).
 		 	\end{split}
 \end{equation}
 In addition, by (\ref{property_poisson_odd_number}), the probability $\widecheck{\mathbb{P}}_{v_1\leftrightarrow w_1,a,b}\big( y\in \cup \mathcal{P}^{(4)} \big)$ is at most
\begin{equation}\label{newadd.5.38}
\begin{split}
 &	\max_{y'\in \partial B_{{v_1}}(\frac{n}{200})} \widetilde{\mathbb{P}}_{y'}\big(\tau_{y} < \tau_{\{v_1,w_1\}} \mid \tau_{\{v_1,w_1\}}=\tau_{w_1}<\infty \big)\\
		\lesssim  & N^{d-2}(|v_1-y|+1)^{2-d}(|y-w_1|+1)^{2-d} \overset{(|y-w_1|\gtrsim N)}{\lesssim } (|v_1-y|+1)^{2-d}.
\end{split}
\end{equation} 
	 By (\ref{newadd.5.37}), (\ref{newadd.5.38}) and $\mathbb{P}\big(y \xleftrightarrow{ } \widetilde{B}_z(1) \big)\asymp (|y-z|+1)^{2-d}$, we have  
	 \begin{equation}\label{revise_new_572}
	 	\begin{split}
	 		\mathbb{U}_{\mathrm{iii}}  \lesssim & \sum\nolimits_{y,z\in \mathbb{Z}^d} (|v_1-y|+1)^{2-d} (|y-z|+1)^{2-d}\\
	 		&\ \ \ \ \ \ \ \ \ \ \  \cdot (|v_2-z|+1)^{2-d}(|z-w_2|+1)^{2-d}\\
	 	\overset{(\ref{computation_2-d_2-d})}{	\lesssim }&\sum\nolimits_{z\in \mathbb{Z}^d} (|v_1-z|+1)^{4-d}(|v_2-z|+1)^{2-d}(|z-w_2|+1)^{2-d}. 
	 	\end{split}
	 \end{equation}
	 As in (\ref{addnew.5.33}), restricted to $|v_2-z|\gtrsim N$, this sum is bounded from above by 
	 \begin{equation}\label{newadd.5.40}
	 	CN^{2-d} \sum\nolimits_{z\in \mathbb{Z}^d} (|v_1-z|+1)^{4-d}(|z-w_2|+1)^{2-d}\lesssim N^{8-2d}.
	 \end{equation}
	 Here we used the elementary estimate that (see e.g. \cite[(4.10)]{cai2024high})
	 \begin{equation}\label{computation_4-d_2-d}
	 	\sum\nolimits_{x\in \mathbb{Z}^d} (|v-x|+1)^{4-d}(|x-w|+1)^{2-d} \lesssim (|v-w|+1)^{6-d}.
	 \end{equation}
	 This estimate also implies that restricted to $|z-w_2|\gtrsim N$, the sum is at most 
	 \begin{equation}\label{newadd.5.42}
	 	CN^{2-d}\sum\nolimits_{z \in \mathbb{Z}^d} (|v_1-z|+1)^{4-d}(|v_2-z|+1)^{2-d}\lesssim n^{6-d}N^{2-d}. 
	 \end{equation}
	 Combining (\ref{revise_new_572}), (\ref{newadd.5.40}) and (\ref{newadd.5.42}), we have  
	 	 \begin{equation}\label{newadd.543}
	 	\mathbb{U}_{\mathrm{iii}}  \lesssim  N^{8-2d}+ n^{6-d}N^{2-d}\lesssim n^{6-d}N^{2-d}. 
	 \end{equation}

	 Plugging (\ref{revise_newadd_568}), (\ref{newadd.536}), (\ref{newadd.543}) into (\ref{new.add.5.28}), we obtain
	 \begin{equation}\label{addnew5.44}
	 	\begin{split}
	 		\mathbb{J}_2\lesssim N^{2-d}\cdot \big( n^{6-d}N^{2-d}+ N^{8-2d}   \big) +  [\ln(n)]^2 n^{2-d}N^{4-2d}\lesssim n^{6-d}N^{4-2d}. 
	 	\end{split}
	 \end{equation}
	 By substituting (\ref{addnew.5.21}), (\ref{addnew.5.22}), (\ref{addnew5.23}),  (\ref{addnew5.44}) into (\ref{addnew5.18}), we derive (\ref{newadd.5.3}).  
	   \end{proof}

 Now we are ready to prove Lemma \ref{lemma_highd_volume_two_arm}.

 \begin{proof}[Proof of Lemma \ref{lemma_highd_volume_two_arm}]
	According to Lemma \ref{lemma_hit_capacity}, there exists $c_\dagger>0$ such that for any $\psi=(v_1,v_2,w_1,w_2)\in \Psi(n,N)$,
 \begin{equation}\label{addnew55}
 	\mathbb{P}\big(\mathcal{V}_{v_i} (M)\ge c_\dagger M^{4} , v_i\xleftrightarrow{} w_i \big) \ge \tfrac{3}{4}\cdot \mathbb{P}\big(v_i\xleftrightarrow{} w_i \big), \ \ \forall i\in \{1,2\}. 
 \end{equation}
 For any $i\in \{1,2\}$, we denote the event 
 \begin{equation}
 	\mathsf{F}_{i}:=  \big\{ \mathcal{V}_{v_i} (M)\ge c_\dagger M^{4},v_i \xleftrightarrow{} w_i  \big\}. 
 \end{equation}
 By the restriction property, one has 
 \begin{equation}\label{newadd57}
 	\begin{split}
 		  \mathbb{P}\big(  \mathcal{V}_{v_i} (M)   \ge c_\dagger M^{4} ,   
		\mathsf{C}[\psi] \big)  
	\ge  & \mathbb{E}\Big[ \mathbbm{1}_{\mathsf{F}_{i } }\cdot \mathbb{P}\big( v_{3-i}\xleftrightarrow{(\mathcal{C}_{v_i})} w_{3-i}  \big)  \Big]\\
	\overset{(\text{Lemma} \ \ref{lemma_different_two_point})}{\ge} &\mathbb{I}_1-C\cdot  \mathbb{I}_2 ,
 	\end{split}
 \end{equation}
 where $\mathbb{I}_1$ and $\mathbb{I}_2$ are defined by 
 \begin{equation}
 	\mathbb{I}_1 := \mathbb{P}\big(\mathsf{F}_{i } \big)\cdot  \mathbb{P}\big( v_{3-i}\xleftrightarrow{} w_{3-i}  \big), 
 \end{equation}
 \begin{equation}
 	\mathbb{I}_2:= \mathbb{E}\big[ \mathbbm{1}_{v_i \xleftrightarrow{} w_i}\cdot \big(  \widetilde{G}(v_{3-i},w_{3-i}) -  \widetilde{G}_{\mathcal{C}_{v_i}}(v_{3-i},w_{3-i}) \big) \big].
 \end{equation}

 For $\mathbb{I}_1$, by (\ref{addnew55}) we have 
\begin{equation}\label{newadd.511}
	\mathbb{I}_1 \ge  \tfrac{3}{4}  \cdot  \mathbb{P}\big( v_{1}\xleftrightarrow{} w_{1}  \big)\cdot  \mathbb{P}\big( v_{2}\xleftrightarrow{} w_{2}  \big). 
\end{equation} 
 For $\mathbb{I}_2$, it directly follows from Lemma \ref{lemma_51_high_dimension} that 
\begin{equation}\label{newadd.512}
	\mathbb{I}_2 \lesssim n^{6-d}N^{4-2d}\asymp n^{6-d}\cdot \mathbb{P}\big( v_{1}\xleftrightarrow{} w_{1}  \big)\cdot  \mathbb{P}\big( v_{2}\xleftrightarrow{} w_{2}  \big).  
\end{equation}
Plugging (\ref{newadd.511}) and (\ref{newadd.512}) into (\ref{newadd57}), we have 
\begin{equation}\label{ineq513}
\begin{split}
		   \mathbb{P}\big(  \mathcal{V}_{v_i} (M) \ge c_\dagger M^{4} ,   
		\mathsf{C}[\psi] \big)  
	   \ge   &   \big[ \tfrac{3}{4}  -C  n^{6-d}  \big]\cdot \mathbb{P}\big( v_{1}\xleftrightarrow{} w_{1}  \big)\cdot  \mathbb{P}\big( v_{2}\xleftrightarrow{} w_{2}  \big) \\
	  \overset{}{\ge} & \tfrac{2}{3}\cdot \mathbb{P}\big( v_{1}\xleftrightarrow{} w_{1}  \big)\cdot  \mathbb{P}\big( v_{2}\xleftrightarrow{} w_{2}  \big).
\end{split}
\end{equation}
Thus, applying the inclusion-exclusion principle, we obtain (\ref{revise_lemma547}) as follows:   
\begin{equation}\label{addnew513}
\begin{split}
		 & \mathbb{P}\big(  \mathcal{V}_{v_1} (M) \land \mathcal{V}_{v_2} (M)  \ge c_\dagger M^{4} ,   
		\mathsf{C}[\psi] \big) \\
		\ge & \sum\nolimits_{i\in \{1,2\}}  \mathbb{P}\big(  \mathcal{V}_{v_i} (M) \ge c_\dagger M^{4} ,   
		\mathsf{C}[\psi] \big) -  \mathbb{P}\big( \mathsf{C}[\psi] \big)\\
		 \overset{(\ref{new.45}),(\ref{ineq513})}{\ge} &  \tfrac{1}{3} \cdot \mathbb{P}\big( v_{1}\xleftrightarrow{} w_{1}  \big)\cdot  \mathbb{P}\big( v_{2}\xleftrightarrow{} w_{2}  \big)  \gtrsim N^{4-2d}. \qedhere
\end{split}
\end{equation}
\end{proof}

 {\color{ForestGreen}

}

By Lemmas \ref{lemma_prob_Cpsi_small_n}, \ref{lemma_prob_Cpsi} and \ref{lemma_highd_volume_two_arm}, we have verified the upper bound in (\ref{thm1_small_n_four_point}) and the lower bound in (\ref{thm1_large_n_four_point}). 
\qed

\vspace{0.2cm}



We record the analogue of Lemma \ref{lemma_separation} for $d\ge 7$ as follows.

\begin{lemma}

Under the same assumptions as in Lemma \ref{lemma_highd_avoid_box_twopoint}, one has 
 \begin{equation}
		\mathbb{P}\big( z_i^{\diamond} \xleftrightarrow{}  \widetilde{B}_x(\cref{const_highd_avoid_box_twopoint2} m_\diamond ) \cup \mathfrak{B}^{\diamond}_{m_\diamond,\cref{const_highd_avoid_box_twopoint2}}  \mid  \mathsf{C}[\psi] \big) \le  \epsilon. 
\end{equation}
\end{lemma}
\begin{proof}
	As in Lemma \ref{lemma_avoid_path_two_point}, we only consider the case $\diamond=\mathrm{out}$ and $i=1$. Using the BRK inequality and Lemma \ref{lemma_highd_avoid_box_twopoint}, we have 
	\begin{equation}
		\begin{split}
			& 	\mathbb{P}\big( w_1 \xleftrightarrow{}  B_x(\delta N ) \cup [B(\delta^{-1}N)]^c  ,   \mathsf{C}[\psi] \big) \\
				\lesssim & \mathbb{P}\big( w_1 \xleftrightarrow{}  B_x(\delta N ) \cup [B(\delta^{-1}N)]^c  ,   w_1\xleftrightarrow{} v_1 \big) \cdot   \mathbb{P}\big(   v_2 \xleftrightarrow{} w_2 \big)\lesssim o_{\delta}(1)N^{4-2d}. 
		\end{split}
	\end{equation}
	Combined with (\ref{order_highd_four_points}), it implies this lemma. 
\end{proof}

 \section{Estimates on heterochromatic two-arm probabilities}\label{section_estimate_coexistence}

 This section is mainly devoted to proving Theorem \ref{thm1}, which provides the exact order of the probability $\mathbb{P}(\mathsf{H}_{v'}^{v}(N))$. Unless specified otherwise, we assume that $N$ is sufficiently large and that $v,v'\in \widetilde{B}(cN)$. By the isomorphism theorem, it suffices to prove the analogue of Theorem \ref{thm1} for $\overline{\mathsf{H}}_{v'}^{v}(N):=\big\{v\xleftrightarrow{} \partial B(N) \Vert v'\xleftrightarrow{} \partial B(N) \big\}$: 
  	 \begin{align}
	 	 &\text{when}\ \chi=:|v-v'|\le C, \  \mathbb{P}\big(\overline{\mathsf{H}}_{v'}^{v}(N) \big)\asymp \chi^{\frac{3}{2}}N^{-[(\frac{d}{2}+1)\boxdot 4]};   \label{thm1_small_n_bar}\\
	&\text{when}\ \chi\ge C,\   \ \ \ \  \  \ \ \ \ \ \ \   \  \ 	\mathbb{P}\big(\overline{\mathsf{H}}_{v'}^{v}(N) \big)\asymp \chi^{(3-\frac{d}{2})\boxdot 0}N^{-[(\frac{d}{2}+1)\boxdot 4]}.   \label{thm1_large_n_bar} 
	\end{align}
	  We first present a quantitative relation between the probability of $\mathsf{C}[\psi]$ and that of a modified event where certain points are replaced by boxes and some restrictions are imposed on the volume and capacity of the cluster. We then establish (\ref{thm1_small_n_bar}) and (\ref{thm1_large_n_bar}) in Sections \ref{newsubsection4.1} and \ref{newsubsection4.2} respectively. In Section \ref{subsection_theorem_four_point_remaining}, we prove the remaining estimates in Theorem \ref{theorem_four_point}.

	 Before showing the result, we record some notations as follows.


\begin{itemize}


	\item  For $\psi=(v_1,v_2,w_1,w_2)\in \Psi(n,N)$, $\beta \in (0,1)$ and $\diamond \in \{\mathrm{in},\mathrm{out}\}$, we define 
 \begin{equation}\label{newfix489}
 	\mathsf{C}^{\diamond}[\psi,\beta]:= \big\{ B_{{z}_1^{\diamond}}(\beta m_{\diamond})\xleftrightarrow{} z_1^{-\diamond} \Vert B_{{z}_2^{\diamond}}(\beta m_{\diamond})\xleftrightarrow{} z_2^{-\diamond} \big\}. 
 \end{equation}

	\item  For any $\delta \in (0,1)$, $i\in \{1,2\}$ and $\diamond \in \{\mathrm{in},\mathrm{out}\}$, we denote
 \begin{equation}\label{newfix490}
 	 \mathsf{S}^{\diamond}_i[\psi,\delta]:=\big\{ 	z_i^{-\diamond}\xleftrightarrow{} B_{{z}_{3-i}^{\diamond}}(\delta m_{\diamond})\cup B_{{z}_{3-i}^{-\diamond}}( \delta m_{-\diamond}) \cup B(\delta n) \cup [B(\delta^{-1}N)]^c \big\}.  
 \end{equation}
   In addition, we denote $\mathsf{S} ^{\diamond}[\psi,\delta]:=\cup_{i\in \{1,2\}}\mathsf{S}_i^{\diamond}[\psi,\delta]$.


	\item  For any $z,z'  \in \widetilde{\mathbb{Z}}^d$ and $M\ge 1$, we define	\begin{equation}\label{def353}
		\mathcal{V}_{z,z'}(M):= \mathrm{vol}\big(\mathcal{C}_{z}\cap \widetilde{B}_{{z}'}(M)\big), \ \  \mathcal{T}_{z,z'}(M):= \mathrm{cap}\big(\mathcal{C}_{z}\cap \widetilde{B}_{{z}'}(M)\big). 
	\end{equation}
	Especially, when $z=z'$, we denote $\mathcal{V}_{z}:=\mathcal{V}_{z,z}$ and $\mathcal{T}_{z}:=\mathcal{T}_{z,z}$ for brevity.

	\item For $\diamond\in \{\mathrm{in},\mathrm{out}\}$, we define the truncated event $\widehat{\mathsf{C}}^{\diamond}$ separately in low- and high-dimensional cases as follows:
	 \begin{enumerate}

	 \item[(i)] When $3\le d \le 5$, for any $\gamma>0$, define 
	 \begin{equation*}
	 	\widehat{\mathsf{C}}^{\diamond}[\psi,\beta,\delta,\gamma]:=    \bigcap_{i\in \{1,2\}} \big\{ \mathcal{V}_{z_i^{-\diamond},z_i^{\diamond}}(\beta m_{\diamond})\ge \gamma m_{\diamond}^{ \frac{d}{2}+1 }, \mathcal{T}_{z_i^{-\diamond},z_i^{\diamond}}(\beta m_{\diamond})\ge \gamma m_{\diamond}^{ d-2 }  \big\}\cap \big(\mathsf{S}^{\diamond}[\psi,\delta] \big)^c;
	 \end{equation*}
	 In particular, for any $n\ge 1$, $m\ge Cn$, $N\ge Cm$, $\psi=(v_1,v_2,w_1,w_2)\in \Psi(n,m)$ and $\psi'=(v_1',v_2',w_1',w_2')\in \Psi(10m,N)$, when $\widehat{\mathsf{C}}^{\mathrm{out}}[\psi,\beta,\delta,\gamma]$ and $\widehat{\mathsf{C}}^{\mathrm{in}}[\psi',\beta,\delta,\gamma]$ both occur, since the four involved clusters all have capacity $cm^{d-2}$ inside $\widetilde{B}(100m)$, one may add two loops to $\widetilde{\mathcal{L}}_{1/2}$ that connect $\mathcal{C}_{v_1}$ to $\mathcal{C}_{v_1'}$, and $\mathcal{C}_{v_2}$ to $\mathcal{C}_{v_2'}$ respectively, while changing the probability by only a constant factor; in addition, after this modification, the event $\mathsf{C}^{\mathrm{out}}[(v_1,v_2,w_1',w_2'),\beta]\cap \{\mathcal{V}_{v_1}(100m)\land \mathcal{V}_{v_2}(100m)\ge \gamma m^{\frac{d}{2}+1}\}$ occurs. As shown in later sections, this observation plays a key role in proving the upper bounds in (\ref{thm1_large_n}) and (\ref{thm1_large_n_four_point}), and Theorem \ref{thm3_volume} in the low-dimensional case.

	 \item[(ii)]  When $d\ge 7$, define 
	 \begin{equation}
	 	\widehat{\mathsf{C}}^{\diamond}[\psi,\beta,\delta]:= \mathsf{C}^{\diamond}[\psi,\beta]\cap \big(\mathsf{S}^{\diamond}[\psi,\delta] \big)^c. 
	 \end{equation}
	 Here we do not impose the volume and capacity constraints as above because the proofs in high dimensions follow a different approach from the one described in Item (i).

	 \end{enumerate}

\end{itemize}

  \begin{lemma}\label{lemma_point_to_boundary}
 	 For any $d\ge 3$ with $d\neq 6$, there exist $\Cl\label{const_point_to_boundary1},\cl\label{const_point_to_boundary2},\cl\label{const_point_to_boundary3},\cl\label{const_point_to_boundary4}>0$ such that for any $n\ge \Cref{const_point_to_boundary1}$, $N\ge \Cref{const_point_to_boundary1}n$, $\psi\in \Psi(n,N)$ and $\diamond \in \{\mathrm{in},\mathrm{out}\}$, the following hold:
 	  	   	 \begin{align}
 	&\text{when}\ 3\le d\le 5,\   \mathbb{P}\big( \mathsf{C}^{\diamond}[\psi, \cref{const_point_to_boundary2}]  \big) \asymp  m^{d-2}_{\diamond} \cdot 	\mathbb{P}\big(  \mathsf{C}[\psi]  \big)  \asymp    \mathbb{P}\big( \widehat{\mathsf{C}}^{\diamond}[\psi, \cref{const_point_to_boundary2},\cref{const_point_to_boundary3},\cref{const_point_to_boundary4}]  \big); \label{good360}  \\ 
 	&\text{when}\ d\ge 7,  \ \ \ \ \ \   \mathbb{P}\big( \mathsf{C}^{\diamond}[\psi, \cref{const_point_to_boundary2}]  \big) \asymp  m_{\diamond}^{2d-8}\cdot 	\mathbb{P}\big(  \mathsf{C}[\psi]  \big)\asymp    \mathbb{P}\big( \widehat{\mathsf{C}}^{\diamond}[\psi, \cref{const_point_to_boundary2},\cref{const_point_to_boundary3}]  \big) . \label{good361}
 \end{align} 	 
 \end{lemma}

 \begin{proof}

 	 We provide the proof details only for the case $\diamond=\mathrm{in}$, as the case $\diamond=\mathrm{out}$ is similar. Throughout this proof, we assume that $\psi=(v_1,v_2,w_1,w_2)$.

 Using the restriction property and Lemma \ref{lemma_onecluster_box_point}, we have 
  \begin{equation}\label{ineq364}
 	\begin{split}
 		 &\mathbb{P}\big( \mathsf{C}^{\mathrm{in}}[\psi, \cref{const_point_to_boundary2}]  \big) =  \mathbb{E}\Big[ \mathbbm{1}_{B_{{v}_1}(\cref{const_point_to_boundary2} n)\xleftrightarrow{} w_1 }   \mathbb{P}\big( B_{{v}_2}(\cref{const_point_to_boundary2} n)\xleftrightarrow{(\mathcal{C}_{w_1})} w_2  \mid \mathcal{F}_{\mathcal{C}_{w_1}} \big) \Big] \\ 
 		 \overset{(\text{Lemma}\ \ref{lemma_onecluster_box_point})}{\lesssim } & n^{-(\frac{d}{2}\boxdot 3)} \sum_{x \in \partial \mathcal{B}_{{v}_2}(\cref{const_point_to_boundary2} d n)} \mathbb{E} \Big[ \mathbbm{1}_{B_{{v}_1}(\cref{const_point_to_boundary2} n)\xleftrightarrow{} w_1 }   \mathbb{P}\big( x \xleftrightarrow{(\mathcal{C}_{w_1})} w_2  \mid \mathcal{F}_{\mathcal{C}_{w_1}} \big)  \Big]\\  
 		 =& n^{-(\frac{d}{2}\boxdot 3)} \sum_{x \in \partial \mathcal{B}_{{v}_2}(\cref{const_point_to_boundary2} d n)}   \mathbb{P}\big(B_{{v}_1}(\cref{const_point_to_boundary2} n)\xleftrightarrow{} w_1 \Vert x \xleftrightarrow{} w_2 \big). 
 		  	\end{split}
 \end{equation} 
Similarly, for any $x\in \partial \mathcal{B}_{{v}_2}(\cref{const_point_to_boundary2} d n)$, we have 
\begin{equation}\label{ineq364_newadd}
\begin{split}
	 & \mathbb{P}\big(B_{{v}_1}(\cref{const_point_to_boundary2} n)\xleftrightarrow{} w_1 \Vert x \xleftrightarrow{} w_2 \big) \\
	 = & \mathbb{E}\Big[ \mathbbm{1}_{x \xleftrightarrow{} w_2}   \mathbb{P}\big( B_{{v}_1}(\cref{const_point_to_boundary2} n)\xleftrightarrow{(\mathcal{C}_{w_2})} w_1 \big) \Big] \\
	 	 \overset{(\text{Lemma}\ \ref{lemma_onecluster_box_point})}{\lesssim } & n^{-(\frac{d}{2}\boxdot 3)}  \sum\nolimits_{y\in \partial \mathcal{B}_{{v}_1}(\cref{const_point_to_boundary2} d n) } \mathbb{P}\big(x \xleftrightarrow{} w_2\Vert y\xleftrightarrow{} w_1 \big)\\
	 	 	  \overset{(\text{Lemma}\ \ref{lemma_roots}),(\ref{order_highd_four_points})}{\asymp }& n^{ (\frac{d}{2}-1)\boxdot (d-4)} \cdot  \mathbb{P}\big( \mathsf{C}[\psi]  \big). 
\end{split}
\end{equation}
 Plugging (\ref{ineq364_newadd}) into (\ref{ineq364}), we obtain 
 \begin{equation}
 	\begin{split}
 		\mathbb{P}\big( \mathsf{C}^{\mathrm{in}}[\psi, \cref{const_point_to_boundary2}]  \big) \lesssim n^{(d-2)\boxdot (2d-8)}  \cdot \mathbb{P}\big( \mathsf{C}[\psi]  \big). 
 	\end{split}
 \end{equation} 
Thus, since $\widehat{\mathsf{C}}^{\mathrm{in}}$ is a sub-event of $\mathsf{C}^{\mathrm{in}}$, it remains to show that the following hold:
 	  	   	 \begin{align}
 	&\text{when}\ 3\le d\le 5,\       \mathbb{P}\big( \widehat{\mathsf{C}}^{\mathrm{in}}[\psi, \cref{const_point_to_boundary2},\cref{const_point_to_boundary3},\cref{const_point_to_boundary4}]  \big)\gtrsim  n^{d-2}  \cdot 	\mathbb{P}\big(  \mathsf{C}[\psi]  \big) ;  \label{good363}  \\ 
 	&\text{when}\ d\ge 7,  \ \ \ \ \ \  \  \mathbb{P}\big( \widehat{\mathsf{C}}^{\mathrm{in}}[\psi, \cref{const_point_to_boundary2},\cref{const_point_to_boundary3} \big)\gtrsim  n^{2d-8}  \cdot 	\mathbb{P}\big(  \mathsf{C}[\psi]  \big).  \label{good364}
 \end{align} 	
We achieve these estimates through the second moment method, as shown below.


 \textbf{When $3\le d\le 5$.} In previous works, the volume and capacity of the sign cluster can be estimated via the second moment method \cite{cai2024quasi} and the exploration martingale \cite{ding2020percolation} respectively. However, due to the correlation between clusters, these methods are no longer valid for the analysis of the involved clusters in two-arm events. To resolve this, we introduce an independent Brownian motion and calculate its hitting probability and expected occupation time in the target cluster, which in turn yields the required lower bounds on the volume and capacity. To be precise, we arbitrarily take a point $z\in \partial B_{{v}_1}(4\cref{const_point_to_boundary2} n)$ and let $\widetilde{S}^z_{\cdot}\sim \widetilde{\mathbb{P}}_z$ be a Brownian motion starting from $z$, independent of the loop soup. Recall the event $\widehat{\mathsf{C}}[\cdot]$ below (\ref{def353}). Consider the quantity
\begin{equation}\label{newfixadd4.63}
	\mathbf{X}:= \sum\nolimits_{x\in B_{{v}_1}(\cref{const_point_to_boundary2} n)}  \mathbbm{1}_{\{\widetilde{S}^z_{\cdot}\ \text{hits}\ x\}\cap \widehat{\mathsf{C}}[(x,v_2,w_1,w_2), \cref{const_lemma_separation3}]  }
\end{equation}
Using Lemmas \ref{lemma_roots} and \ref{lemma_separation}, one has 
\begin{equation}\label{367}
	\begin{split}
		\mathbb{E}\big[ \mathbf{X} \big] =   \sum_{x\in B_{{v}_1}(\cref{const_point_to_boundary2} n)}  \widetilde{\mathbb{P}}_z\big(\tau_{x} <\infty \big)\cdot \mathbb{P}\big( \widehat{\mathsf{C}}[(x,v_2,w_1,w_2), \cref{const_lemma_separation3}]  \big) 
		\asymp   n^2 \cdot  \mathbb{P}\big( \mathsf{C}[\psi]  \big) .
			\end{split}
\end{equation}
For the second moment of $\mathbf{X}$, we have 
\begin{equation}\label{good368}
	\begin{split}
		\mathbb{E}\big[\mathbf{X}^2 \big] \le & \sum_{x_1,x_2\in B_{{v}_1}(\cref{const_point_to_boundary2} n)} \widetilde{\mathbb{P}}_z\big(\tau_{x_1} <\infty,\tau_{x_2} <\infty \big) \\
		 & \ \ \ \ \ \ \ \ \ \     \cdot  \mathbb{P}\big( x_1\xleftrightarrow{} w_1, x_2\xleftrightarrow{} w_1 \Vert  v_2\xleftrightarrow{} w_2, \{ v_2\xleftrightarrow{} B_{v_1}(\cref{const_lemma_separation3}n ) \}^c \big). 
	\end{split}
\end{equation}
For the first probability on the right-hand side, by the strong Markov probability of $\widetilde{S}_{\cdot}^z$ and (\ref{310}), one has
 \begin{equation} \label{app1}
  	\begin{split}
 		\widetilde{\mathbb{P}}_z(\tau_{x_1}<\infty,\tau_{x_2}<\infty )\le &\sum\nolimits_{i\in \{1,2\}} \widetilde{\mathbb{P}}_z(\tau_{x_i}\le \tau_{x_{3-i}}<\infty)\\
 		\lesssim &\sum\nolimits_{i\in \{1,2\}} (|z-x_i|+1)^{2-d}  (|x_1-x_{2}|+1)^{2-d}\\
 		\lesssim & n^{2-d}  (|x_1-x_{2}|+1)^{2-d}. 
 	\end{split}
 \end{equation}
 For the second probability, using the restriction property, we have 
 \begin{equation}\label{add369}
 	\begin{split}
 		& \mathbb{P}\big( x_1\xleftrightarrow{} w_1, x_2\xleftrightarrow{} w_1 \Vert  v_2\xleftrightarrow{} w_2, \{ v_2\xleftrightarrow{} B_{v_1}(\cref{const_lemma_separation3}n ) \}^c \big) \\
 	=& \mathbb{E}\big[\mathbbm{1}_{v_2\xleftrightarrow{} w_2, \{ v_2\xleftrightarrow{} B_{v_1}(\cref{const_lemma_separation3}n ) \}^c} \cdot \mathbb{P}\big(x_1\xleftrightarrow{(\mathcal{C}_{v_2})} w_1, x_2\xleftrightarrow{(\mathcal{C}_{v_2})} w_1  \big)  \big]\\
 	\overset{(\text{Lemma}\ \ref{lemma_cite_three_point})}{\lesssim } & (|x_1-x_{2}|+1)^{-\frac{d}{2}+1} \cdot \mathbb{P}(\mathsf{C}[(x_1,v_2,w_1,w_2)])\\ 
 	 	\overset{(\text{Lemma}\ \ref{lemma_roots})}{\lesssim } & (|x_1-x_{2}|+1)^{-\frac{d}{2}+1} \cdot \mathbb{P}(\mathsf{C}[\psi]). 
 	\end{split}
 \end{equation} 
Plugging (\ref{app1}) and (\ref{add369}) into (\ref{good368}), we have  
\begin{equation}\label{371}
\begin{split}
		\mathbb{E}\big[\mathbf{X}^2 \big] \lesssim  & n^{2-d}\cdot \mathbb{P}\big(\mathsf{C}[\psi]  \big) \sum\nolimits_{x_1,x_2\in B_{{v}_1}(\cref{const_point_to_boundary2} n)} (|x_1-x_{2}|+1)^{-\frac{3d}{2}+3} \\
	 \overset{}{ \lesssim}  & n^{5-\frac{d}{2}}\cdot \mathbb{P}\big(\mathsf{C}[\psi]  \big). 
\end{split}
\end{equation}
 where the last inequality follows from the following basic fact (see e.g. \cite[(4.5)]{cai2024high}): 
  \begin{equation}\label{computation_d-a}
  	\max\nolimits_{z\in \mathbb{Z}^d} \sum\nolimits_{z'\in B(M)} (|z-z'|+1)^{-a}\lesssim M^{(d-a)\vee 0}, \ \ \forall a\neq d.
  	  \end{equation}
By applying the Paley-Zygmund inequality and using (\ref{367}) and (\ref{371}), we get 
 \begin{equation}\label{372}
  	\mathbb{P}\big(\mathbf{X}>0 \big)\ge \frac{\big( 	\mathbb{E}\big[ \mathbf{X} \big] \big)^2  }{	\mathbb{E}\big[\mathbf{X}^2 \big] } \gtrsim n^{\frac{d}{2}-1}\cdot  \mathbb{P}\big( \mathsf{C}[\psi]  \big). 
  \end{equation}
   Meanwhile, since $\{\mathbf{X}>0\}\subset  \mathring{\mathsf{C}}[\psi] := \big\{ B_{{v}_1}(\cref{const_point_to_boundary2} n) \xleftrightarrow{(\mathcal{C}_{w_2})} w_1 \Vert v_2  \xleftrightarrow{} w_2\big\}\cap \big(  \mathsf{S}^{\mathrm{in}}[\psi,\cref{const_lemma_separation3}] \big)^c$, it follows from (\ref{ineq364_newadd}) that  
   \begin{equation}\label{new372}
  \mathbb{P}\big(\mathbf{X}>0 \big)\le	\mathbb{P}\big( \mathring{\mathsf{C}}[\psi] \big)  \overset{(\ref{ineq364_newadd})}{\lesssim} n^{\frac{d}{2}-1} \cdot  \mathbb{P}\big( \mathsf{C}[\psi]  \big),    
  \end{equation}
  Combining (\ref{372}) and (\ref{new372}), we obtain that 
    \begin{equation}\label{newnew374}
  	\mathbb{P}\big( \mathring{\mathsf{C}}[\psi] \big)   \asymp n^{\frac{d}{2}-1} \cdot  \mathbb{P}\big( \mathsf{C}[\psi]  \big), 
  \end{equation}
  which in turn, together with (\ref{367}) and (\ref{371}), yields that 
  \begin{equation}\label{add373}
  	\mathbb{E}\big[ \mathbf{X} \mid \mathring{\mathsf{C}}[\psi]  \big] \asymp   n^{3-\frac{d}{2} } \ \  \text{and}  \ \ 	\mathbb{E}\big[ \mathbf{X}^2 \mid \mathring{\mathsf{C}}[\psi]  \big]\lesssim n^{6-d}. 
  \end{equation} 
  By the Paley-Zygmund inequality and (\ref{add373}), we derive that for some $c_\dagger,c_\spadesuit >0$, 
 \begin{equation}\label{ineq376}
 	\mathbb{P}\big(\mathbf{X}\ge c_\dagger n^{3-\frac{d}{2}} \mid \mathring{\mathsf{C}}[\psi] \big) \ge c_\spadesuit. 
 \end{equation} 
 Note that given $\mathcal{C}_{\{ w_1,w_2\} }$, the quantity $\mathbf{X}$ depends only on $\widetilde{S}^z_{\cdot}$. Hence, the event $\mathsf{G}:= \big\{ \mathbb{P} \big(\mathbf{X}\ge c_\dagger n^{3-\frac{d}{2}} \mid  \mathcal{F}_{ \mathcal{C}_{\{ w_1,w_2\} } } \big)\ge \frac{1}{2}c_\spadesuit  \big\} $ is measurable with respect to the cluster $\mathcal{C}_{\{ w_1,w_2\} }$. Therefore, it follows from (\ref{ineq376}) that  
  \begin{equation}
 \begin{split}
 	 	c_\spadesuit\le & \mathbb{P}\big(\mathsf{G}  \mid \mathring{\mathsf{C}}[\psi] \big)+  \mathbb{P}\big(\mathbf{X}\ge c_\dagger n^{3-\frac{d}{2}}, \mathsf{G}^c \mid \mathring{\mathsf{C}}[\psi] \big) \\
 	 	\le &\mathbb{P}\big(\mathsf{G}  \mid \mathring{\mathsf{C}}[\psi] \big) + \frac{\mathbb{E}\big[\mathbbm{1}_{ \mathsf{G}^c,\mathring{\mathsf{C}}[\psi] } \cdot \mathbb{P} \big(\mathbf{X}\ge c_\dagger n^{3-\frac{d}{2}} \mid  \mathcal{F}_{ \mathcal{C}_{\{ w_1,w_2\} } } \big)  \big] }{\mathbb{P}\big(\mathring{\mathsf{C}}[\psi] \big) }\\
 	 	\le & \mathbb{P}\big(\mathsf{G}  \mid \mathring{\mathsf{C}}[\psi] \big) + \tfrac{1}{2}c_\spadesuit,
 	 	 \end{split}
 \end{equation}
which further implies  
 \begin{equation}\label{revise_new_4.112}
 	\mathbb{P}\big(\mathsf{G}  \mid \mathring{\mathsf{C}}[\psi] \big) \ge \tfrac{1}{2}c_\spadesuit. 
 \end{equation}

 We claim that for some $c_{ \clubsuit}>0$, 
 \begin{equation}\label{newnew378}
 	\mathsf{G} \subset \overline{\mathsf{G}} := \big\{\mathcal{V}_{w_1,v_1}(\cref{const_point_to_boundary2} n)\ge c_{ \clubsuit} n^{\frac{d}{2}+1}  \big\}\cap \big\{\mathcal{T}_{w_1,v_1}(\cref{const_point_to_boundary2} n)\ge c_{ \clubsuit}  n^{d-2} \big\}. 
 \end{equation}
  To see this, since $\mathring{\mathsf{C}}[\psi]$ is independent from $\widetilde{S}^{z}_{\cdot}$, one has 
  \begin{equation}\label{379}
  	\begin{split}
  		 \mathbb{E} \big(\mathbf{X}  \mid  \mathcal{F}_{ \mathcal{C}_{\{ w_1,w_2\} } } \big)
  		=  \sum\nolimits_{x\in  \mathcal{C}_{w_1}\cap B_{{v}_1}(\cref{const_point_to_boundary2} n)} \widetilde{\mathbb{P}}_{z}\big(\tau_x<\infty  \big)  
  		 \asymp   n^{2-d}\cdot \mathcal{V}_{w_1,v_1}(\cref{const_point_to_boundary2} n),
  	\end{split}
  \end{equation}
  \begin{equation}\label{379.5}
  	\mathbb{P} \big(\mathbf{X}>0 \mid  \mathcal{F}_{ \mathcal{C}_{\{ w_1,w_2\} } } \big) \asymp n^{2-d}\cdot \mathcal{T}_{w_1,v_1}(\cref{const_point_to_boundary2} n). 
  \end{equation}
  Meanwhile, on the event $\mathsf{G}$, we have 
 \begin{equation}\label{380}
 \begin{split}
 	 \mathbb{E} \big(\mathbf{X}  \mid  \mathcal{F}_{ \mathcal{C}_{\{ w_1,w_2\} } } \big)  
 		 \ge   c_\dagger n^{3-\frac{d}{2}}\cdot \mathbb{P} \big(\mathbf{X}\ge c_\dagger n^{3-\frac{d}{2}} \mid  \mathcal{F}_{ \mathcal{C}_{\{ w_1,w_2\} } } \big)\ge \tfrac{1}{2} c_\spadesuit c_\dagger n^{3-\frac{d}{2}}, 
 		  \end{split}
 \end{equation}
 \begin{equation}\label{380.5}
 	\begin{split}
 		\mathbb{P} \big(\mathbf{X}>0 \mid  \mathcal{F}_{ \mathcal{C}_{\{ w_1,w_2\} } } \big)  \ge \mathbb{P} \big(\mathbf{X}\ge c_\dagger n^{3-\frac{d}{2}} \mid  \mathcal{F}_{ \mathcal{C}_{\{ w_1,w_2\} } } \big)\ge \tfrac{1}{2} c_\spadesuit. 
 	\end{split}
 \end{equation}
 Combining (\ref{379}) with (\ref{380}), and (\ref{379.5}) with (\ref{380.5}), we obtain (\ref{newnew378}).

  By (\ref{newnew374}), (\ref{revise_new_4.112}) and (\ref{newnew378}), we have 
  \begin{equation}\label{ineq383}
  	\begin{split}
  		\mathbb{P}\big( \big\{ \overline{\mathsf{G}}  \Vert v_2 \xleftrightarrow{} w_2  \big\}\cap  \big(  \mathsf{S}^{\mathrm{in}}[\psi,\cref{const_lemma_separation3}] \big)^c  \big) \gtrsim n^{\frac{d}{2}-1}\cdot  \mathbb{P}\big( \mathsf{C}[\psi]  \big). 
  	\end{split}
  \end{equation}
  We repeat the second moment method used above for the quantity 
  \begin{equation}
  	\begin{split}
  		\mathbf{X}':= \sum\nolimits_{x\in B_{{v}_2}(\cref{const_point_to_boundary2} n)}  \mathbbm{1}_{\{\widetilde{S}^{z'}_{\cdot}\ \text{hits}\ x\}\cap  \{ \overline{\mathsf{G}}  \Vert x \xleftrightarrow{} w_2   \}\cap   (  \mathsf{S}^{\mathrm{in}}[\psi,\cref{const_lemma_separation3}]  )^c}, 
  	\end{split}
  \end{equation}
  where $z'\in B_{{v}_2}(4\cref{const_point_to_boundary2} n)$, and $\widetilde{S}^{z'}_{\cdot}\sim \widetilde{\mathbb{P}}_{z'}$ is an independent Brownian motion starting from $z'$. This yields that for some constants $\cref{const_point_to_boundary2},\cref{const_point_to_boundary3},\cref{const_point_to_boundary4}>0$, 
  \begin{equation}\label{newnew385}
  		\mathbb{P}\big( \widehat{\mathsf{C}}^{\mathrm{in}}[\psi, \cref{const_point_to_boundary2},\cref{const_point_to_boundary3},\cref{const_point_to_boundary4}]  \big) \gtrsim  n^{\frac{d}{2}-1}\cdot \mathbb{P}\big( \big\{ \overline{\mathsf{G}}  \Vert v_2 \xleftrightarrow{} w_2  \big\}\cap  \big(  \mathsf{S}^{\mathrm{in}}[\psi,\cref{const_lemma_separation3}] \big)^c  \big). 
  \end{equation}
The only difference from the previous case is the need for the following inequality, which plays the role of Lemma \ref{lemma_roots} during the proof: 
\begin{equation}
\mathbb{P}\big( \big\{ \overline{\mathsf{G}}  \Vert v \xleftrightarrow{} w_2  \big\}\cap  \big(  \mathsf{S}^{\mathrm{in}}[\psi,\cref{const_lemma_separation3}] \big)^c  \big) \asymp 	\mathbb{P}\big( \big\{ \overline{\mathsf{G}}  \Vert v' \xleftrightarrow{} w_2  \big\}\cap  \big(  \mathsf{S}^{\mathrm{in}}[\psi,\cref{const_lemma_separation3}] \big)^c  \big)  
\end{equation}
for all $v,v'\in \widetilde{B}_{{v}_2}(\cref{const_point_to_boundary2} n)$. It can be derived as follows: 
\begin{equation}\label{3.18_new_6.37}
	\begin{split}
		& \mathbb{P}\big( \big\{ \overline{\mathsf{G}}  \Vert v\xleftrightarrow{} w_2  \big\}\cap  \big(  \mathsf{S}^{\mathrm{in}}[\psi,\cref{const_lemma_separation3}] \big)^c  \big) \\
	\overset{(\ref{ineq383})}{	\gtrsim}  & n^{\frac{d}{2}-1}\cdot  \mathbb{P}\big( \mathsf{C}[(v,v_2,w_1,w_2)]  \big) \\
	 \overset{(\ref{3.18_ineq_lemma4.1})}{\asymp } &  n^{\frac{d}{2}-1}\cdot  \mathbb{P}\big( \mathsf{C}[(v',v_2,w_1,w_2)]  \big) 
	\overset{(\ref{new372})}{ \gtrsim }  \mathbb{P}\big( \big\{ \overline{\mathsf{G}}   \Vert v' \xleftrightarrow{} w_2  \big\}\cap  \big(  \mathsf{S}^{\mathrm{in}}[\psi,\cref{const_lemma_separation3}] \big)^c  \big). 
	\end{split}
\end{equation}
By plugging (\ref{ineq383}) into (\ref{newnew385}), we obtain (\ref{good363}).

     \textbf{When $d\ge 7$.} Compared to the proof in low dimensions, this one follows the same spirit and is more straightforward. We consider the quantity 
    \begin{equation}
     \mathbf{Y}:=	\sum\nolimits_{x\in B_{{v}_1}(\cref{const_point_to_boundary2} n)}  \mathbbm{1}_{ \widehat{\mathsf{C}}[(x,v_2,w_1,w_2)}. 
    \end{equation}
    For the first moment of $\mathbf{Y}$, as in (\ref{367}), one has 
    \begin{equation}\label{good386}
    	\mathbb{E}\big[\mathbf{Y} \big] =   \sum\nolimits_{x\in B_{{v}_1}(\cref{const_point_to_boundary2} n)}   \mathbb{P}\big( \widehat{\mathsf{C}}[(x,v_2,w_1,w_2), \cref{const_lemma_separation3}]  \big)  \asymp   n^d \cdot  \mathbb{P}\big( \mathsf{C}[\psi]  \big).
    \end{equation}
    For the second moment, by the BKR inequality, we have 
    \begin{equation}\label{good387}
    	\begin{split}
    			\mathbb{E}\big[\mathbf{Y}^2 \big]\le & \sum\nolimits_{x_1,x_2\in B_{{v}_1}(\cref{const_point_to_boundary2} n)} \mathbb{P}\big(  x_1\xleftrightarrow{} w_1,x_2\xleftrightarrow{} w_1   \big) \cdot  \mathbb{P}\big( v_2\xleftrightarrow{} w_2 \big)  \\
    			\overset{(\ref{two-point1}),(\text{Lemma}\ \ref{lemme_technical_high_dimension})}{\lesssim } &   |B(\cref{const_point_to_boundary2} n)|\cdot n^4N^{2-d}\cdot N^{2-d} \overset{(\ref{order_highd_four_points})}{\asymp}  n^{d+4}\cdot \mathbb{P}\big(\mathsf{C}[\psi]  \big). 
    	\end{split}
    \end{equation}
 Recall the event $\mathring{\mathsf{C}}[\psi]$ below (\ref{372}). By the Paley-Zygmund inequality, (\ref{good386}) and (\ref{good387}), we obtain that 
    \begin{equation}\label{good388}
    	\mathbb{P}\big( \mathring{\mathsf{C}}[\psi] \big)= \mathbb{P}\big(\mathbf{Y}>0 \big)\gtrsim n^{d-4}\cdot \mathbb{P}\big(\mathsf{C}[\psi]  \big). 
    \end{equation}
   As in (\ref{newnew385}), repeating this second moment method yields that 
   \begin{equation}\label{good389}
   	\mathbb{P}\big( \widehat{\mathsf{C}}^{\mathrm{in}}[\psi, \cref{const_point_to_boundary2},\cref{const_point_to_boundary3}]  \big) \gtrsim  n^{d-4}\cdot  \mathbb{P}\big( \mathring{\mathsf{C}}[\psi] \big). 
   \end{equation}
    Combining (\ref{good388}) and (\ref{good389}), we get (\ref{good364}), thereby completing the proof.  
 \end{proof}

 \subsection{Proof of (\ref{thm1_small_n_bar})}\label{newsubsection4.1}
 Without loss of generality, we assume that $v\neq v'\in \widetilde{B}(\Cref{const_prob_Cpsi_small_n1})$ (recall the constant $\Cref{const_prob_Cpsi_small_n1}$ from Lemma \ref{lemma_prob_Cpsi_small_n}).

 For the upper bound in (\ref{thm1_small_n_bar}), let $\Omega_1$ (resp. $\Omega_2$) denote a collection of $O(1)$-many boxes of the form $B_x(\cref{const_point_to_boundary2}N)$ whose union contains $\partial B(\frac{N}{4})$ (resp. $\partial B(\frac{N}{2})$). By the union bound, one has (recalling $\mathsf{C}^{\mathrm{out}}[\cdot]$ from (\ref{newfix489})), 
  \begin{equation}\label{new.462}
 	\begin{split}
 		\mathbb{P}\big( \overline{\mathsf{H}}_{v'}^{v}(N) \big) \le  \sum_{\psi=(v,v',w_1,w_2):w_1\in \Omega_1,w_2\in \Omega_2 }\mathbb{P}\big( \mathsf{C}^{\mathrm{out}}[\psi, \cref{const_point_to_boundary2}]  \big). 
 	\end{split}
 \end{equation}
 Moreover, by the same argument as in the proof of (\ref{ineq364_newadd}), we have 
 \begin{equation}\label{new.463}
 	\begin{split}
 		\mathbb{P}\big( \mathsf{C}^{\mathrm{out}}[\psi, \cref{const_point_to_boundary2}]  \big) \lesssim N^{(d-2)\boxdot (2d-8)}   \cdot \mathbb{P}\big( \mathsf{C}[\psi]  \big)  
 	\end{split}
 \end{equation}
 for all $\psi=(v,v',w_1,w_2)$ with $w_1\in \Omega_1$ and $w_2\in \Omega_2$.  Plugging (\ref{new.463}) into (\ref{new.462}), and then using Lemma \ref{lemma_prob_Cpsi_small_n}, we obtain the upper bound in (\ref{thm1_small_n_bar}).

 For the lower bound in (\ref{thm1_small_n_bar}), we prove the following stronger result, which will be useful for the companion paper \cite{inpreparation_second_monent}.

 \begin{lemma}\label{additional_lemma_two_arm}
	For any $d\ge 3$ with $d\neq 6$, there exist $\Cl\label{const_additional_lemma_two_arm0},\Cl\label{const_additional_lemma_two_arm1},\Cl\label{const_additional_lemma_two_arm2},\cl\label{const_additional_lemma_two_arm3} >0$ such that for any $v\neq v'\in \widetilde{B}(\Cref{const_additional_lemma_two_arm0})$ and any sufficiently large $N\ge 1$,
	\begin{equation}\label{newfixto65}
	\mathbb{P}\big(	\widehat{\mathsf{H}}_{v'}^{v}(N) \big)	\gtrsim   |v-v'|^{\frac{3}{2}}N^{-[(\frac{d}{2}+1)\boxdot 4]},
	\end{equation}
	where $\widehat{\mathsf{H}}_{v'}^{v}(N)$ is defined as $\mathsf{H}_{v'}^{v}(N)\cap \{v\xleftrightarrow{\ge 0} \partial B(\Cref{const_additional_lemma_two_arm1}N)  \}^c \cap \{v' \xleftrightarrow{\le 0} \partial B(\Cref{const_additional_lemma_two_arm1}N) \}^c 
			 	  \cap \big\{ |\widetilde{\phi}_{v}|,|\widetilde{\phi}_{v'}| \in [\cref{const_additional_lemma_two_arm3}, \Cref{const_additional_lemma_two_arm2}] \big\} $.  
\end{lemma}
\begin{proof}
	 Recall $\mathsf{C}^{\diamond}[\cdot]$ in (\ref{newfix489}), and recall $\mathsf{S}^{\diamond}[\cdot]$ below (\ref{newfix490}). By Lemmas \ref{lemma_point_to_boundary}, \ref{lemma_prob_Cpsi} and \ref{lemma_highd_volume_two_arm}, there exist $C_\dagger, c_\star ,c_{ \triangle}  > 0$ such that for any $\psi\in \Psi(C_\dagger,N)$, 
  \begin{equation}\label{464}
  \begin{split}
   \mathbb{P}\big( \mathsf{C}^{\mathrm{out}}[\psi, c_\star ]  \big) \overset{ \text{Lemma}\ \ref{lemma_point_to_boundary} }{\asymp} &	\mathbb{P}\big( \mathsf{C}^{\mathrm{out}}[\psi, c_\star ]\cap ( \mathsf{S}^{\mathrm{out}}[\psi,c_{ \triangle}])^c  \big) \\
\overset{ \text{Lemma}\ \ref{lemma_point_to_boundary} }{\asymp}   &N^{(d-2)\boxdot (2d-8)} \mathbb{P}\big( \mathsf{C}[\psi ]  \big) \overset{ \text{Lemmas}\ \ref{lemma_prob_Cpsi}\ \text{and}\ \ref{lemma_highd_volume_two_arm} }{\gtrsim}  N^{-[(\frac{d}{2}+1)\boxdot 4]}.
  \end{split}
 \end{equation}
	 We arbitrarily take $v_1\in B(2C_\dagger)$, $v_2\in \partial B(4C_\dagger)$, $w_1\in \partial B(2N)$ and $w_2\in \partial B(4N)$. Note that (\ref{464}) holds for $\psi_{\spadesuit}=(v_1,v_2,w_1,w_2)$. On $ \mathsf{C}^{\mathrm{out}}[\psi_{\spadesuit}, c_\star ]\cap ( \mathsf{S}^{\mathrm{out}}[\psi_{\spadesuit},c_{ \triangle}])^c $, there exist two disjoint paths $\widetilde{\eta}_1$ and $\widetilde{\eta}_2$ such that for $j\in \{1,2\}$, $\widetilde{\eta}_j$ is contained in $\mathcal{C}_{v_j}$, starts from $w_j$ and ends at $v_j$. Let $x_j$ be the point where $\widetilde{\eta}_j$ first hits $\widetilde{B}(8C_\dagger)$. It follows from the construction that the events 
   \begin{equation}
  	\mathsf{A}_{x_1,x_2}:= \big\{x_1\xleftrightarrow{ \mathcal{C}_{x_1}\cap  [ \widetilde{B}(8C_\dagger)]^c}  B_{{w}_1}(c_\star N) \Vert x_2\xleftrightarrow{\mathcal{C}_{x_2}\cap  [ \widetilde{B}(8C_\dagger)]^c} B_{{w}_2}(c_\star N)  \big\} 
  \end{equation} 
 and $( \mathsf{S}^{\mathrm{out}}[(x_1,x_2,w_1,w_2),c_{ \triangle}])^c$ both occur. Combined with (\ref{464}), it yields that for some $x_1^*\neq x_2^*\in \widetilde{\partial} \widetilde{B}(8C_\dagger)$ (letting $\psi_*:= (x_1^*,x_2^*,w_1,w_2)$),
 \begin{equation}\label{newadd.466}
\mathbb{P}\big( \widehat{\mathsf{A}}_{x_1^*,x_2^*} \big):= 	\mathbb{P}\big( \mathsf{A}_{x_1^*,x_2^*} \cap ( \mathsf{S}^{\mathrm{out}}[\psi_*,c_{ \triangle}])^c \big)\gtrsim N^{-[(\frac{d}{2}+1)\boxdot 4]}.  
 \end{equation}

	 For any measurable set $F\subset   (0,\infty)$, any $x_1\neq x_2\in \widetilde{\partial} \widetilde{B}(8C_\dagger)$ and $j\in \{1,2\}$, by the restriction property, we have (letting $\psi:=(x_1,x_2,w_1,w_2)$)
   \begin{equation}\label{465}
 	\begin{split}
 		&  \mathbb{P}\big( \mathsf{A}_{x_1,x_2} \cap ( \mathsf{S}^{\mathrm{out}}[\psi,c_{ \triangle}])^c , \widehat{\mathcal{L}}_{1/2}^{x_j}\in F \big) \\
 		  \le  & \mathbb{E}\Big[\mathbbm{1}_{x_{3-j} \xleftrightarrow{} B_{w_{3-j}}( c_\star N), \{x_{3-j} \xleftrightarrow{} B_{x_j }(c_{ \triangle}C_\dagger ) \}^c} \\
 		  &\ \ \ \  \cdot \mathbb{P}\big( x_j   \xleftrightarrow{(\mathcal{C}_{x_{3-j} })} B_{w_j}( c_\star N), \widehat{\mathcal{L}}_{1/2}^{\mathcal{C}_{x_{3-j} }, x_j}\in F \mid  \mathcal{F}_{\mathcal{C}_{x_{3-j}}} \big) \Big]. 
 	\end{split}
 \end{equation}
 On the event $\{x_{3-j} \xleftrightarrow{} B_{x_{j} }(c_{ \triangle}C_\dagger ) \}^c$, one has $\widetilde{G}_{\mathcal{C}_{x_{3-j}}}(x_{j} ,x_{j} )\asymp 1$ and hence, 
  \begin{equation}\label{466}
 	\mathbb{P}\big( \widehat{\mathcal{L}}_{1/2}^{\mathcal{C}_{x_{3-j} }, x_j}\le t  \mid  \mathcal{F}_{\mathcal{C}_{x_{3-j}}} \big)  =  \mathbb{P}^{\mathcal{C}_{x_{3-j} }}\big( |  \widetilde{\phi}_{x_j} | \le \sqrt{2t} \big)\lesssim  t^{\frac{1}{2}}, \ \ \forall t>0. 
 	\end{equation}
  Therefore, applying the FKG inequality, we get 
 \begin{equation}\label{newaddto470}
 \begin{split}
 &	\mathbb{P}\big( x_j \xleftrightarrow{(\mathcal{C}_{x_{3-j} })} B_{w_j}( c_\star N), \widehat{\mathcal{L}}_{1/2}^{\mathcal{C}_{x_{3-j} }, x_j}\le t \mid  \mathcal{F}_{\mathcal{C}_{x_{3-j}}}  \big) \\
 	 \lesssim  & t^{\frac{1}{2}} \cdot \mathbb{P}\big( x_j  \xleftrightarrow{(\mathcal{C}_{x_{3-j} })} B_{w_j}( c_\star N) \mid  \mathcal{F}_{\mathcal{C}_{x_{3-j}}}   \big).
 \end{split}
 \end{equation}
 Combined with (\ref{465}), it yields that 
 \begin{equation}\label{newadd.470}
 \begin{split}
 		&\mathbb{P}\big( \mathsf{A}_{x_1 ,x_2 } \cap ( \mathsf{S}^{\mathrm{out}}[\psi ,c_{ \triangle}])^c , \widehat{\mathcal{L}}_{1/2}^{x_j }\le t \big) \\\lesssim &t^{\frac{1}{2}} \cdot   \mathbb{P}\big( \mathsf{C}^{\mathrm{out}}[\psi , c_\star ]  \big)  \le t^{\frac{1}{2}} \cdot  \mathbb{P}\big( \overline{\mathsf{H}}_{x_2}^{x_1}(N)  \big)\lesssim t^{\frac{1}{2}}N^{-[(\frac{d}{2}+1)\boxdot 4]}.
 \end{split}
  \end{equation}
  Here in the last inequality we used the upper bound in (\ref{thm1_small_n_bar}), which has been confirmed below (\ref{new.463}). Moreover, for any $T>0$, by Markov's inequality, one has  
  \begin{equation}
	\begin{split}
		&\mathbb{P}\big( x_j   \xleftrightarrow{(\mathcal{C}_{x_{3-j} })} B_{w_j}( c_\star N), \widehat{\mathcal{L}}_{1/2}^{\mathcal{C}_{x_{3-j}}, x_j}\ge T \mid  \mathcal{F}_{\mathcal{C}_{x_{3-j}}}   \big) \\
		\le & T^{-1} \cdot  \mathbb{E}\Big[ \mathbbm{1}_{x_j   \xleftrightarrow{(\mathcal{C}_{x_{3-j} })} B_{w_j}( c_\star N)}\cdot \widehat{\mathcal{L}}_{1/2}^{\mathcal{C}_{x_{3-j} }, x_j} \mid  \mathcal{F}_{\mathcal{C}_{x_{3-j}}}  \Big]\\
		 \lesssim & T^{-1}\cdot \mathbb{P}\big( x_j   \xleftrightarrow{(\mathcal{C}_{x_{3-j} })} B_{w_j}( c_\star N) \mid  \mathcal{F}_{\mathcal{C}_{x_{3-j}}}  \big), 
	\end{split}
\end{equation}
where the last inequality follows from \cite[(2.67)]{cai2024quasi}. Combined with (\ref{465}), it implies 
\begin{equation}\label{newadd.472}
\begin{split}
	& 	\mathbb{P}\big( \mathsf{A}_{x_1 ,x_2 } \cap ( \mathsf{S}^{\mathrm{out}}[\psi ,c_{ \triangle}])^c , \widehat{\mathcal{L}}_{1/2}^{x_j}\ge T \big)\\
		\lesssim & T^{-1}\cdot \mathbb{P}\big( \mathsf{C}^{\mathrm{out}}[\psi , c_\star ]  \big) \lesssim T^{-1} N^{-[(\frac{d}{2}+1)\boxdot 4]}. 
\end{split}
\end{equation}
Putting (\ref{newadd.466}), (\ref{newadd.470}) and (\ref{newadd.472}) together, we obtain that for some $C_{\clubsuit}>0$, 
\begin{equation}\label{new4.73}
	\begin{split}
		\mathbb{P}\Big(\widehat{\mathsf{A}}_{x_1^*,x_2^*}, \bigcap_{y\in (\mathcal{C}_{x_1^*}\cup \mathcal{C}_{x_2^*})\cap \widetilde{\partial} \widetilde{B}(8C_\dagger)}\big\{   \widehat{\mathcal{L}}_{1/2}^{y }\in [C_{\clubsuit}^{-1},C_{\clubsuit}]   \big\} \Big)\gtrsim N^{-[(\frac{d}{2}+1)\boxdot 4]}. 
	\end{split}
\end{equation} 
We define the clusters 
\begin{equation}
	\mathcal{C}_1^*:=  \big\{ v\in \widetilde{\mathbb{Z}}^d:  x_1^*\xleftrightarrow{ \widetilde{E}^{\ge 0}\cap  [ \widetilde{B}(8C_\dagger)]^c} v \big\}, \ \ \mathcal{C}_2^*:=  \big\{ v\in \widetilde{\mathbb{Z}}^d:  x_2^*\xleftrightarrow{ \widetilde{E}^{\le 0}\cap  [ \widetilde{B}(8C_\dagger)]^c} v \big\}. 
\end{equation}
The estimate (\ref{new4.73}) together with the isomorphism theorem implies that 
  \begin{equation}\label{newadd474}
  	\mathbb{P}\big(\mathsf{A}_{*}^{\pm} \big)\gtrsim N^{-[(\frac{d}{2}+1)\boxdot 4]}.
  \end{equation}
 Here the event $\mathsf{A}_{*}^{\pm}$ is defined by  
  \begin{equation*}
  \begin{split}
  	  	\mathsf{A}_{*}^{\pm}:=& \bigcap_{j\in \{1,2\}}\big\{   \mathcal{C}^*_j\cap B_{{w}_j}(c_\star N) \neq \emptyset \big\}\cap \bigcap_{y\in (\mathcal{C}_1^*\cup \mathcal{C}_2^*)\cap \widetilde{\partial} \widetilde{B}(8C_\dagger)} \big\{ \sqrt{2C_{\clubsuit}^{-1}} \le |\widetilde{\phi}_{y}| \le \sqrt{2C_{\clubsuit}}\big\},  \end{split}
  \end{equation*}
which is measurable with respect to $\mathcal{F}_{\mathcal{C}_1^*\cup \mathcal{C}_2^*}$.

We enumerate the points in $[(\mathcal{C}_1^*\cup \mathcal{C}_2^*)\cap \widetilde{\partial} \widetilde{B}(8C_\dagger)]\setminus \{x_1^*,x_2^*\}$ as $y_1,...,y_K$. For each $1\le j\le K$, we denote by $I^*_j$ the interval incident to $y_j$ such that $I_j^*\subset \mathcal{C}_1^*\cup \mathcal{C}_2^*$. When the distance between $y_j$ and the boundary with GFF values fixed is at least $d$ and the absolute values of all boundary conditions lie in $[\sqrt{2C_{\clubsuit}^{-1}}, \sqrt{2C_{\clubsuit}}]$, the conditional probability of $\mathsf{D}_j:=\big\{ y_j\xleftrightarrow{\widetilde{E}^{\ge 0}\setminus I_j^*} \widetilde{\partial}\mathbf{B}_{y_j}(1)  \big\}^c \cup \big\{ y_j\xleftrightarrow{\widetilde{E}^{\le 0}\setminus I_j^*} \widetilde{\partial}\mathbf{B}_{y_j}(1)  \big\}^c$ is of order $1$. As a result, on the event $\mathsf{A}_{*}^{\pm}$, 
\begin{equation}\label{newfix6.18}
	\begin{split}
		 \mathbb{P}\big(\cap_{1\le j\le K} \mathsf{D}_j \mid \mathcal{F}_{\mathcal{C}_1^*\cup \mathcal{C}_2^*} \big) 
		= \prod\nolimits_{1\le j\le K}  \mathbb{P}\big(  \mathsf{D}_j  \mid \mathcal{F}_{\mathcal{C}_1^*\cup \mathcal{C}_2^*}, \cap_{1\le l\le j-1}\mathsf{D}_{l}  \big) \overset{K=O(1)}{\gtrsim} 1,
	\end{split}
\end{equation}
 which together with (\ref{newadd474}) implies that 
\begin{equation}\label{newfix6.19}
	\mathbb{P}\big( \widehat{\mathsf{A}}_*^{\pm} \big):= \mathbb{P}\big(\cap_{1\le j\le K} \mathsf{D}_j \cap  \mathsf{A}_*^{\pm}  \big) \gtrsim N^{-[(\frac{d}{2}+1)\boxdot 4]}.
\end{equation}
 For each $i\in \{1,2\}$ and $y\in [\mathcal{C}_i^* \cap \widetilde{\partial} \widetilde{B}(8C_\dagger)]\setminus \{x_i^*\}$, we define the cluster 
\begin{equation}
	\mathcal{C}_{i,y}:=\big\{z \in \mathbf{B}_{y}(1) :y \xleftrightarrow{\widetilde{E}^{\ge 0}\setminus I_j^*} z \ \text{or}\  y_j\xleftrightarrow{\widetilde{E}^{\le 0}\setminus I_j^*} z  \big\}. 
\end{equation}
We then define $\widehat{\mathcal{C}}_i^*$ as the union of all these clusters and $\mathcal{C}_i^*$. Note that the event $\widehat{\mathsf{A}}_*^{\pm} $ is measurable with respect to $\mathcal{F}_{\widehat{\mathcal{C}}_1^*\cup \widehat{\mathcal{C}}_2^*}$.


We first present the remainder of the proof for the case $\chi := |v-v'| \le 1$. The case $\chi > 1$ follows from a similar argument, which is considerably simpler. Let $y_1^*$ and $y_2^*$ be two lattice points that are contained in the same interval as $v$ and $v'$ respectively, and satisfy $\mathrm{dist}\big(I_{[y_1^*, v] }, I_{[y_2^*, v'] } \big)=\chi$. Note that when $\chi \le 1$, at least one of the intervals $I_{[y_1^*, v] }$ and $ I_{[y_2^*, v'] }$ has length at least $1$ (recall that the length of a single interval is $d \ge 3$). Without loss of generality, we assume that $I_{[y_1^*, v] }$ does. We then take two disjoint paths $\eta_1^*$ and $\eta_2^*$ on $\mathbb{Z}^d$ (with lengths $l_1^*$ and $l_2^*$) such that $\eta_i^*$ starts from $x_i^*$ and ends at $y_i^*$. The $k$-th point of $\eta_i^*$ is denoted by $\eta_i^*(k)$. For each $0\le k\le l_1^*-1$, we define the event 
	 \begin{equation}\label{newfix6.20}
	 	\begin{split}
	 		\mathsf{F}_k^+:= & \big\{\sqrt{2C_{\clubsuit}^{-1}} \le \widetilde{\phi}_{z} \le \sqrt{2C_{\clubsuit}} ,\  \forall z\in I_{\{\eta_1^*(k),\eta_1^*(k+1)  \}}  \big\}\\&\cap \big\{\eta_1^*(k) \xleftrightarrow{\widetilde{E}^{\ge 0}\setminus I_{\{\eta_1^*(k),\eta_1^*(k+1)  \}}} \widetilde{\partial} \mathbf{B}_{\eta_1^*(k)}(1)  \big\}^c.
	 	\end{split}
	 \end{equation}
	  For the same reason as in (\ref{newfix6.18}), we have: on the event $ \widehat{\mathsf{A}}_*^{\pm} $, 
	 \begin{equation}\label{newfix6.21}
	 		 \mathbb{P}\big(\cap_{0\le k\le l_1^*-1} \mathsf{F}_k^+  \mid \mathcal{F}_{\widehat{\mathcal{C}}_1^*\cup \widehat{\mathcal{C}}_2^*}  \big)  = \prod\nolimits_{0\le k\le l_1^*-1}  \mathbb{P}\big( \mathsf{F}_k^+  \mid \mathcal{F}_{\mathcal{C}_1^*\cup \mathcal{C}_2^*}, \cap_{0\le l\le k-1} \mathsf{F}_k^+   \big){\gtrsim} 1. 
	 \end{equation}
	 Moreover, we consider the event 
	 \begin{equation}
	 	\mathsf{G}^{+}:=  \big\{  \widetilde{\phi}_{v}\in [\chi^{\frac{1}{2}}, 2\chi^{\frac{1}{2}}] \big\} \cap \big\{ y_1^*\xleftrightarrow{ \widetilde{E}^{\ge 0}\cap I_{[y_1^*,v]} } v \big\}\cap \big\{ v\xleftrightarrow{\widetilde{E}^{\ge 0}\setminus I_{[y_1^*,v]} } \widetilde{\partial}\mathbf{B}_{v}(\tfrac{\chi}{10})  \big\}^c.  
	 \end{equation}
	 Conditioned on $\cap_{0\le k\le l_1^*-1} \mathsf{F}_k^+\cap \widehat{\mathsf{A}}_*^{\pm}$, since the distance between $v$ and the boundary with GFF values fixed is at least $1$ (here we used the assumption that the length of $I_{[y_1^*, v] }$ is at least $1$) and the absolute values of all boundary conditions are contained in $[\sqrt{2C_{\clubsuit}^{-1}}, \sqrt{2C_{\clubsuit}}]$, the probability of $\widetilde{\phi}_{v}\in [\chi^{\frac{1}{2}}, 2\chi^{\frac{1}{2}}]$ is of order $\chi^{\frac{1}{2}}$. Moreover, conditioned on $\big\{\widetilde{\phi}_{v}\in [\chi^{\frac{1}{2}}, 2\chi^{\frac{1}{2}}]\big\}$, the formula (\ref{formula_two_point}) implies that the  probability of $\big\{ y_1^*\xleftrightarrow{ \widetilde{E}^{\ge 0}\cap I_{[y_1^*,v]} } v \big\}$ is at least of order $\chi^{\frac{1}{2}}$. Furthermore, after sampling $\widetilde{\phi}_\cdot$ on $I_{[y_1^*,v]}$ with $\widetilde{\phi}_{v}\in [\chi^{\frac{1}{2}}, 2\chi^{\frac{1}{2}}]$, the conditional probability of the event $\big\{ v\xleftrightarrow{\ge 0 } \widetilde{\partial}\mathbf{B}_{v}(\tfrac{\chi}{10}) \big\}^c$ is of order $1$. Combined with (\ref{newfix6.21}), it yields that on the event $ \widehat{\mathsf{A}}_*^{\pm} $, 
	 \begin{equation}\label{fixnewto6.23}
	 	 \mathbb{P}\big(\cap_{0\le k\le l_1^*-1} \mathsf{F}_k^+ \cap \mathsf{G}^{+}  \mid \mathcal{F}_{\widehat{\mathcal{C}}_1^*\cup \widehat{\mathcal{C}}_2^*}  \big)\gtrsim \chi. 
	 \end{equation}
	 	 Note that $\mathring{\mathsf{A}}:=\cap_{0\le k\le l_1^*-1} \mathsf{F}_k^+ \cap \mathsf{G}^{+}  \cap  \widehat{\mathsf{A}}_*^{\pm} $ is measurable with respect to $\mathcal{F}_{\mathcal{C}_{x_1^*}^+\cup \widehat{\mathcal{C}}_2^*}$.

	 As in (\ref{newfix6.20}), for each $0\le k\le l_2^*-1$, we define the event 
	 \begin{equation}
	 	\begin{split}
	 		\mathsf{F}_k^-:= & \big\{- \sqrt{2C_{\clubsuit}} \le \widetilde{\phi}_{v} \le  -\sqrt{2C_{\clubsuit}^{-1}},\  \forall v\in I_{\{\eta_2^*(k),\eta_2^*(k+1)  \}}  \big\}\\&\cap \big\{\eta_2^*(k) \xleftrightarrow{\widetilde{E}^{\le 0}\setminus I_{\{\eta_2^*(k),\eta_2^*(k+1)  \}}} \widetilde{\partial} \mathbf{B}_{\eta_2^*(k)}(1)  \big\}^c.
	 	\end{split}
	 \end{equation}
	 Similar to (\ref{newfix6.18}), we have: on the event $\mathring{\mathsf{A}}$,  
	 \begin{equation}\label{newfixto6.25}
	 	 \mathbb{P}\big(\cap_{0\le k\le l_2^*-1} \mathsf{F}_k^-  \mid \mathcal{F}_{\mathcal{C}_{x_1^*}^+\cup \widehat{\mathcal{C}}_2^*}  \big)  = \prod\nolimits_{0\le k\le l_2^*-1}  \mathbb{P}\big( \mathsf{F}_k^-  \mid \mathcal{F}_{\mathcal{C}_{x_1^*}^+\cup \widehat{\mathcal{C}}_2^*}, \cap_{0\le l\le k-1} \mathsf{F}_k^-   \big){\gtrsim} 1. 
	 \end{equation}
	 In fact, conditioned on $\cap_{0\le k\le l_2^*-1} \mathsf{F}_k^-\cap \mathring{\mathsf{A}}$, since the distance between $v'$ and the boundary with GFF values fixed is of order $\chi$, the mean and variance of $\widetilde{\phi}_{v'}$ are both $O(\chi)$. As a result, the event $\big\{\widetilde{\phi}_{v'}\in [\chi^{\frac{1}{2}}, 2\chi^{\frac{1}{2}}]\big\}$ occurs with uniform positive probability. Moreover, conditioned on this event, the probability of $\mathsf{G}^{-}:=\{y_2^*\xleftrightarrow{\le 0} v' \}\cap \{v'\xleftrightarrow{\widetilde{E}^{\le 0}\setminus I_{[y_2^*,v']}} \widetilde{\partial } \mathbf{B}_{v'}(1) \}^c$ is proportional to $\chi^{\frac{1}{2}}$ (using (\ref{formula_two_point})). Combined with (\ref{fixnewto6.23}) and (\ref{newfixto6.25}), it yields that on the event $\widehat{\mathsf{A}}_*^{\pm} $, 
\begin{equation}
	 \mathbb{P}\big(\cap_{0\le k\le l_2^*-1} \mathsf{F}_k^- \cap  \mathsf{G}^{-} \cap ( \cap_{0\le k\le l_1^*-1} \mathsf{F}_k^+ )\cap   \mathsf{G}^{+} \mid \mathcal{F}_{\widehat{\mathcal{C}}_1^* \cup \widehat{\mathcal{C}}_2^*}  \big)\gtrsim \chi^{\frac{3}{2}}. 
\end{equation}
	 This together with (\ref{newfix6.19}) implies that  
	 \begin{equation}
	 	\mathbb{P}\big(\cap_{0\le k\le l_2^*-1} \mathsf{F}_k^- \cap \mathsf{G}^{-}\cap \mathring{\mathsf{A}}  \big)  \gtrsim \chi^{\frac{3}{2}} \cdot  N^{-[(\frac{d}{2}+1)\boxdot 4]}. 
	 \end{equation}
	 Noting that the event on the left-hand side implies $\widehat{\mathsf{H}}_{v'}^{v}(N)$ with $\Cref{const_additional_lemma_two_arm1}=c_\star^{-1}$, $\cref{const_additional_lemma_two_arm3}=\sqrt{2C_{\clubsuit}^{-1}}$ and $\Cref{const_additional_lemma_two_arm2}=\sqrt{2C_{\clubsuit}}$, we obtain the desired bound (\ref{newfixto65}) for $\chi \le 1$.

	 For the case $\chi > 1$, we select two paths $\widetilde{\eta}_1$ and $\widetilde{\eta}_2$ on $\widetilde{\mathbb{Z}}^d$ such that $\widetilde{\eta}_1$ connects $x_{1}^*$ and $v$, $\widetilde{\eta}_2$ connects $x_{2}^*$ and $v'$, and $\mathrm{dist}(\mathrm{ran}(\widetilde{\eta}_1),\mathrm{ran}(\widetilde{\eta}_2))\ge 1$. We then define $\mathsf{H}_1$ (resp. $\mathsf{H}_2$) as the event that $x_{1}^*$ and $v$ (resp. $x_{2}^*$ and $v'$) are connected by a positive (resp. negative) cluster within $\cup_{z\in \mathrm{ran}(\widetilde{\eta}_1) }\mathbf{B}_z(\frac{1}{10})$ (resp. $\cup_{z\in \mathrm{ran}(\widetilde{\eta}_2)}\mathbf{B}_z(\frac{1}{10})$). Note that $\widehat{\mathsf{A}}_*^{\pm}\cap \mathsf{H}_1\cap \mathsf{H}_2\subset \widehat{\mathsf{H}}_{v'}^{v}(N)$. Moreover, following an argument similar to that for the case $\chi \le 1$, we can also derive that given $ \widehat{\mathsf{A}}_*^{\pm}$ occurs, the conditional probability of $\mathsf{H}_1\cap \mathsf{H}_2 $ is of order $1$. To sum up, we obtain (\ref{newfixto65}) for $\chi > 1$. 
	   \end{proof}

	 	 Since $\widehat{\mathsf{H}}_{v'}^{v}(N)\subset \mathsf{H}_{v'}^{v}(N)$, the proof of (\ref{thm1_small_n_bar}) is now complete.   \qed

	Lemmas \ref{lemma_prob_Cpsi} and \ref{lemma_highd_volume_two_arm} imply that for any sufficiently large $N\ge 1$ and any $\psi=(v,v',w,w')\in \Psi(C,N)$, 
	\begin{equation}\label{newfix_6.27}
		\mathbb{P}\big(  \mathsf{H}_{v',w'}^{v,w}  \big)\gtrsim  N^{-[(\frac{3d}{2}-1)\boxdot (2d-4)]}. 
	\end{equation}
	By employing the arguments in the proof of Lemma \ref{additional_lemma_two_arm} (where we only need to replace the bound (\ref{464}) with (\ref{newfix_6.27})), we derive the following strengthened version of the lower bound in (\ref{thm1_small_n_four_point}).

	 \begin{lemma}\label{additional_lemma_four_point}
	 	For any $d\ge 3$ and $d\neq 6$, there exist $\Cl\label{const_additional_lemma_four_point1},\Cl\label{const_additional_lemma_four_point2}, \Cl\label{const_additional_lemma_four_point3},\cl\label{const_additional_lemma_four_point4}>0$ such that for any $v\neq v'\in \widetilde{B}(\Cref{const_additional_lemma_four_point1})$, any sufficiently large $N\ge 1$, and $w,w'\in \widetilde{B}(10N)\setminus \widetilde{B}(N)$ with $|w-w'|\ge N$, 
	 	\begin{equation}
	 		\mathbb{P}\big( \widehat{\mathsf{H}}_{v',w'}^{v,w}   \big)\gtrsim |v-v'|^{\frac{3}{2}}N^{-[(\frac{3d}{2}-1)\boxdot (2d-4)]},
	 	\end{equation}
	 	where $\widehat{\mathsf{H}}_{v',w'}^{v,w}$ is defined as $\mathsf{H}_{v',w'}^{v,w} \cap \{v\xleftrightarrow{\ge 0} \partial B(\Cref{const_additional_lemma_four_point2}N)  \}^c \cap \{v' \xleftrightarrow{\le 0} \partial B(\Cref{const_additional_lemma_four_point2}N) \}^c 
			 	  \cap \big\{ |\widetilde{\phi}_{v}|,|\widetilde{\phi}_{v'}| \in [\cref{const_additional_lemma_four_point4}, \Cref{const_additional_lemma_four_point3}] \big\} $. 
	 \end{lemma}

 \subsection{Proof of (\ref{thm1_large_n_bar})}\label{newsubsection4.2}
We begin by verifying the lower bound in (\ref{thm1_large_n_bar}). By Lemmas \ref{lemma_point_to_boundary}, \ref{lemma_prob_Cpsi} and \ref{lemma_highd_volume_two_arm}, we have: for any $\chi \ge C$, $N\ge C'\chi $ and $\psi\in \Psi(\chi,2N)$, 
\begin{equation}\label{new4.65}
	\mathbb{P}\big( \mathsf{C}^{\mathrm{out}}[\psi,\cref{const_point_to_boundary2} ] \big) \gtrsim  \chi^{(3-\frac{d}{2})\boxdot 0}N^{-[(\frac{d}{2}+1)\boxdot 4]}. 
\end{equation}
 For any $\psi=(v_1,v_2,w_1,w_2)\in \Psi(\chi,2N)$, since $w_1,w_2$ are contained in $[B(2N)]^c$, one has $B_{w_i}(\cref{const_point_to_boundary2}N)\subset [B(N)]^c$ for $i\in \{1,2\}$. Hence, we have $\mathsf{C}^{\mathrm{out}}[\psi,\cref{const_point_to_boundary2} ]\subset \overline{\mathsf{H}}_{v'}^{v}(N)$. Combined with (\ref{new4.65}), it implies the lower bound in (\ref{thm1_large_n_bar}).

  For the upper bound in (\ref{thm1_large_n_bar}), when $d\ge 7$, by the BKR inequality we have 
 \begin{equation}
 	\mathbb{P}\big( \overline{\mathsf{H}}_{v'}^{v}(N) \big)\lesssim  \big[\theta_d(cN)\big]^2\lesssim N^{-4}. 
 \end{equation}
 It remains to prove the upper bound for $3\le d\le 5$. Using the decomposition in (\ref{new.462}), it suffices to prove that for any $\psi=(v_1,v_2,w_1,w_2)\in \Psi(\chi,N)$,  
  \begin{equation}\label{489}
  	\begin{split}
  		\mathbb{P}\big( \mathsf{C}^{\mathrm{out}}[ \psi, \cref{const_point_to_boundary2}  ]  \big)\lesssim \chi^{ 3-\frac{d}{2} }N^{-  \frac{d}{2}-1 }. \
  	\end{split}
  \end{equation}
   To achieve this, we arbitrarily take $\psi':=(x_1,x_2,y_1,y_2)\in \Psi(\Cref{const_point_to_boundary1},c_\dagger \chi)$, where $c_\dagger$ will be determined later. We take two paths $\widetilde{\eta}_1$ and $\widetilde{\eta}_2$ satisfying the following: 
  \begin{itemize}
  	\item    $\widetilde{\eta}_i$ starts from $v_j$ and ends at $y_j$;

  	\item   	 $\widetilde{\eta}_i$ consists of at most $2d$ line segments, and is contained in $\widetilde{B}(10\chi)\setminus \widetilde{B}(\frac{c_\dagger \chi}{10})$;

  	\item $	\mathrm{dist}\big( \mathrm{ran}(\widetilde{\eta}_1), \mathrm{ran}(\widetilde{\eta}_2)  \big) \ge c_\star \chi $ for some constant $c_\star>0$.

  \end{itemize}
   We denote $D_j^{\delta}:=\cup_{v\in \mathrm{ran}(\widetilde{\eta}_j)} \widetilde{B}_v(\delta \chi)$. It follows from (\ref{3.5_ineq_lemma_separation_1}) that 
    \begin{equation}
    	\mathbb{P}\big( \mathsf{D}_\delta     \mid \mathsf{C}[\psi] \big)  := 	\mathbb{P}\big(\cup_{j\in \{1,2\}}  \{w_j\xleftrightarrow{}  D_{3-j}^{\delta} \}     \mid \mathsf{C}[\psi] \big) 
   \end{equation} 
  uniformly converges to zero as $\delta\to 0$.
 Combined with (\ref{3.5_ineq_lemma_separation_3}), it yields that (recall the notation $\widehat{\mathsf{C}}[\cdot]$ before Lemma \ref{lemma_separation})
   \begin{equation}
   	  	\mathbb{P}\big( \widehat{\mathsf{C}}[\psi, \cref{const_lemma_separation3}]\cap (\mathsf{D}_{\delta_\ddagger})^c    \big)  \asymp   	\mathbb{P}\big(  \mathsf{C}[\psi] \big) 
   \end{equation}
   for a sufficiently small constant $\delta_\ddagger >0$. Based on this estimate, we may employ the second moment method as in the proof of Lemma \ref{lemma_point_to_boundary} to derive that for some $c^{(1)}_{\ddagger},c^{(2)}_{\ddagger},c^{(3)}_{\ddagger}>0$ depending on $d$ and $\delta_\ddagger$ (recall $\widehat{\mathsf{C}}^{\diamond}[\cdot]$ before Lemma \ref{lemma_point_to_boundary}), 
    \begin{equation}\label{493}
\mathbb{P}\big( \mathsf{A}  \big) :=   	 \mathbb{P}\big(   \widehat{\mathsf{C}}^{\mathrm{in}}[\psi, c^{(1)}_{\ddagger} ,c^{(2)}_{\ddagger},c^{(3)}_{\ddagger}] \cap (\mathsf{D}_{\delta_\ddagger}) ^c  \big) \asymp \mathbb{P}\big(  \mathsf{C}[\psi] \big) \cdot \chi^{ d-2 }. 
   \end{equation}
   Specifically, similar to (\ref{newfixadd4.63}), we take a point $z\in \partial B_{v_1}(4c^{(1)}_{\ddagger}\chi  )$, define $\widetilde{S}_\cdot^z\sim \widetilde{\mathbb{P}}_z$, and then consider the quantity 
   \begin{equation}
   	\mathbf{X}:= \sum\nolimits_{x\in B_{v_1}(c^{(1)}_{\ddagger}\chi  )} \mathbbm{1}_{ \{\widetilde{S}_\cdot^z\ \text{hits}\ x\} \cap  \widehat{\mathsf{C}}^{\mathrm{in}}[(x,v_2,w_1,w_2), c^{(1)}_{\ddagger} ,c^{(2)}_{\ddagger},c^{(3)}_{\ddagger}] \cap (\mathsf{D}_{\delta_\ddagger}) ^c}.
   \end{equation}
    By (\ref{493}), the first moment of $\mathbf{X}$ satisfies $\mathbb{E}[\mathbf{X}]\asymp \chi^2  \mathbb{P}\big(  \mathsf{C}[\psi] \big)$. Moreover, by (\ref{good368}), the second moment of $\mathbf{X}$ is at most $C\chi^{5-\frac{d}{2}}\cdot \mathbb{P}\big(  \mathsf{C}[\psi] \big)$. Thus, as in (\ref{372}), applying the Paley-Zygmund inequality, one has $\mathbb{P}\big(\mathbf{X} >0 \big)\gtrsim \chi^{\frac{d}{2}-1}\mathbb{P}\big(  \mathsf{C}[\psi] \big)$. Next, following the same arguments as in (\ref{new372})-(\ref{380.5}), this yields 
    \begin{equation}
    	\begin{split}
    		\mathbb{P}\big(  \{\overline{\mathsf{G}}  \Vert v_2\xleftrightarrow{} w_2 \} \cap  (\mathsf{D}_{\delta_\ddagger}) ^c \big) \gtrsim \chi^{\frac{d}{2}-1}\cdot \mathbb{P}\big(  \mathsf{C}[\psi] \big),
    	\end{split}
    \end{equation}
     where $\overline{\mathsf{G}}$ is defined by $ \big\{\mathcal{V}_{w_1,v_1}(c^{(2)}_{\ddagger}\chi )\ge c^{(3)}_{\ddagger} \chi^{\frac{d}{2}+1}  \big\}\cap \big\{\mathcal{T}_{w_1,v_1}(c^{(2)}_{\ddagger}\chi )\ge c^{(3)}_{\ddagger}  \chi^{d-2} \big\}$. We then repeat this second moment method for the quantity (here $\widetilde{S}_\cdot^{z'}\sim \widetilde{\mathbb{P}}_{z'}$ is an independent Brownian motion with $z'\in \partial B_{v_2}(4c^{(1)}_{\ddagger}\chi  )$)
    \begin{equation}
    	\mathbf{X}':= \sum\nolimits_{x\in B_{v_2}(c^{(1)}_{\ddagger}\chi  )} \mathbbm{1}_{ \{\widetilde{S}_\cdot^{z'}\ \text{hits}\ x\} \cap \{\overline{\mathsf{G}}  \Vert v_2\xleftrightarrow{} w_2 \} \cap  (\mathsf{D}_{\delta_\ddagger}) ^c},  
    \end{equation}
    thereby deriving the claimed bound (\ref{493}).

   Similarly to (\ref{493}), we also have (letting $\mathsf{D}_\delta':=\cup_{j\in \{1,2\}}  \{x_j\xleftrightarrow{}  D_{3-j}^{\delta} \}$)
   \begin{equation}\label{494}
   	\mathbb{P}\big( \mathsf{A}'  \big) :=  	 \mathbb{P}\big(   \widehat{\mathsf{C}}^{\mathrm{out}}[\psi',c^{(1)}_{\ddagger} ,c^{(2)}_{\ddagger},c^{(3)}_{\ddagger}] \cap (\mathsf{D}_{\delta_\ddagger}' )^c  \big)  \asymp \mathbb{P}\big(  \mathsf{C}[\psi'] \big) \cdot \chi^{ d-2 }. 
   \end{equation}
We require that $c_\dagger<\frac{1}{10}(c^{(2)}_{\ddagger})^2$. Since the event $\mathsf{A}$ depends only on the loops outside $\widetilde{B}(c^{(2)}_{\ddagger}\chi)$, while $\mathsf{A}'$ depends only on the loops within $\widetilde{B}((c^{(2)}_{\ddagger})^{-1}c_\dagger \chi)$, one has  
\begin{equation}\label{495}
 	\begin{split}
 		 	\mathbb{P}\big(\mathsf{A}\cap  \mathsf{A}'  \big) = \mathbb{P}\big(\mathsf{A}  \big) \cdot \mathbb{P}\big(  \mathsf{A}'  \big).
 	\end{split}
 \end{equation}

   Conditioned on $\{ \mathcal{C}_{w_1}=A_1,\mathcal{C}_{w_2}=A_2, \mathcal{C}_{x_1}=A_1',\mathcal{C}_{x_2}=A_2'\}$, where $A_1,A_2,A_1'$ and $A_2'$ are proper realizations of the corresponding clusters such that $\mathsf{A}\cap  \mathsf{A}'$ occurs, the distribution of the remaining loops is given by $\widetilde{\mathcal{L}}_{1/2}^{\overline{A}}$ with $\overline{A}:=A_1\cup A_2\cup A_1'\cup A_2'$. For any $U\subset \widetilde{\mathbb{Z}}^d$ and $\epsilon>0$, we define the event  
    \begin{equation}
    	\mathsf{F}_\epsilon^{U}:=  \cup_{j\in \{1,2\}}\big\{D_j^{\epsilon \cdot \delta_\ddagger} \xleftrightarrow{ (U) }  \widetilde{\partial} D_j^{\delta_\ddagger}   \big\}.  
    \end{equation}
    In addition, we define $\mathsf{F}_\epsilon$ as $\mathsf{F}_\epsilon^{\overline{A}}$ when $\{ \mathcal{C}_{w_1}=A_1,\mathcal{C}_{w_2}=A_2, \mathcal{C}_{x_1}=A_1',\mathcal{C}_{x_2}=A_2'\}$ occurs. Since $\widetilde{\mathcal{L}}_{1/2}^{\overline{A}}\subset \widetilde{\mathcal{L}}_{1/2}$, it follows from (\ref{crossing_low}) and the FKG inequality that for a small constant $\epsilon_\ddagger>0$, the conditional probability of $(\mathsf{F}_{\epsilon_\ddagger})^c $ given $\mathcal{F}_{\mathcal{C}_{\{w_1,w_2,x_1,x_2\}}}$ is uniformly bounded away from zero. Thus, we have 
    \begin{equation}\label{addnew496}
    		\mathbb{P}\big(\mathsf{A}\cap  \mathsf{A}' \cap  (\mathsf{F}_{\epsilon_\ddagger})^c \big) \asymp   	\mathbb{P}\big(\mathsf{A}\cap  \mathsf{A}'  \big) . 
    \end{equation}  
 Moreover, on the event $\mathsf{A}\cap  \mathsf{A}' \cap  (\mathsf{F}_{\epsilon_\ddagger})^c$, if one adds two loops $\widetilde{\ell}_1$ and $\widetilde{\ell}_2$ to $\widetilde{\mathcal{L}}_{1/2}$ such that $\widetilde{\ell}_j$ is contained in $D_j^{\epsilon_\ddagger\cdot \delta_\ddagger}$ and connects $\mathcal{C}_{w_j}$ and $\mathcal{C}_{x_j}$, then the following occur (see Figure \ref{fig1} for an illustration): 
 \begin{enumerate}

 	\item[(i)] For $j\in \{1,2\}$, the new cluster of $\widetilde{\mathcal{L}}_{1/2}$ containing $x_j$ (denoted by $\overline{\mathcal{C}}_j$) consists of $\mathcal{C}_{w_j}$, $\mathcal{C}_{x_j}$, $\widetilde{\ell}_j$ and loops of $\widetilde{\mathcal{L}}_{1/2}$ contained in $D_j^{\delta_\ddagger}$.

 	\item[(ii)]   The clusters $\overline{\mathcal{C}}_1$ and $\overline{\mathcal{C}}_2$ are disjoint.

 	\item[(iii)] The cluster $\overline{\mathcal{C}}_j$ connects $x_j$ and $w_j$, and has volume at least $c^{(3)}_{\ddagger}n^{d-2}$ within $\widetilde{B}_{{x}_j}(20\chi)$.

 \end{enumerate}

\begin{figure}[h]
	\centering
	\includegraphics[width=0.6\textwidth]{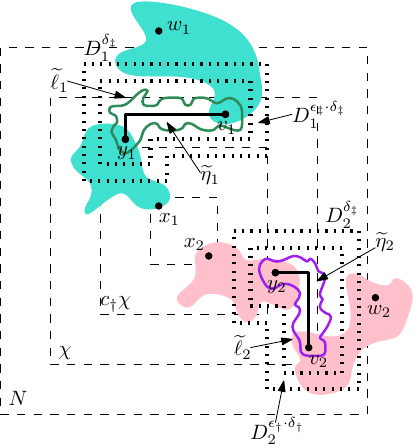}
	\caption{In this illustration, the cyan (resp. pink) regions represent the disjoint clusters $\mathcal{C}_{w_1}$ and $\mathcal{C}_{x_1}$ (resp. $\mathcal{C}_{w_2}$ and $\mathcal{C}_{x_2}$). In order to connect these two pairs of clusters separately, the loops $\widetilde{\ell}_1$ and $\widetilde{\ell}_2$ (i.e., the green and purple curves) are added to $\widetilde{\mathcal{L}}_{1/2}$. For $j\in \{1,2\}$, since $\widetilde{\ell}_j$ is contained in $D_j^{\epsilon_\ddagger\cdot \delta_{\ddagger}}$, the absence of the event $\mathsf{F}_{\epsilon_\ddagger}$ ensures that after adding $\widetilde{\ell}_j$, the enlarged part of the new cluster $\overline{\mathcal{C}}_j$ containing $x_j$ (i.e., the union of clusters intersecting $\widetilde{\ell}_j$) is contained in $D_j^{\delta_{\ddagger}}$ and hence, is disjoint from $\overline{\mathcal{C}}_{3-j}$. Thus, after this modification, the event $\mathsf{C}[(x_1,x_2,w_1,w_2)]$ occurs. 
		 }\label{fig1}
\end{figure}

 We denote the intersection of these events by $\mathsf{G}$. In particular, one has $\mathsf{G}\subset \mathsf{C}[(x_1,x_2,w_1,w_2)]$. Note that the total mass of loops $\widetilde{\ell}_1$ and $\widetilde{\ell}_2$ is of order $1$ since $\mathrm{cap}(\mathcal{C}_{w_j}\cap \widetilde{B}_{v_j}(c_\ddagger^{(1)}n))\ge c_\ddagger^{(3)}\chi^{d-2}$ and $\mathrm{cap}(\mathcal{C}_{x_j}\cap \widetilde{B}_{y_j}(c_\ddagger^{(1)}\chi))\ge c_\ddagger^{(3)}(c_\dagger \chi)^{d-2}$. Thus, 
 \begin{equation}\label{addnew497}
 \begin{split}
 	 	&\mathbb{P}\big(	\mathsf{G}\big)\asymp \mathbb{P}\big(\mathsf{A}\cap  \mathsf{A}' \cap ( \mathsf{F}_{\epsilon_\ddagger})^c \big)   \\\overset{(\ref{495}),(\ref{addnew496})}{\gtrsim}& \mathbb{P} (\mathsf{A}  )  \cdot  \mathbb{P} (  \mathsf{A}'  ) \overset{ (\ref{493}), (\ref{494})}{\gtrsim}   \chi^{2d-4}\cdot \mathbb{P}\big(  \mathsf{C}[\psi] \big)\cdot \mathbb{P}\big(  \mathsf{C}[\psi'] \big). 
 \end{split}
 \end{equation}
 Combined with $\mathsf{G}\subset \mathsf{C}[(x_1,x_2,w_1,w_2)]$ and Lemmas \ref{lemma_prob_Cpsi_small_n} and \ref{lemma_prob_Cpsi}, this implies (\ref{489}), which in turn yields the upper bound in (\ref{thm1_large_n_bar}). In conclusion, we complete the proof of Theorem \ref{thm1}.  \qed

 \begin{remark}[extension with zero boundary conditions]\label{remark_extension_zero_boundary}
 Throughout the proof of Theorem \ref{thm1}, all relevant quantities---including the hitting probabilities of Brownian motion, Green's functions, two-point functions, one-arm probabilities, and crossing probabilities---will change by at most a constant factor when zero boundary conditions are imposed on either a distant region or a small neighborhood (see e.g., \cite[Proposition 1.8, Lemmas 2.1 and 2.2]{cai2024quasi}). Thus, applying the same argument yields analogous bounds in Theorem \ref{thm1} for the heterochromatic two-arm probability under appropriate zero boundary conditions. Specifically, for any $d\ge 3$ with $d\neq 6$, there exist $C,C',c,c'>0$ such that for any $N\ge C$ and $v,v'\in \widetilde{B}(cN)$, the following hold:
 \begin{enumerate}

 	\item  When $ |v-v'|\le 1$, for any $D\subset [\widetilde{B}(C'N)]^c$, 
 	\begin{equation}
 		\mathbb{P}^{D}\big( \mathsf{H}_{v'}^v(N) \big)\asymp |v-v'|^{\frac{3}{2}}N^{-[(\frac{d}{2}+1)\boxdot 4]}; 
 	\end{equation}

 	\item When $ |v-v'|\ge 1$, if there exists $n\in [1,c'N]$ such that $v,v'\in \widetilde{B}(10n)\setminus \widetilde{B}(n)$, then for any $D\subset \widetilde{B}(\frac{n}{C'})\cup [\widetilde{B}(C'N)]^c$, 
 	\begin{equation}
 		\mathbb{P}^{D}\big( \mathsf{H}_{v'}^v(N) \big)\asymp |v-v'|^{(3-\frac{d}{2})\boxdot 0}N^{-[(\frac{d}{2}+1)\boxdot 4]}. 
 	\end{equation}

 \end{enumerate}

 \end{remark}

 \subsection{Proof of Theorem \ref{theorem_four_point}}\label{subsection_theorem_four_point_remaining}
 Recall that (\ref{thm1_small_n_four_point}) has been verified by Lemmas \ref{lemma_prob_Cpsi_small_n} and \ref{additional_lemma_four_point}, and that the lower bound in (\ref{thm1_large_n_four_point}) has been proved in Lemmas \ref{lemma_prob_Cpsi} and \ref{lemma_highd_volume_two_arm}. Hence, it remains to consider the upper bound in (\ref{thm1_large_n_four_point}). By the isomorphism theorem and Lemma \ref{lemma_point_to_boundary}, we have: for any $\psi=(v,v',w,w')\in \Psi(\chi,N)$, 
\begin{equation}\label{final640}
	\begin{split}
		\mathbb{P}\big( \mathsf{H}_{v',w'}^{v,w}\big)\asymp 	\mathbb{P}\big( \mathsf{C}[\psi]\big) \asymp & N^{-[(d-2)\boxdot (2d-8)]} \cdot \mathbb{P}\big( \mathsf{C}^{\mathrm{out}}[\psi,\cref{const_point_to_boundary2} ]\big)\\
		\asymp & N^{-[(d-2)\boxdot (2d-8)]} \cdot \mathbb{P}\big( \mathsf{H}^{v, B_{{w}_1}(\cref{const_point_to_boundary2} N)}_{v', B_{{w}_2}(\cref{const_point_to_boundary2} N)} \big). 
	\end{split}
\end{equation}
  Meanwhile, by $\mathsf{H}^{v, B_{{w}_1}(\cref{const_point_to_boundary2} N)}_{v', B_{{w}_2}(\cref{const_point_to_boundary2} N)} \subset \mathsf{H}^{v}_{v'}(\frac{N}{10})$ and Theorem \ref{thm1}, one has 
 \begin{equation}\label{final641}
 	\mathbb{P}\big( \mathsf{H}^{v, B_{{w}_1}(\cref{const_point_to_boundary2} N)}_{v', B_{{w}_2}(\cref{const_point_to_boundary2} N)} \big) \le 	\mathbb{P}\big( \mathsf{H}^{v}_{v'}(\tfrac{N}{10}) \big) \lesssim \chi^{(3-\frac{d}{2})\boxdot 0}\cdot N^{-[(\frac{d}{2}+1)\boxdot 4]}.
 \end{equation}
 Plugging (\ref{final641}) into (\ref{final640}), we get the upper bound in (\ref{thm1_large_n_four_point}), and thus conclude the proof of Theorem \ref{theorem_four_point}.       \qed

  As noted in Remark \ref{remark_extension_zero_boundary}, imposing suitable zero boundary conditions also preserves the estimates in Theorem \ref{theorem_four_point}, which leads to the following result:

 \begin{lemma}
 	 For any $d\ge 3$ with $d\neq 6$, there exist $C,C',c,c'>0$ such that for any $N\ge C$, $v,v'\in \widetilde{B}(cN)$ and $w,w'\in \widetilde{B}(10N)\setminus \widetilde{B}(N)$ with $|w-w'|\ge N$, 
 	  \begin{enumerate}

 	\item  When $ |v-v'|\le 1$, for any $D\subset [\widetilde{B}(C'N)]^c$, 
 	\begin{equation}
 		\mathbb{P}^{D}\big( \mathsf{H}_{v',w'}^{v,w}  \big)\asymp |v-v'|^{\frac{3}{2}}N^{-[(\frac{3d}{2}-1)\boxdot (2d-4)]}; 
 	\end{equation}

 	\item When $ |v-v'|\ge 1$, if there exists $n\in [1,c'N]$ such that $v,v'\in \widetilde{B}(10n)\setminus \widetilde{B}(n)$, then for any $D\subset \widetilde{B}(\frac{n}{C'})\cup [\widetilde{B}(C'N)]^c$, 
 	\begin{equation}
 		\mathbb{P}^{D}\big( \mathsf{H}_{v',w'}^{v,w}  \big)\asymp |v-v'|^{(3-\frac{d}{2})\boxdot 0}N^{-[(\frac{3d}{2}-1)\boxdot (2d-4)]}. 
 	\end{equation}

 \end{enumerate}
 \end{lemma}

\section{Cluster volumes under heterochromatic two-arm events}\label{section_clusters_volume}

In this section, we aim to establish Theorem \ref{thm3_volume}, which states that conditioned on the two-arm events, the volume growth of each involved cluster is of the same order as that of an unconditioned cluster (as shown in (\ref{result_typical_volume})). The proof is divided into the following two parts.
\begin{enumerate}

	\item[\textbf{Section \ref{section5.1_upper_volume}:}] We prove through constructions that given the occurrence of two distinct large clusters, both clusters have sufficiently large volume with a uniform positive probability.

	\item[\textbf{Section \ref{section5.2_lower_volume}:}]  We first establish an upper bound on the expected volume of the clusters involved in two-arm events. This bound, combined with Markov's inequality, yields the desired upper bound on their typical volumes.


\end{enumerate}

\subsection{Lower bound on the growth rate}\label{section5.1_upper_volume}
Recall $\mathcal{V}_{z,z'}$ and $\mathcal{T}_{z,z'}$ in (\ref{def353}). For brevity, we denote $\mathcal{V}_{z}:=\mathcal{V}_{z,z}$ and $\mathcal{T}_{z}:=\mathcal{T}_{z,z}$. We aim to prove that for some $c>0$, 
\begin{equation}\label{volume_part1}
		\mathbb{P}\big(  \mathcal{V}_v (M) \land  \mathcal{V}_{v'}(M) \ge c M^{(\frac{d}{2}+1)\boxdot 4}  \mid \overline{\mathsf{H}}_{v'}^{v}(N) \big) \gtrsim 1. 
		\end{equation}
		In what follows, we first prove this inequality under the restriction $|v-v'|\ge C$, and then extend it to the case when $|v - v'| \leq C$.


			Arbitrarily take $z_1,z_2\in \partial B(2N)$ with $|z_1-z_2|\ge d N$, and two sufficiently small constants $c_\ddagger,c_{ \star}>0$. For any $w_1\in B_{z_1}(c_\ddagger^{2}N)$ and $w_2\in B_{z_2}(c_\ddagger^{2}N)$, we define 
			\begin{equation}
				\begin{split}
					\mathsf{A}_{w_1,w_2}:=&  \mathsf{C}[(v,v',w_1,w_2)] \cap \{ \mathcal{V}_{v} (M) \land \mathcal{V}_{v'} (M) \ge c_{ \star} M^{(\frac{d}{2}+1)\boxdot 4} \} \\
					& \cap \big\{ v\xleftrightarrow{}  B_{z_2}(c_\ddagger N)  \big\}^c   
					  \cap \big\{ v' \xleftrightarrow{}   B_{z_1}(c_\ddagger N)    \big\}^c. 
				\end{split}
			\end{equation}
			We then consider the quantity
			\begin{equation}
 	\mathbf{X}_{w_2}:= \sum\nolimits_{w_1\in B_{z_1}(c_\ddagger^{2}N)} \mathbbm{1}_{\mathsf{A}_{w_1,w_2}}. 
 \end{equation}
			For the first moment of $\mathbf{X}_{w_2}$, we claim that 
\begin{equation}\label{3.11_7.3}
	\mathbb{E}[\mathbf{X}_{w_2}] \gtrsim  |v-v'|^{(3-\frac{d}{2})\boxdot 0}N^{-[(\frac{d}{2}-1)\boxdot (d-4)]}. 
\end{equation}

\textbf{When $3\le d\le 5$.} We employ $\chi $, $N$, $\psi=(v_1,v_2,w_1,w_2)$, $\psi':=(x_1,x_2,y_1,y_2)$ and $\mathsf{G}$ as in Section \ref{newsubsection4.2}. When $M\ge c_\ddagger^{-2}|v - v'|$, we take $\chi =M$, $x_1=v$ and $x_2=v'$. Recall the event $\mathsf{G}$ before (\ref{addnew497}), and note that $\mathsf{G}\subset \mathsf{A}_{w_1,w_2}$. As a result,  
    \begin{equation}\label{new52}
	\begin{split}
		 \mathbb{P} (  \mathsf{A}_{w_1,w_2} ) 
		\ge  \mathbb{P}\big(	\mathsf{G}\big) \overset{(\ref{addnew497})}{\gtrsim }\chi^{2d-4}\cdot \mathbb{P}\big(  \mathsf{C}[\psi] \big) \cdot \mathbb{P}\big(  \mathsf{C}[\psi'] \big)\overset{(\ref{new44})}{\gtrsim } |v-v'|^{3-\frac{d}{2}} N^{-\frac{3d}{2}+1}. 
	\end{split}
\end{equation}

 When $1\le M\le c_\ddagger^{-2}|v - v'|$, we denote $\hat{M}:=c_\ddagger^2 \cdot (M\land  |v - v'|)$. Arbitrarily take $x\in \partial B_v(\hat{M})$ and $x'\in \partial B_{v'}(\hat{M})$. For any $y\in B_x(c_\ddagger \hat{M})$ and $y'\in B_{x'}(c_\ddagger \hat{M})$, define 
\begin{equation}
	\begin{split}
		\mathsf{F}_{w_1,w_2}^{y,y'} :=&  \mathsf{C}[(y,y',w_1,w_2)]\cap \big\{ w_1 \xleftrightarrow{} B_v(c_\ddagger\hat{M})  \cup B_{v'}(c_\ddagger |v-v'|)  \cup B_{z_2}(c_\ddagger N) \big\}^c \\
		& \cap \big\{ w_2 \xleftrightarrow{} B_{v'}(c_\ddagger\hat{M})  \cup B_{v}(c_\ddagger |v-v'|) \cup B_{z_1}(c_\ddagger N)  \big\}^c. 
	\end{split}
\end{equation}
 By Theorem \ref{theorem_four_point} and Lemma \ref{lemma_separation}, one has 
\begin{equation}\label{3.18_ineq_7.7}
	\begin{split}
		\mathbb{P}\big( \mathsf{F}_{w_1,w_2}^{y,y'} \big)  \gtrsim  |v-v'|^{3-\frac{d}{2}} N^{-\frac{3d}{2}+1}.   
	\end{split}
\end{equation}
Repeating the arguments in the proof of (\ref{good363}), with $\widehat{\mathsf{C}}[(\cdot,\cdot,w_1,w_2), \cref{const_lemma_separation3}]$ replaced by $\mathsf{F}_{w_1,w_2}^{\cdot,\cdot}$, and with $B_{{v}_1}(\cref{const_point_to_boundary2} n),B_{{v}_2}(\cref{const_point_to_boundary2} n)$ replaced by $B_x(c_\ddagger \hat{M}),B_{x'}(c_\ddagger \hat{M})$, we obtain  
 \begin{equation}\label{3.11_7.8}
	\mathbb{P}\big( \widehat{\mathsf{F}}_{w_1,w_2} \big) \gtrsim \hat{M}^{d+2} |v-v'|^{3-\frac{d}{2}} N^{-\frac{3d}{2}+1}, 
\end{equation}
where $\widehat{\mathsf{F}}_{w_1,w_2}$ is defined by 
\begin{equation*}
	\begin{split}
 	 \widehat{\mathsf{F}}_{w_1,w_2}:=& \big\{ \mathrm{cap}\big( \mathcal{C}_{w_1}\cap \widetilde{B}_x(c_\ddagger \hat{M}) \big)  \ge c\hat{M}^{d-2}  \Vert \mathrm{cap}\big( \mathcal{C}_{w_2}\cap \widetilde{B}_{x'}(c_\ddagger \hat{M}) \big)  \ge c\hat{M}^{d-2}  \big\} \\
 	& \cap \big\{ w_1 \xleftrightarrow{} B_v(c_\ddagger\hat{M})  \cup B_{v'}(c_\ddagger |v-v'|)  \cup B_{z_2}(c_\ddagger N) \big\}^c  \\
		&  \cap \big\{ w_2 \xleftrightarrow{} B_{v'}(c_\ddagger\hat{M})  \cup B_{v}(c_\ddagger |v-v'|) \cup B_{z_1}(c_\ddagger N)  \big\}^c. 
	\end{split}
\end{equation*}


Let $u \in \{v,v'\}$ and $D\subset [B_u(c_\ddagger\hat{M})]^c$. It follows from (\ref{29}) and Lemma \ref{lemma_avoid_path_two_point} that  
\begin{equation}\label{3.18_add_7.9}
	\begin{split}
		\mathbb{P}\big( u \xleftrightarrow{(D)}  y, \{u \xleftrightarrow{(D)}\partial  B_u(c_\ddagger\hat{M})  \}^c\big) \asymp (|u-y|+1)^{2-d}, \ \forall y\in \widetilde{B}_{u}(c_\ddagger^2\hat{M}). 
	\end{split}
\end{equation}
 In addition, Lemma \ref{lemma_cite_three_point} implies that for any $y_1,y_2\in \widetilde{B}_{u}(c_\ddagger^2\hat{M})$, 
\begin{equation}\label{3.18_add_7.10}
	\mathbb{P}\big( u \xleftrightarrow{(D)}  y_1 ,u \xleftrightarrow{(D)}  y_2 \big)\lesssim (|y_1-y_2|+1)^{-\frac{d}{2}+1} \mathbb{P}\big( u \xleftrightarrow{(D)}  y_1  \big). 
\end{equation}
In fact, applying the second moment method as in Lemma \ref{lemma_point_to_boundary}, one can obtain
 \begin{equation}\label{3.11_7.11}
\mathbb{P}\big( \mathsf{G}_u^D \big)  \gtrsim 	\hat{M}^{-\frac{d}{2}+1}, 
 \end{equation}
where the event $\mathsf{G}_u^D$ is defined by 
\begin{equation*}
\begin{split}
	\mathsf{G}_u^D:= &  \big\{ \mathrm{cap}\big(\mathcal{C}_u^D\cap \widetilde{B}_{u}(c_\ddagger^2\hat{M}) \big) \ge c\hat{M}^{d-2} \big\}\cap \big\{ \mathrm{vol}\big(\mathcal{C}_u^D\cap \widetilde{B}_{u}(c_\ddagger^2\hat{M}) \big) \ge c \hat{M}^{\frac{d}{2}+1} \big\} \\
	& \cap \big\{u\xleftrightarrow{(D)}  \partial  B_u(c_\ddagger\hat{M}) \big\}^c. 
\end{split}	
\end{equation*} 
Precisely, as in (\ref{newfixadd4.63}), we take an independent Brownian motion $\widetilde{S}^z_\cdot$ starting from some $z\in \partial B_{u}(4c_{\ddagger}^2\hat{M})$, and then consider the quantity 
\begin{equation}
	\mathbf{M}:= \sum\nolimits_{y\in B_{u}(c_\ddagger^2\hat{M})\setminus B_{u}(c_\ddagger^2\hat{M}/2)}  \mathbbm{1}_{\{\widetilde{S}^z_{\cdot}\ \text{hits}\ y\}\cap \{u \xleftrightarrow{(D)}  y\}\cap  \{u \xleftrightarrow{(D)}\partial  B_u(c_\ddagger\hat{M})  \}^c}. 
\end{equation} 
By applying the Paley-Zygmund inequality, and using (\ref{3.18_add_7.9}) and (\ref{3.18_add_7.10}) to control the first and second moments of $\mathbf{M}$ respectively, one has 
\begin{equation}\label{3.18_nice_7.13}
	\mathbb{P}\big( \mathbf{M}>0 \big)\gtrsim  \hat{M}^{-\frac{d}{2}+1}. 
\end{equation}
Meanwhile, since $\{\mathbf{M}>0\}\subset \{u\xleftrightarrow{} \partial B_{u}(c_\ddagger^2\hat{M}/2)\}$, it follows from (\ref{one_arm_low}) that (\ref{3.18_nice_7.13}) is sharp up to multiplicative constants. Finally, by repeating the arguments leading from (\ref{add373}) to (\ref{380.5}), we derive (\ref{3.11_7.11}).


 By the restriction property, (\ref{3.11_7.8}) and (\ref{3.11_7.11}), we have 
\begin{equation}
	\begin{split}
\mathbb{P}(\mathsf{H}):=\mathbb{P}\big(\widehat{\mathsf{F}}_{w_1,w_2} , \mathsf{G}_v^{\emptyset}, \mathsf{G}_{v'}^{\emptyset},  \{v\xleftrightarrow{} w_1\}^c,  \{v'\xleftrightarrow{} w_2\}^c  \big) \gtrsim |v-v'|^{3-\frac{d}{2}} N^{-\frac{3d}{2}+1}. 
	\end{split}
\end{equation}
On the event $\mathsf{H}$, if one adds two loops $\ell$ and $\ell'$ satisfying 
 \begin{itemize}

 	\item   $\mathrm{ran}(\ell)\subset \widetilde{B}_{v}(c_\ddagger |v-v'|)$ and $\mathrm{ran}(\ell')\subset \widetilde{B}_{v'}(c_\ddagger |v-v'|)$;

 	\item  $\ell$ connects $\mathcal{C}_{v}$ and $\mathcal{C}_{w_1}$, and $\ell'$ connects $\mathcal{C}_{v'}$ and $\mathcal{C}_{w_2}$,

 \end{itemize}
then $\mathsf{A}_{w_1,w_2}$ occurs (since $M\asymp \hat{M}$). Noting that the total mass of such loops is uniformly bounded away from zero (by the lower bounds on capacities), we have 
 \begin{equation}\label{3.11_7.13}
 	\mathbb{P}(\mathsf{A}_{w_1,w_2}  ) \gtrsim \mathbb{P}(\mathsf{H}) \gtrsim |v-v'|^{3-\frac{d}{2}} N^{-\frac{3d}{2}+1}. 
 \end{equation}
 The estimates (\ref{new52}) and (\ref{3.11_7.13}) together yield (\ref{3.11_7.3}) for $3\le d\le 5$.

\textbf{When $d\ge 7$.} The claim (\ref{3.11_7.3}) follows from Lemma \ref{lemma_highd_volume_two_arm}: 
 \begin{equation}
 	\mathbb{E} [ \mathbf{X}_{w_2} ]\gtrsim | B_{z_1}(c_\ddagger^{2}N)| \cdot N^{4-2d}\asymp N^{4-d}. 
 \end{equation}

 \vspace*{0.25cm}

\noindent For the second moment of $\mathbf{X}_{w_2}$, we claim that 
\begin{equation}\label{3.11_7.15}
	\begin{split}
		\mathbb{E}\big[\mathbf{X}_{w_2}^2\big] \lesssim   |v-v'|^{(3-\frac{d}{2})\boxdot 0}N^{2\boxdot (8-d)} . 
	\end{split}
\end{equation}

 \textbf{When $3\le d\le 5$.} For any $w_1,w_1'\in B_{z_1}(c_\ddagger^{2}N)$, using the restriction property and Lemma \ref{lemma_cite_three_point}, we have  
\begin{equation}
	\begin{split}
		\mathbb{P}\big(\mathsf{A}_{w_1,w_2} \cap \mathsf{A}_{w_1',w_2}  \big) \le&  \mathbb{P}\big(v \xleftrightarrow{}  w_1\xleftrightarrow{} w_1'   \Vert v'\xleftrightarrow{} w_2, \{v'\xleftrightarrow{}B_{z_1}(c_\ddagger N)  \}^c \big) \\
		\lesssim & (|w_1-w_1'|+1)^{-\frac{d}{2}+1} \mathbb{P}\big( \mathsf{C}[(v,v',w_1,w_2)] \big)\\
		\overset{}{\lesssim} &  (|w_1-w_1'|+1)^{-\frac{d}{2}+1}|v-v'|^{3-\frac{d}{2}} N^{-\frac{3d}{2}+1}. 
	\end{split}
\end{equation}
Summing over all $w_1,w_1'\in B_{z_1}(c_\ddagger^{2}N)$, this yields the claim (\ref{3.11_7.15}).

 \textbf{When $d\ge 7$.} By the BKR inequality we have 
 \begin{equation}
 	\begin{split}
 		\mathbb{E}\big[\mathbf{X}_{w_2}^2\big] \le   \sum_{w_1,w_1'\in B_{z_1}(c_\ddagger^{2}N)} \mathbb{P}\big( v\xleftrightarrow{} w_1\xleftrightarrow{} w_1'   \big) \cdot \mathbb{P}\big(  v'\xleftrightarrow{} w_2  \big) 
 		\overset{(\text{Lemma}\ \ref{lemme_technical_high_dimension})}{\lesssim } N^{8-d}. 
 	\end{split}
 \end{equation}

{\color{red}


 
 }


 \vspace*{0.25cm}

For any $w_2\in B_{z_2}(c_\ddagger^{2}N)$, we define the event 
\begin{equation}
\begin{split}
		\widehat{\mathsf{A}}_{w_2} := & \big\{ v\xleftrightarrow{} \partial B(N) \Vert v'\xleftrightarrow{} w_2 \big\} \cap \big\{ v\xleftrightarrow{} B_{z_2}(c_\ddagger N) \big\}^c \\
		& \cap  \big\{ \mathcal{V}_{v} (M) \land \mathcal{V}_{v'} (M) \ge c_{ \star} M^{(\frac{d}{2}+1)\boxdot 4} \big\}. 
\end{split}
\end{equation}
By the Paley-Zygmund inequality, (\ref{3.11_7.3}) and (\ref{3.11_7.15}), we have 
  \begin{equation}
  	\mathbb{P}\big(\widehat{\mathsf{A}}_{w_2}  \big) \ge \mathbb{P}\big(\mathbf{X}_{w_2} >0 \big)\ge  \frac{(\mathbb{E}\big[ \mathbf{X}_{w_2}  \big])^2}{\mathbb{E}\big[ \mathbf{X}_{w_2} ^2\big]} \gtrsim   |v-v'|^{(3-\frac{d}{2})\boxdot 0} N^{-d}. 
  \end{equation}
   Repeating the second moment method above to the quantity 
   \begin{equation}
   	\widehat{\mathbf{X}}:= \sum\nolimits_{w_2\in B_{z_2}(c_\ddagger^{2}N)} \mathbbm{1}_{\widehat{\mathsf{A}}_{w_2}}, 
   \end{equation}
  we derive the following bound: 
  \begin{equation}
  \begin{split}
  	  	 & \mathbb{P}\big(  \mathcal{V}_v (M) \land  \mathcal{V}_{v'}(M) \ge c_\star  M^{(\frac{d}{2}+1)\boxdot 4} ,  \overline{\mathsf{H}}_{v'}^{v}(N) \big)  \\
  	  	 \ge & \mathbb{P}\big( 	\widehat{\mathbf{X}} >0 \big)  \gtrsim  |v-v'|^{(3-\frac{d}{2})\boxdot 0} N^{-[(\frac{d}{2}+1)\boxdot 4]}, 
  \end{split}
  \end{equation}   
  which together with Theorem \ref{thm1} implies (\ref{volume_part1}) for $|v-v'| \ge C$.


As for the case $|v-v'|\le C$, recall the argument in Section \ref{newsubsection4.1}, which was used to extend the lower bound in (\ref{thm1_small_n_bar}) from $|v-v'|\ge C$ to $|v-v'|\le C$. Specifically, we apply a union bound to identify two distinct points $x_1^*,x_2^*\in \partial B(C')$ such that the parts of their sign clusters outside $\widetilde{B}(C')$ reach $\partial B(N)$ with significant probability, and then explore $\widetilde{\phi}_\cdot$ inside $\widetilde{B}(C')$ to connect these two partial clusters to $v$ and $v'$ respectively. By repeating the same argument, we can similarly extend (\ref{volume_part1}) from $|v-v'|\ge C$ to $|v-v'| \le C$ with an additional factor of $|v-v'|^{\frac{3}{2}}$, i.e., 
 \begin{equation}
 	\mathbb{P}\big(  \mathcal{V}_{v} (M) \land  \mathcal{V}_{v'}(M) \ge c M^{(\frac{d}{2}+1)\boxdot 4} ,  \overline{\mathsf{H}}_{x_2}^{x_1}(N) \big) \gtrsim |v-v'|^{\frac{3}{2}}N^{-\frac{d}{2}-1}. 
 \end{equation}
Combined with Theorem \ref{thm1}, it yields (\ref{volume_part1}) for $|v-v'| \le C$.  \qed

\subsection{Upper bound on the growth rate }\label{section5.2_lower_volume}

 The goal of this subsection is to prove that for any $\lambda>1$,  
\begin{equation}\label{volume_part2}
	\mathbb{P}\big(  \mathcal{V}_v (M) \vee   \mathcal{V}_{v'}(M) \ge \lambda M^{(\frac{d}{2}+1)\boxdot 4}  \mid \overline{\mathsf{H}}_{v'}^{v}(N) \big) \lesssim  \lambda^{-1}.
  \end{equation}
By Markov's inequality and Theorem \ref{thm1}, it suffices to prove that for $x\in \{v,v'\}$, 
     \begin{equation}\label{revisenew_78}
  \sum\nolimits_{y\in B_{x}(M)} 	 \mathbb{P}\big(  x\xleftrightarrow{} y ,  \overline{\mathsf{H}}_{v'}^{v}(N) \big)  \lesssim  f(|v-v'|) \cdot \big(\tfrac{M}{N} \big)^{(\frac{d}{2}+1)\boxdot 4}, 
   \end{equation} 
  where $f(a)$ is defined by $f(a) = a^{\frac{3}{2}}$ for $0<a\le 1$, and $f(a) = a^{(3 - \frac{d}{2})\boxdot 0}$ for $a\ge 1$.

  Without loss of generality, we only consider the case $x=v$. By the restriction property, we have 
   \begin{equation}\label{new.54}
   	\begin{split}
   		 \mathbb{P}\big(  v\xleftrightarrow{} y,  \overline{\mathsf{H}}_{v'}^{v}(N) \big) 
  	\le &  \mathbb{E} \big[ \mathbbm{1}_{ v\xleftrightarrow{} y, v\xleftrightarrow{} \partial B(N) }\cdot \mathbb{P}^{\mathcal{C}_v}\big(v'\xleftrightarrow{} \partial B(\tfrac{N}{2})  \big) \big]\\
  		\overset{(\text{Lemma}\ \ref{lemma_onecluster_box_point})}{\lesssim } & N^{- (\frac{d}{2} \boxdot 3)}\sum_{w_2\in \partial \mathcal{B}(\frac{N}{2d})} \mathbb{P} \big(  v\xleftrightarrow{} y,v\xleftrightarrow{} \partial B(N) \Vert v'\xleftrightarrow{} w_2    \big). 
   	\end{split}
   \end{equation}    
  As in (\ref{3.4_4.33}), $\mathbb{P} (  v\xleftrightarrow{} y,v\xleftrightarrow{} \partial B(N) \Vert v'\xleftrightarrow{} w_2 )$ can be written as 
 \begin{equation}\label{3.12_7.26}
  \begin{split}
  	 			\int_{a_1,b_1,a_2,b_2>0}    \mathbb{P} \big(   \{ \cup \mathfrak{L}_1   \xleftrightarrow{\cup \mathcal{P}} \cup \mathfrak{L}_2   \}^c, \cup \mathfrak{L}_1 \xleftrightarrow{\cup \mathcal{P}}   \partial B(N) \big) 
 		 \bar{\mathfrak{p}}   \mathrm{d}a_1\mathrm{d}b_1\mathrm{d}a_2\mathrm{d}b_2. 
  \end{split}
 \end{equation} 
 Here $\mathfrak{L}_1$ (resp. $\mathfrak{L}_2$) has the same law as $\mathcal{P}^{(4)}$ under $\widecheck{\mathbb{P}}_{v\leftrightarrow y,a_1,b_1}$ (resp. $\widecheck{\mathbb{P}}^{\{v,y\}}_{v'\leftrightarrow w_2,a_2,b_2}$), $\mathcal{P}:=\mathcal{P}_{v,v',y,w_2}^{a_1,a_2,b_1,b_2}$ (defined below (\ref{3.9_4.17})), and $\bar{\mathfrak{p}}:=\mathfrak{p}_{v \leftrightarrow y ,a_1,b_1}\cdot \mathfrak{p}_{v'\leftrightarrow w_2,a_2,b_2}^{\{v,y\}}$. For each $z\in \{v,v',y,w_2\}$, let $\mathcal{P}_z$ denote the point measure consisting of excursions in $\mathcal{P}$ starting from $z$. On $\{ \cup \mathfrak{L}_1   \xleftrightarrow{\cup \mathcal{P}} \cup \mathfrak{L}_2   \}^c$, the cluster $\mathcal{C}_2:=\{ z\in \widetilde{\mathbb{Z}}^d: z\xleftrightarrow{\cup \mathcal{P} } \cup \mathfrak{L}_2 \}$ must be disjoint from $\cup (\mathcal{P}_v+\mathcal{P}_y +\mathfrak{L}_1)$. Moreover, if $\{\cup \mathfrak{L}_1 \xleftrightarrow{\cup \mathcal{P}}   \partial B(N)\}$ also occurs, then $\cup (\mathcal{P}_v+\mathcal{P}_y +\mathfrak{L}_1)$ is connected to $\partial B(N)$ by $\widetilde{\mathcal{L}}_{1/2}^{\mathcal{C}_2\cup \{v,y\}}$. Thus, letting $\mathsf{F}:=\{ \cup (\mathcal{P}_v+\mathcal{P}_y +\mathfrak{L}_1 ) \subset \widetilde{B}(cN) \}$ and applying Lemma \ref{lemma_onecluster_box_point}, we have 
 \begin{equation}\label{3.12_7.27}
 	 \begin{split}
 	& 	  \mathbb{P} \big(   \{ \cup \mathfrak{L}_1   \xleftrightarrow{\cup \mathcal{P}} \cup \mathfrak{L}_2   \}^c, \cup \mathfrak{L}_1 \xleftrightarrow{\cup \mathcal{P}}   \partial B(N) , \mathsf{F} \big) \\
   \lesssim &  N^{- (\frac{d}{2} \boxdot 3)}\sum\nolimits_{w_1\in \partial \mathcal{B}(d^{-1}N)}   \mathbb{P} \big(   \{ \cup \mathfrak{L}_1   \xleftrightarrow{\cup \mathcal{P}} \cup \mathfrak{L}_2   \}^c, \cup \mathfrak{L}_1 \xleftrightarrow{\cup \mathcal{P}} w_1 \big). 
 	 \end{split}
 \end{equation}
 Next, we estimate the probability of $\mathsf{G}:=\{ \cup \mathfrak{L}_1   \xleftrightarrow{\cup \mathcal{P}} \cup \mathfrak{L}_2   \}^c\cap \{ \cup \mathfrak{L}_1 \xleftrightarrow{\cup \mathcal{P}}   \partial B(N)\}\cap  \mathsf{F}^c$. Let $\chi:=|v-v'|$, $\hat{\chi}:=\chi\land 1$ and $r:=|v-y|+1$. Note that $\mathsf{G}$ implies $v'\notin \cup \mathcal{P}_v$, whose probability is at most $e^{-ca_1 \hat{\chi}^{-1}}$ (by (\ref{3.11_3.71})). Meanwhile, using Lemma \ref{lemma_new216} and (\ref{property_poisson_odd_number}), one has 
 \begin{equation}
 	\begin{split}
 		\mathbb{P}( \mathsf{F}^c)  \lesssim (a_1+ b_1) N^{2-d} + \sqrt{a_1b_1}\big(\tfrac{r}{N}\big)^{d-2}.  
 	\end{split}
 \end{equation}
Conditioned on $\mathsf{F}^c$, $\cup \mathfrak{L}_1$ and $\cup \mathfrak{L}_2$ both stochastically dominate a Brownian excursion from $\partial B_v(C(\chi+1))$ to $\partial B(cN)$; in addition, as noted in Remark \ref{remark_order_escape_prob}, the probability of two such excursions not being connected by the loop soup is proportional to $N^{2d-4}\cdot \mathbb{P}(\mathsf{C}[(\psi)])$ for $\psi\in \Psi(C(\chi+1),N)$, which is of order $[(\chi+1)/N]^{(3-\frac{d}{2})\boxdot 0}$ by Theorem \ref{theorem_four_point}. To sum up, 
  \begin{equation}\label{3.12_7.29}
  	\begin{split}
  		\mathbb{P}(\mathsf{G})\lesssim e^{-ca_1\hat{\chi}^{-1}} \big[ (a_1+ b_1) N^{2-d} + (a_1b_1)^{\frac{1}{2}}\big(\tfrac{r}{N}\big)^{d-2}\big] \cdot \big( \tfrac{\chi+1}{N }\big)^{(3-\frac{d}{2})\boxdot 0}.
  		  	\end{split}
  \end{equation}
 Combining (\ref{3.12_7.26}), (\ref{3.12_7.27}) and (\ref{3.12_7.29}), we have
 \begin{equation}\label{3.12_7.30}
 	\begin{split}
 		\mathbb{P} (  v\xleftrightarrow{} y,v\xleftrightarrow{} \partial B(N) \Vert v'\xleftrightarrow{} w_2 ) \lesssim N^{- (\frac{d}{2} \boxdot 3)} \mathbb{J}_1(y,w_2)+ \big( \tfrac{\chi+1}{N }\big)^{(3-\frac{d}{2})\boxdot 0} \mathbb{J}_2, 
 	\end{split}
 \end{equation}
 where $\mathbb{J}_1(y,w_2):=\sum\nolimits_{w_1\in \partial \mathcal{B}(d^{-1}N)} \mathbb{P} (  v\xleftrightarrow{} y,v\xleftrightarrow{} w_1 \Vert v'\xleftrightarrow{} w_2 ) $ and 
 \begin{equation*}
 	\mathbb{J}_2:= \mathbb{P}\big(v'\xleftrightarrow{(\{v,y\})} w_2\big) \int_{a_1,b_1>0} e^{-ca_1 \hat{\chi}^{-1}}\big[(a_1+ b_1) N^{2-d} + (a_1b_1)^{\frac{1}{2}}\big(\tfrac{r}{N}\big)^{d-2}  \big] \mathfrak{p}_{v \leftrightarrow y ,a_1,b_1}  \mathrm{d}a_1\mathrm{d}b_1.  
 \end{equation*}
  It follows from (\ref{250}) that $\mathfrak{p}_{v \leftrightarrow y ,a_1,b_1}  \lesssim r^{2-d}e^{-c(a_1+b_1)}$, which implies 
  \begin{equation}
  	\begin{split}
  			\mathbb{J}_2 \lesssim & \hat{\chi}^{\frac{1}{2}}r^{2-d}N^{4-2d} \int_{a_1,b_1>0} e^{-c(a_1 \hat{\chi}^{-1}+b_1)}\big[ a_1+ b_1   + (a_1b_1)^{\frac{1}{2}}r^{d-2}  \big]   \mathrm{d}a_1\mathrm{d}b_1 \\
  			\lesssim &  \hat{\chi}^{\frac{1}{2}}r^{2-d}N^{4-2d}(\hat{\chi}^2+ \hat{\chi}+ \hat{\chi}^{\frac{3}{2}}r^{d-2})\lesssim \hat{\chi}^{\frac{3}{2}}N^{4-2d}. 
  	\end{split}
  \end{equation}
  Combined with (\ref{new.54}) and (\ref{3.12_7.30}), it yields 
  \begin{equation}\label{3.12_7.32}
  	\begin{split}
  &	  \sum\nolimits_{y\in B_{v}(M)}  	\mathbb{P}\big(  v\xleftrightarrow{} y,  \overline{\mathsf{H}}_{v'}^{v}(N) \big) \\
  	   \lesssim   &N^{-(d\boxdot 6)}  \sum\nolimits_{y\in B_{v}(M),w_2\in \partial \mathcal{B}(\frac{N}{2d})} \mathbb{J}_1(y,w_2)  + f(\chi)\cdot \big( \tfrac{M}{N}\big)^d. 
  	\end{split}
  \end{equation}

 In what follows, we prove (\ref{revisenew_78}) for $3\le d\le 5$ and $d\ge 7$ separately.

 \textbf{When $3\le d\le 5$.} By Lemma \ref{lemma_five_point} and its extension in Remark \ref{3.12_remark_4.7},
 \begin{equation}\label{3.12_7.33}
 	\begin{split}
 		& \sum\nolimits_{y\in B_{v}(M),w_2\in \partial \mathcal{B}(\frac{N}{2d})} \mathbb{J}_1(y,w_2) \\
 		\lesssim & \sum\nolimits_{y\in B_{v}(M),w_1\in \partial \mathcal{B}(\frac{N}{d}) ,w_2\in \partial \mathcal{B}(\frac{N}{2d})}  (|v-y|+1)^{-\frac{d}{2}+1}\cdot \mathbb{P}\big( \mathsf{C}[(v,v',w_1,w_2)] \big) \\
 		\lesssim &f(\chi) \cdot N^{\frac{d}{2}-1} \sum\nolimits_{y\in B_{v}(M) }  (|v-y|+1)^{-\frac{d}{2}+1}   \lesssim f(\chi) \cdot  M^{\frac{d}{2}+1}N^{\frac{d}{2}-1}. 
 	\end{split}
 \end{equation}
 Plugging (\ref{3.12_7.33}) into (\ref{3.12_7.32}), we derive (\ref{revisenew_78}): 
 \begin{equation*}
 	\begin{split}
 		 \sum\nolimits_{y\in B_{v}(M)}  	\mathbb{P}\big(  v\xleftrightarrow{} y,  \overline{\mathsf{H}}_{v'}^{v}(N) \big)  \lesssim f(\chi)\cdot \big( \tfrac{M}{N}\big)^{\frac{d}{2}+1} + f(\chi)\cdot \big( \tfrac{M}{N}\big)^d \lesssim f(\chi)\cdot \big( \tfrac{M}{N}\big)^{\frac{d}{2}+1}. 
 	\end{split}
 \end{equation*}

  \textbf{When $d\ge 7$.} By the restriction property, one has 
     \begin{equation}
  	\begin{split}
  		& \sum\nolimits_{y\in B_{v}(M)}  \mathbb{P} (  v\xleftrightarrow{} y,v\xleftrightarrow{} w_1 \Vert v'\xleftrightarrow{} w_2 )  \\
  		= &   \sum\nolimits_{y\in B_{v}(M)} \mathbb{E}\big[ \mathbbm{1}_{  v\xleftrightarrow{} y,v\xleftrightarrow{} w_1,\{v\xleftrightarrow{} v'\}^c} \cdot \mathbb{P}\big(  v'\xleftrightarrow{(\mathcal{C}_v)} w_2 \big) \big] \\ 
  		\overset{v\in \mathcal{C}_v,(\text{FKG})}{\le }  & \mathbb{P}\big(  v'\xleftrightarrow{(v)} w_2 \big) \cdot \mathbb{P}(\{v\xleftrightarrow{} v'\}^c) \cdot \sum\nolimits_{y\in B_{v}(M)}   \mathbb{P} (  v\xleftrightarrow{} y,v\xleftrightarrow{} w_1  ) \\
  		\overset{ (\ref{29}), (\ref{3.12_3.69}), (\ref{3.12_3.39})}{\lesssim } & \hat{\chi}^{\frac{1}{2}} N^{2-d} \cdot \hat{\chi}  \cdot  M^4N^{2-d} = \hat{\chi}^{\frac{3}{2}}M^{4}N^{4-2d}. 
  	\end{split}
  \end{equation}
  Combined with (\ref{3.12_7.32}), it yields (\ref{revisenew_78}):  
  \begin{equation*}
  	 \sum\nolimits_{y\in B_{v}(M)}  	\mathbb{P}\big(  v\xleftrightarrow{} y,  \overline{\mathsf{H}}_{v'}^{v}(N) \big)  \lesssim  f(\chi)\cdot \big( \tfrac{M}{N}\big)^4 + f(\chi)\cdot \big( \tfrac{M}{N}\big)^d \lesssim  f(\chi)\cdot \big( \tfrac{M}{N}\big)^4.
  \end{equation*}

 In conclusion, we obtain the desired bound (\ref{volume_part2}). \qed





    {\color{red}

    }


   {\color{blue}

  }

With (\ref{volume_part1}) and (\ref{volume_part2}) in hand, the proof of Theorem \ref{thm3_volume} is straightforward.

 \begin{proof}[Proof of Theorem \ref{thm3_volume}]
 	By (\ref{volume_part1}) and (\ref{volume_part2}), there exists $C>0$ such that 
 	\begin{equation}
 		\begin{split}
 				\mathbb{P}\big(  \mathcal{V}_v(M), \mathcal{V}_{v'}(M) \in [C^{-1}M^{(\frac{d}{2}+1)\boxdot 4},CM^{(\frac{d}{2}+1)\boxdot 4}]   \mid \overline{\mathsf{H}}_{v'}^{v}(N) \big) \gtrsim 1. 
 		\end{split}
 	\end{equation}
 	This together with the isomorphism theorem implies the desired bound (\ref{ineq_thm3}). 
 \end{proof}

\section*{Acknowledgments} 

We warmly thank Gady Kozma for valuable discussions. We are also very grateful to Wendelin Werner for his insightful comments, which significantly improved the presentation of the manuscript. J. Ding is supported by the National Natural Science Foundation of China (Grant No. 12231002, 12595284, 12595280), and by the New Cornerstone Science Foundation through the New Cornerstone Investigator Program and XPLORER PRIZE.


\appendix

\section{Proof of Lemma \ref{lemme_technical_high_dimension}}\label{appendix2}

 For convenience, we set $0^{-a}:=1$ for all $a>0$.

The tree expansion argument (see \cite[Section 3.4]{cai2024high}) implies that on the event $\{\bm{0}\xleftrightarrow{} x_1,\bm{0}\xleftrightarrow{} x_2\}$, there exists a loop $\widetilde{\ell}_*\in \widetilde{\mathcal{L}}_{1/2}$ and three clusters $\mathcal{C}_1$, $\mathcal{C}_2$ and $\mathcal{C}_3$, each consisting of disjoint collections of loops in $\widetilde{\mathcal{L}}_{1/2}-\mathbbm{1}_{\widetilde{\ell}_*}$, such that the cluster $\mathcal{C}_1$, $\mathcal{C}_2$ and $\mathcal{C}_3$ connect $\mathrm{ran}(\widetilde{\ell}_*)$ to $x_1$, $x_2$ and $\bm{0}$ respectively. As a result, there exist $w_1,w_2,w_3\in \mathbb{Z}^d$ such that $\widetilde{\ell}_*$ intersects $\widetilde{B}_{w_j}(1)$ for $j\in \{1,2,3\}$, and that the events $w_1\xleftrightarrow{} x_1$, $w_2\xleftrightarrow{} x_2$ and $w_3\xleftrightarrow{} \bm{0}$ happen disjointly and without using $\widetilde{\ell}_*$. Recall that the total mass of loops intersecting $\widetilde{B}_{w_j}(1)$ for $j\in \{1,2,3\}$ is at most  $C|w_1-w_2|^{2-d}|w_2-w_3|^{2-d}|w_3-w_1|^{2-d}$ (see e.g.,  \cite[Lemma 2.3]{cai2024high}). In conclusion,  
\begin{equation}\label{B1}
	\begin{split}
		&\sum\nolimits_{x_2\in B_y(M)}\mathbb{P}\big( \bm{0}\xleftrightarrow{} x_1,\bm{0}\xleftrightarrow{} x_2 \big) \\
		\lesssim &\mathbb{J}:= \sum_{x_2\in B_y(M)}\sum_{w_1,w_2,w_3\in \mathbb{Z}^d}|w_1-w_2|^{2-d}|w_2-w_3|^{2-d}\\
		&\ \ \ \  \ \ \ \ \ \ \ \ \cdot |w_3-w_1|^{2-d}|w_1-x_1|^{2-d}|w_2-x_2|^{2-d}|w_3|^{2-d}. 
	\end{split}
\end{equation}
For $i\in \{1,2\}$, we denote by $\mathbb{J}_i$ the partial sum of $\mathbb{J}$ restricted to $w_i\in [B_{x_i}(\frac{N}{10})]^c$. In addition, let $\mathbb{J}_3$ denote the partial sum of $\mathbb{J}$ restricted to $w_3\in [B(\frac{N}{10})]^c$, and let $\mathbb{J}_4:=\mathbb{J}-\sum_{i\in \{1,2,3\}}\mathbb{J}_i$.

 For $i\in \{1,2\}$, since $|w_i-x_i|\gtrsim N$, we have 
 \begin{equation*}
 	\begin{split}
 		\mathbb{J}_i \lesssim & N^{2-d} \sum_{x_2\in B_y(M)}\sum_{w_1,w_2,w_3\in \mathbb{Z}^d}|w_1-w_2|^{2-d}|w_2-w_3|^{2-d}\\
		&\ \ \ \ \ \ \ \ \ \ \ \ \  \ \ \ \ \ \ \ \ \ \ \ \ \ \ \ \ \  \cdot |w_3-w_1|^{2-d}|w_{3-i}-x_{3-i}|^{2-d}|w_3|^{2-d}\\
		\overset{(\ref{computation_2-d_2-d})}{\lesssim } & N^{2-d}  \sum_{x_2\in B_y(M)} |x_{3-i}|^{4-d} \overset{(|x_{3-i}|\gtrsim N)}{\lesssim } N^{6-2d}M^d, 
 	\end{split}
 \end{equation*}
 where in the second inequality we first summed over $w_i$, then over $w_{3-i}$, and finally over $w_3$; in addition, when summing over $w_{3-i}$, we used the following fact: 
 \begin{equation}\label{computation_6-2d_2-d}
 	\sum\nolimits_{x\in \mathbb{Z}^d} |v-x|^{6-2d}|x-w|^{2-d}\lesssim |v-w|^{2-d}. 
 \end{equation}

For $\mathbb{J}_3$, since $|w_3|\gtrsim N$, one has 
\begin{equation*}
	\begin{split}
		\mathbb{J}_3 \lesssim &N^{2-d} \sum_{x_2\in B_y(M)}\sum_{w_1,w_2,w_3\in \mathbb{Z}^d}|w_1-w_2|^{2-d}|w_2-w_3|^{2-d}\\
		&\ \ \ \ \ \ \ \ \ \ \ \ \  \ \ \ \ \ \ \ \ \ \ \ \ \ \ \ \ \  \cdot |w_3-w_1|^{2-d}|w_{1}-x_{1}|^{2-d}|w_{2}-x_{2}|^{2-d}\\
		\overset{(\ref{computation_2-d_2-d}), (\text{\ref{computation_6-2d_2-d}})}{\lesssim } & N^{2-d}  \sum_{x_2\in B_y(M)} |x_1-x_2|^{4-d} \overset{(\ref{computation_d-a})}{\lesssim } N^{2-d}M^4. 
	\end{split}
\end{equation*}
Here in the second inequality we first summed over $w_3$, and then over $w_{1}$ and $w_2$.

Now we estimate $\mathbb{J}_4$. For each $i\in \{1,2\}$, note that $w_i\in B_{x_i}(\frac{N}{10})$ and $w_3\in B(\frac{N}{10})$ imply $|w_i-w_3|\gtrsim N$. Therefore, we have 
\begin{equation*}
	\begin{split}
		\mathbb{J}_4 \lesssim & N^{4-2d} \sum_{x_2\in B_y(M)}\sum_{w_1,w_2,w_3\in \mathbb{Z}^d}|w_1-w_2|^{2-d}|w_{1}-x_{1}|^{2-d}|w_{2}-x_{2}|^{2-d}|w_3|^{2-d} \\
	 \overset{(\ref{computation_d-a}) }{\lesssim } & N^{6-2d}\sum_{x_2\in B_y(M)}\sum_{w_1,w_2\in \mathbb{Z}^d}|w_1-w_2|^{2-d}|w_{1}-x_{1}|^{2-d}|w_{2}-x_{2}|^{2-d} \\
	\overset{(\ref{computation_2-d_2-d}),(\ref{computation_4-d_2-d}) }{\lesssim }& N^{6-2d} \sum_{x_2\in B_y(M)} |x_1-x_2|^{6-d}  \overset{(\ref{computation_d-a}) }{\lesssim } N^{6-2d}M^6. 
	\end{split}
\end{equation*}
Plugging these estimates for $\{\mathbb{J}_i\}_{1\le i\le 4}$ into (\ref{B1}), we obtain the lemma: 
\begin{equation}
	\begin{split}
		&\sum\nolimits_{x_2\in B_y(M)}\mathbb{P}\big( \bm{0}\xleftrightarrow{} x_1,\bm{0}\xleftrightarrow{} x_2 \big) \\
		 \lesssim & N^{6-2d}M^d + N^{2-d}M^4+ N^{6-2d}M^6 \lesssim N^{2-d}M^4.  
\pushQED{\qed} 
 \qedhere
\popQED
	\end{split}
\end{equation}


\section{Proof of Lemma \ref{lemma_hit_capacity}}\label{appendix3}

We follow the proof strategy of \cite[Theorem 1.2]{cai2024incipient} for $d\ge 7$ (see \cite[Section 5.1]{cai2024incipient}). Recall that $y\in \partial B(N)$. For any $m\in [1,M]$, we consider the partial cluster 
\begin{equation}
	\widehat{\mathcal{C}}_m^+:= \big\{ v\in \widetilde{B}_y(m) : v\xleftrightarrow{\widetilde{E}^{\ge 0}\cap \widetilde{B}_y(m) } y \big\},  
\end{equation}
 and we define the harmonic average $\widehat{\mathcal{H}}_{m,v}^+:=\mathcal{H}_v(\widehat{\mathcal{C}}_m^+)$ for $v\in \widetilde{\mathbb{Z}}^d$ (recall $\mathcal{H}_v(\cdot)$ below (\ref{revise_new_334})).  According to \cite[Lemma 2.1]{cai2024incipient}, one has: for any $w\in [B_y(2dm)]^c$, 
\begin{equation}\label{C2}
\begin{split}
		\mathbb{P}\big( z\xleftrightarrow{\ge 0} y  \mid \mathcal{F}_{\widehat{\mathcal{C}}_m^+} \big)\asymp &\big( \tfrac{m}{|w-y|} \big)^{2-d}\cdot |\partial \mathcal{B}_y(dm)|^{-1}\sum\nolimits_{z\in \partial \mathcal{B}_y(dm) } \widehat{\mathcal{H}}_{m,z}^+\\
		\overset{(\ref{310})}{\asymp} & \big( \tfrac{N}{ |w-y|} \big)^{d-2} \cdot \widehat{\mathcal{H}}_{m,\bm{0}}^+. 
\end{split}
\end{equation}

 The following lemma, analogous to \cite[Lemma 5.3]{cai2024incipient}, shows that conditioning on $\{\bm{0}\xleftrightarrow{\ge 0} y\}$, the aforementioned harmonic average is typically of order $M^2N^{2-d}$.

\begin{lemma}\label{lemma_C1}
	For any $d\ge 7$ and $\lambda>2$, there exist $\Cl\label{const_prepare_highd1}(d,\lambda),\Cl\label{const_prepare_highd2}(d)>0$ such that for any $m\ge \Cref{const_prepare_highd1}$, $N\ge 10dm$ and $y\in \partial B(N)$, 
	\begin{equation}\label{C3}
		\mathbb{P}\big( \widehat{\mathcal{H}}_{m,\bm{0}}^+ \in \big[\lambda^{-1}m^2N^{2-d},\lambda m^2N^{2-d} \big]  \mid \bm{0}\xleftrightarrow{\ge 0} y \big)\le  \Cref{const_prepare_highd2} \lambda^{-1}. 
	\end{equation}
\end{lemma}
\begin{proof}
 On the one hand, it follows from (\ref{C2}) that 
\begin{equation}\label{C4}
	\begin{split}
		&\mathbb{P}\big(0< \widehat{\mathcal{H}}_{m,\bm{0}}^+ \le \lambda^{-1}m^2N^{2-d} , \bm{0}\xleftrightarrow{\ge 0} y \big) \\
		=&\mathbb{E}\big[\mathbbm{1}_{0< \widehat{\mathcal{H}}_{m,\bm{0}}^+ \le \lambda^{-1}m^2N^{2-d}}\cdot \mathbb{P}\big(\bm{0}\xleftrightarrow{\ge 0} y \mid \mathcal{F}_{\widehat{\mathcal{C}}_m^+} \big) \big]\\
		\lesssim & \lambda^{-1}m^2N^{2-d}  \cdot \mathbb{P}\big(0< \widehat{\mathcal{H}}_{m,\bm{0}}^+ \le \lambda^{-1}m^2N^{2-d} \big)\lesssim  \lambda^{-1}N^{2-d}, 
			\end{split}
\end{equation}
Here in the last inequality we used the fact that $\{\widehat{\mathcal{H}}_{m,\bm{0}}^+>0 \}= \{y\xleftrightarrow{\ge 0} \partial B_y(m) \}$, whose probability is  of order $m^{-2}$ (by (\ref{one_arm_high})).

On the other hand, by Lemma \ref{lemme_technical_high_dimension} one has 
\begin{equation}\label{C5}
\mathbb{I}:=	\sum\nolimits_{x\in B_y(40dm)\setminus B_y(10dm)} \mathbb{P}\big( \bm{0}\xleftrightarrow{\ge 0} x, \bm{0}\xleftrightarrow{\ge 0} y \big)\lesssim N^{2-d}m^4. 
\end{equation}
Meanwhile, using the FKG inequality and (\ref{C2}), we have   
\begin{equation}
	\begin{split}
		\mathbbm{I}\ge  &	\sum_{x\in B_y(40dm)\setminus B_y(10dm)} \mathbb{E}\big[ \mathbbm{1}_{ \widehat{\mathcal{H}}_{m,\bm{0}}^+ \ge \lambda m^2N^{2-d}}  \mathbb{P}\big(\bm{0}\xleftrightarrow{\ge 0} x,\bm{0}\xleftrightarrow{\ge 0} y \mid \mathcal{F}_{\widehat{\mathcal{C}}_m^+} \big) \big]\\
		\overset{}{\ge} &\sum _{x\in B_y(40dm)\setminus B_y(10dm)} \mathbb{E}\big[ \mathbbm{1}_{ \widehat{\mathcal{H}}_{m,\bm{0}}^+ \ge \lambda m^2N^{2-d}} \mathbb{P}\big(\bm{0}\xleftrightarrow{\ge 0} x  \mid \mathcal{F}_{\widehat{\mathcal{C}}_m^+} \big)   \mathbb{P}\big( \bm{0}\xleftrightarrow{\ge 0} y \mid \mathcal{F}_{\widehat{\mathcal{C}}_m^+} \big) \big]\\
		\gtrsim  &  \sum _{x\in B_y(40dm)\setminus B_y(10dm)}  \lambda m^{4-d}\cdot \mathbb{P}\big(\widehat{\mathcal{H}}_{m,\bm{0}}^+ \ge \lambda m^2N^{2-d},   \bm{0}\xleftrightarrow{\ge 0} y\big)\\
		\asymp &    \lambda m^{4} \cdot \mathbb{P}\big(\widehat{\mathcal{H}}_{m,\bm{0}}^+ \ge \lambda m^2N^{2-d},   \bm{0}\xleftrightarrow{\ge 0} y\big). 
	\end{split}
\end{equation}
Combined with (\ref{C5}), it implies that 
\begin{equation}\label{C7}
	 \mathbb{P}\big(\widehat{\mathcal{H}}_{m,\bm{0}}^+ \ge \lambda m^2N^{2-d},   \bm{0}\xleftrightarrow{\ge 0} y\big) \lesssim \lambda^{-1}N^{2-d}. 
\end{equation}

By (\ref{C4}) and (\ref{C7}), we obtain the desired bound (\ref{C3}). 
\end{proof}

For any $\epsilon>0$, we denote $C_\star(d,\epsilon):=4\Cref{const_prepare_highd2}\epsilon^{-1}$. For any $M\ge 1$, let $\hat{M}:=\frac{M}{C_\dagger C_\star}$, where $C_\dagger>0$ is a large constant to be determined later. For any $m\ge 1$, we define the event $\mathsf{G}^\star_m:= \{ C_{\star}^{-1}m^{2}N^{2-d} \le \widehat{\mathcal{H}}_{m,\bm{0}}^+ \le  C_{\star}m^{2}N^{2-d}\}$. Lemma \ref{lemma_C1} implies that 
\begin{equation}
	\mathbb{P}\big( (\mathsf{G}^\star_M)^c\cup (  \mathsf{G}^\star_{\hat{M}})^c  \mid \bm{0} \xleftrightarrow{\ge 0} y \big)\le \tfrac{\epsilon}{2}. 
	\end{equation}
Thus, to verify Lemma \ref{lemma_hit_capacity}, it remains to show that for some $c_{\clubsuit}(d,\epsilon)>0$,  
\begin{equation}\label{C9}
 \mathbb{P}\big( \mathcal{V}^+_y(M)\le c_{\clubsuit}M^4, \mathsf{G}^\star_M, \mathsf{G}^\star_{\hat{M}} \mid \bm{0} \xleftrightarrow{\ge 0} y  \big) \le \tfrac{\epsilon}{2}. 
\end{equation}

To achieve (\ref{C9}), we employ the exploration martingale argument (see \cite{lupu2018random,ding2020percolation}). For any $t>0$, we define $\mathcal{I}_t:= \{v\in \widetilde{B}_y(M): \mathrm{dist}(v,  \widetilde{B}_y(\hat{M}))\le t\}$ and  
\begin{equation}
	\widetilde{\mathcal{C}}_t:= \big\{ v\in \mathcal{I}_t: y\xleftrightarrow{\widetilde{E}^{\ge 0}\cap \mathcal{I}_t} v \big\}. 
\end{equation}
Note that $\widetilde{\mathcal{C}}_0=\widehat{\mathcal{C}}_{\hat{M}}^+$ and $\widetilde{\mathcal{C}}_\infty=\widehat{\mathcal{C}}_M^+$. Since the $\sigma$-field $\mathcal{F}_{\widetilde{\mathcal{C}}_t}$ is increasing in $t$, the process $\mathcal{M}_t:= \mathbb{E}\big[\widetilde{\phi}_{\bm{0}}\mid  \mathcal{F}_{\widetilde{\mathcal{C}}_t}\big]$ for $t\ge 0$ is a martingale. In particular, $\mathcal{M}_0=\widehat{\mathcal{H}}_{\hat{M},\bm{0}}^+$ and $\mathcal{M}_\infty=\widehat{\mathcal{H}}_{M,\bm{0}}^+$. Let $\langle M \rangle_t$ denote the quadratic variation of $\mathcal{M}_t$. According to \cite[Corollary 10]{ding2020percolation}, $\langle M \rangle_t$ can be equivalently written as 
\begin{equation}
	\begin{split}
		\widetilde{G}_{\widetilde{\mathcal{C}}_0}(\bm{0},\bm{0}) - \widetilde{G}_{\widetilde{\mathcal{C}}_t}(\bm{0},\bm{0}) \overset{(\ref{formula_two_green})}{=}&  \sum\nolimits_{v\in \widetilde{\partial }\widetilde{\mathcal{C}}_t} \widetilde{\mathbb{P}}_{\bm{0}}\big(\tau_{\widetilde{\mathcal{C}}_t}  =\tau_{v}< \tau_{\widetilde{\mathcal{C}}_0}\big)\cdot \widetilde{G}_{\widetilde{\mathcal{C}}_0}(v,\bm{0})\\
		\overset{(\widetilde{\mathcal{C}}_t\subset \widehat{\mathcal{C}}_{M}^+ \subset \widetilde{B}_y(M))}{\lesssim} & N^{2-d}\cdot \widetilde{\mathbb{P}}_{\bm{0}}\big(\tau_{\widehat{\mathcal{C}}_{M}^+}  <\infty\big) \\
		\overset{(\ref{310})}{\asymp } &  N^{4-2d}\cdot \mathrm{cap}(\widehat{\mathcal{C}}_{M}^+)  \lesssim  N^{4-2d}\cdot \mathcal{V}^+_y(M), 
	 \end{split}
\end{equation}
  where the last inequality follows from the sub-additivity of the capacity. Therefore, there exists $C_\dagger>0$ such that 
  \begin{equation}\label{C12}
  	\big\{ \mathcal{V}^+_y(M)\le c_{\clubsuit}M^4  \big\} \subset \big\{\langle M \rangle_\infty\le  C_\dagger c_{\clubsuit}M^4N^{4-2d} \big\}.  
  \end{equation}
  By the martingale representation theorem (see e.g. \cite[Theorem 11]{ding2020percolation}), under certain time change, $\mathcal{M}_t$ behaves as a standard Brownian motion with duration $\langle M \rangle_\infty$. Meanwhile, on the $\mathsf{G}^\star_M \cap  \mathsf{G}^\star_{\hat{M}}$, one has  
  \begin{equation}
  \begin{split}
  	  	\mathcal{M}_\infty - \mathcal{M}_0=& \widehat{\mathcal{H}}_{M,\bm{0}}^+-\widehat{\mathcal{H}}_{\hat{M},\bm{0}}^+ \\
  	  	\ge &C_\star^{-1}M^2N^{2-d}-C_\star\hat{M}^2 N^{2-d}  \ge \tfrac{1}{2}C_\star^{-1}M^2N^{2-d},
  \end{split}
  \end{equation}
  implying that the aforementioned Brownian motion must reach $\tfrac{1}{2}C_\star^{-1}M^2N^{2-d}$ before time $\langle M \rangle_\infty$. As shown in \cite[(5.66)]{cai2024incipient}, this observation yields that for a sufficiently small $c_{\clubsuit}>0$, on the event $\mathsf{G}^\star_{\hat{M}}$, 
  \begin{equation}
  	\begin{split}
  		\mathbb{P}\big( \langle M \rangle_\infty\le  C_\dagger c_{\clubsuit}M^4N^{4-2d}, \mathsf{G}^\star_{M} \mid \mathcal{F}_{\widehat{\mathcal{C}}_{\hat{M}}^+} \big) \le e^{-\epsilon^{-1}}.
  		  	\end{split}
  \end{equation}
 By taking the integral of both sides, and using $\mathsf{G}^\star_{\hat{M}}\subset \{y\xleftrightarrow{} \partial B_y(\hat{M})\} $, we get 
 \begin{equation}\label{C15}
 	\begin{split}
 	  \mathbb{P}\big( \langle M \rangle_\infty\le  C_\dagger c_{\clubsuit}M^4N^{4-2d}, \mathsf{G}^\star_{M}, \mathsf{G}^\star_{\hat{M}} \big)  
 	\le e^{-\epsilon^{-1}} \cdot \mathbb{P}\big(\mathsf{G}^\star_{\hat{M}} \big) \lesssim  e^{-\epsilon^{-1}}\hat{M}^{-2}. 
 	\end{split}
 \end{equation}
  Moreover, since the events $\{\langle M \rangle_\infty\le  C_\dagger c_{\clubsuit}M^4N^{4-2d} \}$, $\mathsf{G}^\star_M$ and $\mathsf{G}^\star_{\hat{M}}$ are all measurable with respect to $\mathcal{F}_{\widehat{\mathcal{C}}_M^+}$, it follows from (\ref{C2}) and (\ref{C15}) that 
  \begin{equation}
  \begin{split}
  	  	&\mathbb{P}\big(\langle M \rangle_\infty\le  C_\dagger c_{\clubsuit}M^4N^{4-2d}, \mathsf{G}^\star_M, \mathsf{G}^\star_{\hat{M}}, \bm{0}\xleftrightarrow{\ge 0} y \big)  \\
  	  	=& \mathbb{E}\big[ \mathbbm{1}_{\langle M \rangle_\infty\le  C_\dagger c_{\clubsuit}M^4N^{4-2d}, \mathsf{G}^\star_M, \mathsf{G}^\star_{\hat{M}}} \cdot \mathbb{P}\big( \bm{0}\xleftrightarrow{\ge 0} y   \mid \mathcal{F}_{\widehat{\mathcal{C}}_M^+}\big)  \big]\\
  	  	\overset{(\text{\ref{C2}})}{\lesssim}  &  \mathbb{P}\big(\langle M \rangle_\infty\le  C_\dagger c_{\clubsuit}M^4N^{4-2d}, \mathsf{G}^\star_M, \mathsf{G}^\star_{\hat{M}} \big) \cdot C_\star M^2N^{2-d}\\
  	  	\overset{(\text{\ref{C15}})}{\lesssim}  & e^{-\epsilon^{-1}}\hat{M}^{-2} \cdot C_\star M^2N^{2-d}. 
  \end{split}
  \end{equation} Combined with (\ref{C12}), it yields that
    \begin{equation}
  \begin{split}
  	    	\mathbb{P}\big( \mathcal{V}^+_y(M)\le c_{\clubsuit}M^4, \mathsf{G}^\star_M, \mathsf{G}^\star_{\hat{M}} \mid \bm{0} \xleftrightarrow{\ge 0} y  \big) \lesssim  \epsilon^{-3}e^{-\epsilon^{-1}}.
  \end{split}
  \end{equation}
 This immediately implies that (\ref{C9}) holds for all sufficiently small $\epsilon>0$, thereby completing the proof of Lemma \ref{lemma_hit_capacity}. \qed

	\bibliographystyle{plain}
	\bibliography{ref}

@article{lupu2016loop,
  title={From loop clusters and random interlacements to the free field},
  author={Lupu, T.},
  journal={Annals of Probability},
  volume={44},
  number={3},
  pages={2117--2146},
  year={2016}
}

@article{lupu2018random,
	title={{The random pseudo-metric on a graph defined via the zero-set of the Gaussian free field on its metric graph}},
	author={Lupu, T. and Werner, W.},
	journal={Probability Theory and Related Fields},
	volume={171},
	pages={775--818},
	year={2018},
	publisher={Springer}
}

@article{cai2024one,
	title={{One-arm probabilities for metric graph Gaussian free fields below and at the critical dimension}},
	author={Cai, Z. and Ding, J.},
	journal={arXiv preprint arXiv:2406.02397},
	year={2024}
}

@article{cai2024high,
  title={{One-arm exponent of critical level-set for metric graph Gaussian free field in high dimensions}},
  author={Cai, Z. and Ding, J.},
  journal={Probability Theory and Related Fields},
  pages={1--86},
  year={2024},
  publisher={Springer}
}

@article{drewitz2023arm,
  title={{Arm exponent for the Gaussian free field on metric graphs in intermediate dimensions}},
  author={Drewitz, A. and Pr{\'e}vost, A. and Rodriguez, P.-F.},
  journal={arXiv preprint arXiv:2312.10030},
  year={2023}
}

@article{drewitz2024critical,
  title={{Critical one-arm probability for the metric Gaussian free field in low dimensions}},
  author={Drewitz, A. and Pr{\'e}vost, A. and Rodriguez, P.-F.},
  journal={Probability Theory and Related Fields},
  pages={1--24},
  year={2025},
  publisher={Springer}
}

@article{cai2024incipient,
  title={{Incipient infinite clusters and self-similarity for metric graph Gaussian free fields and loop soups}},
  author={Cai, Z. and Ding, J.},
  journal={arXiv preprint arXiv:2412.05709},
  year={2024}
}

@article{drewitz2024cluster,
  title={{Cluster volumes for the Gaussian free field on metric graphs}},
  author={Drewitz, A. and Pr{\'e}vost, A. and Rodriguez, P.-F.},
  journal={arXiv preprint arXiv:2412.06772},
  year={2024}
}

@article{werner2021clusters,
  title={On clusters of {B}rownian loops in $d$ dimensions},
  author={Werner, W.},
  journal={In and Out of Equilibrium 3: Celebrating Vladas Sidoravicius},
  pages={797--817},
  year={2021},
  publisher={Springer}
}

@article{ganguly2024ant,
  title={The ant on loops: {Alexander-Orbach} conjecture for the critical level set of the {Gaussian} free field},
  author={Ganguly, S. and Nam, K.},
  journal={arXiv preprint arXiv:2403.02318},
  year={2024}
}

@article{cai2024quasi,
  title={{Quasi-multiplicativity and regularity for critical metric graph Gaussian free fields}},
  author={Cai, Z. and Ding, J.},
  journal={arXiv preprint arXiv:2412.05706},
  year={2024}
}

@book{morters2010brownian,
  title={Brownian motion},
  author={M{\"o}rters, P. and Peres, Y.},
  volume={30},
  year={2010},
  publisher={Cambridge University Press}
}

@article{ganguly2024critical,
  title={Critical level set percolation for the {GFF} in $d> 6$: comparison principles and some consequences},
  author={Ganguly, S. and Jing, K.},
  journal={arXiv preprint arXiv:2412.17768},
  year={2024}
}

@article{ding2020percolation,
	title={Percolation for level-sets of {Gaussian} free fields on metric graphs},
	author={Ding, J. and Wirth, M.},
	journal={The Annals of Probability},
	volume={48},
	number={3},
	pages={1411--1435},
	year={2020}
}

@article{drewitz2023critical,
  title={Critical exponents for a percolation model on transient graphs},
  author={Drewitz, A. and Pr{\'e}vost, A. and Rodriguez, P.-F.},
  journal={Inventiones mathematicae},
  volume={232},
  number={1},
  pages={229--299},
  year={2023},
  publisher={Springer}
}

@article{werner2025switching,
  title={A switching identity for cable-graph loop soups and {Gaussian} free fields},
  author={Werner, W.},
  journal={arXiv preprint arXiv:2502.06754},
  year={2025}
}

@article{inpreparation,
  title={{On the gap between cluster dimensions of loop soups on $\mathbb{R}^3$ and the metric graph of $\mathbb{Z}^3$}},
  author={Cai, Z. and Ding, J.},
  	journal={arXiv preprint arXiv:2510.20526}, 
  year={2025}
}

@article{inpreparation_second_monent,
  title={{Separation and cut edge in macroscopic clusters for metric graph Gaussian free fields}},
  author={Cai, Z. and Ding, J.},
  	journal={arXiv preprint arXiv:2510.20516}, 
  year={2025}
}

@article{kesten1986incipient,
	title={The incipient infinite cluster in two-dimensional percolation},
	author={Kesten, H.},
	journal={Probability theory and related fields},
	volume={73},
	pages={369--394},
	year={1986},
	publisher={Springer}
}

@article{sheffield2012conformal,
  title={{Conformal loop ensembles: the Markovian characterization and the loop-soup construction}},
  author={Sheffield, S. and Werner, W.},
  journal={Annals of Mathematics},
  pages={1827--1917},
  year={2012},
  publisher={JSTOR}
}

@incollection{werner2016spatial,
  title={{On the spatial Markov property of soups of unoriented and oriented loops}},
  author={Werner, W.},
  booktitle={S{\'e}minaire de Probabilit{\'e}s XLVIII},
  pages={481--503},
  year={2016},
  publisher={Springer}
}

@book{revuz2013continuous,
  title={Continuous martingales and Brownian motion},
  author={Revuz, D. and Yor, M.},
  volume={293},
  year={2013},
  publisher={Springer Science \& Business Media}
}

@book{blumenthal2012excursions,
  title={Excursions of Markov processes},
  author={Blumenthal, R. M.},
  year={2012},
  publisher={Springer Science \& Business Media}
}

@article{kesten1987scaling,
  title={Scaling relations for $2${D}-percolation},
  author={Kesten, H.},
  journal={Communications in Mathematical Physics},
  volume={109},
  pages={109--156},
  year={1987},
  publisher={Springer}
}

@article{smirnov2001critical,
  title={Critical exponents for two-dimensional percolation},
  author={Smirnov, S. and Werner, W.},
  journal={Mathematical Research Letters},
  volume={8},
  number={6},
  pages={729--744},
  year={2001},
  publisher={International Press of Boston}
}

@article{nolin2008near,
  title={Near-critical percolation in two dimensions},
  author={Nolin, P.},
  journal={Electronic Journal of Probability},
  volume={13},
  pages={1562--1623},
  year={2008},
  publisher={Institute of Mathematical Statistics}
}

@article{schramm2010quantitative,
  title={Quantitative noise sensitivity and exceptional times for percolation},
  author={Schramm, O. and Steif, J. E.},
  journal={Annals of Mathematics},
  pages={619--672},
  year={2010},
  publisher={JSTOR}
}

@article{garban2013pivotal,
  title={Pivotal, cluster, and interface measures for critical planar percolation},
  author={Garban, C. and Pete, G. and Schramm, O.},
  journal={Journal of the American Mathematical Society},
  volume={26},
  number={4},
  pages={939--1024},
  year={2013}
}

@article{du2024sharp,
  title={Sharp asymptotics for arm probabilities in critical planar percolation},
  author={Du, H. and Gao, Y. and Li, X. and Zhuang, Z.},
  journal={Communications in Mathematical Physics},
  volume={405},
  number={8},
  pages={182},
  year={2024},
  publisher={Springer}
}

@article{Lawler1998,
author = {Lawler, G. F.},
journal = {Mathematical Physics Electronic Journal},
language = {eng},
pages = {1--67},
publisher = {University of Barcelona},
title = {Strict concavity of the intersection exponent for {Brownian} motion in two and three dimensions.},
url = {http://eudml.org/doc/225103},
volume = {4},
year = {1998},
}

@incollection{lawler2002sharp,
  title={Sharp estimates for {Brownian} non-intersection probabilities},
  author={Lawler, G. F. and Schramm, O. and Werner, W.},
  booktitle={In and Out of Equilibrium: Probability with a Physics Flavor},
  pages={113--131},
  publisher={Springer}
}

@article{lawler2012fast,
  title={Fast convergence to an invariant measure for non-intersecting 3-dimensional {Brownian} paths},
  author={Lawler, G. F. and Vermesi, B.},
  journal={ALEA},
  volume={9},
  number={2},
  pages={717--738},
  year={2012}
}

@article{lawler1996hausdorff,
  title={Hausdorff Dimension of Cut Points for {Brownian} Motion},
  author={Lawler, G. F.},
  journal={Electronic Journal of Probability},
  volume={1},
  pages={1--20},
  year={1996},
  publisher={Institute of Mathematical Statistics}
}

@article{masson2009growth,
  title={The growth exponent for planar loop-erased random walk.},
  author={Masson, R.},
  journal={Electronic Communications in Probability},
  volume={14},
  pages={1012--1073},
  year={2009}
}

@article{lawler2021convergence,
  title={Convergence of loop-erased random walk in the natural parameterization},
  author={Lawler, G. F. and Viklund, F.},
  journal={Duke Mathematical Journal},
  volume={170},
  number={10},
  pages={2289--2370},
  year={2021},
  publisher={Duke University Press}
}

@article{4966635c-2c6a-31bf-9599-a4e32c6a97f5,
 ISSN = {00911798, 2168894X},
 URL = {https://www.jstor.org/stable/26402367},
 author = {Daisuke S.},
 journal = {The Annals of Probability},
 number = {2},
 pages = {687--774},
 publisher = {Institute of Mathematical Statistics},
 title = {GROWTH EXPONENT FOR LOOP-ERASED RANDOM WALK IN THREE DIMENSIONS},
 urldate = {2025-05-19},
 volume = {46},
 year = {2018}
}

@article{gao2024percolation,
  title={{Percolation of discrete GFF in dimension two I. Arm events in the random walk loop soup}},
  author={Gao, Y. and Nolin, P. and Qian, W.},
  journal={arXiv preprint arXiv:2409.16230},
  year={2024}
}

@article{gao2024percolation2,
  title={{Percolation of discrete GFF in dimension two II. Connectivity properties of two-sided level sets}},
  author={Gao, Y. and Nolin, P. and Qian, W.},
  journal={arXiv preprint arXiv:2409.16273},
  year={2024}
}

@book{lawler2010random,
  title={Random walk: a modern introduction},
  author={Lawler, G. F. and Limic, V.},
  volume={123},
  year={2010},
  publisher={Cambridge University Press}
}

@book{lyons2017probability,
  title={Probability on trees and networks},
  author={Lyons, R. and Peres, Y.},
  volume={42},
  year={2017},
  publisher={Cambridge University Press}
}

@article{lupu2019convergence,
  title={Convergence of the two-dimensional random walk loop soup clusters to {CLE}},
  author={Lupu, T.},
  journal={J. Eur. Math. Soc},
  volume={21},
  number={4},
  pages={1201--1227},
  year={2019}
}

@article{leone1961folded,
  title={The folded normal distribution},
  author={Leone, F. C. and Nelson, L. S. and Nottingham, R. B. },
  journal={Technometrics},
  volume={3},
  number={4},
  pages={543--550},
  year={1961},
  publisher={Taylor \& Francis}
}

@book{whittaker1924calculus,
  title={The calculus of observations: a treatise on numerical mathematics},
  author={Whittaker, E. T. and Robinson, G.},
  year={1924},
  publisher={Blackie and Son limited}
}
	
\end{document}